\documentclass[12pt]{article}
\usepackage[paperwidth=230mm, paperheight=310mm]{geometry}
\usepackage{graphicx}
\usepackage{amssymb}
\usepackage{mathtools}
\usepackage{amsmath}
\usepackage{hyperref}
\usepackage{enumerate}
\usepackage{amsthm}
\usepackage{amsmath}
\usepackage{float}
\usepackage{color}   
\usepackage{lineno}

\makeatletter
\renewenvironment{proof}[1][\proofname]{%
   \par\pushQED{\qed}\normalfont%
   \topsep6\p@\@plus6\p@\relax
   \trivlist\item[\hskip\labelsep\bfseries#1\@addpunct{.}]%
   \ignorespaces
}{%
   \popQED\endtrivlist\@endpefalse
}
\makeatother

\newtheorem{theorem}{Theorem}
\newtheorem{proposition}[theorem]{Proposition}
 \numberwithin{theorem}{section}
 \newtheorem{corollary}[theorem]{Corollary}
\newtheorem{lemma}[theorem]{Lemma}
\newtheorem{mydef}[theorem]{Definition}

\numberwithin{equation}{section}

\renewcommand{\P}{\mathbb{P}}
\newcommand{\E}{\mathbb{E}}
\newcommand{\R}{\mathbb{R}}
\newcommand{\mR}{\mathcal{R}}

\newcommand{\Z}{\mathbb{Z}}

\newcommand{\N}{\mathbb{N}}

\newcommand{\cA}{\mathcal A}
\newcommand{\cC}{\mathcal C}

\newcommand{\cN}{\mathcal{N}}

\newcommand{\cH}{\mathcal{H}}

\newcommand{\cL}{\mathcal{L}}

\newcommand{\cF}{\mathcal F}

\newcommand{\cU}{\mathcal U}

\newcommand{\cG}{\mathcal G}

\newcommand{\eps}{\varepsilon}

 \newcommand{\nn}{\nonumber}
 \newcommand{\no}{\noindent}

\begin{document}
\author{
Jieliang Hong\footnote{Department of Mathematics, University of British Columbia, Canada, E-mail: {\tt jlhong@math.ubc.ca} }
}
\title{An upper bound for $p_c$ in range-$R$ bond percolation in two and three dimensions}
\date{\today}
\maketitle
\begin{abstract}
An upper bound for the critical probability of long range bond percolation in $d=2$ and $d=3$ is obtained by connecting the bond percolation with the SIR epidemic model, thus complementing the lower bound result in Frei and Perkins \cite{FP16}. A key ingredient is that we establish a uniform bound for the local times of branching random walk by calculating their exponential moments and by using the discrete versions of Tanaka's formula and Garsia's Lemma.
\end{abstract}
\section{Introduction} \label{4s1}

\subsection{Range-$R$ bond percolation and the main result}

For any $R\in \N$, we set $\Z_R^d=\Z^d/R=\{x/R: x\in \Z^d\}$. Let $x, y\in \Z^d_R$ be neighbours if $0<\|x-y\|_\infty\leq 1$ where $\|\cdot \|_\infty$ denotes the $l^\infty$ norm on $\R^d$ and we write $x\sim y$ if $x, y\in \Z^d_R$ are neighbours. Let $\cN(x)$ denote the set of neighbours of $x$ and denote its size by 
\[V(R):=|\cN(x)|=|\{y\in \Z^d_R: 0<\|y-x\|_\infty\leq 1\}|=(2R+1)^d-1,\] where $|S|$ is the cardinality of a finite set $S$. If $x\sim y$ in $\Z^d_R$, we let $(x,y)$ or $(y,x)$ denote the edge between $x$ and $y$ and let $E(\Z_R^d)$ be the set of all the edges in $\Z_R^d$. Assign a collection of i.i.d. Bernoulli random variable $\{B(e): e\in E(\Z_R^d)\}$ with parameter $p>0$ to the edges. If $B(e)=1$, we say the edge $e$ is open; if $B(e)=0$, we say the edge $e$ is closed. 
Denote by $G=G_R$ the resulting subgraph with vertex set $\Z_R^d$ and edge set being the set of  open edges. For any $x,y\in \Z_R^d$, we write $x\leftrightarrow y$ if $x=y$ or there is a path between $x$ and $y$ consisting of open edges. Denote the cluster $\cC_x$ in $G$ containing $x$ by 
\[
\cC_x:=\{y\in \Z_R^d: x\leftrightarrow y\}.
\]
Define the percolation probability $q(p)$ to be
\[
q(p)=\P_p(|\cC_0|=\infty).
\]
The critical probability is then defined by
\[
p_c=p_c(R)=\inf\{p: q(p)>0\}.
\]
One can check by monotonicity in $p$ that $q(p)=0$ for $p\in [0,p_c)$ and $q(p)>0$ for $p\in (p_c,1]$. Write $f(R)\sim g(R)$ as $R\to \infty$ iff $f(R)/g(R) \to 1$ as $R\to \infty$. It is shown in M. Penrose \cite{Pen93} that 
\[
p_c(R) \sim \frac{1}{V(R)} \text{ as } R\to \infty.
\]
In higher dimensions $d> 6$, Van der Hofstad and Sakai \cite{HS05} use lace expansion to get finer asymptotics on $p_c(R)$:
\begin{align}\label{4e10.10}
p_c(R)V(R)-1 \sim \frac{\theta_d}{R^d},
\end{align}
where $\theta_d$ is given in terms of a probability concerning random walk with uniform steps on $[-1,1]^d$.
The extension of \eqref{4e10.10} to $d>4$ has been conjectured by Edwin Perkins [private communication] while in the critical dimension $d=4$, it is believed that 
\begin{align}\label{4ec10.10}
p_c(R)V(R)-1 \sim \frac{\theta_4 \log R}{R^4} \text{ in } d=4,
\end{align}
where the constant $\theta_4$ can be explicitly determined. In lower dimensions $d=2,3$, the correct asymptotics for $p_c(R)V(R)-1$, suggested by Lalley and Zheng \cite{LZ10} (see also Conjecture 1.2 of \cite{FP16}),  should be $\frac{\theta_d}{R^{\gamma}}$ where $\gamma=\frac{2d}{6-d}$. Therefore a parallel conjecture states that
\begin{align}\label{4e10.11}
p_c(R)V(R)-1\sim \frac{\theta_d}{R^{\gamma}},
\end{align}
for some constant $\theta_d>0$ that depends on the dimension. When $d=2$ or $d=3$, one may check that $\frac{2d}{6-d}=d-1$ and so for simplicity we will proceed with $\gamma=d-1$.
The lower bound implied by \eqref{4e10.11} is already obtained in \cite{FP16}: there is some constant $\theta=\theta(d)>0$ such that for all $R\in \N$,
\begin{align}\label{4e10.12}
p_c(R)V(R) \geq 1+\frac{\theta}{R^{d-1}}.
\end{align}
In this paper, we complement this result by establishing a corresponding upper bound for $p_c$.\\

\no ${\bf Convention\ on\ Functions\ and\ Constants.}$ Constants whose value is unimportant and may change from line to line are denoted $C, c, c_d, c_1,c_2,\dots$, while constants whose values will be referred to later and appear initially in say, Lemma~i.j are denoted $c_{i.j}$ or $C_{i.j}$. 

\begin{theorem}\label{4t0}
Let $d=2$ or $d=3$. There exist some constants  $\theta_d>0$ and $c_{\ref{4t0}}(d)>0$ so that for any positive integer $R>c_{\ref{4t0}}(d)$, we have
\begin{align}\label{4e10.13}
p_c(R)V(R) \leq 1+\frac{\theta_d}{R^{d-1}}.
\end{align}
\end{theorem}


\subsection{SIR epidemic models} \label{4s1.2}
We define the SIR epidemic process on $\Z^d_R$ as follows: For each vertex $x\in \Z^d_R$, it's either infected, susceptible or recovered. Define
\begin{align}
 \eta_n&=\text{the set of infected vertices at time $n$; }\nn\\
\xi_n&= \text{the set of susceptible vertices at time $n$; }\nn\\
\rho_n&= \text{the set of recovered vertices at time $n$.  }
 \end{align}
 
Given the finite initial configurations of infected sites, $\eta_0$, and recovered sites, $\rho_0$, the epidemic evolves as follows: an infected site $x\in \eta_n$ infects its susceptible neighbor $y\in \xi_n$, $y\sim x$ with probability $p=p(R)$, where the infections are conditionally independent given the current configuration. Infected sites at time $n$ become recovered at time $n+1$, and recovered sites will be immune from further infection and stay recovered.
Recall the edge percolation variables $\{B(e): e\in E(\Z_R^d)\}$ with parameter $p=p(R)$. The above process can be described below:
\begin{align}\label{4e10.14}
\eta_{n+1}=&\bigcup_{x\in \eta_n} \{y\in \xi_n: B(x,y)=1\},\nn\\
\rho_{n+1}=&\rho_{n}\cup \eta_n,\\
\xi_{n+1}=&\xi_{n}\backslash \eta_{n+1}.\nn
\end{align}
For any disjoint finite sets $\eta_0$ and $\rho_0$, one may use the above and an easy induction to conclude $\eta_n$ and $\rho_n$ are finite for all $n\geq 0$. Throughout the rest of this paper, we will only consider the epidemic with finite initial condition $(\eta_0, \rho_0)$. Denote by $\cF_n^\eta=\sigma(\eta_k, k\leq n)$ the $\sigma$-field generated by the epidemic process $\eta=(\eta_n)$.

Recall the percolation graph $G$ on $\Z_R^d$. We let $d_G(x,y)$ be the graph distance in $G$ between $x,y\in \Z_R^d$. By convention we let $d_G(x,y)=\infty$ if there is no path between $x$ and $y$ on $G$. For a set of vertices $A$, define $d_G(A,x)=\inf\{d_G(y,x): y\in A\}$. Given a pair of disjoint finite sets in $\Z_R^d$, $(\eta_0,\rho_0)$, we denote by $G(\rho_0)$ the percolation graph by deleting all the edges containing a vertex in $\rho_0$. 
For an SIR epidemic starting from $(\eta_0,\rho_0)$, it is shown in (1.9) of \cite{FP16} that
\begin{align}\label{4e10.15}
\eta_n=\{x\in \Z^d_R: d_{G(\rho_0)}(\eta_0, x)=n\}:=\eta_n^{\eta_0,\rho_0}.
\end{align}
For any integer $k\geq 0$, conditioning on $\cF_k^\eta$, by the Markov property of $(\eta_k,\rho_k)$ as in (1.7) of \cite{FP16}, we have for all $n\geq k$,
\begin{align}\label{4ea3.26}
&\eta_n=\eta_n^{\eta_0,\rho_0}=\{x\in \Z^d_R: d_{G(\rho_{k})}(\eta_{k}, x)=n-k\}=\eta_{n-k}^{\eta_{k},\rho_{k}}.
\end{align}
This is saying that starting from time $k$, the process $(\eta_{n+k},n\geq 0)$ is a usual SIR epidemic starting from $(\eta_k, \rho_k)$.

The total infection set is given by
\begin{align}\label{4e11.3}
\cup_{k=0}^n \eta_k=\{x\in \Z^d_R: d_{G(\rho_0)}(\eta_0, x)\leq n\}.
\end{align}
By shrinking the initial infection set $\eta_0$, it is clear that the total number of infected sites will be decreased. We state this intuition in the following lemma.

\begin{lemma}\label{4l0}
Let $(\eta_0,\rho_0)$ and $(\eta'_0,\rho_0)$ be two finite initial conditions with $\eta'_0\subseteq \eta_0$. For $\eta$ starting from $(\eta_0,\rho_0)$ and $\eta'$  starting from $(\eta'_0,\rho_0)$ given by \eqref{4e10.15}, we have
\[
\cup_{k=0}^n \eta'_k \subseteq \cup_{k=0}^n \eta_k ,\quad \forall n\geq 0.
\]
\end{lemma}
\begin{proof}
On the percolation graph $G(\rho_0)$, we have $d_{G(\rho_0)}(\eta'_0, x)\leq n$ implies $d_{G(\rho_0)}(\eta_0, x)\leq n$ since $\eta'_0\subseteq \eta_0$. So the result follows from \eqref{4e11.3}.
\end{proof}

\begin{mydef}\label{4def1.3}
We say that an SIR epidemic {\bf survives} if with positive probability we have $\eta_n\neq \emptyset$ for all $n\geq 1$; we say the epidemic becomes {\bf extinct} if with probability one, we have $\eta_n= \emptyset$ for some finite $n\geq 1$.
\end{mydef}
 For any $p=p(R) \in [0,1]$, if the epidemic $\eta$ starting from $(\{0\},\emptyset)$ survives, then with positive probability, there is an infinite sequence of infected sites $\{x_k,k\geq 0\}$ such that $x_k\in \eta_k$, $x_k\sim x_{k-1}$ and $x_{k-1}$ infects $x_k$ at time $k$. Hence we have the edge $(x_{k-1},x_k)$ is open and $B(x_{k-1},x_k)=1$. Therefore with positive probability, we have percolation from $\eta_0=\{0\}$ to infinity in range-$R$ bond percolation. This implies $p\geq p_c$ and so an upper bound for $p_c$ is obtained. On the other hand, by Lemma \ref{4l0} and a trivial union inclusion and translation invariance, one may easily check that it is equivalent to prove the survival of $\eta$ starting from $(\eta_0,\emptyset)$ for any finite $\eta_0\subseteq \Z_R^d$.

 From now on, we set 
 \begin{align}\label{4ea3.1}
 p=p(R)=\frac{1+\frac{\theta}{R^{{d-1}}}}{V(R)} \quad \text{ for $\theta\geq 100$ and $R\geq 4\theta$.}
 \end{align}
For the required upper bound, it suffices to find some large $\theta$ so that the SIR epidemic survives. To do this, we will use a comparison to supercritical oriented percolation and apply the methods from Lalley, Perkins and Zheng \cite{LPZ14} with some necessary adjustments and new ideas. Let $\Z_+^2=\{x=(x_1,x_2) \in \Z^2: x_i\geq 0, i=1,2\}$. Set the grid $\Gamma$ to be $\Z_+^2$ in $d=2$ and $\Z_+^2\times \{0\}$ in $d=3$. Define a total order $\prec$ on $\Gamma$ by
\begin{align}\label{4ea1.2}
x\prec y 
\begin{cases}
\text{ if } \|x\|_1<\|y\|_1 \text{ or }\\
\|x\|_1=\|y\|_1 \text{ and } x_1<y_1,
\end{cases}
\end{align} 
 where $\|x\|_1=\sum_{i=1}^d |x_i|$ is the $l^1$-norm on $\R^d$. 
Hence we can write $\Gamma=\{x(1), x(2), \cdots\}$ with $0=x(1)\prec x(2) \prec \cdots$. For any $x\in \Gamma$, define $\cA(x)=\{(x_1, x_2+1), (x_1+1, x_2)\}$ in $d=2$ and $\cA(x)=\{(x_1, x_2+1,0), (x_1+1, x_2, 0)\}$ in $d=3$. This is the set of ``immediate offspring'' of $x$.

For any $M>0$ and $x\in \R^d$, set $Q_M(x)=\{y\in \R^d: \|y-x\|_\infty \leq M\}$ to be the rectangle centered at $x$. Write $Q(y)$ for $Q_1(y)$. For any $T\geq 100$, we define 
\begin{align}\label{4e10.05}
T_\theta^R=[TR^{d-1}/\theta], \text{ and } R_\theta=\sqrt{R^{d-1}/\theta}
\end{align}
for $\theta \geq 100$ and $R\geq 4\theta \geq 400$. These quantities in \eqref{4e10.05} are from the usual Brownian scaling for time and space. One can check that
\begin{align}\label{4e10.06}
200\leq \frac{1}{2}  \frac{TR^{d-1}}{\theta} \leq T_\theta^R\leq \frac{TR^{d-1}}{\theta}.
\end{align}
For any $\theta\geq 100$, define
\begin{align}\label{4e10.20}
f_d(\theta)=
\begin{cases}
\sqrt{\theta}, &\text{ in } d=2,\\
\log {\theta}, &\text{ in } d=3,
\end{cases}
\end{align}
and set for any $R\geq 400$,
\begin{align}\label{4ea10.45}
\beta_d(R)=
\begin{cases}
\log R,&\text{ in } d=2,\\
1,& \text{ in } d=3.
\end{cases}
\end{align}
For any finite set $A\subseteq \Z_R^d$, denote by $|A|$ the number of vertices in $A$. Consider some finite $\eta_0\subseteq \Z_R^d$ such that
\begin{align}\label{4ea10.23}
\begin{dcases}
\text{(i) }\eta_0\subseteq Q_{R_\theta}(0);\\
\text{(ii) } R^{d-1} f_d(\theta)/\theta \leq |\eta_0|\leq 1+R^{d-1} f_d(\theta)/\theta;\\
\text{(iii) }|\eta_0 \cap Q(y)|\leq K \beta_d(R),\ \forall y\in \Z^d,
\end{dcases}
\end{align}
where $K\geq 100$ is some large constant that will be chosen below in Proposition \ref{4p4}. We note that the assumption (iii) in \eqref{4ea10.23} will only be used in Proposition \ref{4p4} (in fact it is only used in the proof of Lemma \ref{4l10.01}). The existence of such a set is trivial if one observes that the finer lattice in $\Z_R^d$ has enough space to place those $|\eta_0|$ vertices.

For any set $Y\subseteq \Z_R^d$, we denote by $\hat{Y}^K\subseteq Y$ a ``thinned'' version of $Y$ so that there are at most $K \beta_d(R)$ vertices in the set $\hat{Y}^K\cap Q(y)$ for all $y\in \Z^d$. This ``thinned'' version idea comes from the ``crabgrass'' paper by Bramson, Durrett and Swindle \cite{BDS89}. The $\beta_d(R)$ in \eqref{4ea10.45} are the typical size of particles in each unit box $Q(y)$ in a branching random walk at time $T_\theta^R$. The procedure for ``thinning''  can be done in a fairly arbitrary way. For example, in Proposition \ref{4p4} below we may proceed by deleting all the vertices in $Y\cap Q(y)$ for each $y\in \Z^d$ if $|Y\cap Q(y)|>K \beta_d(R)$.

Choose $T\geq 100$ large such that
\begin{align}\label{4ea10.04}
\inf_{z\in Q(0)} \inf_{y\in Q(0)} e^{T/4} \P(\zeta_T^z \in Q(y)) \geq 16,
\end{align}
where $\zeta_T^z$ is a $d$-dimensional Gaussian random variable with mean $z$ and variance $T/3$.
The following result is an analogue to Lemma 7.1 of \cite{BDS89} with our SIR epidemic setting.

\begin{proposition}\label{4p4}
For any $\eps_0\in (0,1)$, $\kappa>0$, and $T\geq 100$ satisfying \eqref{4ea10.04}, there exist positive constants $\theta_{\ref{4p4}}$, $K_{\ref{4p4}}$ depending only on $T, \eps_0,\kappa$ such that for all $\theta \geq \theta_{\ref{4p4}}$, there is some $C_{\ref{4p4}}(\eps_0, T,\kappa, \theta)\geq 4\theta$ such that for any $R\geq C_{\ref{4p4}}$, any finite initial condition $(\eta_0,\rho_0)$ where $\eta_0$ is as in \eqref{4ea10.23} with $K_{\ref{4p4}}$, if the SIR epidemic process $\eta$ starts from $(\eta_0,\rho_0)$, then we have
\begin{align*}
\P \Big(\Big\{|\hat{\eta}_{T_\theta^R}^{K_{\ref{4p4}} }\cap Q_{R_\theta}(y R_\theta)|< |\eta_0| \text{ for some } y\in \cA(0) \Big\} \cap N(\kappa)\Big)\leq \eps_0,
\end{align*}
where 
 \[N(\kappa)=\{|\rho_{T_\theta^R}\cap \cN(x) |\leq \kappa R, \forall x\in \Z_R^d\}.\] 
\end{proposition}

We will show in Proposition \ref{4p2} below that under certain conditions, the event $N(\kappa)$ in fact occurs with high probability (see more discussions in  Section \ref{4s1.3}). Then the above result implies that for an SIR epidemic $\eta$ starting from an appropriate initial infection set $\eta_0\subseteq Q_{R_\theta}(0)$ as in \eqref{4ea10.23}, with high probability we have $|\hat{\eta}_{T_\theta^R}^{K_{\ref{4p4}} }\cap Q_{R_\theta}(y R_\theta)|\geq |\eta_0|$ for both $y\in \cA(0)$, that is, the SIR epidemic will generate a sufficiently large total mass in each of the adjacent cubes $Q_{R_\theta}(y R_\theta)$ for $y\in \cA(0)$, even after ``thinning''.  Restart the SIR epidemic with the ``thinned'' infection set $\hat{\eta}_{T_\theta^R}^{K_{\ref{4p4}} }$ restricted to $Q_{R_\theta}(y R_\theta)$ so that the initial condition in \eqref{4ea10.23} recurs (with a spatial translation). By Proposition \ref{4p4}, we may reproduce the infection to the next adjacent cubes with high probability. In this way, infection to the adjacent cubes can be iterated by carefully choosing the initial condition at each step so that it satisfies the necessary assumptions. Of course we need more conditions to make $N(\kappa)$ occur with high probability at each iteration, which we will discuss more in Section \ref{4s1.3} below. By a comparison to oriented percolation, with positive probability this iterated infection will last forever and so the epidemic $\eta$ survives. A rigorous proof for the above arguments leading to the survival of the epidemic can be found in Section \ref{4s2.2}. The proof of Proposition \ref{4p4} is deferred to Section \ref{4s8}.

We next introduce the branching random walk (BRW) dominating the epidemic to show that with high probability the event $N(\kappa)$ holds, i.e. the epidemic will not accumulate enough recovered sites in each unit cube up to time $T_\theta^R$.

\subsection{Branching envelope}\label{4s1.3}

Following Section 2.2 of Frei and Perkins \cite{FP16}, we will couple the epidemic $\eta$ with a dominating branching random walk $Z=(Z_n, n\geq 0)$ on $\Z_R^d$. We first give a brief introduction.  The state space for our branching random walk in this paper is the space of finite measures on $\Z_R^d$  taking values in nonnegative integers, which we denote by $M_F(\Z_R^d)$. For any $\phi: \Z^d_R \to \R$, write $\mu(\phi)=\sum_{x\in \Z^d_R} \phi(x) \mu(x)$ for $\mu \in M_F(\Z_R^d)$. We set $|\mu|=\mu(1)$ to be the total mass for $\mu \in M_F(\Z_R^d)$. We will use a slightly different labelling system here than that in \cite{FP16} in order to keep track of the initial position for each particle. 

Totally order the set $\cN(0)$ as $\{e_1, \cdots, e_{V(R)}\}$ and then totally order each $\cN(0)^n$ lexicographically by $<$. We use the following labelling system borrowed from Section II.3 of \cite{Per02} for our branching particle system:
\begin{align}\label{4e1.16}
I=\bigcup_{n=0}^\infty \N \times \cN(0)^n=\{(\alpha_0, \alpha_1, \cdots, \alpha_n): \alpha_0\in \N, \alpha_i \in \cN(0), 1\leq i\leq n\},
\end{align}
where $\alpha_0$ labels the ancestor of particle $\alpha$. Let $|(\alpha_0, \alpha_1, \cdots, \alpha_n)|=n$ be the generation of $\alpha$ and write
$\alpha|i=(\alpha_0, \cdots, \alpha_i)$ for $0\leq i\leq n$. Let $\pi \alpha=(\alpha_0, \alpha_1, \cdots, \alpha_{n-1})$ be the parent of $\alpha$ and let $\alpha \vee e_i=(\alpha_0, \alpha_1, \cdots, \alpha_n, e_i)$ be an offspring of $\alpha$ whose position relative to its parent is $e_i$.  Recall $p(R)$ from \eqref{4ea3.1}. Assign an i.i.d. collection of Bernoulli random variables $\{B^\alpha: \alpha \in I, |\alpha|>0\}$ to the edge connecting the locations of $\alpha$ and its parent $\pi \alpha$ so that the birth in this direction is valid with probability $p(R)$ and invalid with probability $1-p(R)$.  Set 
\begin{align}
\cG_n=\sigma(\{B^\alpha: 0<|\alpha|\leq n\})\text{ for each $n\geq 0$.}
\end{align}

Fix any $Z_0\in M_F(\Z_R^d)$. Recall that $M_F(\Z_R^d)$ is the space of finite measures taking values in nonnegative integers. So the total mass $|Z_0|$ is the number of initial particles. Label these particles by $1,2,3,\cdots, |Z_0|$ and denote by $x_1, x_2, \cdots, x_{|Z_0|}$ their locations. We note that these $\{x_i\}$ do not have to be distinct; for example, if $Z_0=3\delta_0$, then we have $x_1=x_2=x_3=0$ with initial particles $1,2,3$. Hence we may rewrite $Z_0$ as $Z_0=\sum_{i=1}^{|Z_0|} \delta_{x_i}$. For any $i>|Z_0|$, we set $x_i$ to be the cemetery state $\Delta$.
 For each $n\geq 0$, we write $\alpha \approx n$ iff $x_{\alpha_0}\neq \Delta$, $|\alpha|=n$ and $B^{\alpha|i}=1$ for all $1\leq i\leq n$ so that such an $\alpha$ labels a particle alive in generation $n$.  For each $\alpha \in I$, define its current location by 
\begin{align}\label{4e1.17}
Y^\alpha=
\begin{cases}
x_{\alpha_0}+\sum_{i=1}^{|\alpha|}  \alpha_i, &\text{ if } \alpha\approx |\alpha|,\\
\Delta, &\text{ otherwise. }
\end{cases}
\end{align}
In this way, $Z_n=\sum_{|\alpha|= n} \delta_{Y^\alpha} 1(Y^\alpha\neq \Delta)$ defines the empirical distribution of a branching random walk where in generation $n$, each particle gives birth to one offspring to its $V(R)$ neighboring positions independently with probability $p(R)$. So it follows that
\begin{align}\label{4ea4.5}
\text{ $({{Z}}_n(1),n\geq 0)$ is a}&\text{  Galton-Watson process with 
}\\& \text{ 
offspring distribution $Bin(V(R),p(R))$.}\nn
 \end{align}
Note the dependence of $Z_n$ on $\theta$ and $R$ is implicit. Define $Z_n(x)=Z_n(\{x\})$ for any $x\in \Z^d_R$. For any Borel function $\phi$, we let 
\begin{align}\label{4eb2.21}
Z_n(\phi)=\sum_{|\alpha|= n} \phi(Y^\alpha)=\sum_{x\in \Z^d_R} \phi(x) Z_n(x),
\end{align}
where it is understood that $\phi(\Delta)=0$. We use $\P^{Z_0}$ to denote the law of $(Z_n, n\geq 0)$ starting from $Z_0$. 

For $\mu,\nu \in M_F(\Z_R^d)$, we say $\nu$ {\bf dominates} $\mu$ if $\nu(x)\geq \mu(x)$ for all $x\in \Z_R^d$. For any set $Y$ on $\Z_R^d$, by slightly abusing the notation, we write $Y(x)=1(x\in Y)$ for $x\in \Z^d_R$ so that the set $Y$ naturally defines a measure on $\Z^d_R$ taking values in $\{0,1\}$. In particular we let $\eta_n(x)=1(x\in \eta_n)$ for any $n\geq 0$ and $x\in \Z^d_R$. By the construction in Section 2.2 of \cite{FP16}, we may define the coupled SIR epidemic $(\eta_n)$ inductively with the dominating $(Z_n)$.

\begin{lemma}\label{4l1.5}
For any finite initial configuration $(\eta_0,\rho_0)$ and any $Z_0 \in M_F(\R^d)$ such that $Z_0$ dominates $\eta_0$, on a common probability space we can define an SIR epidemic processes $\eta$ starting from $(\eta_0, \rho_0)$, and a branching random walk $Z$ starting from $Z_0$, such that 
\[
\eta_n(x)\leq Z_n(x) \text{ for all } x\in \Z^d_R, n\geq 0.
\]
Moreover, we have both $(\eta,\rho)$ and $Z$ satisfy the Markov property with respect to a common filtration $(\cG_n)$.
\end{lemma}
\begin{proof}
The proof is similar to that of Proposition 2.3 in \cite{FP16}.  Although their proof was dealing with $\eta_0=\{0\}$, it works for any finite $\eta_0$ as the arguments there indeed uses induction to prove $\eta_{n+1}(x)\leq Z_{n+1}(x), \forall x\in \Z^d_R$ by assuming $Z_n$ dominates $\eta_n$. The proof of the Markov property is similar.
\end{proof}

To understand the large $R$ behavior of $(Z_n)$, we will also consider a rescaled version of $(Z_n)$ and study its limit as $R\to \infty$. Let $\sigma^2=1/3$ be the variance of  the marginals of the uniform distributions over $[-1,1]^d$. For each $t\geq 0$, we define a random measure $W_t^R$ on $\R^d$ by
\begin{align}\label{4e5.39}
W_t^R=\frac{1}{R^{{d-1}}}\sum_{x\in \Z^d_R} \delta_{{x}/{\sqrt{\sigma^2 R^{d-1}}}} Z_{[tR^{d-1}]}(x).
\end{align}
 Let $\P^{W_0^R}$ denote the law of $(W_t^R, t\geq 0)$. Let $M_F(\R^d)$ be the space of finite measures on $\R^d$ equipped with weak topology and denote by $C_b^2(\R^d)$ the space of twice continuously differentiable functions on $\R^d$. For $\mu \in M_F(\R^d)$, we denote by $|\mu|$ its total mass. For any $\phi: \R^d\to \R$, we write $\mu(\phi)$ for the integral of $\phi$ with respect to $\mu$. Let $X$ be a super-Brownian motion (SBM) with drift $\theta$ that is the unique in law solution to the following martingale problem:
\begin{align}\label{4e10.25}
(MP)_{\theta}: \quad X_t(\phi)=X_0(\phi)+M_t(\phi)+ \int_0^t X_s(\frac{\Delta}{2}  \phi) ds+\theta\int_0^t X_s(\phi) ds, \quad \forall \phi \in C_b^2(\R^d),
\end{align}
where $X$ is a continuous $M_F(\R^d)$-valued process, and $M(\phi)$ is a continuous martingale
with $\langle M(\phi)\rangle_t=\int_0^t X_s(\phi^2)ds$. 
We denote the law of $X$ by $\P^{X_0}$.
If there is some $X_0 \in M_F(\R^d)$ so that (recall $\sigma^2=1/3$)
\begin{align}\label{4e10.19}
W_0^R=\frac{1}{R^{{d-1}}}\sum_{x\in \Z^d_R} Z_{0}(x) \delta_{{x}/{\sqrt{R^{d-1}/3}}}  \to X_0 \text{ in } M_F(\R^d)
\end{align}
as $R\to \infty$,
then by Proposition 4.3 of \cite{FP16}, it follows that
\begin{align}\label{4e10.22}
(W_t^R, t\geq 0) \Rightarrow (X_t, t\geq 0) \text{ on } D([0,\infty), M_F(\R^d))
\end{align}
as $R\to \infty$. Here $D([0,\infty), M_F(\R^d))$ is the Skorohod space of cadlag $M_F(\R^d)$-valued paths, on which $\Rightarrow$ denotes the weak convergence. Note we have scaled the variance $\sigma^2=1/3$ in \eqref{4e5.39} and so the constant in \eqref{4e10.25} will differ from that of \cite{FP16}.\\


We collect the properties of $(Z_n)$ below in Propositions \ref{4p1}, \ref{4p3} and \ref{4p2} while these results will be proved later. In fact these proofs will occupy most of the paper. They are technical results that will be used in the proof of the main theorem in Section \ref{4s2}. We briefly explain their uses: Proposition \ref{4p1} says that the support of $(Z_n)$ up to time $T_\theta^R$ will be contained in a large box; Proposition \ref{4p3} is a technical condition that ensures Proposition \ref{4p2} holds; Proposition \ref{4p2} will be the key condition that guarantees there won't be too many accumulated particles in each unit cube contained in a large box. Together with Proposition \ref{4p1}, we may conclude by the dominance of $(Z_n)$ over $(\eta_n)$ that the event $N(\kappa)$ in Proposition \ref{4p4} occurs with high probability. The assumptions on $Z_0$ for each proposition will vary. Nevertheless, we may choose $Z_0$ carefully so that all the conditions will be satisfied for each iteration.

 \no Let $\text{Supp}(\mu)$ denote the closed support of a measure $\mu$. Consider $Z_0\in M_F(\Z_R^d)$ such that
\begin{align}\label{4e10.23}
\begin{dcases}
\text{(i) }\text{Supp}(Z_0)\subseteq Q_{R_\theta}(0); \\
\text{(ii) } R^{d-1} f_d(\theta)/\theta\leq |Z_0|\leq 1+R^{d-1} f_d(\theta)/\theta.
\end{dcases}
\end{align}

\begin{proposition}\label{4p1}
For any $\eps_0\in (0,1)$, $T\geq 100$, there are constants $\theta_{\ref{4p1}}\geq 100, M_{\ref{4p1}}\geq 100$ depending only on $\eps_0, T$ such that for all $\theta \geq \theta_{\ref{4p1}}$,  there is some $C_{\ref{4p1}}(\eps_0, T,\theta)\geq 4\theta$ such that for any $R\geq C_{\ref{4p1}}$ and any $Z_0$ satisfying \eqref{4e10.23}, we have
\[
\P^{Z_0}\Big(\text{Supp}(\sum_{n=0}^{T_\theta^R} Z_n) \subseteq Q_{M_{\ref{4p1}} \sqrt{\log f_d(\theta)}R_\theta} (0) \Big)\geq 1-\eps_0.
\]

\end{proposition}

Next we turn to the crucial event $N(\kappa)=\{|\rho_{T_\theta^R}\cap \cN(y) |\leq \kappa R, \forall y\in \Z_R^d\}$ in Proposition \ref{4p4}. To show that $N(\kappa)$ occurs with high probability, we will show the corresponding result for the dominating branching random walk $Z=(Z_n, n\geq 0)$, i.e. we will bound $\sum_{n=0}^{T_\theta^R} Z_n(\cN(y))$ for all $y\in \Z_R^d$. We call this the ``local time'' process of $Z$ as we indeed conjecture that $\sum_{n=0}^{T_\theta^R} Z_n(\cN(y))$ will converge to the local time of super-Brownian motion as $R\to \infty$. By applying a discrete version of Tanaka's formula (see \eqref{4e6.13} and \eqref{4e6.14}), we need a regularity condition on $Z_0$ to get bounds for the local time of $Z$. For any $x,u\in \R^d$, define 
\begin{align}\label{4e10.31}
g_{u,d}(x)=
\begin{dcases}
\sum_{n=1}^\infty e^{-n\theta/R} \frac{1}{n} e^{-|x-u|^2/(32n)}, &\text{ in } d=2,\\
R\sum_{n=1}^\infty  \frac{1}{n^{3/2}} e^{-|x-u|^2/(32n)}, &\text{ in } d=3.
\end{dcases}
\end{align}
Again we have suppressed the dependence of $g_{u,d}$ on $R,\theta$. One can show that (see Lemma \ref{4l4.1} and Lemma \ref{4l3.2}) there is some universal constant $C>0$ such that for any $x\neq u$,
\begin{align}\label{4e10.32}
g_{u,d}(x)\leq
\begin{cases}
C\Big(1+\log^+ \big(\frac{R}{\theta |x-u|^2} \big)\Big),& \text{ in } d=2,\\
C\frac{R}{|x-u|}, &\text{ in } d=3,
\end{cases}
\end{align}
where $\log^+(x)=0\vee \log x$ for $x>0$. The reason for defining $g_{u,d}$ as in \eqref{4e10.31} will be clearer in Section \ref{4s4} when we introduce the appropriate potential kernels and Tanaka’s formula.

Now consider $Z_0\in M_F(\Z_R^d)$ such that
\begin{align}\label{4e11.24}
\begin{dcases}
\text{(i) }\text{Supp}(Z_0)\subseteq Q_{R_\theta}(0); \\
\text{(ii) } Z_0(1)\leq 2 R^{d-1} f_d(\theta)/\theta;\\
\text{(iii) } Z_{0}(g_{u,d})\leq m {R^{d-1}}/\theta^{1/4}, \quad \forall u\in \R^d.
\end{dcases}
\end{align}

\begin{proposition}\label{4p2}
For any $\eps_0\in (0,1)$, $T\geq 100$ and $m>0$, there exist constants $\theta_{\ref{4p2}}\geq 100, \chi_{\ref{4p2}}>0$ depending only on $\eps_0, T,m$ such that for all $\theta \geq \theta_{\ref{4p2}}$,  there is some $C_{\ref{4p2}}(\eps_0, T,\theta,m)\geq 4\theta$ such that for any $R\geq  C_{\ref{4p2}}$ and any $Z_0$ satisfying \eqref{4e11.24}, we have
\[
\P^{Z_0}\Big(\sum_{n=0}^{T_\theta^R} Z_n(\cN(x)) \leq  \chi_{\ref{4p2}} {R}, \quad \forall x\in \Z^d_R \cap Q_{2M_{\ref{4p1}} \sqrt{\log f_d(\theta)}R_\theta} (0) \Big)\geq 1-\eps_0.
\]
\end{proposition}

Finally we show that the extra condition (iii) of \eqref{4e11.24} indeed holds with high probability, which allows us to iterate this initial condition for $Z_0$. The following theorem gives an analogue to the ``admissible'' regularity condition for super-Brownian motion in (5.4) of \cite{LPZ14}. For the next two results, instead of \eqref{4e11.24} we only assume
\begin{align}\label{4e12.01}
Z_0(1)\leq 2R^{d-1} f_d(\theta)/\theta.
\end{align}
\begin{proposition}\label{4p3}
For any $\eps_0\in (0,1)$, $T\geq 100$, there exist constants $\theta_{\ref{4p3}}\geq 100, m_{\ref{4p3}}>0$ depending only on $\eps_0, T$ such that for all $\theta \geq \theta_{\ref{4p3}}$,  there is some $C_{\ref{4p3}}(\eps_0, T,\theta)\geq 4\theta$ such that for any $R\geq C_{\ref{4p3}}$ and any $Z_0$ satisfying \eqref{4e12.01}, we have
\begin{align}\label{4e11.33}
\P^{Z_0}\Big(Z_{T_\theta^R}(g_{u,d})\leq m_{\ref{4p3}} \frac{R^{d-1}}{\theta^{1/4}},\quad \forall u\in Q_{8\sqrt{\log f_d(\theta)}R_\theta} (0) \Big)\geq 1-\eps_0.
\end{align}
\end{proposition}


By restricting the measure $Z_{T_\theta^R}$ to a finite rectangle $Q_{4R_\theta}(0)$, we may be able to assume the above holds for all $u$.

\begin{corollary}\label{4c0.1}
For any $\eps_0\in (0,1)$, $T\geq \eps_0^{-1}+100$, there are constants $\theta_{\ref{4c0.1}}\geq 100, m_{\ref{4c0.1}}>0$ depending only on $\eps_0, T$ such that for all $\theta \geq \theta_{\ref{4c0.1}}$,  there is some $C_{\ref{4c0.1}}(\eps_0, T,\theta)\geq 4\theta$ such that for any $R\geq  C_{\ref{4c0.1}}$ and any $Z_0$ satisfying \eqref{4e12.01}, we have
\begin{align}\label{4e11.37}
\P^{Z_0}\Big(\tilde{Z}_{T_\theta^R}(g_{u,d})\leq m_{\ref{4c0.1}} \frac{R^{d-1}}{\theta^{1/4}},\quad \forall u\in \R^d \Big)\geq 1-2\eps_0,
\end{align}
where $\tilde{Z}_{T_\theta^R}(\cdot)=Z_{T_\theta^R}(\cdot \cap Q_{4R_\theta}(0))$.
\end{corollary}

\no Corollary \ref{4c0.1} is an easy refinement of Proposition \ref{4p3}. Its proof is given in Section \ref{4s7}. 
The proofs of Propositions \ref{4p1},  \ref{4p2} and   \ref{4p3}  will be the main parts of this paper and are deferred to Sections \ref{4s3}, \ref{4s5}, \ref{4s6}, \ref{4s7}. Assuming the above results, we will prove the survival of the SIR epidemic in Section \ref{4s2}, thus giving our main result Theorem \ref{4t0}. \\

\no $\mathbf{Organization\ of\ the\ paper.}$ In Section \ref{4s2},  assuming Propositions \ref{4p4}, \ref{4p1}, \ref{4p2} and Corollary \ref{4c0.1}, we give the proof of our main result Theorem \ref{4t0} by showing the survival of the SIR epidemic. We use a comparison with supercritical oriented percolation inspired by that in \cite{LPZ14}, along with some new ideas and some necessary adjustments to our setting. In Section \ref{4s3}, we will prove Proposition \ref{4p1} for the support propagation and state some preliminary results, including the $p$-th moments, exponential moments and the martingale problem, for the branching random walk. Section \ref{4s4} introduces the potential kernel, and by applying it to the martingale problem, we get a discrete version of Tanaka's formula for the ``local times'' of the branching random walk. Using this Tanaka's formula and a discrete Garsia's Lemma,  we give the proof of Proposition \ref{4p2} for $d=2$ in Section \ref{4s5}  and $d=3$ in Section \ref{4s6}. In Section \ref{4s7}, the proofs of Proposition \ref{4p3} and Corollary \ref{4c0.1} for the regularity of branching random walk is completed. Finally in Section \ref{4s8}, we prove Proposition \ref{4p4} that will imply the survival of the SIR  epidemic.


\section*{Acknowledgements}
This work was done as part of the author's graduate studies at the University of British Columbia. I would like to thank my advisor, Professor Edwin Perkins, for suggesting this problem and for the helpful discussions throughout this work, especially during the global pandemic.

\section{Oriented percolation and proof of survival} \label{4s2}

\subsection{SIR epidemic with immigration}
Recall from \eqref{4e10.15} the SIR epidemic process $\eta$ starting from $(\eta_0,\rho_0)$:
\begin{align}\label{4ea4.8}
\eta_n=\{x\in \Z^d_R: d_{G(\rho_0)}(\eta_0, x)=n\}:=\eta_n^{\eta_0,\rho_0}, \quad \forall n\geq 0.
\end{align}
In order to prove the survival of $\eta$, we need some coupled SIR epidemic process to serve as a lower bound.  Let $\mu_0, \nu_0$ be two finite subsets of $\Z^d_R$ and set $\rho_0$ to be a finite set disjoint from $\mu_0\cup \nu_0$. Recall from Lemma \ref{4l1.5} that $(\eta_n,\rho_n)$  satisfies the Markov property w.r.t. $(\cG_n)$ where 
\begin{align}\label{4ea4.18}
\cG_n=\sigma(\{B^\alpha: 0<|\alpha|\leq n\}), \quad \forall n\geq 0.
\end{align}
 We say $\eta^{*}$ is an {\bf SIR epidemic process with immigration at time $k_*$} if 
\begin{align}\label{4ea3.25}
&\eta_0^*=\mu_0,\quad \rho_0^*=\rho_0, \quad \rho_{n+1}^*=\rho_n^* \cup \eta_n^* \quad \text{ and}\nn\\
&\eta_n^*=\{x\in \Z^d_R: d_{G(\rho_0^*)}(\eta_0^*, x)=n\}=\eta_n^{\eta_0^*,\rho_0^*}, \text{ if } n\leq k_*;\nn\\
&\eta_n^*=\{x\in \Z^d_R: d_{G(\rho_{k_*}^*)}(\eta_{k_*}^* \cup \nu_0, x)=n-k_*\}=\eta_{n-k_*}^{\eta_{k_*}^*\cup \nu_0,\rho_{k_*}^*}, \text{ if } n> k_*,
\end{align}
where $G(\rho_{k_*}^*)$ is the percolation graph by deleting all the edges containing a vertex in $\rho_{k_*}^*$. The dependence of $\eta^*$ on $\mu_0, \nu_0,\rho_0, k_*$ will be implicit. One can check that $(\eta^*,\rho^*)$ satisfies the Markov property w.r.t. $(\cG_n)$.

Briefly speaking, at time $k_*$ all the non-recovered sites in $\nu_0$ are suddenly infected. This could be due to the infection caused by, say,  intercontinental travel. Before time $k_*$, $\eta_n^*$ is the usual SIR epidemic starting from $(\mu_0, \rho_0)$. At time $k_*$, we let all the non-recovered sites in $\nu_0$ become infected. Afterwards $\eta_n^*$ will evolve as the usual SIR epidemic starting from $(\eta_{k_*}^* \cup \nu_0,\rho_{k_*}^*)$. 
The following lemma tells us that the SIR epidemic with immigration will give a lower bound of the original epidemic.

\begin{lemma}\label{4l0.2}
Let $\mu_0, \nu_0$ be finite subsets of $\Z^d_R$ and set $\rho_0$ to be a finite set disjoint from $\mu_0\cup \nu_0$.  For any integer $k_*\geq 0$, and any finite $\eta_0$ with $\mu_0\cup \nu_0 \subseteq \eta_0$, if $\eta$ and $\eta^*$ are given as in \eqref{4ea4.8} and \eqref{4ea3.25}, we have 
\[
\cup_{k=0}^n \eta^*_k \subseteq \cup_{k=0}^n \eta_k ,\quad \forall n\geq 0.
\]
\end{lemma}
\begin{proof}
For any $n\leq k_*$, $\eta_n^*$ is a usual SIR epidemic starting from $(\mu_0,\rho_0)$. Since $\mu_0\subseteq \eta_0$, by Lemma \ref{4l0} we have
\begin{align}\label{4ea3.40}
\cup_{k=0}^n \eta_k^{*} \subseteq \cup_{k=0}^n \eta_k, \quad \forall n\leq k_*.
\end{align}
Moreover, by \eqref{4ea3.25} we have
\begin{align}\label{4ea3.10}
&\cup_{k=0}^n \eta^*_k= \{x\in \Z^d_R: d_{G(\rho_{0}^*)}(\eta_0^*, x)\leq n\}, \quad \forall n\leq k_*.
\end{align}
For $n\geq k_*+1$, use \eqref{4ea3.25} again to get
\begin{align}
\eta_n^*=&\{x\in \Z^d_R: d_{G(\rho_{k_*}^*)}(\eta_{k_*}^* \cup \nu_0, x)=n-k_*\}\nn\\
\subseteq &\{x\in \Z^d_R: d_{G(\rho_{k_*}^*)}(\eta_{k_*}^*, x)=n-k_*\} \cup \{x\in \Z^d_R:  d_{G(\rho_{k_*}^*)}(\nu_0, x)=n-k_*\}\nn\\
=&\{x\in \Z^d_R: d_{G(\rho_{0}^*)}(\eta_{0}^*, x)=n\} \cup \{x\in \Z^d_R:  d_{G(\rho_{k_*}^*)}(\nu_0, x)=n-k_*\},
\end{align}
where the last equality is by \eqref{4ea3.26}.  Apply the above and \eqref{4ea3.10} to see that for $n\geq k_*+1$,
\begin{align}\label{4e11.1}
&\cup_{k=0}^n \eta^*_k \subseteq \{x\in \Z^d_R: d_{G(\rho_{0}^*)}(\eta_0^*, x)\leq n\} \cup \{x\in \Z^d_R:  1\leq d_{G(\rho_{k_*}^*)}(\nu_0, x)\leq n-k_*\}.
\end{align}

\no On the other hand, by using \eqref{4e11.3} and $\eta_0\supseteq \mu_0 \cup \nu_0$, for $n\geq k_*+1$, we have
\begin{align}\label{4e11.2}
\cup_{k=0}^n \eta_k =&\{x\in \Z^d_R: d_{G(\rho_{0})}(\eta_0, x)\leq n\}\supseteq\{x\in \Z^d_R: d_{G(\rho_{0})}(\mu_0 \cup \nu_0, x)\leq n\}\nn\\
=&\{x\in \Z^d_R: d_{G(\rho_{0})}(\mu_0, x)\leq n\} \cup \{x\in \Z^d_R: d_{G(\rho_{0})}(\nu_0, x)\leq n\}.
\end{align}
Recall that $\eta_0^*=\mu_0$, $\rho_0^*=\rho_0$.
 Since $\rho_0 \subseteq \rho_{k_*}^* $, one can check that for any $x$ with $d_{G(\rho_{k_*}^*)}(\nu_0, x)\leq n-k_*$, we have $d_{G(\rho_0)}(\nu_0, x)\leq n$. So it follows from \eqref{4e11.1} and \eqref{4e11.2} that
\[
\cup_{k=0}^n \eta_k^{*} \subseteq \cup_{k=0}^n \eta_k,\quad \forall n\geq k_*+1.
\]
The proof is complete by \eqref{4ea3.40}.
\end{proof}

We may also consider immigration at random times. Let $\tau$ be some finite stopping time with respect to $(\cG_n)$. We say $\eta^{*}$ is an {\bf SIR epidemic process with immigration at time $\tau$} if 
\begin{align}\label{4e10.17}
&\eta_0^*=\mu_0,\quad \rho_0^*=\rho_0, \quad \rho_{n+1}^*=\rho_n^* \cup \eta_n^* \quad \text{ and}\nn\\
&\eta_n^*=\{x\in \Z^d_R: d_{G(\rho_0^*)}(\eta_0^*, x)=n\}=\eta_n^{\eta_0^*,\rho_0^*}, \text{ if } n\leq \tau;\nn\\
&\eta_n^*=\{x\in \Z^d_R: d_{G(\rho_{\tau}^*)}(\eta_{\tau}^* \cup \nu_0, x)=n-\tau\}=\eta_{n-\tau}^{\eta_{\tau}^*\cup \nu_0,\rho_{\tau}^*}, \text{ if } n> \tau,
\end{align}
where $G(\rho_{\tau}^*)$ is the percolation graph by deleting all the edges containing a vertex in $\rho_{\tau}^*$. The dependence of $\eta^*$ on $\mu_0, \nu_0,\rho_0, \tau$ will be implicit.

\begin{lemma}\label{4l0.3}
Let $\mu_0, \nu_0$ be finite subsets of $\Z^d_R$ and set $\rho_0$ to be a finite set disjoint from $\mu_0\cup \nu_0$.  For any finite stopping time $\tau$, and any finite $\eta_0$ with $\mu_0\cup \nu_0 \subseteq \eta_0$, if $\eta$ and $\eta^*$ are given as in \eqref{4ea4.8} and \eqref{4e10.17}, we have 
\[
\cup_{k=0}^n \eta^*_k \subseteq \cup_{k=0}^n \eta_k ,\quad \forall n\geq 0.
\]
\end{lemma}

\begin{proof}
The proof is similar to that of Lemma \ref{4l0.2} by conditioning on $\tau=k_*$ for $k_*\geq 0$.
\end{proof}

Finally we consider immigration at an increasing sequence of random times $0=\tau_0 \leq \tau_1\leq \tau_2\leq \cdots<\infty$. Here $\{\tau_i\}$ are finite stopping times with respect to $(\cG_n)$. Let $\mu_0, \nu_0$ be two finite subsets of $\Z^d_R$. For any finite subset $\rho_0$ disjoint from $\mu_0\cup \nu_0$, we say $\eta^{*}=(\eta_n^*, n\geq 0)$ is an {\bf SIR epidemic process with immigration at times $\{\tau_i, i\geq 0\}$} if
\begin{align}\label{4e10.18}
&\eta_0^*=\mu_0, \quad \rho_0^*=\rho_0,\nn\\
&\eta_n^*=\{x\in \Z^d_R: d_{G(\rho_{\tau_i}^*)}(\mu_i , x)=n-\tau_i\} \text{ for }\tau_i+1\leq n\leq \tau_{i+1},\nn\\
&\rho_{\tau_i+1}^*=\rho_{\tau_i}^* \cup \mu_i \text{ and } \rho_{n+1}^*=\rho_n^* \cup \eta_n^*\ \text{ for }\tau_i+1\leq n\leq \tau_{i+1}, 
\end{align}
where for $i\geq 1$, $\mu_i$ and $\nu_i$ are $\cG_{\tau_i}$-measurable random sets such that 
\begin{align}\label{4ea3.12}
(\mu_i \cup \nu_i) \subseteq (\eta_{\tau_i}^* \cup \nu_{i-1}).
\end{align}

Briefly speaking, at time $\tau_{i}$ we introduce the immigration set $\nu_{i-1}$ and choose subsets $\mu_i$, $\nu_i$ from $\eta_{\tau_i}^* \cup \nu_{i-1}$. Restart the SIR epidemic with initial condition $\mu_i$ starting from time $\tau_i$. In the mean time, we keep $\nu_i$ for the next immigration at time $\tau_{i+1}$ while ``forgetting'' other infected sites in $\eta_{\tau_i}^*$, which is done by defining $\rho_{\tau_i+1}^*=\rho_{\tau_i}^* \cup \mu_i$ in \eqref{4e10.18}.
If $\tau_{k}=\tau_i$ for all $k\geq i$ for some $i\geq 0$, we may ``freeze'' the epidemic by letting $\eta_n^*=\eta_{\tau_i}^*$ for all $n\geq \tau_i$.

\begin{proposition}\label{4p2.3}
Let $\mu_0, \nu_0$ be finite subsets of $\Z^d_R$ and set $\rho_0$ to be a finite set disjoint from $\mu_0\cup \nu_0$.  For any finite stopping times $0=\tau_0 \leq \tau_1\leq \tau_2\leq \cdots<\infty$, and any finite $\eta_0$ with $\mu_0\cup \nu_0 \subseteq \eta_0$, if $\eta$ and $\eta^*$ are given as in \eqref{4ea4.8} and \eqref{4e10.18}, we have 
\begin{align}\label{4ea3.24}
\cup_{k=0}^n \eta^*_k \subseteq \cup_{k=0}^n \eta_k ,\quad \forall n\geq 0.
\end{align}
\end{proposition}
\begin{proof}
We will iteratively define a sequence of epidemic processes $\{\eta^{*,i}, i\geq 1\}$ such that
\begin{align}\label{4ea3.31}
\eta^{*}_n=\eta^{*,i}_{n-\tau_{i-1}}, \quad \forall \tau_{i-1}<n\leq \tau_i, \quad \forall i\geq 1.
\end{align}

\no Given $\mu_0, \nu_0, \eta_0$ and $\rho_0$ as above, we first consider the epidemic process $\eta^{*,1}$ such that
\begin{align}\label{4ea3.41}
&\eta_0^{*,1}=\mu_0,\quad \rho_0^{*,1}=\rho_0, \quad \rho_{n+1}^{*,1}=\rho_n^{*,1} \cup \eta_n^{*,1} \quad \text{ and}\nn\\
&\eta_n^{*,1}=\{x\in \Z^d_R: d_{G(\rho_0^{*,1})}(\eta_0^{*,1}, x)=n\}=\eta_n^{\eta_0^{*,1},\rho_0^{*,1}}, \text{ if } n\leq \tau_1;\nn\\
&\eta_n^{*,1}=\{x\in \Z^d_R: d_{G(\rho_{\tau_1}^{*,1})}(\eta_{\tau_1}^{*,1} \cup \nu_0, x)=n-\tau_1\}=\eta_{n-\tau_1}^{\eta_{\tau_1}^{*,1}\cup \nu_0,\rho_{\tau_1}^{*,1}}, \text{ if } n> \tau_1.
\end{align}
By Lemma \ref{4l0.3}, we have
\begin{align}\label{4ea3.32}
\cup_{k=0}^n \eta^{*,1}_k\subseteq \cup_{k=0}^n \eta_k, \quad \forall n\geq 0.
\end{align}
It is easy to check that $\eta^{*}_n=\eta^{*,1}_{n}$ for all $1\leq n\leq \tau_1$. Apply \eqref{4ea3.32} to get
\begin{align}\label{4ea3.20}
\cup_{k=0}^n \eta^{*}_k\subseteq \cup_{k=0}^n \eta_k, \quad \forall 0\leq n\leq \tau_1.
\end{align}
 Since $\rho_0^{*,1}=\rho_{0}^*$ and $\eta_0^{*,1}=\mu_0$, we also have $\rho_{\tau_1}^{*,1}=\rho_{\tau_1}^{*}$. By \eqref{4ea3.41} and $\eta^{*}_{\tau_1}=\eta^{*,1}_{\tau_1}$,  we conclude that conditioning on $\cG_{\tau_1}$, the process $(\eta^{*,1}_{k+\tau_1}, k\geq 1)$ will be a usual SIR epidemic starting from $(\eta_{\tau_1}^* \cup \nu_0, \rho_{\tau_1}^*)$.
Next, choose random sets $\mu_1$, $\nu_1$ which are $\cG_{\tau_1}$-measurable such that $(\mu_1 \cup \nu_1) \subseteq (\eta_{\tau_1}^* \cup \nu_{0})$. We 
consider the epidemic process $\eta^{*,2}$ such that
\begin{align}\label{4ea3.22}
&\eta_0^{*,2}=\mu_1,\quad \rho_0^{*,2}=\rho_{\tau_1}^*, \quad \rho_{n+1}^{*,2}=\rho_n^{*,2} \cup \eta_n^{*,2} \quad \text{ and}\nn\\
&\eta_n^{*,2}=\{x\in \Z^d_R: d_{G(\rho_0^{*,2})}(\eta_0^{*,2}, x)=n\}=\eta_n^{\eta_0^{*,2},\rho_0^{*,2}}, \text{ if } n\leq \tau_2-\tau_1;\nn\\
&\eta_n^{*,2}=\{x\in \Z^d_R: d_{G(\rho_{\tau_2-\tau_1}^{*,2})}(\eta_{\tau_2-\tau_1}^{*,2} \cup \nu_1, x)=n-(\tau_2-\tau_1)\}=\eta_{n-(\tau_2-\tau_1)}^{\eta_{\tau_2-\tau_1}^{*,2}\cup \nu_1,\rho_{\tau_2-\tau_1}^{*,2}},\nn\\
&\quad \quad \quad \text{ if } n> \tau_2-\tau_1.
\end{align}
By Lemma \ref{4l0.3} applied to $(\eta^{*,1}_{k+\tau_1}, k\geq 1)$ and $\eta^{*,2}$, we have for all $n\geq 0$,
\begin{align}\label{4ea3.33}
\cup_{k=0}^n \eta^{*,2}_k\subseteq (\eta_{\tau_1}^* \cup \nu_0) \bigcup \cup_{k=1}^n \eta^{*,1}_{k+\tau_1}=\nu_0\bigcup \cup_{k=\tau_1}^{n+\tau_1} \eta^{*,1}_{k} \subseteq  \cup_{k=0}^{n+\tau_1} \eta_k,
\end{align}
where the equality uses $\eta_{\tau_1}^*=\eta_{\tau_1}^{1,*}$ and the last subset relation uses \eqref{4ea3.32} and $\nu_0\subseteq \eta_0$.
By \eqref{4e10.18}, conditioning on $\cG_{\tau_1}$,  the process $\{\eta_{n+\tau_1}^*, 0<n\leq \tau_2-\tau_1\}$ is a usual SIR epidemic starting from $(\mu_1, \rho_{\tau_1}^*)$. Therefore $\eta^{*}_n=\eta^{*,2}_{n-\tau_1}$ for all $\tau_1<n\leq \tau_2$ and it follows that for any $\tau_1<n\leq \tau_2$,
\begin{align}
&\cup_{k=\tau_1+1}^{n} \eta^{*}_k=\cup_{k=\tau_1+1}^{n} \eta^{*,2}_{k-\tau_1}=\cup_{k=1}^{n-\tau_1} \eta^{*,2}_{k}\subseteq \cup_{k=0}^{n} \eta_{k},
\end{align}
where the last subset relation uses \eqref{4ea3.33}.
Together with \eqref{4ea3.20}, we conclude
\begin{align}\label{4ea3.21}
&\cup_{k=0}^{n} \eta^{*}_k \subseteq \cup_{k=0}^{n} \eta_{k}, \quad \forall 0\leq n\leq \tau_2.
\end{align}
Since $\rho_0^{*,2}=\rho_{\tau_1}^*$ and $\eta_0^{*,2}=\mu_1$, we also have $\rho_{\tau_2-\tau_1}^{*,2}=\rho_{\tau_2}^{*}$. By \eqref{4ea3.22} and $\eta^{*,2}_{\tau_2-\tau_1}=\eta^{*}_{\tau_2}$, we conclude the process \mbox{$(\eta^{*,2}_{k+\tau_2-\tau_1}, k\geq 1)$} is a usual SIR epidemic starting from $(\eta_{\tau_2}^* \cup \nu_1, \rho_{\tau_2}^*)$. 
Next, choose random set $\mu_2$, $\nu_2$ which are $\cG_{\tau_2}$-measurable such that $(\mu_2 \cup \nu_2) \subseteq (\eta_{\tau_2}^* \cup \nu_{1})$. We may repeat the above and consider some epidemic process $\eta^{*,3}$ with $\eta^{*,3}_0=\mu_2$ and $\rho^{*,3}_0=\rho_{\tau_2}^*$ in a way similar to \eqref{4ea3.22}. Similar arguments will give that
\begin{align}\label{4ea3.23}
&\cup_{k=0}^{n} \eta^{*}_k \subseteq \cup_{k=0}^{n} \eta_{k}, \quad \forall 0 \leq n\leq \tau_3.
\end{align}
Therefore by induction we conclude \eqref{4ea3.24} holds.
\end{proof}

\subsection{Proofs of Theorem \ref{4t0} and the survival of the epidemic}\label{4s2.2}

Now we return to the original SIR epidemic process $\eta$. By our discussion in the paragraph following Definition \ref{4def1.3}, the main result in Theorem \ref{4t0} is immediate from the proposition below. The proof will be patterned after that of Proposition 5.5 in \cite{LPZ14}.

\begin{proposition}\label{4p0.5} 
Let $d=2$ or $d=3$. There exist some constants  $\theta_d>0$ and $K_{\ref{4p0.5}}(d)>0$ so that for all $R>K_{\ref{4p0.5}}(d)$, we have the SIR epidemic process $\eta$ starting from $(\{0\},\emptyset)$ satisfies
\[
\P(\eta_n\neq \emptyset, \forall n\geq 0) >0.
\]
\end{proposition}

\begin{mydef}
For any constant $m>0$ and $\mu \in M_F(\Z^d_R)$, we say $\mu$ is $m$-${\bf admissible}$ if  
\begin{align}
\mu(g_{u,d})\leq m\frac{R^{d-1}}{\theta^{1/4}}, \forall u\in \R^d,
\end{align}
where $g_{u,d}$ is as in \eqref{4e10.31}.
 \end{mydef}
\no  For any $\mu\in M_F(\Z_R^d)$ and $K\subseteq \R^d$, write $\mu|_K(\cdot)=\mu(\cdot \cap K)$ for the measure $\mu$ restricted to $K$.
In the setting of Corollary \ref{4c0.1}, we see that with high probability, $Z_{T_\theta^R}|_{Q_{4R_\theta}(0)}$ is $m_{\ref{4c0.1}}$-admissible. Since $Z_{T_\theta^R}$ dominates $\eta_{T_\theta^R}$, it follows that $\eta_{T_\theta^R}|_{Q_{4R_\theta}(0)}$ will be $m_{\ref{4c0.1}}$-admissible as well.

Let $Y=(Y_n,n\geq 0)$ be a stochastic process taking values in the set of finite subsets of $\Z_R^d$. As usual we write $Y_n(x)=1(x\in Y_n), \forall x\in \Z_R^d$ so that $Y_n \in M_F(\Z^d_R)$ for all $n$.  Recall the grid $\Gamma$ defined in Section \ref{4s1.2}. Choose $T\geq 100$ as in \eqref{4ea10.04}. For any $x\in \Gamma$, any $m,M, K, \chi>0$, $\theta\geq 100$ and $R\geq 4\theta$, define
\begin{align}\label{4e1.18}
&F_1(Y; M,x)=\{\text{Supp}(\sum_{n=0}^{T_\theta^R} Y_n) \subseteq Q_{M R_\theta }(xR_\theta)\};\nn\\
&F_2(Y; \chi)=\{\sum_{n=0}^{T_\theta^R} Y_n(\cN(y)) \leq \chi R, \forall y\in \Z_R^d\};\nn\\
&F_3(Y; K,x)=\{\hat{Y}_{T_\theta^R}^K(Q_{R_\theta}(y R_\theta))\geq  |Y_0|, \forall y\in \cA(x)\};\nn\\
&F_4(Y; m,x)=\{{Y}_{T_\theta^R}|_{Q_{ R_\theta }(y R_\theta)}\text{ is } m \text{-admissible for all } y\in \cA(x)\}.
\end{align}
Here $\hat{Y}_{T_\theta^R}^K$ is the ``thinned'' version of ${Y}_{T_\theta^R}$ such that $|\hat{Y}_{T_\theta^R}^K \cap Q(y)|\leq K \beta_d(R)$ for any $y\in \Z^d$, where $ \beta_d(R)$ is defined in \eqref{4ea10.45}.  By using Propositions \ref{4p4}, \ref{4p1}, \ref{4p2} and Corollary \ref{4c0.1}, we show below that the above conditions will hold with high probability for $Y=\eta$, the SIR epidemic. 
Define
\begin{align}\label{4ea3.2}
&\widetilde{M}=\widetilde{M}(M,\theta)=[M\sqrt{\log f_d(\theta)}]+1, \text{ and }\nn\\
&\kappa=\kappa(\chi,\widetilde{M})=(4\widetilde{M}+4)^2 \cdot \chi.
\end{align}

\begin{proposition}\label{4p0.7}
For any $\eps_0\in (0,1)$ and $T\geq \eps_0^{-1}+100$ as in \eqref{4ea10.04}, there exist positive constants $\theta, m,M,K, \chi$ depending only on $T, \eps_0$, and \mbox{$C_{\ref{4p0.7}}(\theta, m,M,K, \chi)\geq 4\theta$} such that for any $R\geq C_{\ref{4p0.7}}$, any finite $\eta_0$ as in \eqref{4ea10.23} which is $m$-admissible, and any finite $\rho_0$ disjoint from $\eta_0$ with 
\begin{align}\label{4ea3.4}
|\rho_{0}\cap \cN(y) | \leq \kappa R,\quad  \forall y\in \Z_R^d,
\end{align}
 the SIR epidemic process $\eta$ starting from $(\eta_0,\rho_0)$ satisfies
\[
\P\Big(F_1(\eta; \widetilde{M},0) \cap F_2(\eta; \chi)\cap F_3(\eta; K,0)\cap F_4(\eta; m, 0) \Big)\geq 1-7\eps_0.
\]
\end{proposition}
\begin{proof}
Fix $\eps_0 \in (0,1)$ and $T\geq \eps_0^{-1}+100$ satisfying \eqref{4ea10.04}. Let $\theta>\max\{\theta_{\ref{4p4}}, \theta_{\ref{4p1}}, \theta_{\ref{4p2}}, \theta_{\ref{4c0.1}}\}$ and $m=m_{\ref{4c0.1}}(\eps_0, T)$. We will choose other constants $M, K, \chi$ along the proof. Set $C_{\ref{4p0.7}}=\max\{C_{\ref{4p4}}, C_{\ref{4p1}}, C_{\ref{4p2}}, C_{\ref{4c0.1}}\}$ and fix $R\geq C_{\ref{4p0.7}}$. Let $\eta_0$ be as in \eqref{4ea10.23} such that $\eta_0$ is $m$-admissible. Set $Z_0=\eta_0$. Use Lemma \ref{4l1.5} to see that there is some BRW $(Z_n)$ starting from $Z_0$ such that $Z_n$ dominates $\eta_n$ for all $n\geq 0$. A brief plan for the proof is as follows: we apply Proposition \ref{4p1} with $(Z_n)$ to show that with high probability (w.h.p.) $F_1(\eta; \widetilde{M},0)$ holds. Next, on the event $F_1(\eta; \widetilde{M},0)$, we use Proposition \ref{4p2} with $(Z_n)$ to get w.h.p. $F_2(\eta; \chi)$ holds; on $F_1(\eta; \widetilde{M},0) \cap F_2(\eta; \chi)$, we prove w.h.p. $F_3(\eta; K,0)$ holds by Proposition \ref{4p4}. Finally we finish the proof by showing that w.h.p. $F_4(\eta;m, 0)$ holds by applying Corollary \ref{4c0.1} with $(Z_n)$. \\

\no (i) Since $Z_0=\eta_0$ is as in \eqref{4ea10.23}, we have $Z_0$ satisfies the assumption of Proposition \ref{4p1}. By letting $M=M_{\ref{4p1}}(\eps_0, T)$, we may apply Proposition \ref{4p1} to get for $\theta\geq \theta_{\ref{4p1}}$ and $R\geq C_{\ref{4p0.7}} \geq C_{\ref{4p1}}$, with probability larger than $1-\eps_0$ we have
\[\text{Supp}(\sum_{n=0}^{T_\theta^R} Z_n) \subseteq Q_{M \sqrt{\log f_d(\theta)}R_\theta} (0)\subseteq Q_{\widetilde{M} R_\theta} (0),\] and so $F_1(\eta; \widetilde{M},0)$ holds since $Z_n$ dominates $\eta_n$ for all $n$. This gives
\begin{align}\label{4ea4.11}
\P(F_1(\eta; \widetilde{M}, 0))\geq 1-\eps_0.
\end{align}

\no (ii) Next, recall $m=m_{\ref{4c0.1}}(\eps_0, T)$. We have the $m$-admissible $Z_0=\eta_0$ (as in \eqref{4ea10.23}) satisfies the assumption of Proposition \ref{4p2}. By letting $\chi=\chi_{\ref{4p2}}(\eps_0, T, m)$, we get for $\theta\geq \theta_{\ref{4p2}}$ and $R\geq C_{\ref{4p0.7}}\geq C_{\ref{4p2}}$, with probability larger than $1-\eps_0$ we have
\begin{align}\label{4ea4.9}
\sum_{n=0}^{T_\theta^R} Z_n(\cN(x)) \leq  \chi {R}, \quad \forall x\in \Z^d_R \cap Q_{2M_{\ref{4p1}} \sqrt{\log f_d(\theta)}R_\theta} (0).
\end{align}
Recall $M=M_{\ref{4p1}}\geq 100$ and $\theta\geq 100$. So we have $\widetilde{M}< 2M \sqrt{\log f_d(\theta)}$ by \eqref{4ea3.2}. Since $Z_n$ dominates $\eta_n$ for all $n$, on the event $F_1(\eta; \widetilde{M},0)$, we conclude from \eqref{4ea4.9} that
\begin{align}\label{4ea3.3}
\sum_{n=0}^{T_\theta^R} \eta_n(\cN(y)) \leq   \chi {R}, \quad \forall y\in \Z^d_R.
\end{align}
Let $A$ denote the event in \eqref{4ea4.9}. Then $\P(A)\geq 1-\eps_0$ and it follows that 
\begin{align}\label{4ea4.12}
\P(F_2(\eta; \chi) \cap F_1(\eta; \widetilde{M},0))\geq& \P(A \cap F_1(\eta; \widetilde{M},0))\nn\\
\geq& 1-\P(A^c)- \P(F_1(\eta; \widetilde{M},0)^c)\geq 1-2\eps_0,
\end{align}
where in the last inequality we have used \eqref{4ea4.11}.

\no (iii) On the event $F_2(\eta; \chi)$, we may use the assumption on $\rho_0$ in \eqref{4ea3.4} to conclude for all $y\in \Z_R^d$,
\begin{align}\label{4ea3.35}
|\rho_{T_\theta^R}\cap \cN(y) |\leq |\rho_{0}\cap \cN(y) |+\sum_{n=0}^{T_\theta^R} \eta_n(\cN(y)) \leq (\kappa+\chi) R. 
\end{align}
 Let $\kappa'=\kappa+\chi$ and set $N(\kappa')=\{|\rho_{T_\theta^R}\cap \cN(y) | \leq \kappa' R, \forall y\in \Z_R^d\}$. It follows that 
\begin{align}\label{4ea4.10}
\P(N(\kappa'))\geq \P(F_2(\eta; \chi))\geq \P(F_1(\eta; \widetilde{M},0)\cap F_2(\eta; \chi))\geq 1-2\eps_0,
\end{align}
where the last inequality is by \eqref{4ea4.12}.
Let $K=K_{\ref{4p4}}(T,\eps_0, \kappa')$. Apply Proposition \ref{4p4} to see for $\theta\geq \theta_{\ref{4p4}}$ and $R\geq C_{\ref{4p0.7}}\geq C_{\ref{4p4}}$, we have 
\begin{align}
\P(F_3(\eta; K, 0)^c \cap N(\kappa'))\leq \eps_0.
\end{align} 
Therefore we get
\begin{align}
1-\eps_0\leq \P(F_3(\eta; K, 0)\cup N(\kappa')^c)&\leq \P(F_3(\eta; K, 0))+\P(N(\kappa')^c)\nn\\
&\leq \P(F_3(\eta; K, 0))+2\eps_0,
\end{align}
where the last inequality is by \eqref{4ea4.10}. This gives
\begin{align}\label{4ea4.13}
\P(F_3(\eta; K, 0))\geq 1-3\eps_0
\end{align}

\no (iv) Turning to $F_4(\eta; m, 0)$, recall we set $m=m_{\ref{4c0.1}}(\eps_0, T)$. Since $Z_0=\eta_0$ is as in \eqref{4ea10.23}, we may apply Corollary \ref{4c0.1} to get for $\theta\geq \theta_{\ref{4c0.1}}$ and  $R\geq C_{\ref{4p0.7}}\geq C_{\ref{4c0.1}}$, with probability larger than $1-2\eps_0$ we have ${Z}_{T_\theta^R}|_{Q_{ 4R_\theta }(0)}$ is  $m$-admissible. Since $Q_{ R_\theta }(y R_\theta) \subseteq Q_{ 4R_\theta }(0)$ for all $y\in \cA(0)$, it follows that ${\eta}_{T_\theta^R}|_{Q_{ R_\theta }(y R_\theta)}$ is also $m$-admissible and so $F_4(\eta; m, 0)$ holds. We conclude
\begin{align}\label{4ea4.14}
\P(F_4(\eta; m, 0))\geq 1-2\eps_0.
\end{align}
 
Now we have \eqref{4ea4.12}, \eqref{4ea4.13}, \eqref{4ea4.14} hold and so
\begin{align}
\P(F_4(\eta; m, 0)\cap F_3(\eta; K, 0)\cap  F_2(\eta; \chi) \cap F_1(\eta; \widetilde{M},0))\geq 1-7\eps_0.
\end{align}
The proof is then complete.
\end{proof}

We are ready to give the proof of Proposition \ref{4p0.5}, thus finishing the proof of the main result Theorem \ref{4t0}.
\begin{proof}[Proof of Proposition \ref{4p0.5}]
By a trivial union inclusion and translation invariance, it suffices to prove the survival of the SIR epidemic process $\eta$ starting from $(\eta_0, \emptyset)$ for some finite $\eta_0\subseteq \Z_R^d$.
Let $\eps_0\in (0,1)$ be small so that any $3$-dependent oriented site percolation process on $\Z_+^2$ with density at least $(1-14\eps_0)$ has positive probability of percolation. For this $\eps_0$, we fix $T\geq \eps_0^{-1}+100$ satisfying \eqref{4ea10.04}. Let $\theta, m,M,K, \chi>0$ be as in Proposition \ref{4p0.7} and let $R\geq C_{\ref{4p0.7}}$. Set $\rho_0=\emptyset$ and choose a finite set $\eta_0$ such that it satisfies the hypothesis of Proposition \ref{4p0.7}. The existence of such $\eta_0$ is immediate from Proposition \ref{4p4} and Corollary \ref{4c0.1}.  Let $\eta=(\eta_n,n\geq 0)$ be a usual SIR epidemic starting from $(\eta_0, \emptyset)$. Since our initial infection set $\eta_0$ is finite, one can check by \eqref{4e10.14} that
\begin{align}\label{4equiv}
 \cup_{n=0}^\infty \eta_n \text{ is not a compact set} \Rightarrow\eta_n\neq \emptyset, \forall n\geq 0 .
\end{align}

Write $\rho_\infty=\cup_{n=0}^\infty \eta_n$. By slightly abusing the notation, we let $\rho_\infty$ be a measure on $\Z_R^d$ such that  $\rho_\infty(x)=1(x\in \rho_\infty)$ for $x\in \Z_R^d$. Note we also write $\eta_n$ for the measure $\eta_n(x)=1(x\in \eta_n)$. By \eqref{4equiv}, it suffices to show that with positive probability, the measure $\rho_\infty$ is not compactly supported. To do this, we will produce a random set $\Omega$ on the two-dimensional grid $\Gamma$ such that 
\begin{align}\label{4e10.34}
\begin{dcases}
\text{ (i) }  &\rho_\infty (Q_{R_\theta}(x R_\theta))>0 \text{ for all $x\in \Omega$};\\
\text{ (ii) } & \Omega \text{ is infinite with positive probability.}
\end{dcases}
\end{align}

Before describing the algorithm used to construct $\Omega$, we first introduce some notations. We will frequently use the stopping rule $\tau=\tau(Y, x)$ defined as follows: for $x\in \R^d$ and for the stochastic process $Y=(Y_n,n\geq 0)$ taking values in the set of finite subsets of $\Z^d_R$, let
\begin{align}
\tau(Y,x)=\inf\Big\{n\geq 0: &\sup_{y\in \Z_R^d} \sum_{k=0}^n Y_k(\cN(y)) > \chi R \text{ or } \nn\\
&\text{Supp}(\sum_{k=0}^n Y_k) \nsubseteq Q_{\widetilde{M}R_\theta}(xR_\theta)\Big\} \wedge T_\theta^R.
\end{align}
Recall that $\Gamma=\{x(1), x(2), \cdots\}$ where $0 = x(1) \prec x(2) \prec \cdots$ with the total order defined by \eqref{4ea1.2}. Set $\tau_0=0$, $\mu_0=\eta_0$ and $\nu_0=\emptyset$.  Starting from $x(1)=0$, following the total order we will define stopping times $\tau_i$ using $\tau(Y,x)$ above. Let $\eta^*$ be the SIR epidemic with immigration at times $\{\tau_i, i\geq 0\}$ satisfying \eqref{4e10.18}. Below we will choose $\cG_{\tau_i}$-measurable finite sets $\mu_i,\nu_i$ in a way such that $|\mu_i|=|\eta_0|$ and $(\mu_i \cup \nu_i)\subseteq (\eta_{\tau_i}^*\cup \nu_{i-1})$  for all $i\geq 1$. Then we may apply Proposition \ref{4p2.3} to couple $\eta$ with $\eta^*$ so that $\cup_{k=0}^n \eta_k^*\subseteq \cup_{k=0}^n \eta_k$ for all $n\geq 0$.

  For each $i\geq 1$, we let $Y_0^i=\mu_{i-1}$ and $Y_{n}^i=\eta_{n+\tau_{i-1}}^*$ for $n\geq 1$ to denote the epidemic process $\eta^*$ between $\tau_{i-1}$ and $\tau_i$. Then $Y^i$ is a usual SIR epidemic starting from $(\mu_{i-1}, \rho_{i-1}^{0,*})$.
Define the ``good'' events
\begin{align}
G^i=F^1(Y^i; \widetilde{M}, x(i)) &\cap F^2(Y^i; \chi) \cap F^3(Y^i; K,x(i)) \cap F^4(Y^i; m, x(i)).
\end{align}
On the good event, $F^1$ and $F^2$ ensures that before time $T_\theta^R$, the epidemic $Y^i$ has not accumulated the recovered set with more than $\chi R$ sites in each unit cube $\cN(y)$ and has not escaped $Q_{\widetilde{M} R_\theta}(x(i)R_\theta)$; $F^3$ guarantees that at time $T_\theta^R$,  the epidemic has spread at least $|Y_0^i|=|\mu_{i-1}|=|\eta_0|$ infected sites in all the cubes $Q_{R_\theta}(yR_\theta)$ for $y\in \cA(x(i))$ after thinning; finally $F^4$ is a technical restriction needed for the proof of Proposition \ref{4p0.7}, the $m$-admissible property. This also allows us to carefully choose $\{Y_0^{i}\}$ so that the good events will propagate with high probability. 

The recovered set $\rho_0^{i,*}$ will determined as follows: $\rho_0^{0,*}\equiv \emptyset$, and for $i\geq 1$,
\begin{align}
\rho_0^{i,*}=\rho_0^{i-1,*} \bigcup \bigcup_{n=0}^{\tau_{i}-\tau_{i-1}-1} Y^i_n.
\end{align}
Recall $Y_0^i=\mu_{i-1}$ and $Y_{n}^i=\eta_{n+\tau_{i-1}}^*$ for $n\geq 1$. One can easily check by induction that $\rho_0^{i,*}$ is the total recovered set of $\eta^*$ up to time $\tau_i$, i.e. $\rho_0^{i,*}=\rho_{\tau_i}^*=\cup_{n=0}^{\tau_i-1} \eta_n^*$. Below we will set $\tau_{i}-\tau_{i-1}$ to be $0$ or $\tau(Y^i, x(i))$ for different cases. In either case, one may check by induction that $\tau_i$ is a stopping time with respect to $(\cG_n)$ if $\tau_{i-1}$ is.

 If $\tau_{i}-\tau_{i-1}=\tau(Y^i, x(i))$, then the definition of $\tau(Y^i, x(i))$ gives that
\[
\Big|\bigcup_{n=0}^{\tau_{i}-\tau_{i-1}-1} Y^i_n \cap \cN(y)\Big|=\sum_{n=0}^{\tau(Y^i, x(i))-1} Y^i_n(\cN(y))\leq \chi R \cdot 1_{\{\cN(y) \cap Q_{\widetilde{M} R_\theta}(x(i)R_\theta)\neq \emptyset\}}, \quad \forall y\in \Z^d_R.
\]
The case for $\tau_{i}-\tau_{i-1}=0$ is trivial. So it follows that for each $i\geq 1$,
\begin{align}\label{4ea3.9}
|\rho_0^{i,*} \cap \cN(y)|\leq \chi R \cdot \sum_{j=1}^i  1_{\{\cN(y) \cap Q_{\widetilde{M} R_\theta}(x(j)R_\theta)\neq \emptyset\}}, \quad \forall y\in \Z^d_R.
\end{align}
Notice that each unit cube $\cN(y)$ has non-empty intersection with at most $(4\widetilde{M}+4)^2$ cubes of the form $Q_{\widetilde{M} R_\theta}(x(j)R_\theta)$ for $x(j)$ in the $2$-dimensional grid $\Gamma$. Hence for any $i\geq 1$, by \eqref{4ea3.9} we have
\begin{align}\label{4ea3.8}
|\rho_0^{i,*} \cap \cN(y)| &\leq \chi R \cdot \sum_{j=1}^\infty  1_{\{\cN(y) \cap Q_{\widetilde{M} R_\theta}(x(j)R_\theta)\neq \emptyset\}}\nn\\
& \leq \chi R \cdot (4\widetilde{M}+4)^2=\kappa R, \quad \forall y\in \Z^d_R,
\end{align}
where the last equality is from \eqref{4ea3.2}.
Therefore the assumption \eqref{4ea3.4} on $\rho_0$ of Proposition \ref{4p0.7} will always be satisfied. For notation ease, we write
\[
\widetilde{Q}(x)=Q_{R_\theta}(x R_\theta) \text{ for any } x\in \Z^d.
\]

Now we are ready to introduce the algorithm. We start with $x(1)=0$. Set $\tau_0=0$, $\mu_0=\eta_0$, $\nu_0=\emptyset$ and $\rho_0^{*}=\rho_0^{0,*}=\emptyset$. We first let $\eta^*$ proceed as a usual SIR epidemic starting from $(\mu_0, \rho_0^{0,*})$. Let $Y_{0}^1=\mu_0$ and $Y_{n}^1=\eta_{n+\tau_{0}}^*$ for $n\geq 1$. Let $\tau_1=\tau(Y^1; x(1))$. By Proposition \ref{4p0.7}, the good event $G^1$ occurs with probability $\geq 1-7\eps_0$. If the good event occurs, we have $\tau_1=T_\theta^R$ and we change the status of site $x(1)=0$ to be occupied. Since $F^3(Y^1; K,x(1))$ holds, we have
\begin{align}\label{4ea3.60}
 |\hat{\eta}^{*,K}_{\tau_1}\cap \widetilde{Q}(z)|= |\hat{Y}^{1}_{T_\theta^R}( \widetilde{Q}(z))|\geq |Y_{0}^1|=|\mu_0|=|\eta_0| \text{ for all } z\in \cA(x(1)).
\end{align}
Totally order $\Z_R$ by $\{0,1/R,-1/R,2/R,-2/R,\cdots\}$ and then totally order $\Z_R^d$ lexicographically. By \eqref{4ea3.60} we may choose $\widetilde{\eta}^{*,K}_{\tau_1}\subseteq \hat{\eta}^{*,K}_{\tau_1}$ following the above total order on $\Z_R^d \cap \hat{\eta}^{*,K}_{\tau_i}$ such that
\begin{align}\label{4ea3.6}
 |\widetilde{\eta}^{*,K}_{\tau_1}\cap \widetilde{Q}(z)|=|Y_{0}^1|=|\mu_0|=|\eta_0| \text{ for all } z\in \cA(x(1)).
\end{align}
Recall that we also obtain the ``thinned'' version $\hat{\eta}^{*,K}_{\tau_1}$ from ${\eta}^{*}_{\tau_1}$ in a deterministic way in Proposition \ref{4p4}. Since ${\eta}^{*}_{\tau_1}\in \cG_{\tau_1}$, it follows that $\hat{\eta}^{*,K}_{\tau_1}\in \cG_{\tau_1}$ and hence $\widetilde{\eta}^{*,K}_{\tau_1}\in \cG_{\tau_1}$.

Next, $F^4$ ensures that for each $z\in \cA(x(1))$, we have $\widetilde{\eta}^{*,K}_{\tau_1}|_{\widetilde{Q}(z)}$ is $m$-admissible. Further define
\begin{align}\label{4ea3.7}
w_1=
\begin{dcases}
\bigcup_{z\in \cA(x(1))} (\widetilde{\eta}^{*,K}_{\tau_1}  \cap \widetilde{Q}(z)), & \text{ if } G^1 \text{ occurs,}\\
\emptyset,& \text{ otherwise. }
\end{dcases}
\end{align}
In this way if $G^1$ occurs, then $w_1$ has exactly $|\eta_0|$ infected sites in each cube $\widetilde{Q}(z)$ for $z\in \cA(x(1))$ and the assumption of $\eta_0$ in Proposition \ref{4p0.7} will be satisfied.

We now work with site $y=x(i)$ for $i\geq 2$. 

\no {\bf Case I.}  If $y=x(i)$ is an immediate offspring of some occupied site $x(j)$ with $j<i$ (i.e. $x(i)\in \cA(x(j))$ and the good event $G^j$ occurs). Define
\[
(\mu_{i-1}, \nu_{i-1})=(w_{i-1}\cap \widetilde{Q}(y), w_{i-1} \cap \widetilde{Q}(y)^c).
\]
By \eqref{4ea3.6} and \eqref{4ea3.7}, we have $\mu_{i-1}=\widetilde{\eta}^{*,K}_{\tau_j} \cap \widetilde{Q}(y)$ with total mass $|\mu_{i-1}|=|\eta_0|$. Since $G^j$ occurs, we have $\mu_{i-1}$ is $m$-admissible and hence satisfies the assumption of $\eta_0$ in Proposition \ref{4p0.7}. Let $Y_{0}^i=\mu_{i-1}$ and $Y_{n}^i=\eta_{n+\tau_{i-1}}^*$ for $n\geq 1$ so that $Y^i$ is a usual SIR epidemic starting from $(\mu_{i-1}, \rho_{i-1}^{0,*})$. Set $\tau_i=\tau_{i-1}+\tau(Y^i, x(i))$. By Proposition \ref{4p0.7} with a spatial translation, the good event $G^i$ occurs with probability $\geq 1-7\eps_0$. In this case, we change the status of site $y=x(i)$ to occupied. Again since $F^3(Y^i; K,x(i))$ holds, as in \eqref{4ea3.6} we may choose some $\cG_{\tau_i}$-measurable set $\widetilde{\eta}^{*,K}_{\tau_i} \subseteq \hat{\eta}^{*,K}_{\tau_i}$ such that
\begin{align}\label{4ea3.42}
 |\widetilde{\eta}^{*,K}_{\tau_i}\cap \widetilde{Q}(z)|=|Y_{0}^i|=|\mu_{i-1}|=|\eta_0| \text{ for all } z\in \cA(x(i)).
\end{align}
Moreover, $F^4$ gives that for each $z\in \cA(x(i))$, we have $\widetilde{\eta}^{*,K}_{\tau_i}|_{\widetilde{Q}(z)}$ is $m$-admissible. Further we define
 \begin{align*}
w_i=
\begin{dcases}
\nu_{i-1}\bigcup \bigcup_{z\in \widetilde{\cA}(y)} (\widetilde{\eta}^{*,K}_{\tau_i} \cap \widetilde{Q}(z)), & \text{ if } G^i \text{ occurs,}\\
\nu_{i-1},& \text{ otherwise, }
\end{dcases}
\end{align*}
where
\[
\widetilde{\cA}(y)=\{ z\in \cA(y): z\notin \cA(u) \text{ for } u \text{ which is occupied and} \prec y \}.
\]
One can check that $\widetilde{\cA}(y)$ will contain at least one member of $\cA(y)$. The definition of $\widetilde{\cA}(y)$ is to avoid duplicate of particles on $\widetilde{Q}(z)$ for $z\in \cA(y)$ as $\{\nu_{i-1}\}$ will carry and freeze the infected sites in each cube $\widetilde{Q}(z)$ until we reach it.

\no {\bf Case II.}  Site $y$ is not an immediate offspring of any occupied site. Then we set $\tau_i=\tau_{i-1}$, $(\mu_{i-1}, \nu_{i-1})=(\emptyset,w_{i-1})$ and  $w_i=w_{i-1}$. In this case, we simply skip the cube $\widetilde{Q}(y)$ and move to the next site in our total ordering of $\Gamma$.

In either case, we will move to site $x(i+1)$ at time $\tau_i$. The definitions of $\{w_i\}$, $\{\nu_i\}$ and $\{\mu_i\}$ ensure that if $y=x(k)$ for some $k\geq 2$ is an immediate offspring of some occupied site, then the infected set $\mu_{k-1}$ contained in the cube $\widetilde{Q}(y)$ will satisfy the assumption of $\eta_0$ in Proposition \ref{4p0.7}. Restart the SIR epidemic with $\mu_{k-1}$ so that the good event $G^k$ will occur with high probability and so $y=x(k)$ will be occupied with high probability as well.\\

Since $\mu_i \cup \nu_i=w_i \subseteq (\eta^*_{\tau_i} \cup \nu_{i-1})$ by construction, we have such defined $\mu_i$, $\nu_i$ and $\tau_i$ satisfy the conditions of Proposition \ref{4p2.3}. Therefore the processes $\eta$ and $\eta^*$ can be coupled such that
\[
\cup_{k=0}^n \eta_k^* \subseteq \cup_{k=0}^n \eta_k \text{ for any } n\geq 0.
\]
In particular, since $\rho_0^*=\rho_0=\emptyset$, if we let $\rho_\infty^*=\cup_{k=0}^\infty \eta_k^*$, we have $\rho^{*}_\infty \subseteq \rho_\infty$. Again we abuse the notation $\rho_\infty^{*}$ for the measure $\rho_\infty^{*}(x)=1(x\in \rho_\infty^{*}), \forall x\in \Z_R^d$. If we let $\Omega$ be the set of all occupied sites, then the construction above implies for any $x=x(i) \in \Omega$, there is some occupied $x(j)$ with $j<i$ such that $x(i)\in \cA(x(j))$ and the good event $G^j$ occurs. Therefore $F^3(Y^{j}; K, x(j))$ guarantees that the infection from $\widetilde{Q}(x(j))$ will spread enough mass to its adjacent cube $\widetilde{Q}(x(i))$ so that $\hat{\eta}_{\tau_j}^{*,K}(\widetilde{Q}(x(i)))\geq |\eta_0|$. It follows that 
\[
\rho_\infty(Q_{R_\theta}(x R_\theta))\geq  \rho_\infty^{*}(Q_{R_\theta}(x R_\theta))=\rho_\infty^{*}(\widetilde{Q}(x(i)))\geq  \eta_{\tau_j}^{*}(\widetilde{Q}(x(i)))\geq |\eta_0|>0,
\]
and hence $\Omega$ satisfies condition (i) in \eqref{4e10.34}.\\

To show that $\Omega$ is infinite with positive probability, we define a $3$-dependent oriented site percolation on $\Gamma$ with density at least $(1-14\eps_0)$ following \cite{LPZ14}. Recall we have picked $\eps_0 \in (0,1)$ small so that such an oriented site percolation has positive probability of percolation from the origin.  For each $x\in \Gamma$, if $x$ is occupied, then $\xi(x)=1$ if both $y\in \cA(x)$ are occupied, and set $\xi(x)=0$ otherwise; if $x$ is vacant, then we let $\xi(x)$ be Bernoulli $(1-14\eps_0)$ independent of everything else. We know that the origin and both $y\in \cA(0)$ are occupied with positive probability and so $\xi(0)=1$ with positive probability. Assuming $\xi(0)=1$, we have both $y\in \cA(0)$ are occupied. By induction one may conclude that $\Omega$ contains the collection of sites reachable from the origin. In other words, if percolation to infinity occurs, we have $\Omega$ is infinite. It remains to show that such defined site percolation is a $3$-dependent site percolation with density at least $1-14\eps_0$, i.e. for any $n\geq 1$ and any $1\leq i_1<\cdots<i_n$ such that $\|x(i_j)-x(i_k)\|_1\geq 3$ for any $j\neq k$,
\begin{align}
P(\xi(x(i_j))=0, \forall 1\leq j\leq n ) \leq (14\eps_0)^{n}.
\end{align}
Recall that we have let $\xi(x)$ be Bernoulli $(1-14\eps_0)$ independent of everything else when $x$ is vacant. By using the total probability formula and conditioning on whether $x(i_j)$ is occupied or vacant, it suffices to show that
\begin{align}\label{4ea8.1}
P\Big(\xi(x(i_j))=0, \forall 1\leq j\leq n |\text{all $x(i_j)'s$ are occupied}\Big) \leq (14\eps_0)^{n}.
\end{align}

We prove the above by induction. When $n=1$, if $x:=x(i_1)$ is occupied, we have each $y\in \cA(x)$ is occupied with probability larger than $1-7\eps_0$, and so $\xi(x)=1$ occurs with probability larger than $1-14\eps_0$ by letting both $y\in \cA(x)$ be occupied. Hence \eqref{4ea8.1} holds for $n=1$.

Turning to induction step, for each $m\geq 0$, we let $\cH_m=\cG_{\tau_m}$ so that the good event $G^i \in \cH_i$ for all $i\geq 1$. Hence the random variable $\xi(x(i))$ is measurable with respect to $\cH_\ell$ where $\ell$ is the index of the second $y\in \cA(x(i))$. Let $\ell_j$ be the index of the second $y\in \cA(x(i_j))$. Since $\|x(i_j)-x(i_k)\|_1\geq 3$ for any $j\neq k$, we conclude 
\[
\ell_j<\ell_n-2, \quad \forall 1\leq j\leq n-1.
\]
Hence by conditioning on $\cH_{\ell_n-2}$, we reduce \eqref{4ea8.1} to the $n=1$ case and so by induction hypothesis the conclusion follows.
\end{proof}

\section{Preliminaries for branching random walk}\label{4s3}

\subsection{Support propagation of branching random walk}

We first give the proof of Proposition \ref{4p1}. Let ${U}$ be a super-Brownian motion with drift $1$, that is, the solution to the martingale problem $(MP)_1$ in \eqref{4e10.25}. Similarly we let $X$ be a super-Brownian motion with drift $\theta$. By using the scaling of SBM from Lemma 2.27 of \cite{LPZ14}, we have
\begin{align}\label{4e2.11}
 \int \psi(x) {U}_t(dx)\overset{\text{law}}{=} \theta\int \psi(\sqrt{\theta} x) {X}_{t/\theta}(dx), \quad \forall t\geq 0.
\end{align}
In particular, if we use \eqref{4e2.11} to define $X$ and $U$ on a common probability space, then it follows that ${U}_0(1)=\theta {X}_0(1)$ and for any $t\geq 0$, $\text{Supp}({U}_t)=\sqrt{\theta}\ \text{Supp}({X}_{t/\theta})$ where $kA=\{kx: x\in A\}$ for $k\in \R$ and $A\subseteq \R^d$. The lemma below is an easy consequence of Lemma 3.12 in \cite{LPZ14}.

\begin{lemma}\label{4l3.21}
For any $\eps_0\in (0,1)$ and $T\geq 100$, there exists some constant $M_{\ref{4l3.21}}=M_{\ref{4l3.21}}(\eps_0,T)\geq 100$ such that for any $\theta\geq 100$, any $\lambda\geq e$ and any $X_0\in M_F(\R^d)$ satisfying $|X_0|=\lambda/\theta$ and $\text{Supp}({X}_0)\subseteq Q_{\sqrt{3/\theta}}(0)$, if $X$ is a super-Brownian motion with drift $\theta$ starting from $X_0$, then
\[
\P^{X_0}\Big(\text{Supp}\Big(\int_0^{2T/\theta} X_s ds\Big)\subseteq Q_{M_{\ref{4l3.21}}\sqrt{(\log \lambda)/\theta}}(0)\Big)\geq 1-\frac{\eps_0}{8}.
\]
\end{lemma}
\begin{proof}
Fix $\eps_0\in (0,1)$ and $T\geq 100$. Let $\theta\geq 100$, $\lambda\geq e$ and choose $X_0\in M_F(\R^d)$ such that $|X_0|=\lambda/\theta$ and $\text{Supp}({X}_0)\subseteq Q_{\sqrt{3/\theta}}(0)$. If $X$ is a super-Brownian motion with drift $\theta$ starting from $X_0$, then we may use \eqref{4e2.11} to define a super-Brownian motion $U$ with drift $1$ starting from $U_0$ where $U_0$ satisfies
\begin{align}\label{4ea4.21}
|{U}_0|=\theta |{X}_0|=\lambda\text{ and }  \text{Supp}({U}_0)=\sqrt{\theta}\ \text{Supp}({X}_{0})\subseteq Q_{\sqrt{3}}(0).
\end{align}
Now apply Lemma 3.12 of \cite{LPZ14} with above $U_0$ to see that there is some $M=M(T,\eps_0)\geq 100$ so that 
\begin{align}\label{4ea4.2}
 \P^{U_0}\Big(\text{Supp}\Big(\int_0^{2T} U_s ds\Big)\subseteq Q_{M\sqrt{\log \lambda}}(0)\Big)\geq 1-\frac{\eps_0}{8}.
\end{align}
The proof of Lemma 3.12 in \cite{LPZ14} goes back to Theorem A of \cite{Pin95}, which allows us to accommodate a slightly different assumption on $\text{Supp}({U}_0)$ as in \eqref{4ea4.21}. Use  \eqref{4e2.11} and \eqref{4ea4.2} to conclude
\begin{align}\label{4e6.23}
&\P^{X_0}\Big(\text{Supp}\Big( \int_0^{2T/\theta} X_s ds\Big)\subseteq Q_{M\sqrt{(\log \lambda)/\theta} }(0)\Big)\nn\\
=&\P^{U_0}\Big(\text{Supp}\Big(\int_0^{2T} U_s ds\Big)\subseteq Q_{M\sqrt{\log \lambda}}(0)\Big)\geq 1-\frac{\eps_0}{8},
\end{align}
as required.
\end{proof}
Now we are ready to prove Proposition \ref{4p1}.

\begin{proof}[Proof of Proposition \ref{4p1}]
Fix $\eps_0\in (0,1)$ and $T\geq 100$.  Let $\theta_{\ref{4p1}}= 100$. For any $\theta\geq \theta_{\ref{4p1}}$ and $R\geq 4\theta$, let $Z_0$ be as in \eqref{4e10.23}. Let $e_1=(1,0)$ in $d=2$ and $e_1=(1,0,0)$ in $d=3$. Set $\widetilde{R_\theta}=[R_\theta \cdot R]/R$ and define  $\widetilde{e}_1=\widetilde{R_\theta} e_1$ so that the vertex $\widetilde{e}_1$ has the largest first coordinate in $Q_{R_\theta}(0)\cap \Z_R^d$.  Let $Z$ be a branching random walk starting from $Z_0$. Define
\begin{align}\label{4ea0.1a}
\mR(Z, T_\theta^R)=\inf \Big\{K\in \R: \text{Supp}(\sum_{n=0}^{T_\theta^R} Z_n) \subseteq H_K\Big\},
\end{align}
where $H_K=\{x\in \R^d: x_1\leq K\}$.  In this way, $\mR(Z, T_\theta^R)$ characterizes the rightmost site that has been reached by $Z$ up to time $ T_\theta^R$. Next we couple $Z$ with another branching random walk $\widetilde{Z}$ starting from $\widetilde{Z}_0=|Z_0| \cdot \delta_{\widetilde{e}_1}$ so that 
\begin{align}\label{4ea0.1}
\mR(Z, T_\theta^R)\leq \mR(\widetilde{Z}, T_\theta^R),
\end{align}
where $\mR(\widetilde{Z}, T_\theta^R)$ is defined in a similar way to $\mR({Z}, T_\theta^R)$ as in \eqref{4ea0.1a} by replacing $Z$ with $\widetilde{Z}$.
This coupling could be done by simply translating all the family trees starting from the ancestors in $Z_0$ to $\widetilde{e}_1$. Since $\widetilde{e}_1$ has the largest first coordinate among all vertices located inside $\text{Supp}(Z_0)\subseteq Q_{R_\theta}(0)\cap \Z_R^d$, we have \eqref{4ea0.1} follows immediately. Let $M_{\ref{4p1}}=2M_{\ref{4l3.21}}(\eps_0,T)$. We claim that it suffices to show the following holds for all $R$ large enough:
\begin{align}\label{4ea0.2}
\P^{\widetilde{Z}_0}\Big(\text{Supp}(\sum_{n=0}^{T_\theta^R} \widetilde{Z}_n) \subseteq Q_{M_{\ref{4p1}} \sqrt{\log f_d(\theta)}R_\theta} (0) \Big)\geq 1-\frac{\eps_0}{6}.
\end{align}
To see this, by assuming \eqref{4ea0.2} we have $\mR(\widetilde{Z}, T_\theta^R)\leq M_{\ref{4p1}}\sqrt{\log f_d(\theta)}R_\theta$ holds with probability $\geq 1-\eps_0/6$. Apply \eqref{4ea0.1} to get 
\begin{align}
\P^{{Z}_0}\Big(\mR({Z}, T_\theta^R)\leq M_{\ref{4p1}}\sqrt{\log f_d(\theta)}R_\theta\Big)\geq 1-\frac{\eps_0}{6}.
\end{align}
By symmetry, we conclude
\begin{align}
\P^{{Z}_0}\Big(\text{Supp}(\sum_{n=0}^{T_\theta^R} {Z}_n) \subseteq Q_{M_{\ref{4p1}} \sqrt{\log f_d(\theta)}R_\theta} (0) \Big)\geq 1-2d\frac{\eps_0}{6}\geq 1-\eps_0,
\end{align}
as required.

It remains to prove \eqref{4ea0.2}. Recall $T_\theta^R=[TR^{d-1}/\theta]$ and $R_\theta= \sqrt{ R^{d-1}/\theta}$. Consider $\widetilde{W}_t^R$ as in \eqref{4e5.39} given by
\begin{align}\label{4ea4.22}
\widetilde{W}_t^R=\frac{1}{R^{{d-1}}}\sum_{x\in \Z^d_R} \delta_{{x}/{\sqrt{R^{d-1}/3}}} \widetilde{Z}_{[tR^{d-1}]}(x), \quad \forall t\geq 0.
\end{align}
It suffices to show that for any $R>0$ large,
\begin{align}
 &\P^{\widetilde{W}_0^R}\Big(\text{Supp}\Big(\int_0^{2T/\theta} \widetilde{W}_s^R ds\Big)\subseteq Q_{\sqrt{3}M_{\ref{4p1}}\sqrt{(\log f_d(\theta))/\theta}}(0)\Big) \geq 1-\frac{\eps_0}{6}.
\end{align}
Assume to the contrary that the above fails for some $\{\widetilde{W}_t^{R_N}, t\geq 0\}$ with $R_N\to \infty$ such that
\begin{align}\label{4ea0.3}
 &\P^{\widetilde{W}_0^{R_N}}\Big(\text{Supp}\Big(\int_0^{2T/\theta} \widetilde{W}_s^{R_N} ds\Big)\subseteq Q_{\sqrt{3}M_{\ref{4p1}}\sqrt{(\log f_d(\theta))/\theta}}(0)\Big) < 1-\frac{\eps_0}{6}, \ \forall R_N.
\end{align}
Recall $\widetilde{Z}_0=|Z_0| \cdot \delta_{\widetilde{e}_1}$. Note by the definition of $\widetilde{e}_1$ and \eqref{4e10.23}, we have 
\begin{align}
\lim_{R\to \infty} \frac{\widetilde{e}_1}{\sqrt{R^{d-1}/3}}=\sqrt{\frac{3}{\theta}} e_1\quad \text{ and }  \lim_{R\to \infty} \frac{|Z_0|}{R^{{d-1}}}=f_d(\theta)/\theta.
\end{align}
It follows that
\begin{align}
\widetilde{W}_0^R=&\frac{1}{R^{{d-1}}}\sum_{x\in \Z^d_R} \widetilde{Z}_{0}(x) \delta_{{x}/{\sqrt{R^{d-1}/3}}}\nn\\
=&\frac{|Z_0|}{R^{{d-1}}}\delta_{{\widetilde{e}_1}/{\sqrt{R^{d-1}/3}}} \to X_0=\frac{f_d(\theta)}{\theta} \delta_{\sqrt{\frac{3}{\theta}} e_1}\in M_F(\R^d).
\end{align}
Therefore by \eqref{4e10.22} we have as $R\to \infty$,
\begin{align}\label{4e10.26}
(\widetilde{W}_t^R, t\geq 0) \Rightarrow (X_t, t\geq 0) \text{ on } D([0,\infty), M_F(\R^d)),
\end{align}
where $X$ is a super-Brownian motion with drift $\theta$ starting from $X_0$. 
Apply Lemma 4.4 of  \cite{FP16} with a slight modification to see that for any $t, M>0$,
\begin{align}
&\limsup_{R\to \infty} \P^{\widetilde{W}_0^R}\Big(\text{Supp}\Big(\int_0^t \widetilde{W}_s^R ds\Big)\cap ((-M,M)^d)^c \neq \emptyset\Big)\nn\\
&\quad \leq \P^{X_0}\Big(\text{Supp}\Big(\int_0^t X_s ds\Big)\cap ((-M,M)^d)^c \neq \emptyset\Big),
\end{align}
thus giving
\begin{align}\label{4e5.38}
&\liminf_{R\to \infty} \P^{\widetilde{W}_0^R}\Big(\text{Supp}\Big(\int_0^t \widetilde{W}_s^R ds\Big)\subseteq (-M,M)^d\Big) \nn\\
&\quad \quad  \quad \geq \P^{X_0}\Big(\text{Supp}\Big(\int_0^t X_s ds\Big)\subseteq (-M,M)^d\Big).
\end{align}
Notice that $X_0=\frac{f_d(\theta)}{\theta} \delta_{\sqrt{\frac{3}{\theta}} e_1}$ will satisfy the assumption of Lemma \ref{4l3.21} since  $\lambda=f_d(\theta)\geq e$ by $\theta \geq 100$, which allows us to get
\begin{align}\label{4ea4.31}
\P^{X_0}\Big(\text{Supp}\Big(\int_0^{2T/\theta} X_s ds\Big)\subseteq Q_{M_{\ref{4l3.21}}\sqrt{(\log f_d(\theta))/\theta}}(0)\Big)\geq 1-\frac{\eps_0}{8}.
\end{align}
Apply \eqref{4e5.38} with $t=2T/\theta$, $M=2M_{\ref{4l3.21}}\sqrt{\frac{\log f_d(\theta)}{\theta}}$ and $\{R_N\}$ to see that
\begin{align}\label{4ea4.3}
 &\liminf_{{R_N}\to \infty}  \P^{\widetilde{W}_0^{R_N}}\Big(\text{Supp}\Big(\int_0^{2T/\theta} \widetilde{W}_s^{R_N} ds\Big)\subseteq \Big(-2M_{\ref{4l3.21}}\sqrt{\frac{\log f_d(\theta)}{\theta}}, 2M_{\ref{4l3.21}}\sqrt{\frac{\log f_d(\theta)}{\theta}}\Big)^d  \Big)\nn\\
& \geq  \P^{X_0}\Big(\text{Supp}\Big(\int_0^{2T/\theta} X_s ds\Big)\subseteq \Big(-2M_{\ref{4l3.21}}\sqrt{\frac{\log f_d(\theta)}{\theta}}, 2M_{\ref{4l3.21}}\sqrt{\frac{\log f_d(\theta)}{\theta}} \Big)^d \Big)\nn\\
&\geq 1-\frac{\eps_0}{8},
\end{align}
where the last inequality is by \eqref{4ea4.31}. This contradicts \eqref{4ea0.3} as we set $M_{\ref{4p1}}=2M_{\ref{4l3.21}}$. So the proof is complete.
\end{proof}

\subsection{Moments and exponential moments of branching random walk}

Let $p_1$ be a probability distribution that is uniform on $\cN(0)$: 
\begin{align}\label{4e0.01}
 p_1(x)=\frac{1}{V(R)}1(x\in \cN(0)).
\end{align}
Let $Y_1, Y_2, \cdots$ be i.i.d. random variables with distribution $p_1$ and write $S_n=Y_1+\cdots+Y_n$ for the random walk on $\Z^d_R$ starting from $0$ with step distribution $p_1$. Define
\begin{align}\label{4eb2.2}
p_n(x)=\P(S_n=x).
\end{align}
 Set $p_0(x)=\delta_0(x)$ by convention where $\delta_0(x)=1$ if $x=0$ and $\delta_0(x)=0$ if $x\neq 0$. It is easy to check by symmetry that $p_n(x)=p_n(-x)$ for any $x\in \Z_R^d$ and $n\geq 0$. We collect the properties of $p_n$ below. Their proofs are rather technical, which can be found in Appendix \ref{a3}.
\begin{proposition}\label{4p1.1}
Let $d\geq 1$. There exist constants $c_{\ref{4p1.1}}=c_{\ref{4p1.1}}(d)>0$, $C_{\ref{4p1.1}}=C_{\ref{4p1.1}}(d)>0$ and $K_{\ref{4p1.1}}=K_{\ref{4p1.1}}(d)>0$ such that the following holds for any $n\geq 1$ and $R\geq K_{\ref{4p1.1}}$.\\
(i) For any $x\in \Z^d_R$, we have
\begin{align}\label{4eb6.1}
p_n(x)\leq \frac{c_{\ref{4p1.1}}}{n^{d/2} R^d} e^{-\frac{|x|^2}{8dn}}.
\end{align} 
(ii) For any $x,y\in \Z^d_R \text{ with } |x-y|\geq 1$ and $\gamma\in (0,1]$, we have
\begin{align}\label{4eb6.2}
  |p_n(x)-{p}_n(y)|\leq  \frac{C_{\ref{4p1.1}}}{n^{d/2} R^d} \Big(\frac{|x-y|}{\sqrt{n}}\Big)^\gamma (e^{-\frac{|x|^2}{16dn}}+e^{-\frac{|y|^2}{16dn}}).
\end{align}
\end{proposition}
Throughout the rest of this paper, we will only consider $R\geq K_{\ref{4p1.1}}$ so that Proposition \ref{4p1.1} holds. Since we assume $d=2$ or $d=3$, for simplicity we will replace $8d$ with $32$ in \eqref{4eb6.1} and replace $16d$ with $64$ in \eqref{4eb6.2} whenever we use Proposition \ref{4p1.1} below. In fact, these constants can be chosen to be any fixed large number.

We state the following results on the moments and exponential moments of branching random walk whose proofs
 are deferred to Appendix \ref{4ap1.1}; the arguments follow essentially from Perkins \cite{Per88}.
 Write $\P^x$ for the law of BRW starting from a single ancestor at $x$ for $x\in \Z_R^d$.

\begin{proposition}\label{4p1.2}
For any $x\in \Z^d_R$, $n\geq 1$ and any Borel function $\phi\geq 0$, we have\\
\no (i)\begin{align*}
\E^{x}({Z}_{n}(\phi))=(1+\frac{\theta}{R^{d-1}})^n  \E(\phi(S_n+x))=(1+\frac{\theta}{R^{d-1}})^n  \sum_{y\in \Z_R^d} \phi(y) p_n(x-y).
\end{align*}
(ii) For any $p\geq 2$, 
\begin{align*}
&\E^{x}({Z}_{n}(\phi)^p)\leq (p-1)! e^{\frac{n\theta(p-1)}{R^{d-1}}} G(\phi,n)^{p-1}\E^{x}({Z}_{n}(\phi)),
\end{align*}
where
\begin{align}\label{4e5.90}
 G(\phi,n)= 3\|\phi\|_\infty +\sum_{k=1}^{n} \sup_{y\in \Z_R^d}\sum_{z\in \Z^d_R} \phi(z) p_k(y-z).
\end{align}
\end{proposition}
\begin{corollary}\label{4c1.2}
For any ${Z}_0\in M_F(\Z^d_R)$, $\phi\geq 0$, $\lambda>0$, $n\geq 1$, if $\lambda e^{\frac{n\theta}{R^{d-1}}} G(\phi,n)<1$ is satisfied, we have
\begin{align*}
&\E^{{Z}_0}(e^{\lambda {Z}_{n}(\phi)})\leq \exp\Big(\lambda \E^{{Z}_0}({Z}_{n}(\phi)) (1-\lambda e^{\frac{n\theta}{R^{d-1}}} G(\phi,n))^{-1}\Big).
\end{align*}
\end{corollary}
The following exponential moment for the occupation measure uses similar arguments; the proof is deferred to Appendix \ref{4ap1.2}.

\begin{proposition}\label{4p1.4}
For any ${Z}_0\in M_F(\Z^d_R)$, $\phi\geq 0$, $\lambda>0$, $n\geq 1$, if $2\lambda n e^{\frac{n\theta}{R^{d-1}}} G(\phi,n)<1$ is satisfied, we have
\begin{align}\label{4e100}
&\E^{{Z}_0}\Big(\exp\Big({\lambda\sum_{k=0}^n {Z}_{k}(\phi)}\Big)\Big)\leq \exp\Big(\lambda |{Z}_0| e^{\frac{n\theta}{R^{d-1}}}  G(\phi,n)  (1-2\lambda  n e^{\frac{n\theta}{R^{d-1}}} G(\phi,n))^{-1}\Big).
\end{align}
\end{proposition}

\subsection{Martingale problem of branching random walk}

Recall the construction and the labelling system of branching random walk $(Z_n)$ in Section \ref{4s1.3}.
Observe that for any $n\geq 0$ and $\phi: \Z_R^d\to \R$, we have
\begin{align*}
Z_{n+1}(\phi)=&\sum_{|\alpha|=n+1} \phi({Y^\alpha})=\sum_{|\alpha|=n} \sum_{i=1}^{V(R)} \phi({Y^\alpha+e_i}) B^{\alpha \vee e_i}.
\end{align*}
In the last expression above, we use $Y^\alpha$ for $|\alpha|=n$ to represent the location of the particle $\alpha$ alive in generation $n$ and so $Y^\alpha+e_i$ are the possible locations of its offspring. We use the convention that if $Y^\alpha=\Delta$, the cemetery state, then $\phi(\Delta+x)=0$ for any $\phi$ and $x$. In the mean time, the Bernoulli random variables $\{B^{\alpha \vee e_i}\}$ with parameter $p(R)$ indicates whether the birth in this direction is valid. Use the above with some arithmetic to further get
\begin{align}\label{4e1.0}
Z_{n+1}(\phi)-Z_n(\phi)=&\sum_{|\alpha|=n} \sum_{i=1}^{V(R)} \Big[\phi({Y^\alpha+e_i}) B^{\alpha \vee e_i}-\phi({Y^\alpha})\frac{1}{V(R)}\Big]\nn\\
=&\sum_{|\alpha|=n} \sum_{i=1}^{V(R)} \Big[\phi({Y^\alpha+e_i}) -\phi({Y^\alpha})\Big]\frac{1}{V(R)} (1+\frac{\theta}{R^{d-1}})\nn\\
&\quad +\sum_{|\alpha|=n} \sum_{i=1}^{V(R)} \phi({Y^\alpha+e_i}) \Big(B^{\alpha \vee e_i}-\frac{1+\frac{\theta}{R^{d-1}}}{V(R)}\Big)\nn\\
&\quad +\sum_{|\alpha|=n} \phi({Y^\alpha})\frac{\theta}{R^{d-1}}.
\end{align}
For any $N\geq 1$, we sum \eqref{4e1.0} over $0\leq n\leq N-1$ to arrive at
\begin{align}\label{4e1.1}
Z_{N}(\phi)=Z_0(\phi)&+(1+\frac{\theta}{R^{d-1}}) \sum_{n=0}^{N-1} \sum_{|\alpha|=n} \frac{1}{V(R)}\sum_{i=1}^{V(R)}  \Big[\phi({Y^\alpha+e_i}) -\phi({Y^\alpha})\Big]\nn\\
& +M_N(\phi)+\frac{\theta}{R^{d-1}}\sum_{n=0}^{N-1} Z_n(\phi).
\end{align}
where (recall $p(R)=(1+{\theta}/{R^{d-1}})/V(R)$)
\begin{align}\label{4e1.22}
M_N(\phi)=\sum_{n=0}^{N-1}\sum_{|\alpha|=n} \sum_{i=1}^{V(R)} \phi({Y^\alpha+e_i}) \Big(B^{\alpha \vee e_i}-p(R)\Big).
\end{align}
Recall $\cG_N=\sigma(\{B^\alpha: |\alpha|\leq N\})$. One can check that
\begin{align*}
\E^{Z_0}(M_{N+1}(\phi)-M_{N}(\phi)|\cG_N)=&\E^{Z_0}\Big(\sum_{|\alpha|=N} \sum_{i=1}^{V(R)} \phi({Y^\alpha+e_i}) (B^{\alpha \vee e_i}-p(R))\Big|\cG_N\Big)\\
=&\sum_{|\alpha|=N} \sum_{i=1}^{V(R)} \phi({Y^\alpha+e_i}) \E^{Z_0}\Big(\big(B^{\alpha \vee e_i}-p(R)\big)\Big|\cG_N\Big)=0,
\end{align*}
where the last equality is by the independence of $\cG_N$ and $B^{\alpha \vee e_i}$ with $|\alpha|=N$.
Then the above gives that $\{M_N(\phi), N\geq 0\}$ is a martingale w.r.t. $\cG_N$, whose conditional quadratic variation will be given by
\begin{align}\label{4eb1.51}
&\langle M(\phi)\rangle_N=\sum_{n=0}^{N-1} \E^{Z_0}\Big((M_{n+1}(\phi)-M_{n}(\phi))^2\Big|\cG_n\Big)\nn\\
=&\sum_{n=0}^{N-1} \sum_{|\alpha|=n} \sum_{i=1}^{V(R)} \phi({Y^\alpha+e_i})^2 \E^{Z_0}\Big( (B^{\alpha \vee e_i}-p(R))^2\Big|\cG_n\Big)\nn\\
=&\sum_{n=0}^{N-1} \sum_{|\alpha|=n} \sum_{i=1}^{V(R)} \phi({Y^\alpha+e_i})^2 p(R) (1-p(R)).
\end{align}
In the second equality, the cross terms are cancelled by the mutual independence of $\{B^{\alpha \vee e_i}\}$.
Use $p(R)=(1+{\theta}/{R^{d-1}})/V(R)$ to get
\begin{align}\label{4e1.30}
\langle M(\phi)\rangle_N=&(1+\frac{\theta}{R^{d-1}}) (1-p(R))\sum_{n=0}^{N-1} \sum_{|\alpha|=n} \frac{1}{V(R)}\sum_{i=1}^{V(R)} \phi({Y^\alpha+e_i})^2\nn\\
\leq &2\sum_{n=0}^{N-1} \sum_{|\alpha|=n}  \frac{1}{V(R)}\sum_{i=1}^{V(R)} \phi({Y^\alpha+e_i})^2\nn\\
= &2\sum_{n=0}^{N-1} \sum_{x\in \Z^d_R} Z_n(x) \cdot \frac{1}{V(R)}\sum_{i=1}^{V(R)} \phi({x+e_i})^2,
\end{align}
where we have used $\theta\leq R^{d-1}$ in the inequality and the last equality is by \eqref{4eb2.21}.

The following proposition will play an important role in computing the exponential moments of $M_N(\phi)$. The proof follows essentially from Freedman \cite{Free75} and can be found in Appendix \ref{4ap1.3}.
\begin{proposition}\label{4p5.1}
Let $d=2$ or $d=3$. Let $N\geq 1$, $\theta \geq 100$, $R\geq 4\theta$ and ${Z}_0\in M_F(\Z^d_R)$. For any $\lambda>0$ and any Borel function $\phi$ so that $\lambda\| \phi\|_\infty\leq 1$, we have
\begin{align*}
\E^{Z_0}( \exp(\lambda |M_{N}(\phi)|))\leq 2\Big(\E^{Z_0}\Big( \exp\Big( 16\lambda^2 \langle M(\phi) \rangle_{N}\Big)\Big)\Big)^{1/2}.
\end{align*}
\end{proposition}

\section{Potential kernel and Tanaka's formula}\label{4s4}


For any function $f: \Z^d_R\to \R$ and $x\in \Z^d_R$, we define the generator of $p_1$ to be
\begin{align}\label{4e1.7}
\cL f(x)=\E(f(x+S_1)-f(x))=\sum_{i=1}^{V(R)} (f(x+e_i)-f(x))\frac{1}{V(R)}.
\end{align}
By Chapman-Kolmogorov's equation, we have
\begin{align}\label{4e2.4}
p_{n+1}(x)=\sum_{y} p_n(y) p_1(y-x)=\frac{1}{V(R)}\sum_{i=1}^{V(R)} p_n(x+e_i),
\end{align}
thus giving
\begin{align}\label{4e6.4}
p_{n+1}(x)-p_n(x)=\frac{1}{V(R)}\sum_{i=1}^{V(R)} (p_n(x+e_i)-p_n(x))=\cL p_n(x).
\end{align}

 In $d=3$, for any $a\in \Z^3_R$, we let 
\begin{align}\label{4e6.3}
\phi_a(x)=RV(R)\sum_{n=1}^\infty p_n(x-a), \quad \forall x\in \Z^3_R.
\end{align}
Recall $g_{u,3}$ from \eqref{4e10.31}. We may use Proposition \ref{4p1.1}(i) to get for any $a,x \in \Z^3_R$,
\begin{align}\label{4e7.10}
\phi_a(x)\leq RV(R) \sum_{n=1}^\infty \frac{c_{\ref{4p1.1}}}{n^{3/2}R^3} e^{-\frac{|x-a|^2}{32n}}\leq  CR \sum_{n=1}^\infty \frac{1}{n^{3/2}} e^{-\frac{|x-a|^2}{32n}}=Cg_{a,3}(x).
\end{align}
Note that
\begin{align}\label{4eb1.32}
\|g_{a,3}\|=R \sum_{n=1}^\infty \frac{1}{n^{3/2}}\leq CR<\infty.
\end{align}
Hence the sum in $\phi_a$ is absolutely convergent. We also have $p_n$ is absolutely summable.
Apply Fubini's theorem to get
\begin{align}\label{4e6.5}
\cL\phi_a(x)=&\sum_{i=1}^{V(R)} (\phi_a(x+e_i)-\phi_a(x))\frac{1}{V(R)}\nn\\
=&R\sum_{n=1}^\infty\sum_{i=1}^{V(R)} (p_n(x-a+e_i)-p_n(x-a))\nn\\
=&RV(R)\sum_{n=1}^\infty (p_{n+1}(x-a)-p_n(x-a))\nn\\
=&-RV(R)p_1(x-a)=  -R\cdot 1(x\in \cN(a)),
\end{align}
where the third equality follows from \eqref{4e2.4}. Replace $\phi$ with $\phi_a$ in \eqref{4e1.1} and use the above to see that for any $N\geq 1$,
\begin{align*}
Z_{N}(\phi_a)=&Z_0(\phi_a)- R(1+\frac{\theta}{R^{d-1}}) \sum_{n=0}^{N-1} \sum_{|\alpha|=n} 1(Y^\alpha\in \cN(a)) +M_N(\phi_a)+\frac{\theta}{R^{d-1}}\sum_{n=0}^{N-1} Z_n(\phi_a).
\end{align*}
Rearrange terms to arrive at
\begin{align}\label{4ea6.13}
(1+\frac{\theta}{R^{d-1}})R\sum_{n=0}^{N-1} Z_n(\cN(a))=&(1+\frac{\theta}{R^{d-1}})R\sum_{n=0}^{N-1} \sum_{|\alpha|=n} 1(Y^\alpha\in \cN(a)) \nn\\
=&Z_{0}(\phi_a)-Z_N(\phi_a)+M_N(\phi_a)+\frac{\theta}{R^{d-1}}\sum_{n=0}^{N-1} Z_n(\phi_a).
 \end{align}
We call \eqref{4ea6.13} the Tanaka formula for the local times of $(Z_n)$ in $d=3$. It is easy to derive the following bounds from the above:
\begin{align}\label{4e6.13}
R\sum_{n=0}^{N-1} Z_n(\cN(a))\leq &Z_{0}(\phi_a)+M_N(\phi_a)+\frac{\theta}{R^{2}}\sum_{n=0}^{N-1} Z_n(\phi_a).
 \end{align}

 In $d=2$, for any $a\in \Z^2_R$ we set
\begin{align}\label{4e6.6}
g_a(x)=V(R)\sum_{n=1}^\infty e^{-n\theta/R} p_n(x-a),\quad \forall x\in \Z^2_R.
\end{align} 
Recall $g_{u,2}$ from \eqref{4e10.31}. We use Proposition \ref{4p1.1}(i) to get
\begin{align}\label{4e10.33}
 g_a(x)&\leq V(R)\sum_{n=1}^\infty  e^{-n\theta/R}\frac{c_{\ref{4p1.1}}}{n R^2} e^{-|x-a|^2/(32n)}\nn\\
 &\leq C\sum_{n=1}^\infty  e^{-n\theta/R}\frac{1}{n} e^{-|x-a|^2/(32n)}=Cg_{a,2}(x).
\end{align}
Note that
\begin{align}\label{4e6.20}
\|g_{a,2}\|_\infty&=\sum_{n=1}^\infty  (e^{-\theta/R})^n\frac{1}{n}= (-\log (1-e^{-\theta/R})) \leq  \log \frac{2R}{\theta}<\infty,
\end{align}
where the second equality uses the Taylor series of $-\log(1-x)$ and the first inequality is by applying $1-e^{-x}\geq x/2$ for $0\leq x\leq 1/4$ and $R\geq 4\theta$.
Hence we conclude from \eqref{4e10.33} and \eqref{4e6.20} that the sum in $g_a$ is absolutely convergent. Similar to the derivation of \eqref{4e6.5}, we do some arithmetic to get
\begin{align}\label{4e6.7}
\cL g_a(x)
=&\sum_{n=1}^\infty e^{-n\theta/R} \sum_{i=1}^{V(R)} (p_n(x-a+e_i)-p_n(x-a))\nn\\
=&V(R)\sum_{n=1}^\infty e^{-n\theta/R} (p_{n+1}(x-a)-p_n(x-a))\nn\\
=&e^{\theta/R} V(R) \sum_{n=1}^\infty e^{-(n+1)\theta/R} p_{n+1}(x-a)-V(R)\sum_{n=1}^\infty e^{-n\theta/R}  p_n(x-a)\nn\\
=&(e^{\theta/R}-1)g_a(x)-V(R)  p_1(x-a)=(e^{\theta/R}-1)g_a(x)-1_{\{x\in \cN(a)\}}.
\end{align}
Replace $\phi$ with $g_a$ in \eqref{4e1.1} and use the above to see that
\begin{align*}
Z_{N}(g_a)=&Z_0(g_a)+ (1+\frac{\theta}{R^{d-1}})\sum_{n=0}^{N-1} \sum_{|\alpha|=n} \Big[(e^{\theta/R}-1)g_a(Y^\alpha)-1(Y^\alpha\in \cN(a))\Big] \nn\\
&\quad +M_N(g_a)+\frac{\theta}{R^{d-1}}\sum_{n=0}^{N-1} Z_n(g_a)\\
=&Z_0(g_a)+ (e^{\theta/R}-1)(1+\frac{\theta}{R^{d-1}}) \sum_{n=0}^{N-1} Z_n(g_a)-(1+\frac{\theta}{R^{d-1}})\sum_{n=0}^{N-1} \sum_{|\alpha|=n} 1(Y^\alpha\in \cN(a)) \nn\\
&\quad +M_N(g_a)+\frac{\theta}{R^{d-1}}\sum_{n=0}^{N-1} Z_n(g_a).
\end{align*}
Note we are in $d=2$. Rearrange terms in the above to get
\begin{align}\label{4ea6.14}
&(1+\frac{\theta}{R}) \sum_{n=0}^{N-1} Z_n(\cN(a))=(1+\frac{\theta}{R}) \sum_{n=0}^{N-1} \sum_{|\alpha|=n} 1(Y^\alpha\in \cN(a))\nn\\
&=Z_{0}(g_a)-Z_N(g_a)+ M_N(g_a)+\Big((e^{\theta/R}-1)(1+\frac{\theta}{R}) +\frac{\theta}{R}\Big)\sum_{n=0}^{N-1} Z_n(g_a).
\end{align}
We call \eqref{4ea6.14} the Tanaka formula for the local times of $(Z_n)$ in $d=2$. By using $1+\frac{\theta}{R}\leq 2$ and $e^{\theta/R}-1\leq 2\theta/R$ when $\theta/R\leq 1/4$, we get
\begin{align}\label{4e6.14}
&\sum_{n=0}^{N-1} Z_n(\cN(a))  \leq Z_{0}(g_a)+ M_N(g_a)+\frac{5\theta}{R} \sum_{n=0}^{N-1} Z_n(g_a).
\end{align}

Using the bounds in \eqref{4e6.13} and \eqref{4e6.14}, we will prove the key Proposition \ref{4p2} in the following two sections for $d=2$ and $d=3$ respectively.

\section{Local time bounds in $d=2$}\label{4s5}
In this section we give the proof of Proposition \ref{4p2} for $d=2$. Throughout this section we let $d=2$ unless otherwise indicated. Recall $Z_0\in M_F(\Z_R^2)$ satisfies
\begin{align}\label{4e7.1}
\begin{dcases}
\text{(i) }\text{Supp}(Z_0)\subseteq Q_{R_\theta}(0); \\
\text{(ii) } Z_0(1)\leq 2 R f_2(\theta)/\theta=2R/\sqrt{\theta};\\
\text{(iii) } Z_{0}(g_{u,2})\leq m {R}/\theta^{1/4}, \quad \forall u\in \R^2.
\end{dcases}
\end{align}
The local time that we aim to bound in Proposition \ref{4p2} is the sum over the branching random walk masses of the unit box centered at $x\in \Z_R^d$, and so it suffices to consider the local time at points in the integer lattice $\Z^d$. We claim Proposition \ref{4p2} in $d=2$ will be an easy consequence of the following result.
\begin{proposition}\label{4t2.0}
Let $d=2$. For any $\eps_0\in (0,1)$, $T\geq 100$ and $m>0$, there exist constants $\theta_{\ref{4t2.0}}\geq 100, \chi_{\ref{4t2.0}}>0$ depending only on $\eps_0, T,m$ such that for all $\theta \geq \theta_{\ref{4t2.0}}$,  there is some $C_{\ref{4t2.0}}(\eps_0, T,\theta,m)\geq 4\theta$ such that for any $R\geq  C_{\ref{4t2.0}}$ and any $Z_0$ satisfying \eqref{4e7.1}, we have
\[
\P^{Z_0}\Big(\sum_{n=0}^{T_\theta^R} Z_n(\cN(a)) \leq  \chi_{\ref{4t2.0}} {R}, \quad \forall a\in \Z^2 \cap Q_{3M_{\ref{4p1}} \sqrt{\log f_2(\theta)}R_\theta}  (0) \Big)\geq 1-\eps_0.
\]
\end{proposition}
\begin{proof}[Proof of Proposition \ref{4p2} in $d=2$ assuming Proposition \ref{4t2.0}]
Fix $\eps_0\in (0,1)$, $T\geq 100$ and $m>0$. Let $\theta, R, Z_0$ be as in Proposition \ref{4t2.0}. Then with probability $\geq 1-\eps_0$, we have
\begin{align}\label{4e10.63}
\sum_{n=0}^{T_\theta^R} Z_n(\cN(a)) \leq  \chi_{\ref{4t2.0}} {R}, \quad \forall a\in \Z^2 \cap Q_{3M_{\ref{4p1}} \sqrt{\log f_2(\theta)}R_\theta}  (0).
\end{align}
For any $x\in \Z_R^2$, let $\cU(x)=\{a\in \Z^2: \|a-x\|_\infty\leq 1\}$. One can easily check that
\begin{align}\label{4ea8.4}
\sum_{n=0}^{T_\theta^R} Z_n(\cN(x))\leq \sum_{n=0}^{T_\theta^R} \sum_{a\in \cU(x)}  Z_n(\cN(a))=\sum_{a\in \cU(x)} \sum_{n=0}^{T_\theta^R} Z_n(\cN(a)).
\end{align}
For any $x\in \Z^2_R \cap Q_{2M_{\ref{4p1}} \sqrt{\log f_2(\theta)}R_\theta}  (0)$, we have $a\in \cU(x)\subseteq Q_{3M_{\ref{4p1}} \sqrt{\log f_2(\theta)}R_\theta}  (0)$. Notice that there are at most $3^2$ elements in $\cU(x)$ for each $x\in \Z^2_R$. Hence one may conclude  by \eqref{4ea8.4} that on the event \eqref{4e10.63}, we have
\begin{align}
\sum_{n=0}^{T_\theta^R} Z_n(\cN(x)) \leq  9\chi_{\ref{4t2.0}} {R}, \quad \forall x\in \Z^2_R \cap Q_{2M_{\ref{4p1}} \sqrt{\log f_2(\theta)}R_\theta}  (0).
\end{align}
 So the proof is complete by letting $\chi_{\ref{4p2}}=9\chi_{\ref{4t2.0}}$.
\end{proof}

It remains to prove Proposition \ref{4t2.0}.
In view of \eqref{4e6.14}, it suffices to get bounds for $Z_0(g_a)$, $M_{{T_\theta^R}+1}(g_a)$ and $\sum_{n=0}^{{T_\theta^R}} Z_n(g_a)$ where $g_a(x)=V(R)\sum_{n=1}^\infty e^{-n\theta/R} p_n(x-a)$. Recall  from \eqref{4e10.33} that $g_a(x)\leq Cg_{a,2}(x)$ for any $a,x\in \Z_R^2$.
Hence \eqref{4e7.1} implies that
\begin{align}\label{4ea10.33}
Z_0(g_a)\leq CZ_0(g_{a,2})\leq CmR/\theta^{1/4}, \quad \forall a\in \Z^2.
\end{align}
Turning to $M_{{T_\theta^R}+1}(g_a)$ and $\sum_{n=0}^{{T_\theta^R}} Z_n(g_a)$, we will calculate their exponential moments and use the following version of Garsia's Lemma from Lemma 3.7 of \cite{LPZ14} to derive the corresponding probability bounds.

\begin{lemma}[\cite{LPZ14}]\label{4l2.2}
Let  $d\geq 1$. Suppose $\{\Upsilon(x): x\in \R^d\}$ is an almost surely continuous random field such that for some $\lambda>0$ and $\eta>0$,
\begin{align}
\begin{cases}
\E\Big(\exp\Big(\lambda \frac{|\Upsilon(x)-\Upsilon(y)|}{|x-y|^\eta}\Big)\Big)\leq C_1,\quad \forall 0<|x-y|\leq \sqrt{d};\\
\E(\exp(\lambda \Upsilon(x)))\leq C_2, \quad \forall x\in \R^d.
\end{cases}
\end{align}
Then for all $M\geq 1$ and $\chi> 0$,
\[
\P\Big(\sup_{x\in Q_M(0)} \Upsilon(x) \geq \chi\Big)\leq (C_1 e^{2d/\eta}+C_2) (2M)^d \exp\Big({-\frac{\lambda \chi}{1+8d^{\eta/2}}}\Big).
\]
\end{lemma}

With our discrete setting, we need the following lemma that serves as an intermediate step towards the ``discrete'' version of the above Garsia's Lemma. The proof is deferred to Appendix \ref{a2}. 
\begin{lemma}\label{4l2.1}
Let  $d\geq 1$. Assume $\{f(n): n\in \Z^d\}$ is a collection of non-negative random variables on some probability space $(\Omega, \cF, \P)$ which satisfies
\begin{align}\label{4eb3.21}
\begin{cases}
\E\Big(\exp\Big(\lambda \frac{|f(n)-f(m)|}{|n-m|^\eta}\Big)\Big)\leq C_1,\quad \forall n\neq m \in \Z^d,\\
\E(\exp(\mu f(n)))\leq C_1, \quad \forall n\in \Z^d,
\end{cases}
\end{align}
for some constants $\lambda,\mu, C_1>0$ and $\eta \in (0,1]$. For each $\omega \in \Omega$, if we linearly interpolate between integer points to obtain a continuous function $g(x)$ for $x\in \R^d$, then there exists some constant $0<c_{\ref{4l2.1}}(d)<1$ such that 
\begin{align}\label{4e8.12}
\begin{cases}
\E\Big(\exp\Big( c_{\ref{4l2.1}} \lambda\frac{|g(x)-g(y)|}{|x-y|^\eta}\Big)\Big)\leq C_1, \quad\forall x\neq y \in \R^d,\\
\E(\exp( c_{\ref{4l2.1}} \mu g(x)))\leq C_1,\quad \forall x\in \R^d.
\end{cases}
\end{align}
\end{lemma}

Combining Lemma \ref{4l2.2} and Lemma \ref{4l2.1}, the probability bounds for random variables indexed by the integer points may follow from their exponential moment bounds, which we now give.

\begin{proposition}\label{4p2.1}
Let $\eta=1/8$. For any $T\geq 100$, there exist constants $C_{\ref{4p2.1}}(T)>0$ and $\theta_{\ref{4p2.1}}(T)\geq 100$ such that for all $\theta \geq \theta_{\ref{4p2.1}}(T)$,  there is some $K_{\ref{4p2.1}}(T,\theta)\geq 4\theta$ such that for any $m>0$, $R\geq  K_{\ref{4p2.1}}$ and any $Z_0$ satisfying \eqref{4e7.1}, we have
\begin{align}
&\text{(i) } \E^{Z_0}\Big(\exp\Big(\theta^{3/2} R^{-2} \sum_{k=0}^{{T_\theta^R}} Z_k(g_a)\Big)\Big)\leq C_{\ref{4p2.1}}(T), \quad \forall a \in \Z^2;\\
&\text{(ii) }\E^{{Z}_0}\Big(\exp\Big( \frac{\theta^{3/2}}{R^2} \frac{(R/\theta)^{\eta/2}}{|a-b|^{\eta}} |\sum_{k=0}^{{T_\theta^R}} {Z}_{k}(g_a)-\sum_{k=0}^{{T_\theta^R}} {Z}_{k}(g_b)|\Big)\Big)
\leq C_{\ref{4p2.1}}(T), \quad \forall a\neq b \in \Z^2.\nn
\end{align}
\end{proposition}

Assuming Proposition \ref{4p2.1}, we first show these exponential moments indeed give us the desired bounds by applying the discrete Garsia's Lemma.

\begin{corollary}\label{4c2.1}
For any $\eps_0\in (0,1)$ and $T\geq 100$, there exist constants $\chi_{\ref{4c2.1}}>0$ and $ \theta_{\ref{4c2.1}}\geq 100$ depending only on $\eps_0, T$ such that for all $\theta \geq \theta_{\ref{4c2.1}}$,  there is some $C_{\ref{4c2.1}}(\eps_0, T,\theta)\geq 4\theta$ such that for any $m>0$, $R\geq  C_{\ref{4c2.1}}$ and any $Z_0$ satisfying \eqref{4e7.1}, we have
\[
\P^{Z_0}\Big(\frac{\theta}{R}\sum_{k=0}^{{T_\theta^R}} Z_k(g_a) \leq \chi_{\ref{4c2.1}} \frac{R}{\theta^{1/4}}, \quad \forall a\in \Z^2 \cap Q_{3M_{\ref{4p1}} \sqrt{\log f_2(\theta)}R_\theta}  (0) \Big)\geq 1-\frac{\eps_0}{2}.
\]
\end{corollary}

\begin{proof}
Fix $\eps_0 \in (0,1)$, $T\geq 100$ and $\eta=1/8$. Let $\theta, m, R$ satisfy the assumptions of Proposition \ref{4p2.1} and set $Z_0$ as in \eqref{4e7.1}. If we define  $\{f(x): x\in \R^2\}$ to be the continuous random field obtained by linearly interpolating $\{\sum_{k=0}^{{T_\theta^R}} Z_k(g_a): a\in \Z^2\}$, then by assuming Proposition \ref{4p2.1}, we may apply Lemma \ref{4l2.1} to get 
\begin{align}\label{4eb4.1} 
\E^{Z_0}\Big(\exp\Big( c_{\ref{4l2.1}}\theta^{3/2 } R^{-2} f(x)\Big)\Big)\leq C_{\ref{4p2.1}}(T), \quad \forall x \in \R^2,
\end{align}
and
\begin{align}\label{4eb4.2} 
\E^{{Z}_0}\Big(\exp\Big(c_{\ref{4l2.1}}  \frac{\theta^{3/2}}{R^2} \frac{(R/\theta)^{\eta/2}}{|x-y|^{\eta}}   |f(x)-f(y)|\Big)\Big)
\leq C_{\ref{4p2.1}}(T), \quad \forall x\neq y \in \R^2.
\end{align}
Recall $R_\theta=\sqrt{R^{d-1}/\theta}=\sqrt{R/\theta}$ and $f_2(\theta)= \sqrt{\theta}$. Define
\begin{align}\label{4e5.41} 
k_\theta =\sqrt{\log f_2(\theta)/\theta}= (2\theta)^{-1/2}\sqrt{\log \theta}
\end{align}
so that $\sqrt{\log f_2(\theta)}R_\theta= R^{1/2} k_\theta$.
Replace $x,y$ in \eqref{4eb4.1} and \eqref{4eb4.2} with $xR^{1/2} k_\theta$, $yR^{1/2} k_\theta$ respectively to see that 
\begin{align}\label{4eb4.3} 
\E^{Z_0}\Big(\exp\Big( c_{\ref{4l2.1}} \theta^{3/2} R^{-2} f( x R^{1/2} k_\theta )\Big)\Big)\leq C_{\ref{4p2.1}}(T), \quad \forall x \in \R^2,
\end{align}
and for all $x\neq y \in \R^2$,
\begin{align}\label{4eb4.4} 
\E^{{Z}_0}\Big(\exp\Big( \frac{2^{\eta/2} c_{\ref{4l2.1}}}{(\log \theta)^{\eta/2}} \frac{\theta^{3/2} {R^{-2}}}{|x-y|^{\eta}} \Big|f( x R^{1/2} k_\theta )-f( y R^{1/2} k_\theta )\Big|\Big)\Big)
\leq C_{\ref{4p2.1}}(T).
\end{align}
 Set $\Upsilon(x)=\theta^{5/4} R^{-2}f( x R^{1/2} k_\theta )$ for $x\in \R^d$. Note we have $\theta^{5/4}\leq \theta^{3/2}$ and $\theta^{5/4}\leq \frac{2^{\eta/2}\theta^{3/2}}{(\log \theta)^{\eta/2}}$ for $\theta\geq 100$. Therefore we conclude from \eqref{4eb4.3}, \eqref{4eb4.4} that
\begin{align}
\begin{cases}
\E^{{Z}_0}\Big(\exp\Big( c_{\ref{4l2.1}}   \frac{|\Upsilon(x)-\Upsilon(y)|}{|x-y|^\eta}\Big)\Big)\leq C_{\ref{4p2.1}}(T),\quad \forall x\neq y \in \R^2,\\
\E^{{Z}_0}(\exp( c_{\ref{4l2.1}} \Upsilon(x)))\leq C_{\ref{4p2.1}}(T), \quad \forall x \in \R^2.
\end{cases}
\end{align}
Apply Lemma \ref{4l2.2} with the above moment bounds to get for any $\chi>0$ and $M\geq 1$,
\begin{align*}
&\P^{\Z_0}\Big(\sup_{x\in Q_M(0)} \theta^{5/4} R^{-2}f( x R^{1/2} k_\theta )\geq \chi\Big)\\
&\leq (C_{\ref{4p2.1}}(T) e^{32}+C_{\ref{4p2.1}}(T)) (2M)^2 \exp\Big({-\frac{c_{\ref{4l2.1}} \chi}{1+8\cdot 2^{1/16}}}\Big).
\end{align*}
Let $M=3M_{\ref{4p1}}(\eps_0,T)\geq 1$. Pick $\chi_{\ref{4c2.1}}=\chi_{\ref{4c2.1}}(M, \eps_0, T)=\chi_{\ref{4c2.1}}(\eps_0,T)>0$ large enough so that
\begin{align}
\P^{\Z_0}\Big(\sup_{x\in Q_{3M_{\ref{4p1}} }(0) } \theta^{5/4} R^{-2}f( x R^{1/2} k_\theta ) \geq \chi_{\ref{4c2.1}}\Big)\leq \frac{\eps_0}{2}.
\end{align}
Hence with probability larger than $1-\eps_0/2$, we have
\begin{align}
\sup_{a\in \Z^2\cap Q_{ 3M_{\ref{4p1}} R^{1/2} k_\theta }(0) }  \theta^{5/4} R^{-2} \sum_{k=0}^{{T_\theta^R}} Z_k(g_a) \leq \sup_{x\in Q_{3M_{\ref{4p1}}}(0)} \theta^{5/4} R^{-2}f( x R^{1/2} k_\theta)\leq \chi_{\ref{4c2.1}}.
\end{align}
The proof is complete by noting $\sqrt{\log f_2(\theta)}R_\theta= R^{1/2} k_\theta$.
\end{proof}

In a similar way we will take care of the martingale term by the following exponential moments.

\begin{proposition}\label{4p2.2}
Let $\eta=1/8$. For any $T\geq 100$, there exist constants $C_{\ref{4p2.2}}(T)>0$ and $\theta_{\ref{4p2.2}}(T)\geq 100$ such that for all $\theta \geq \theta_{\ref{4p2.2}}(T)$,  there is some $K_{\ref{4p2.2}}(T,\theta)\geq 4\theta$ such that for any $m>0$, $R\geq  K_{\ref{4p2.2}}$ and any $Z_0$ satisfying \eqref{4e7.1}, we have
\begin{align*}
&\text{(i) }\E^{Z_0}\Big(\exp\Big(\theta^{3/4} R^{-1} |M_{{T_\theta^R}+1}(g_a)|\Big)\Big)\leq C_{\ref{4p2.2}}(T), \quad \forall a \in \Z^2,\\
&\text{(ii) }\E^{{Z}_0}\Big(\exp\Big( \frac{\theta^{3/4}}{R}  \frac{(R/\theta)^{\eta/2} }{|a-b|^\eta} \Big||M_{{T_\theta^R}+1}(g_a)|-|M_{{T_\theta^R}+1}(g_b)|\Big|\Big)\Big)
\leq C_{\ref{4p2.2}}(T), \quad \forall a\neq b \in \Z^2.
\end{align*}
\end{proposition}

\begin{corollary}\label{4c2.2}
For any $\eps_0\in (0,1)$ and $T\geq 100$, there exist constants $\chi_{\ref{4c2.2}}>0$ and $ \theta_{\ref{4c2.2}}\geq 100$ depending only on $\eps_0, T$ such that for all $\theta \geq \theta_{\ref{4c2.2}}$,  there is some $C_{\ref{4c2.2}}(\eps_0, T,\theta)\geq 4\theta$ such that for any $m>0$, $R\geq  C_{\ref{4c2.2}}$ and any $Z_0$ satisfying \eqref{4e7.1}, we have
\[
\P^{Z_0}\Big( |M_{{T_\theta^R}+1}(g_a)|\leq \chi_{\ref{4c2.2}} \frac{R}{\theta^{1/4}}, \quad \forall a\in \Z^2 \cap Q_{3M_{\ref{4p1}} \sqrt{\log f_2(\theta)}R_\theta}  (0) \Big)\geq 1-\frac{\eps_0}{2}.
\]
\end{corollary}
\begin{proof}
By using Proposition \ref{4p2.2}, the proof follows in a similar way to that of Corollary \ref{4c2.1} and so is omitted.
\end{proof}

Assuming Proposition \ref{4p2.1} and Proposition \ref{4p2.2}, we may finish the proof of Proposition \ref{4t2.0} below.
\begin{proof}[Proof of Proposition \ref{4t2.0}]
Fix $\eps_0\in (0,1)$, $T\geq 100$ and $m>0$. Let $\theta_{\ref{4t2.0}}=\max\{\theta_{\ref{4c2.1}}, \theta_{\ref{4c2.2}}\}$.  For any $\theta \geq \theta_{\ref{4t2.0}}$, we let $C_{\ref{4t2.0}}=\max\{C_{\ref{4c2.1}}(\eps_0, T,\theta), C_{\ref{4c2.2}}(\eps_0, T,\theta)\}$. For any $R\geq C_{\ref{4t2.0}}$, we let $Z_0$ be as in \eqref{4e7.1}.
Apply Corollary \ref{4c2.1} to get with probability $\geq 1-\eps_0/2$,
\begin{align}\label{4ea7.1}
\frac{\theta}{R}\sum_{k=0}^{{T_\theta^R}} Z_k(g_a) \leq \chi_{\ref{4c2.1}} \frac{R}{\theta^{1/4}}, \quad \forall a\in \Z^2 \cap Q_{3M_{\ref{4p1}} \sqrt{\log f_2(\theta)}R_\theta}  (0).
\end{align}
Apply Corollary \ref{4c2.2} to get with probability $\geq 1-\eps_0/2$,
\begin{align}\label{4ea7.2}
|M_{{T_\theta^R}+1}(g_a)|\leq \chi_{\ref{4c2.2}} \frac{R}{\theta^{1/4}}, \quad \forall a\in \Z^2 \cap Q_{3M_{\ref{4p1}} \sqrt{\log f_2(\theta)}R_\theta}  (0).
\end{align}
Therefore with probability $\geq 1-\eps_0$, both \eqref{4ea7.1} and \eqref{4ea7.2} hold. Use \eqref{4e6.14} to get for any $a\in \Z^2 \cap Q_{3M_{\ref{4p1}} \sqrt{\log f_2(\theta)}R_\theta}  (0)$,
\begin{align*}
&\sum_{n=0}^{T_\theta^R} Z_n(\cN(a))  \leq Z_{0}(g_a)+ M_{{T_\theta^R}+1}(g_a)+\frac{5\theta}{R} \sum_{n=0}^{T_\theta^R} Z_n(g_a)\\
&\leq  C\frac{mR}{\theta^{1/4}}+\chi_{\ref{4c2.1}} \frac{R}{\theta^{1/4}}+\chi_{\ref{4c2.2}} \frac{R}{\theta^{1/4}}\leq (Cm+\chi_{\ref{4c2.1}}+\chi_{\ref{4c2.2}}) R,
\end{align*}
where in the second inequality we have also used \eqref{4ea10.33}.
The proof is complete by letting $\chi_{\ref{4t2.0}}=Cm+\chi_{\ref{4c2.1}}+\chi_{\ref{4c2.2}}$. 
\end{proof}

It remains to prove Proposition \ref{4p2.1} and Proposition \ref{4p2.2}.

\subsection{Exponential moments of the drift term}

In this section we will prove Proposition \ref{4p2.1} for the exponential moments of $\sum_{k=0}^{{T_\theta^R}} Z_k(g_a)$ by applying Proposition \ref{4p1.4}. To do this, we need an estimate for 
\[
G(g_a,{T_\theta^R})=3\|g_a\|_\infty+\sum_{k=1}^{{T_\theta^R}} \sup_{y\in \Z_R^d} \sum_{z\in \Z^d_R} p_k(y-z) g_a(z).
\]
 By \eqref{4e10.33}, it is immediate that
 \begin{align}\label{4ec7.2}
 G(g_a,{T_\theta^R})\leq C\cdot G(g_{a,2},{T_\theta^R}),
 \end{align}
 and so it suffices to get bounds for $G(g_{a,2},{T_\theta^R})$.
We first give some preliminary results. With some calculus, one may easily obtain the following lemma, whose proof can be found in Appendix \ref{a2}.
\begin{lemma}\label{4l4.2}
Let $d\geq 1$ and $R\geq 1$.
\no  (i) 
For any $s, t>0$ and any $x_1, x_2\in \Z^d_R$, we have
\[
\sum_{y\in \Z^d_R} e^{-t|y-x_1|^2} e^{-s|y-x_2|^2} \leq 2^d e^{-\frac{st}{s+t}|x_1-x_2|^2} \sum_{y\in \Z^d_R} e^{-t|y|^2} e^{-s|y|^2}.
\]
(ii) There is some constant $c_{\ref{4l4.2}}=c_{\ref{4l4.2}}(d)>0$ such that for any $u\geq 1$ and $R\geq 1$, we have
\[
\sum_{y\in \Z^d_R} e^{-{|y|^2}/{(2u)}}\leq c_{\ref{4l4.2}} u^{d/2} R^d.
\]
\end{lemma}

\no The following result is from Lemma 4.3.2 of \cite{LL10} and will be used repeatedly below.
\begin{lemma}[\cite{LL10}]\label{4l4.1}
For any $\alpha>0$, there exist constants $C_{\ref{4l4.1}}(\alpha)> c_{\ref{4l4.1}}(\alpha)>0$ such that for all $r\geq 1/64$,
\begin{align}\label{4e2.3}
c_{\ref{4l4.1}}(\alpha)\frac{1}{r^{\alpha}}\leq \sum_{k=1}^\infty \frac{1}{k^{1+\alpha}} e^{-r/k}\leq C_{\ref{4l4.1}}(\alpha)\frac{1}{r^{\alpha}}.
 \end{align}
\end{lemma}

\begin{lemma}\label{4l1.3}
Let $d=2$ or $d=3$. For any $1<\alpha<(d+1)/2$, there is some constant $c_{\ref{4l1.3}}=c_{\ref{4l1.3}}(\alpha, d)>0$ so that for any $n\geq 1$, $R\geq K_{\ref{4p1.1}}$, and $a,x\in \Z^d_R$,
\begin{align*}
\sum_{y\in \Z^d_R} p_n(y-x) \sum_{k=1}^\infty \frac{1}{k^{\alpha}} e^{-\frac{|y-a|^2}{64k}} \leq c_{\ref{4l1.3}}\cdot   \frac{1}{n^{\alpha-1}}.
\end{align*}
\end{lemma}
\begin{proof}
This result follows essentianlly from Lemma \ref{4l4.1}. The proof is deferred to Appendix \ref{a2}.
\end{proof}

\begin{lemma}\label{4l3.2}
There is some constant $c_{\ref{4l3.2}}>0$ so that for any $\theta\geq 100$ and  $R\geq 4\theta$, we have
\begin{align}\label{4e8.66}
g_{a,2}(y)\leq c_{\ref{4l3.2}}\Big(1+\log^+ \Big(\frac{R}{\theta |y-a|^2}\Big)\Big), \quad \forall y\neq a\in \R^2.
\end{align}
\end{lemma}
\begin{proof}
Recall from \eqref{4e10.31} that
 \begin{align}\label{4eb1.2}
g_{a,2}(y)=\sum_{n=1}^\infty e^{-n\theta/R}\frac{1}{n} e^{-\frac{|y-a|^2}{32n}}\leq 1+\sum_{n=2}^\infty e^{-n\theta/R}\frac{1}{n} e^{-\frac{|y-a|^2}{32n}}.
\end{align}
For any $n\geq 2$, if $n\leq t\leq n+1$, then $n\geq t-1$ and $n\leq 2(t-1)$. So we have
\[
e^{-n\theta/R}\frac{1}{n} e^{-\frac{|y-a|^2}{32n}}=\int_n^{n+1}  e^{-n\theta/R}\frac{1}{n} e^{-\frac{|y-a|^2}{32n}} dt \leq \int_n^{n+1}  e^{-(t-1)\theta/R}\frac{1}{(t-1)} e^{-\frac{|y-a|^2}{64(t-1)}} dt.
\]
Sum the above for all $n\geq 2$ and use \eqref{4eb1.2} to see that
 \begin{align*}
g_{a,2}(y)\leq  1+\int_1^{\infty}  e^{-t\theta/R}\frac{1}{t} e^{-\frac{|y-a|^2}{64t}} dt.
\end{align*}
For simplicity, we write $k=R/\theta\geq 4$ and $r=|y-a|^2/64> 0$ so that
 \begin{align}\label{4ec1.3}
g_{a,2}(y)\leq  1+\int_1^{\infty}  e^{-t/k}\frac{1}{t} e^{-{r}/{t}} dt:=1+I.
\end{align}
 By a change of variable in $I$, we get
 \begin{align}\label{4eb1.4}
I=& \int_{1/k}^{\infty} e^{-t}\frac{1}{t} e^{-{r}/(tk)} dt=\int_{1/k}^{1} e^{-t}\frac{1}{t} e^{-{r}/(tk)} dt+\int_{1}^{\infty} e^{-t}\frac{1}{t} e^{-{r}/(tk)} dt\nn\\
\leq &\int_{1/k}^{1} \frac{1}{t} e^{-{r}/(tk)} dt+\int_{1}^{\infty} e^{-t} dt=\int_{1/k}^{1} \frac{1}{t} e^{-{r}/(tk)} dt+e^{-1}:=J+e^{-1}.
\end{align}
Another change of variable with $s=r/(tk)$ in $J$ gives us that
\begin{align*}
J=&\int_{r/k}^{r} \frac{1}{s} e^{-s} ds\leq \int_{(r/k)\wedge 1}^{1} \frac{1}{s} e^{-s} ds+\int_{1}^{\infty} \frac{1}{s} e^{-s} ds\leq \int_{(r/k)\wedge 1}^{1} \frac{1}{s} ds+e^{-1}=\log^+(\frac{k}{r})+e^{-1}.
\end{align*}
Hence it follows that $I\leq 2e^{-1}+\log^+(\frac{k}{r})$. Returning to \eqref{4ec1.3}, we get
 \begin{align*}
g_{a,2}(y)\leq  1+\Big(2e^{-1}+\log^+ \Big(\frac{64R}{\theta |y-a|^2}\Big)\Big)\leq C+C\log^+ \Big(\frac{R}{\theta |y-a|^2}\Big),
\end{align*}
as required.
\end{proof}

\begin{lemma}\label{4l3.3}
There is some constant $c_{\ref{4l3.3}}>0$ so that for all $x\in \Z^2_R$ and $a\in \R^2$, $n\geq 1$, $\theta\geq 100$, $R\geq 4\theta+ K_{\ref{4p1.1}}$ and $\beta=1$ or $2$,
we have
\begin{align}\label{4e7.32}
\sum_{y\in \Z^2_R} p_n(x-y) (g_{a,2}(y))^\beta\leq &c_{\ref{4l3.3}} \Big(1+ \frac{1}{n} \Big(\log \frac{2R}{\theta}\Big)^\beta+\Big(\frac{R}{n\theta}\Big)^{1/2}\Big).
\end{align}
\end{lemma}

\begin{proof}
First we use Proposition \ref{4p1.1} and  \eqref{4e6.20} to get
\begin{align}\label{ec1.7}
&\sum_{y\in \Z^d_R, |y-a|<1} p_n(x-y) (g_{a,2}(y))^\beta \leq (2R+1)^2 \cdot \frac{c_{\ref{4p1.1}}}{nR^2}\Big(\log \frac{2R}{\theta}\Big)^\beta\leq  C\frac{1}{n}\Big(\log \frac{2R}{\theta}\Big)^\beta. 
\end{align}
Turning to $|y-a|\geq 1$, we apply Lemma \ref{4l3.2}  to see that for $\beta=1$ or $2$, we have
\begin{align*}
(g_{a,2}(y))^\beta \leq  c_{\ref{4l3.2}}^\beta \Big(1+\log^+ \Big(\frac{R}{\theta |y-a|^2}\Big)\Big)^\beta\leq& C+C\Big(\log^+ \Big(\frac{R}{\theta |y-a|^2}\Big)\Big)^\beta\\
\leq& C+C\Big(\frac{R}{\theta |y-a|^2}\Big)^{1/2},
\end{align*}
where in the last inequality we have used $(\log^+x)^\beta  \leq  x^{1/2}$, $\forall x>0$.
Hence it follows that
\begin{align}\label{ec1.8}
&\sum_{y\in \Z^d_R, |y-a|\geq 1} p_n(x-y) (g_{a,2}(y))^\beta \leq C  +C\Big(\frac{R}{\theta}\Big)^{1/2}\sum_{y\in \Z^d_R, |y-a|\geq 1} p_n(x-y)\frac{1}{|y-a|}.
\end{align}
Since we are summing over $|y-a|\geq 1$, we may apply Lemma \ref{4l4.1} to see that
\begin{align*}
\sum_{k=1}^\infty  \frac{1}{k^{3/2}} e^{-\frac{|y-a|^2}{64k}}\geq c_{\ref{4l4.1}}  \frac{64^{1/2}}{|y-a|}.
\end{align*}
Use the above to get
\begin{align}\label{ec1.9}
\sum_{y\in \Z^d_R, |y-a|\geq 1}  p_n(x-y)\frac{1}{|y-a|}
 \leq & C\sum_{y\in \Z^d_R} p_n(y-x) \sum_{k=1}^\infty  \frac{1}{k^{3/2}} e^{-\frac{|y-a|^2}{64k}}\leq C\frac{c_{\ref{4l1.3}}}{n^{1/2}},
\end{align}
where the last inequality is by Lemma \ref{4l1.3}. Now the result follows from \eqref{ec1.7}, \eqref{ec1.8} and \eqref{ec1.9}.
\end{proof}

 Recall ${T_\theta^R}=[TR/\theta]\leq TR/\theta$.   Apply \eqref{4e6.20} and Lemma \ref{4l3.3} with $\beta=1$ to get
\begin{align}\label{4ec7.4}
G(g_{a,2},{T_\theta^R})=&3\|g_{a,2}\|_\infty+\sum_{k=1}^{{T_\theta^R}} \sup_{y\in \Z_R^d} \sum_{z\in \Z^d_R} p_k(y-z) g_{a,2}(z)\\
\leq &3\log \frac{2R}{\theta}+\sum_{k=1}^{{T_\theta^R}}  c_{\ref{4l3.3}}  \Big(1+ \frac{1}{k}\log \frac{2R}{\theta}+\Big(\frac{R}{\theta}\Big)^{1/2} \frac{1}{k^{1/2}}\Big)\nn
\\
\leq &3 \frac{2R}{\theta} +C {T_\theta^R} +C \log \frac{2R}{\theta} \cdot C\log(T_\theta^R)+C (\frac{R}{\theta})^{1/2} \cdot C({{T_\theta^R}})^{1/2}\leq C(T) \frac{R}{\theta},\nn
\end{align}
where in the last inequality we have used $\log(x)\leq x^{1/2}$ for any $x>0$.  Hence it follows from \eqref{4ec7.2} that
\begin{align}\label{4e7.2}
G(g_{a,2},{T_\theta^R})\leq CG(g_{a,2},{T_\theta^R})\leq c(T) \frac{R}{\theta}.
\end{align}
Now we are ready to give the
\begin{proof} [Proof of Proposition \ref{4p2.1}(i)]
 Let $\lambda=\theta^{3/2} R^{-2}$ and $n=T_\theta^R\leq \frac{TR}{\theta}$. Use \eqref{4e7.2} to get
\begin{align}\label{4ea7.8}
 2\lambda T_\theta^R e^{\frac{T_\theta^R\theta}{ R}} G(g_a,T_\theta^R)\leq 2\theta^{3/2} R^{-2} \frac{TR}{\theta}e^{T} \cdot c(T) \frac{R}{\theta}\leq c(T) \frac{1}{\theta^{1/2}}.
\end{align}
 If we pick $\theta>0$ large enough so that $c(T)/{\theta^{1/2}}\leq 1/2$, then we may apply Proposition \ref{4p1.4} to get (recall $|Z_0|\leq 2R/\sqrt{\theta}$ by \eqref{4e7.1})
\begin{align}\label{4e9.04}
\E^{Z_0}\Big(\exp\Big(\lambda \sum_{k=0}^{{T_\theta^R}} Z_k(g_a)\Big)\Big)\leq& \exp\Big(\lambda |Z_0| e^{\frac{T_\theta^R\theta}{ R}}  G(g_a,{T_\theta^R})  (1-2\lambda {T_\theta^R} e^{\frac{T_\theta^R\theta}{ R}} G(g_a,{T_\theta^R}))^{-1}\Big)\nn\\
\leq &\exp\Big(\lambda \frac{2R}{\sqrt{\theta}} e^{T}  c(T) \frac{R}{\theta} (1-c(T) \frac{1}{\theta^{1/2}} )^{-1}\Big)\nn\\
\leq& \exp\Big(C(T)  (1-c(T) \frac{1}{\theta^{1/2}} )^{-1}\Big)\leq e^{2C(T)},
\end{align}
where we have used \eqref{4e7.2}, \eqref{4ea7.8} in the second inequality and the last inequality is by $c(T)/{\theta^{1/2}}\leq 1/2$. 
\end{proof}

Turning to the difference moments in Proposition \ref{4p2.1} (ii), we need an estimate for $G(|g_a-g_b|,{T_\theta^R})$. Fix $\eta=1/8$ throughout the rest of this section.
For any $a\neq b \in \Z^d$ and $y \in \Z^d_R$, we have $|(y-a)-(y-b)|\geq 1$ and so we may apply Proposition \ref{4p1.1}(ii) to get
\begin{align}\label{4e9.22}
|{g}_a(y)-{g}_b(y)|\leq&  V(R)\sum_{k=1}^\infty e^{-k\theta/R}  \frac{C_{\ref{4p1.1}}}{k R^2} \Big(\frac{|a-b|}{\sqrt{k}}\Big)^\eta (e^{-\frac{|y-a|^2}{64k}}+e^{-\frac{|y-b|^2}{64k}})\nn\\
\leq& C |a-b|^{\eta}   \sum_{k=1}^\infty  \frac{1}{k^{(2+\eta)/2}} (e^{-\frac{|y-a|^2}{64k}}+e^{-\frac{|y-b|^2}{64k}}).
\end{align}

For any $x\in \Z^d_R$ and any $n\geq 1$, we may use \eqref{4e9.22} and Lemma \ref{4l1.3} to see that
\begin{align}\label{4e9.24}
&\sum_{y\in \Z^d_R} p_n(y-x) |g_a(y)-g_b(y)|\nn\\
\leq &C  |a-b|^{\eta}  \sum_{y\in \Z^d_R}  p_n(y-x) \sum_{k=1}^\infty \frac{1}{k^{(2+\eta)/2}} (e^{-\frac{ |a-y|^2}{64k}}+e^{-\frac{ |b-y|^2}{64k}})\nn\\
\leq &C |a-b|^{\eta} \cdot 2 c_{\ref{4l1.3}}  n^{-\eta/2}.
\end{align}
Apply \eqref{4e9.22} and \eqref{4e9.24} to get
\begin{align}\label{4e6.43}
G(|g_a-g_b|,{T_\theta^R})=&3\|g_a-g_b\|_\infty+\sum_{k=1}^{{T_\theta^R}} \sup_{y\in \Z_R^d} \sum_{z\in \Z^d_R} p_k(y-z) |g_a(y)-g_b(y)|\nn\\
\leq& C(\eta) |a-b|^{\eta}+C(\eta)\sum_{k=1}^{{T_\theta^R}} |a-b|^{\eta} k^{-\frac{\eta}{2}}\nn\\
\leq&  c(\eta)  |a-b|^{\eta} (T_\theta^R)^{1-\eta/2}\leq c(T) |a-b|^\eta \frac{R^{1-\eta/2}}{\theta^{1-\eta/2}}.
\end{align}
Now we give the
\begin{proof} [Proof of Proposition \ref{4p2.1}(ii)]
 Let $\lambda={\theta^{(3-\eta)/2}} {R^{-2+\eta/2}}|a-b|^{-\eta}$ and $n=T_\theta^R\leq \frac{TR}{\theta}$.  Note by \eqref{4e6.43} we have
\begin{align}\label{4ea7.7}
 2\lambda T_\theta^R e^{\frac{T_\theta^R\theta}{ R}} G(|g_a-g_b|,{T_\theta^R})\leq& 2{\theta^{(3-\eta)/2}} {R^{-2+\eta/2}}|a-b|^{-\eta} \frac{TR}{\theta}e^{T} \cdot c(T) |a-b|^\eta \frac{R^{1-\eta/2}}{\theta^{1-\eta/2}}\nn\\
 \leq&c(T) \frac{1}{\theta^{1/2}}.
\end{align}
 If we pick $\theta>0$ large enough so that $c(T)/{\theta^{1/2}}\leq 1/2$, then we may apply Proposition \ref{4p1.4} to get (recall $|Z_0|\leq 2R/\sqrt{\theta}$)
\begin{align}\label{4e9.03}
&\E^{{Z}_0}\Big(\exp\Big({\lambda|\sum_{k=0}^{{T_\theta^R}} {Z}_{k}(g_a)-\sum_{k=0}^{{T_\theta^R}} {Z}_{k}(g_b)|}\Big)\Big)\leq \E^{{Z}_0}\Big(\exp\Big({\lambda \sum_{k=0}^{{T_\theta^R}} {Z}_{k}(|g_a-g_b|)}\Big)\Big)\nn\\
\leq&\exp\Big(\lambda  |{Z}_0| e^{\frac{T_\theta^R\theta}{ R}} \cdot G(|g_a-g_b|,{T_\theta^R}) (1-2\lambda  {T_\theta^R} e^{\frac{T_\theta^R\theta}{ R}}  G(|g_a-g_b|,{T_\theta^R}))^{-1}\Big)\nn\\
\leq &\exp\Big(\lambda \frac{2R}{\sqrt{\theta}} e^T \cdot c(T) |a-b|^\eta \frac{R^{1-\eta/2}}{\theta^{1-\eta/2}}  (1-c(T) \frac{1}{\theta^{1/2}})^{-1}\Big)\nn\\
\leq &\exp\Big(C(T)  (1-c(T) \frac{1}{\theta^{1/2}} )^{-1}\Big)\leq e^{2C(T)},
\end{align}
where we have used \eqref{4e6.43}, \eqref{4ea7.7} in the second last inequality and the last inequality is by $c(T)/{\theta^{1/2}}\leq 1/2$. 
\end{proof}

\subsection{Exponential moments of the martingale term}

Now we will turn to the martingale term $M_{{T_\theta^R}+1}(g_a)$ and give the proof of Proposition \ref{4p2.2}. Recall from \eqref{4e1.22} and \eqref{4e1.30} that
\begin{align}\label{4e7.41}
M_N(\phi)=\sum_{n=0}^{N-1}\sum_{|\alpha|=n} \sum_{i=1}^{V(R)} \phi({Y^\alpha+e_i}) \Big(B^{\alpha \vee e_i}-p(R)\Big).
\end{align}
and
\begin{align}\label{4e7.42}
\langle M(\phi)\rangle_N\leq 2\sum_{n=0}^{N-1} \sum_{x\in \Z^d_R} Z_n(x) \cdot \frac{1}{V(R)}\sum_{i=1}^{V(R)} \phi({x+e_i})^2,
\end{align}
We first proceed to the proof of Proposition \ref{4p2.2}(ii) and deal with 
\begin{align}\label{4eb1.1}
\Big||M_{{T_\theta^R}+1}(g_a)|-|M_{{T_\theta^R}+1}(g_b)|\Big|\leq |M_{{T_\theta^R}+1}(g_a)-M_{{T_\theta^R}+1}(g_b)|=|M_{{T_\theta^R}+1}(g_a-g_b)|.
\end{align}
Throughout the rest of this section we fix $\eta=1/8$. Use $R\geq 4\theta$ and \eqref{4e9.22} to see that
\begin{align}\label{4ea7.9}
\theta^{(3-2\eta)/4} R^{(\eta-2)/2}|a-b|^{-\eta}\|g_a-g_b\|_\infty \leq \theta^{(3-2\eta)/4} (4\theta)^{(\eta-2)/2} C(\eta)\leq 1,
\end{align}
if we pick $\theta\geq 100$ to be large. Then we may use \eqref{4eb1.1} and Proposition \ref{4p5.1} with $\phi=g_a-g_b$ and $\lambda=\theta^{(3-2\eta)/4} R^{(\eta-2)/2}|a-b|^{-\eta}$ to get
\begin{align}\label{4e9.05}
&\E^{{Z}_0}\Big(\exp\Big(\theta^{(3-2\eta)/4} \frac{R^{(\eta-2)/2}}{|a-b|^\eta} \Big||M_{{T_\theta^R}+1}(g_a)|-|M_{{T_\theta^R}+1}(g_b)|\Big|\Big)\Big)\nn\\
\leq& \E^{{Z}_0}\Big(\exp\Big(\theta^{(3-2\eta)/4} \frac{R^{(\eta-2)/2}}{|a-b|^\eta} \Big|M_{{T_\theta^R}+1}(g_a-g_b)\Big|\Big)\Big)\nn\\
\leq& 2\Big(\E^{{Z}_0}\Big(\exp\Big(16\theta^{(3-2\eta)/2} \frac{R^{\eta-2}}{|a-b|^{2\eta}} \langle M(g_a-g_b)\rangle_{{T_\theta^R}+1}\Big)\Big)\Big)^{1/2}.
\end{align}
By \eqref{4e7.42}, we have the quadratic variation is bounded by
\begin{align}\label{4ea7.10}
\langle M(g_a-g_b)\rangle_{{T_\theta^R}+1}\leq &2 \sum_{n=0}^{{T_\theta^R}} 
\sum_{x\in \Z^d_R} Z_n(x) \cdot \frac{1}{V(R)}\sum_{i=1}^{V(R)} \Big(g_a({x+e_i})-g_b(x+e_i)\Big)^2.
\end{align}
Use \eqref{4e9.22} again to get for all $a\neq b\in \Z^2$ and $y\in \Z^2_R$,
\begin{align}\label{4e9.25}
|{g}_a(y)-{g}_b(y)|^2\leq C |a-b|^{2\eta}\Big(\Big(\sum_{k=1}^\infty  \frac{1}{k^{(2+\eta)/2}} e^{-\frac{|y-a|^2}{64k}}\Big)^2+\Big(\sum_{k=1}^\infty  \frac{1}{k^{(2+\eta)/2}} e^{-\frac{|y-b|^2}{64k}}\Big)^2\Big).
\end{align}
To take care of the square term on the right-hand side, we need the following lemma.
\begin{lemma}\label{4l4.1.1}
Let $d\geq 1$. For any $\alpha>0$, there is some constant $C_{\ref{4l4.1.1}}(\alpha)>0$ such that for all $a,y\in \Z_R^d$,
\begin{align}\label{4eb1.9}
 \Big(\sum_{k=1}^\infty  \frac{1}{k^{1+\alpha}} e^{-\frac{|y-a|^2}{64k}} \Big)^2\leq C_{\ref{4l4.1.1}}(\alpha) \sum_{k=1}^\infty  \frac{1}{k^{1+2\alpha}} e^{-\frac{|y-a|^2}{64k}}.
 \end{align}
\end{lemma}
\begin{proof}
For any $a, y\in \Z^d_R$, we first consider $|y-a|>1$.  Apply Lemma \ref{4l4.1} with $r=|y-a|^2/64>1/64$ to get
\begin{align*}
&\sum_{k=1}^\infty  \frac{1}{k^{1+\alpha}} e^{-\frac{|y-a|^2}{64k}}\leq C_{\ref{4l4.1}}(\alpha) \frac{64^{\alpha}}{|y-a|^{2\alpha}},\text{ and  } \sum_{k=1}^\infty  \frac{1}{k^{1+2\alpha}} e^{-\frac{|y-a|^2}{64k}}\geq c_{\ref{4l4.1}}(2\alpha) \frac{64^{2\alpha}}{|y-a|^{4\alpha}}.
\end{align*}
Therefore it follows that
\begin{align*}
&\Big(\sum_{k=1}^\infty  \frac{1}{k^{1+\alpha}} e^{-\frac{|y-a|^2}{64k}}\Big)^2 1_{\{|y-a|>  1\}}\leq C_{\ref{4l4.1}}(\alpha)^2 \frac{64^{2\alpha}}{|y-a|^{4\alpha}} 1_{\{|y-a|>  1\}}\nn\\
\leq& C_{\ref{4l4.1}}(\alpha)^2 c_{\ref{4l4.1}}(2\alpha)^{-1} \sum_{k=1}^\infty  \frac{1}{k^{1+2\alpha}} e^{-\frac{|y-a|^2}{64k}} 1_{\{|y-a|>  1\}},
\end{align*}
thus proving \eqref{4eb1.9} for the case $|y-a|>1$. 

Turning to $|y-a|\leq 1$, it is immediate from the definition that 
\begin{align*}
&\Big(\sum_{k=1}^\infty  \frac{1}{k^{1+\alpha}} e^{-\frac{|y-a|^2}{64k}}\Big)^2 1_{\{|y-a|\leq 1\}}\leq\Big(\sum_{k=1}^\infty  \frac{1}{k^{1+\alpha}} \Big)^2 1_{\{|y-a|\leq 1\}}\leq c_1(\alpha) 1_{\{|y-a|\leq  1\}}
\end{align*}
for some constants $c_1(\alpha)>0$.
On the other hand, we have
\begin{align*}
&\sum_{k=1}^\infty  \frac{1}{k^{1+2\alpha}} e^{-\frac{|y-a|^2}{64k}} 1_{\{|y-a|\leq  1\}} \geq  e^{-\frac{1}{64}} 1_{\{|y-a|\leq  1\}}\sum_{k=1}^\infty  \frac{1}{k^{1+2\alpha}} \geq c_2(\alpha)  1_{\{|y-a|\leq  1\}}
\end{align*}
for some constants $c_2(\alpha)>0$.
Therefore it follows that 
\begin{align*}
&\Big(\sum_{k=1}^\infty  \frac{1}{k^{1+\alpha}} e^{-\frac{|y-a|^2}{64k}}\Big)^2 1_{\{|y-a|\leq 1\}}\leq c_1(\alpha) 1_{\{|y-a|\leq  1\}}\\
&=\frac{c_1(\alpha)}{c_2(\alpha)} c_2(\alpha) 1_{\{|y-a|\leq  1\}}\leq \frac{c_1(\alpha)}{c_2(\alpha)} \sum_{k=1}^\infty  \frac{1}{k^{1+2\alpha}} e^{-\frac{|y-a|^2}{64k}} 1_{\{|y-a|\leq  1\}},
\end{align*}
thus proving \eqref{4eb1.9} for the case $|y-a|\leq 1$.
By adjusting constants, we get \eqref{4eb1.9} holds for all $a,y\in \Z^d_R$.
\end{proof}
Apply the above lemma in \eqref{4e9.25} to get
\begin{align}\label{4eb1.10}
|{g}_a(y)-{g}_b(y)|^2\leq C |a-b|^{2\eta} C_{\ref{4l4.1.1}}(\frac{\eta}{2})\Big(\sum_{k=1}^\infty  \frac{1}{k^{1+\eta}} e^{-\frac{|y-a|^2}{64k}}+\sum_{k=1}^\infty  \frac{1}{k^{1+\eta}} e^{-\frac{|y-b|^2}{64k}}\Big).
\end{align}
Define for any $a\in \Z^2$ that
\begin{align}\label{4e5.21}
q_a(x)=\sum_{k=1}^\infty  \frac{1}{k^{1+\eta}} e^{-\frac{|x-a|^2}{64k}}\text{ and write } 
\overline{q_a}(x)=\frac{1}{V(R)}\sum_{i=1}^{V(R)} q_a(x+e_i).
\end{align}
Then we may apply \eqref{4eb1.10} and \eqref{4e5.21} in \eqref{4ea7.10}  to get
\begin{align*}
\langle M(g_a-g_b)\rangle_{{T_\theta^R}+1} \leq& 2\sum_{n=0}^{{T_\theta^R}}\sum_{x\in \Z^d_R} Z_n(x) \cdot C |a-b|^{2\eta}( \overline{q_a}(x) +\overline{q_b}(x))\\
\leq &C |a-b|^{2\eta} \sum_{n=0}^{{T_\theta^R}}  Z_n(\overline{q_a}+\overline{q_b}).
\end{align*}
Returning to \eqref{4e9.05}, we use above to arrive at
\begin{align}\label{4e9.112}
&\E^{{Z}_0}\Big(\exp\Big(\theta^{(3-2\eta)/4} \frac{R^{(\eta-2)/2}}{|a-b|^\eta} \Big||M_{{T_\theta^R}+1}(g_a)|-|M_{{T_\theta^R}+1}(g_b)|\Big|\Big)\Big)\nn\\
\leq& 2\Big(\E^{{Z}_0}\Big(\exp\Big(16 C\theta^{(3-2\eta)/2} R^{\eta-2}  \sum_{n=0}^{{T_\theta^R}}  Z_n(\overline{q_a}+\overline{q_b}) \Big)\Big)\Big)^{1/2}\nn\\
\leq& 2\Big(\E^{{Z}_0}\Big(\exp\Big(32 C\theta^{3/2-\eta} R^{\eta-2}  \sum_{n=0}^{{T_\theta^R}}  Z_n(\overline{q_a}) \Big)\Big)\Big)^{1/4}\nn\\
&\times\Big(\E^{{Z}_0}\Big(\exp\Big(32 C\theta^{3/2-\eta} R^{\eta-2}  \sum_{n=0}^{{T_\theta^R}}  Z_n(\overline{q_b}) \Big)\Big)\Big)^{1/4},
\end{align}
where the last inequality is by the Cauchy-Schwartz inequality. It suffices to bound
\begin{align}\label{4e5.22}
\E^{{Z}_0}\Big(\exp\Big(32 C\theta^{3/2-\eta} R^{\eta-2}  \sum_{n=0}^{{T_\theta^R}}  Z_n(\overline{q_a}) \Big)\Big),\quad \forall a\in \Z^d_R.
\end{align}

Recalling $q_a$ from \eqref{4e5.21}, we may use Lemma \ref{4l1.3} to get for any $a, x\in \Z^d_R$,
\begin{align}\label{4e9.72}
\sum_{y\in \Z^d_R} p_n(y-x) q_a(y)= \sum_{y\in \Z^d_R} p_n(y-x) \sum_{k=1}^\infty  \frac{1}{k^{1+\eta}} e^{-\frac{|y-a|^2}{64k}}\leq c_{\ref{4l1.3}} \cdot  \frac{1}{n^{\eta}}.
\end{align}
Recall $\overline{q_a}$ from \eqref{4e5.21}. The above immediately gives 
\begin{align}
\sum_{y\in \Z^d_R} p_n(y-x) \overline{q_a}(y)\leq  c_{\ref{4l1.3}} \cdot   \frac{1}{n^{\eta}}, \quad \forall a, x\in \Z^d_R.
\end{align}
Therefore we have
\begin{align}\label{4e5.23}
G(\overline{q_a},{T_\theta^R})=&3\|\overline{q_a}\|_\infty+\sum_{k=1}^{{T_\theta^R}} \sup_{y\in \Z_R^d} \sum_{z\in \Z^d_R} p_k(y-z) \overline{q_a}(z)\nn\\
\leq &3C(\eta)+\sum_{k=1}^{{T_\theta^R}}  c_{\ref{4l1.3}} \cdot   \frac{1}{k^{\eta}}\leq C(\eta) (T_\theta^R)^{1-\eta}\leq C(T) \frac{R^{1-\eta}}{\theta^{1-\eta}}.
\end{align}
\begin{proof}[Proof of Proposition \ref{4p2.2}(ii)]
By \eqref{4e9.112}, it suffices to give bounds for \eqref{4e5.22}. Fix any $a\in \Z_R^d$. Let $\lambda=32C \theta^{3/2-\eta} {R^{\eta-2}}$ and $n=T_\theta^R\leq \frac{TR}{\theta}$.  Apply \eqref{4e5.23} to get
\begin{align}\label{4ea7.11}
 2\lambda T_\theta^R e^{\frac{T_\theta^R\theta}{ R}} G(\overline{q_a},{T_\theta^R})\leq& 64C \theta^{3/2-\eta} {R^{\eta-2}} \frac{TR}{\theta}e^{T} \cdot C(T)  \frac{R^{1-\eta}}{\theta^{1-\eta}}\leq c(T) \frac{1}{\theta^{1/2}}.
\end{align}
 If we pick $\theta>0$ large enough so that $c(T)/{\theta^{1/2}}\leq 1/2$, then we may apply Proposition \ref{4p1.4} to get (recall $|Z_0|\leq 2R/\sqrt{\theta}$)
\begin{align}\label{4e10.50}
\E^{{Z}_0}\Big(\exp\Big({\lambda \sum_{k=0}^{{T_\theta^R}} {Z}_{k}(\overline{q_a})}\Big)\Big) \leq& \exp\Big(\lambda  |{Z}_0| e^{\frac{T_\theta^R\theta}{ R}}\cdot  G(\overline{q_a},{T_\theta^R})(1-2\lambda  T_\theta^R e^{\frac{T_\theta^R\theta}{ R}}  G(\overline{q_a},{T_\theta^R}))^{-1}\Big)\nn\\
\leq &\exp\Big(\lambda \frac{2R}{\sqrt{\theta}} e^T \cdot C(T) \frac{R^{1-\eta}}{\theta^{1-\eta}} (1-c(T) \frac{1}{{\theta}^{1/2}} )^{-1}\Big)\nn\\
\leq &\exp\Big(C(T)  (1-c(T) \frac{1}{{\theta}^{1/2}} )^{-1}\Big)\leq e^{2C(T)},
\end{align}
where we have used \eqref{4e5.23}, \eqref{4ea7.11} in the second inequality. The last inequality is by $c(T)/{\theta^{1/2}}\leq 1/2$. 
Returning to \eqref{4e9.112}, we use \eqref{4e10.50} to get
\begin{align}\label{4e5.24}
&\E^{{Z}_0}\Big(\exp\Big(\theta^{(3-2\eta)/4} \frac{R^{(\eta-2)/2}}{|a-b|^\eta} \Big||M_{{T_\theta^R}+1}(g_a)|-|M_{{T_\theta^R}+1}(g_b)|\Big|\Big)\Big)\leq 2e^{C(T)}.
\end{align}
Hence the proof is complete.
\end{proof}

Finally we turn to the exponential moments of $M_{{T_\theta^R}+1}(g_a)$ and prove Proposition \ref{4p2.2}(i). Use \eqref{4e10.33}, \eqref{4e6.20} and  $R\geq 4\theta$, one can check that
\begin{align*}
\theta^{3/4} R^{-1}  \|g_a\|_\infty \leq \theta^{3/4} R^{-1} \cdot C \log \frac{2R}{\theta}\leq \frac{1}{4^{3/4}} R^{-1/4} \cdot C \log \frac{2R}{100}\leq 1,
\end{align*}
if we pick $R\geq 4\theta\geq 400$ to be large. Then we may apply Proposition \ref{4p5.1} with $\phi= g_a$ and $\lambda=\theta^{3/4} R^{-1}$ to get
\begin{align}\label{4e10.51}
&\E^{{Z}_0}\Big(\exp\Big(\theta^{3/4} R^{-1} |M_{{T_\theta^R}+1}(g_a)|\Big)\Big)\nn\\
\leq& 2\Big(\E^{{Z}_0}\Big(\exp\Big(16\theta^{3/2} R^{-2} \langle M(g_a)\rangle_{{T_\theta^R}+1}\Big)\Big)\Big)^{1/2},
\end{align}
where the quadratic variation is bounded by (recall \eqref{4e7.42})
\begin{align*}
&\langle M(g_a)\rangle_{{T_\theta^R}+1}\leq 2 \sum_{n=0}^{{T_\theta^R}} \sum_{x\in Z_n} Z_n(x) \cdot \frac{1}{V(R)}\sum_{i=1}^{V(R)} \Big(g_a({x+e_i})\Big)^2.
\end{align*}
For any $a, x\in \Z^d_R$, we define 
\begin{align}\label{4e5.25}
\overline{g_a}(x)=\frac{1}{V(R)}\sum_{i=1}^{V(R)} (g_a(x+e_i))^2
\end{align}
so that
\begin{align*}
&\langle M(g_a)\rangle_{{T_\theta^R}+1}\leq 2  \sum_{n=0}^{{T_\theta^R}} \sum_{x\in Z_n} Z_n(x) \cdot \overline{g_a}(x)=2  \sum_{n=0}^{{T_\theta^R}} Z_n(\overline{g_a}).
\end{align*}
Therefore \eqref{4e10.51} becomes
\begin{align}\label{4e9.113}
&\E^{{Z}_0}\Big(\exp\Big(\theta^{3/4} R^{-1} |M_{{T_\theta^R}+1}(g_a)|\Big)\Big)\leq 2\Big(\E^{{Z}_0}\Big(\exp\Big(32 \theta^{3/2} R^{-2} \sum_{n=0}^{{T_\theta^R}} Z_n(\overline{g_a}) \Big)\Big)\Big)^{1/2}.
\end{align}
It remains to bound
\begin{align}\label{4e5.26}
\E^{{Z}_0}\Big(\exp\Big(32 \theta^{3/2} R^{-2} \sum_{n=0}^{{T_\theta^R}} Z_n(\overline{g_a}) \Big)\Big), \quad \forall a\in \Z^d_R.
\end{align}

In order to apply Proposition \ref{4p1.4} to get bounds for \eqref{4e5.26}, we will need bounds for $G(\overline{g_a}, {T_\theta^R})$. The definition of $\overline{g_a}$ as in  \eqref{4e5.25} gives
\begin{align}\label{4e9.10}
G(\overline{g_a}, {T_\theta^R})=&3\|\overline{g_a}\|_\infty+\sum_{k=1}^{{T_\theta^R}} \sup_{x \in \Z_R^d} \sum_{y\in \Z^d_R} p_k(x-y) \overline{g_a}(y)\nn\\
\leq& 3\|{g_a}\|_\infty^2+\frac{1}{V(R)}\sum_{i=1}^{V(R)} \sum_{k=1}^{{T_\theta^R}} \sup_{x \in \Z_R^d}   \sum_{y\in \Z^d_R} p_k(x-y) (g_a(y+e_i))^2\nn\\
\leq& 3\|{g_a}\|_\infty^2+ \sum_{k=1}^{{T_\theta^R}} \sup_{x \in \Z_R^d}   \sum_{y\in \Z^d_R} p_k(x-y) (g_a(y))^2\nn\\
\leq& C(\log(\frac{2R}{\theta}))^2+C \sum_{k=1}^{{T_\theta^R}} \sup_{x \in \Z_R^d}   \sum_{y\in \Z^d_R} p_k(x-y) (g_{a,2}(y))^2,
\end{align}
where the last inequality is by \eqref{4e10.33}, \eqref{4e6.20}. Use Lemma \ref{4l3.3} with $\beta=2$ to get
\begin{align}\label{4e5.27}
G(\overline{g_a},{T_\theta^R})\leq& C(\log(\frac{2R}{\theta}))^2 +C\sum_{k=1}^{{T_\theta^R}}  c_{\ref{4l3.3}} \Big(1+ \frac{1}{k} \Big(\log \frac{2R}{\theta}\Big)^2+\Big(\frac{R}{k\theta}\Big)^{1/2}\Big)\\
\leq &C\frac{2R}{\theta}+C{T_\theta^R}+C\Big(\log \frac{2R}{\theta}\Big)^2\cdot C\log {T_\theta^R} +C\Big(\frac{R}{\theta}\Big)^{1/2}\cdot C (T_\theta^R)^{1/2} \nn\\
\leq &C\frac{R}{\theta}+C\frac{TR}{\theta}+C\Big(\frac{2R}{\theta}\Big)^{2/3}\Big( \frac{TR}{\theta}\Big)^{1/3}+C\Big(\frac{R}{\theta}\Big)^{1/2} \Big(\frac{TR}{\theta}\Big)^{1/2} \leq c(T) \frac{R}{\theta},\nn
\end{align}
where  in the second inequality we have used $\log x \leq x^{1/2}$, $\forall x>0$ and the third inequality uses $T_\theta^R\leq TR/\theta$ and  $\log x \leq x^{1/3}$ , $\forall x>0$.
\begin{proof}[Proof of Proposition \ref{4p2.2}(i)]
By \eqref{4e9.113}, it suffices to give bounds for \eqref{4e5.26}. Fix any $a\in \Z_R^d$.
Let $\lambda=32 \theta^{3/2} {R^{-2}}$ and $n=T_\theta^R\leq \frac{TR}{\theta}$.  By \eqref{4e5.27} we have
\begin{align}\label{4eb1.6}
 2\lambda T_\theta^R e^{\frac{T_\theta^R\theta}{ R}} G(\overline{g_a},{T_\theta^R})\leq& 64 \theta^{3/2} {R^{-2}} \frac{TR}{\theta}e^{T} \cdot c(T)  \frac{R}{\theta}\leq c(T) \frac{1}{\theta^{1/2}}.
\end{align}
 If we pick $\theta>0$ large enough so that $c(T)/{\theta^{1/2}}\leq 1/2$, then we may apply Proposition \ref{4p1.4} to get (recall $|Z_0|\leq 2R/\sqrt{\theta}$)
\begin{align}\label{4e10.61}
\E^{{Z}_0}\Big(\exp\Big({\lambda \sum_{k=0}^{{T_\theta^R}} {Z}_{k}(\overline{g_a})}\Big)\Big) \leq& \exp\Big(\lambda  |{Z}_0| e^{\frac{T_\theta^R\theta}{ R}} G(\overline{g_a},{T_\theta^R}) (1- 2\lambda T_\theta^R e^{\frac{T_\theta^R\theta}{ R}} G(\overline{g_a},{T_\theta^R}))^{-1}\Big)\nn\\
\leq& \exp\Big(\lambda \frac{2R}{\sqrt{\theta}} e^{T}  c(T)  \frac{R}{\theta} (1-c(T) \frac{1}{\theta^{1/2}})^{-1}\Big)\nn\\
\leq &\exp\Big(C(T)  (1-c(T) \frac{1}{{\theta}^{1/2}} )^{-1}\Big)\leq e^{2C(T)},
\end{align}
where the second inequality uses \eqref{4e5.27} and \eqref{4eb1.6} and the last inequality follows by $c(T)/{\theta^{1/2}}\leq 1/2$. Returning to \eqref{4e9.113}, we use \eqref{4e10.61} to get
\begin{align}\label{4e5.28}
&\E^{{Z}_0}\Big(\exp\Big(\theta^{3/4} R^{-1} |M_{{T_\theta^R}+1}(g_a)|\Big)\Big)\leq 2e^{C(T)},
\end{align}
thus completing the proof.
\end{proof}

\section{Local time bounds in $d=3$}\label{4s6}
In this section we give the proof of Proposition \ref{4p2} for $d=3$. Recall $Z_0\in M_F(\Z_R^3)$ satisfies
\begin{align}\label{4eb1.7}
\begin{dcases}
\text{(i) }\text{Supp}(Z_0)\subseteq Q_{R_\theta}(0); \\
\text{(ii) } Z_0(1)\leq 2 R^2 f_3(\theta)/\theta=2 R^2 \log \theta/\theta;\\
\text{(iii) } Z_{0}(g_{u,3})\leq m {R^2}/\theta^{1/4}, \quad \forall u\in \R^3.
\end{dcases}
\end{align}
Similar to $d=2$, it suffices to get bounds for the local time at  points in the integer lattice. 
\begin{proposition}\label{4t6.1}
Let $d=3$. For any $\eps_0\in (0,1)$, $T\geq 100$ and $m>0$, there exist constants $\theta_{\ref{4t6.1}}\geq 100, \chi_{\ref{4t6.1}}>0$ depending only on $\eps_0, T,m$ such that for all $\theta \geq \theta_{\ref{4t6.1}}$,  there is some $C_{\ref{4t6.1}}(\eps_0, T,\theta,m)\geq 4\theta$ such that for any $R\geq  C_{\ref{4t6.1}}$ and any $Z_0$ satisfying \eqref{4eb1.7}, we have
\[
\P^{Z_0}\Big(\sum_{n=0}^{T_\theta^R} Z_n(\cN(a)) \leq  \chi_{\ref{4t6.1}} {R}, \quad \forall a\in \Z^3 \cap Q_{3M_{\ref{4p1}} \sqrt{\log f_3(\theta)}R_\theta}  (0) \Big)\geq 1-\eps_0.
\]
\end{proposition}
\begin{proof}[Proof of Proposition \ref{4p2} in $d=3$ assuming Proposition \ref{4t6.1}]
This follows from similar arguments used for $d=2$.
\end{proof}

It remains to prove Proposition \ref{4t6.1}.
In view of \eqref{4e6.13}, it suffices to get bounds for $Z_0(\phi_a)$, $M_{{T_\theta^R}+1}(\phi_a)$ and $\sum_{n=0}^{{T_\theta^R}} Z_n(\phi_a)$ where 
$\phi_a(x)=RV(R)\sum_{n=1}^\infty p_n(x-a)$.
Recall from \eqref{4e7.10} that for any $a,x \in \Z^d_R$,
\begin{align}\label{4e7.10a}
\phi_a(x)\leq   CR \sum_{n=1}^\infty \frac{1}{n^{3/2}} e^{-\frac{|x-a|^2}{32n}}=Cg_{a,3}(x).
\end{align}
Therefore we may use the above and \eqref{4eb1.7} to see that
\begin{align}\label{4e5.33}
Z_0(\phi_a)\leq CZ_0(g_{a,3}) \leq C\frac{mR^2}{\theta^{1/4}}, \quad \forall a\in \Z^d.
\end{align}

Turning to $M_{{T_\theta^R}+1}(g_a)$ and $\sum_{n=0}^{{T_\theta^R}} Z_n(g_a)$, we will also calculate their exponential moments.

\begin{proposition}\label{4p3.1}
Let $\eta=1/8$. For any $T\geq 100$, there exist constants $C_{\ref{4p3.1}}(T)>0$ and $\theta_{\ref{4p3.1}}(T)\geq 100$ such that for all  $\theta \geq \theta_{\ref{4p3.1}}(T)$,  there is some $K_{\ref{4p3.1}}(T,\theta)\geq 4\theta$ such that for any $m>0$, $R\geq  K_{\ref{4p3.1}}$ and any $Z_0$ satisfying \eqref{4eb1.7}, we have
\begin{align*}
&\text{(i) }
\E^{Z_0}\Big(\exp\Big(\frac{\theta^{3/2}R^{-4} }{\log \theta}  \sum_{k=0}^{{T_\theta^R}} Z_k(\phi_a)\Big)\Big)\leq C_{\ref{4p3.1}}(T), \quad \forall a \in \Z^3,\\
&\text{(ii) }
\E^{{Z}_0}\Big(\exp\Big(\frac{\theta^{3/2}R^{-4} }{\log \theta}  \frac{(R^2/\theta)^{\eta/2} }{ |a-b|^{\eta}}  |\sum_{k=0}^{{T_\theta^R}} {Z}_{k}(\phi_a)-\sum_{k=0}^{{T_\theta^R}} {Z}_{k}(\phi_b)|\Big)\Big)
\leq C_{\ref{4p3.1}}(T), \quad \forall a\neq b \in \Z^3.
\end{align*}
\end{proposition}

\begin{corollary}\label{4c3.1}
For any $\eps_0\in (0,1)$ and $T\geq 100$, there exist constants $\chi_{\ref{4c3.1}}>0$ and $ \theta_{\ref{4c3.1}}\geq 100$ depending only on $\eps_0, T$ such that for all $\theta \geq \theta_{\ref{4c3.1}}$,  there is some $C_{\ref{4c3.1}}(\eps_0, T,\theta)\geq 4\theta$ such that for any $m>0$,  $R\geq  C_{\ref{4c3.1}}$ and any $Z_0$ satisfying \eqref{4eb1.7}, we have
\begin{align*}
\P^{Z_0}\Big(\frac{\theta}{R^2} \sum_{k=0}^{{T_\theta^R}} Z_k(\phi_a) \leq \chi_{\ref{4c3.1}} \frac{R^2}{\theta^{1/16}}, \quad \forall a\in  \Z^3 \cap Q_{3M_{\ref{4p1}} \sqrt{\log f_3(\theta)}R_\theta}  (0) \Big)\geq 1-\frac{\eps_0}{2}.
\end{align*}
\end{corollary}

\begin{proof}
By using Proposition \ref{4p3.1}, the proof follows in a similar way to that of Corollary \ref{4c2.1} .
\end{proof}

In the previous calculation of the exponential moments, we never use the regularity condition (iii) of $Z_0$ in \eqref{4eb1.7}. It was also not used in the corresponding calculation for $d=2$ in Section \ref{4s5}. The case for the martingale term in $d=3$ is slightly different--condition (iii) of $Z_0$ will enter in the calculation of its exponential moments (see the proof of Proposition \ref{4p5.5}). This makes the arguments rather tedious compared to other terms. 

\begin{proposition}\label{4p3.2}
Let $\eta=1/8$. For any $T\geq 100$ and $m>0$, there exist constants $C_{\ref{4p3.2}}(T,m)>0$ and $\theta_{\ref{4p3.2}}(T,m)\geq 100$ such that for all $\theta \geq \theta_{\ref{4p3.2}}(T,m)$,  there is some $K_{\ref{4p3.2}}(T,\theta, m)\geq 4\theta$ such that for any $R\geq  K_{\ref{4p3.2}}$ and any $Z_0$ satisfying \eqref{4eb1.7}, we have
\begin{align*}
&\text{(i) }
\E^{Z_0}\Big(\exp\Big(\theta^{\eta} R^{-2}  |M_{{T_\theta^R}+1}(\phi_a)|\Big)\Big)\leq C_{\ref{4p3.2}}(T,m), \quad \forall a \in \Z^3,\\
&\text{(ii) }
\E^{{Z}_0}\Big(\exp\Big( \theta^{\eta} R^{-2} \frac{(R^2/\theta)^{\eta/2}}{|a-b|^\eta} \Big||M_{{T_\theta^R}+1}(\phi_a)|-|M_{{T_\theta^R}+1}(\phi_b)|\Big|\Big)\Big)
\leq C_{\ref{4p3.2}}(T,m), \quad \forall a\neq b \in \Z^3.
\end{align*}
\end{proposition}

\begin{corollary}\label{4c3.2}
For any $\eps_0\in (0,1)$, $T\geq 100$ and $m>0$, there exist constants $\chi_{\ref{4c3.2}}>0$ and $ \theta_{\ref{4c3.2}}\geq 100$ depending only on $\eps_0, T, m$ such that for all $\theta \geq \theta_{\ref{4c3.2}}$,  there is some $C_{\ref{4c3.2}}(\eps_0, T,\theta,m)\geq 4\theta$ such that for any $R\geq  C_{\ref{4c3.2}}$ and any $Z_0$ satisfying \eqref{4eb1.7}, we have
\[
\P^{Z_0}\Big(  |M_{{T_\theta^R}+1}(\phi_a)|\leq \chi_{\ref{4c3.2}}\frac{R^2}{\theta^{1/16}}, \quad \forall a\in  \Z^3 \cap Q_{3M_{\ref{4p1}} \sqrt{\log f_3(\theta)}R_\theta}  (0) \Big)\geq 1-\frac{\eps_0}{2}.
\]

\end{corollary}
\begin{proof}
By using Proposition \ref{4p3.2}, the proof follows in a similar way to that of Corollary \ref{4c2.1}.
\end{proof}

Assuming Proposition \ref{4p3.1}, Proposition \ref{4p3.2}, we first finish the proof of Proposition \ref{4t6.1}.
\begin{proof}[Proof of Proposition \ref{4t6.1}]
Fix $\eps_0\in (0,1)$, $T\geq 100$ and $m>0$. Let $\theta_{\ref{4t6.1}}=\max\{\theta_{\ref{4c3.1}}, \theta_{\ref{4c3.2}}\}$.  For any $\theta \geq \theta_{\ref{4t6.1}}$, we let $C_{\ref{4t6.1}}=\max\{C_{\ref{4c3.1}}(\eps_0, T,\theta), C_{\ref{4c3.2}}(\eps_0, T,\theta, m)\}$. For any $R\geq C_{\ref{4t6.1}}$, we let $Z_0$ be as in \eqref{4eb1.7}.
Apply Corollary \ref{4c3.1} to get with probability $\geq 1-\eps_0/2$,
\begin{align}\label{4eac7.1}
\frac{\theta}{R^2} \sum_{k=0}^{{T_\theta^R}} Z_k(\phi_a) \leq \chi_{\ref{4c3.1}} \frac{R^2}{\theta^{1/16}}, \quad \forall a\in  \Z^3 \cap Q_{3M_{\ref{4p1}} \sqrt{\log f_3(\theta)}R_\theta}  (0).
\end{align}
Apply Corollary \ref{4c3.2} to get with probability $\geq 1-\eps_0/2$,
\begin{align}\label{4eac7.2}
 |M_{{T_\theta^R}+1}(\phi_a)|\leq \chi_{\ref{4c3.2}}\frac{R^2}{\theta^{1/16}}, \quad\forall a\in  \Z^3 \cap Q_{3M_{\ref{4p1}} \sqrt{\log f_3(\theta)}R_\theta}  (0).
  \end{align}
Therefore with probability $\geq 1-\eps_0$, both \eqref{4eac7.1} and \eqref{4eac7.2} hold. Use \eqref{4e6.13} to get for any $a\in \Z^3 \cap Q_{3M_{\ref{4p1}} \sqrt{\log f_3(\theta)}R_\theta}  (0)$,
\begin{align*}
R\sum_{n=0}^{T_\theta^R} Z_n(\cN(a))\leq &Z_{0}(\phi_a)+M_{T_\theta^R+1}(\phi_a)+\frac{\theta}{R^{2}}\sum_{n=0}^{T_\theta^R} Z_n(\phi_a)\\
\leq &  C\frac{mR^2}{\theta^{1/4}}+\chi_{\ref{4c3.1}} \frac{R^2}{\theta^{1/16}}+\chi_{\ref{4c3.2}}\frac{R^2}{\theta^{1/16}}\leq (Cm+\chi_{\ref{4c3.1}}+\chi_{\ref{4c3.2}}) R^2,
\end{align*}
where the first inequality is by \eqref{4e5.33}.
The proof is complete by letting $\chi_{\ref{4t6.1}}=Cm+\chi_{\ref{4c3.1}}+\chi_{\ref{4c3.2}}$. 
\end{proof}

It remains to prove Proposition \ref{4p3.1} and Proposition \ref{4p3.2}.

\subsection{Exponential moments of the drift term}
In this section we will prove Proposition \ref{4p3.1} for the exponential moments of $\sum_{k=0}^{{T_\theta^R}} Z_k(\phi_a)$. For any $x\in \Z_R^d$ and $n\geq 1$, we apply \eqref{4e7.10a} and Lemma \ref{4l1.3} to get
\begin{align}\label{4e7.11}
\sum_{y\in \Z^d_R} p_n(y-x)  \phi_a(y)\leq CR \sum_{y\in \Z^d_R} p_n(y-x) \sum_{k=1}^\infty \frac{1}{k^{3/2}} e^{-\frac{|y-a|^2}{64k}} \leq  CR \cdot c_{\ref{4l1.3}} n^{-1/2}.
\end{align}
It follows that
\begin{align}\label{4e7.21}
G(\phi_a, {T_\theta^R})=& 3\|\phi_a\|_\infty+\sum_{k=1}^{T_\theta^R} \sup_{y \in \Z_R^d} \sum_{z\in \Z^d_R} p_k(y-z) \phi_a(z)\nn\\
\leq &CR+\sum_{k=1}^{T_\theta^R} CR\cdot c_{\ref{4l1.3}}  k^{-1/2}\leq  CR\sqrt{T_\theta^R}\leq c(T)\frac{R^2}{\theta^{1/2}},
\end{align}
where the first inequality is by \eqref{4e7.10a}, \eqref{4e7.11}, and the last inequality is by $T_\theta^R\leq TR^2/\theta$. 
\begin{proof}[Proof of Proposition \ref{4p3.1}(i)]
Let $\lambda=\theta^{3/2}R^{-4}/\log \theta$ and $n=T_\theta^R\leq \frac{TR^2}{\theta}$.  By \eqref{4e7.21} we have
\begin{align}\label{4e5.43}
 2\lambda T_\theta^R e^{\frac{T_\theta^R\theta}{R^{2}}} G(\phi_a,{T_\theta^R})\leq&2\frac{\theta^{3/2}R^{-4}}{\log \theta} \frac{TR^2}{\theta}e^T \cdot c(T)\frac{R^2}{\theta^{1/2}}\leq C(T) \frac{1}{\log \theta}.
\end{align}
 If we pick $\theta>0$ large enough so that $C(T)/{\log \theta}\leq 1/2$, then we may apply Proposition \ref{4p1.4} to get (recall $|Z_0|\leq 2R^2\log \theta/\theta$)
\begin{align}\label{4e5.30}
\E^{Z_0}\Big(\exp\Big(\lambda \sum_{k=0}^{{T_\theta^R}} Z_k(\phi_a)\Big)\Big)\leq& \exp\Big(\lambda |Z_0| e^{\frac{T_\theta^R\theta}{R^{2}}} G(\phi_a,{T_\theta^R})  (1-2\lambda {T_\theta^R} e^{\frac{T_\theta^R\theta}{R^{2}}} G(\phi_a,T_\theta^R))^{-1}\Big)\nn\\
\leq &\exp\Big(\lambda \frac{2R^2 \log \theta}{\theta} e^{T}  c(T) \frac{R^2}{{\theta}^{1/2}} (1-C(T)/{\log \theta})^{-1}\Big)\nn\\
\leq &\exp\Big( c(T) (1-C(T)/{\log \theta})^{-1}\Big)\leq e^{2c(T)},
\end{align}
where in the second inequality we have used \eqref{4e7.21}, \eqref{4e5.43} and the last inequality is by $C(T)/{\log \theta}\leq 1/2$.
\end{proof}

Turning to the difference moments, we fix $\eta=1/8$ throughout the rest of this section. 
For any $a\neq b \in \Z^d$ and $y \in \Z^d_R$, we have $|(y-a)-(y-b)|\geq 1$. So we may apply Proposition \ref{4p1.1} (ii) to get
\begin{align}\label{4e7.12}
|\phi_a(y)-\phi_b(y)|\leq&  RV(R)\sum_{k=1}^\infty   \frac{C_{\ref{4p1.1}}}{k^{3/2} R^3} \Big(\frac{|a-b|}{\sqrt{k}}\Big)^\eta (e^{-\frac{|y-a|^2}{64k}}+e^{-\frac{|y-b|^2}{64k}})\nn\\
\leq& CR |a-b|^{\eta}   \sum_{k=1}^\infty  \frac{1}{k^{(3+\eta)/2}} (e^{-\frac{|y-a|^2}{64k}}+e^{-\frac{|y-b|^2}{64k}}).
\end{align}

\no For any $x\in \Z^d_R$, by \eqref{4e7.12} and  Lemma \ref{4l1.3} we have for any $n\geq 1$,
\begin{align}\label{4eb1.8}
&\sum_{y\in \Z^d_R} p_n(y-x) |\phi_a(y)-\phi_b(y)|\nn\\
\leq &CR  |a-b|^{\eta}  \sum_{y\in \Z^d_R}  p_n(y-x) \sum_{k=1}^\infty \frac{1}{k^{(3+\eta)/2}} (e^{-\frac{ |a-y|^2}{64k}}+e^{-\frac{ |b-y|^2}{64k}})\nn\\
\leq &CR |a-b|^{\eta} \cdot 2 c_{\ref{4l1.3}}   n^{-(1+\eta)/2}.
\end{align}
Hence we may apply \eqref{4e7.12} and \eqref{4eb1.8} to get
\begin{align}\label{4e5.45}
G(|\phi_a-\phi_b|,{T_\theta^R})=&3\|\phi_a-\phi_b\|_\infty+\sum_{k=1}^{{T_\theta^R}} \sup_{y \in \Z_R^d} \sum_{z\in \Z^d_R} p_k(y-z) |\phi_a(y)-\phi_b(y)|\nn\\
\leq& C(\eta) R |a-b|^{\eta}+CR |a-b|^{\eta} \sum_{k=1}^{{T_\theta^R}}   2 c_{\ref{4l1.3}}  k^{-(1+\eta)/2} \nn\\
\leq&  c(\eta) R |a-b|^{\eta} (T_\theta^R)^{(1-\eta)/2}\leq c(T)  |a-b|^{\eta} \frac{R^{2-\eta}}{\theta^{(1-\eta)/2}}.
\end{align}
\begin{proof}[Proof of Proposition \ref{4p3.1}(ii)]
Let $\lambda={\theta^{(3-\eta)/2}} {R^{\eta-4}}/(|a-b|^{\eta}\log \theta)$ and $n=T_\theta^R\leq \frac{TR^2}{\theta}$.  Note by \eqref{4e5.45} we have
\begin{align}\label{4e5.44}
 2\lambda T_\theta^R e^{\frac{T_\theta^R\theta}{R^{d-1}}} G(|\phi_a-\phi_b|,{T_\theta^R})\leq&2\frac{{\theta^{(3-\eta)/2}} {R^{\eta-4}}}{|a-b|^{\eta}\log \theta} \frac{TR^2}{\theta}e^T \cdot c(T)  |a-b|^{\eta} \frac{R^{2-\eta}}{\theta^{(1-\eta)/2}}\nn\\
 \leq &C(T)/\log \theta.
\end{align}
 If we pick $\theta>0$ large enough so that $C(T)/{\log \theta}\leq 1/2$, then we may apply Proposition \ref{4p1.4} to get (recall $|Z_0|\leq 2R^2\log \theta/\theta$)
\begin{align}\label{4e5.31}
&\E^{{Z}_0}\Big(\exp\Big({\lambda\Big|\sum_{k=0}^{{T_\theta^R}} {Z}_{k}(\phi_a)-\sum_{k=0}^{{T_\theta^R}} {Z}_{k}(\phi_b)\Big|}\Big)\Big)\leq \E^{{Z}_0}\Big(\exp\Big({\lambda \sum_{k=0}^{{T_\theta^R}} {Z}_{k}(|\phi_a-\phi_b|)}\Big)\Big)\nn\\
\leq&\exp\Big(\lambda  |{Z}_0| e^{\frac{T_\theta^R\theta}{R^{2}}} G(|\phi_a-\phi_b|,{T_\theta^R}) (1- 2\lambda T_\theta^R e^{\frac{T_\theta^R\theta}{R^{2}}} G(|\phi_a-\phi_b|,{T_\theta^R}))^{-1}\Big)\nn\\
\leq&\exp\Big(\lambda \frac{2R^2\log \theta}{\theta} e^T \cdot c(T)  |a-b|^{\eta} \frac{R^{2-\eta}}{\theta^{(1-\eta)/2}}  (1-C(T)/{\log \theta})^{-1}\Big)\nn\\
\leq &\exp\Big( c(T)(1-C(T)/{\log \theta})^{-1}\Big)\leq e^{2c(T)},
\end{align}
where in the third inequality we have used \eqref{4e5.45}, \eqref{4e5.44}, and the last inequality is by $C(T)/{\log \theta}\leq 1/2$. The second last inequality uses $\lambda={\theta^{(3-\eta)/2}} {R^{\eta-4}}/(|a-b|^{\eta}\log \theta)$. So the proof is complete.
\end{proof}

\subsection{Exponential moments of the martingale term}

Now we will turn to complicated martingale term $M_{{T_\theta^R}+1}(\phi_a)$ and give the proof of Proposition \ref{4p3.2}.

\subsubsection{Proof of Proposition \ref{4p3.2}(ii)}
We first prove Proposition \ref{4p3.2}(ii) and deal with 
\begin{align*}
\Big||M_{{T_\theta^R}+1}(\phi_a)|-|M_{{T_\theta^R}+1}(\phi_b)|\Big|\leq |M_{{T_\theta^R}+1}(\phi_a)-M_{{T_\theta^R}+1}(\phi_b)|=|M_{{T_\theta^R}+1}(\phi_a-\phi_b)|.
\end{align*}
 Use \eqref{4e7.12} and $R\geq 4\theta$ to get 
\begin{align*}
\theta^{\eta/2} R^{\eta-2}|a-b|^{-\eta}\|\phi_a-\phi_b\|_\infty \leq& \theta^{\eta/2} R^{\eta-2}|a-b|^{-\eta} \cdot CR|a-b|^\eta\\
\leq&C \theta^{3\eta/2-1}\leq 1,
\end{align*}
if we pick $\theta\geq 100$ to be large. Then we may apply Proposition \ref{4p5.1} with $\phi=\phi_a-\phi_b$ and $\lambda=\theta^{\eta/2} R^{\eta-2}|a-b|^{-\eta}$  to get
\begin{align}\label{4e8.46}
&\E^{{Z}_0}\Big(\exp\Big(\theta^{\eta/2} R^{\eta-2}|a-b|^{-\eta} \Big||M_{{T_\theta^R}+1}(\phi_a)|-|M_{{T_\theta^R}+1}(\phi_b)|\Big|\Big)\Big)\nn\\
\leq& \E^{{Z}_0}\Big(\exp\Big(\theta^{\eta/2}R^{\eta-2} |a-b|^{-\eta} |M_{{T_\theta^R}+1}(\phi_a-\phi_b)|\Big)\Big)\nn\\
\leq& 2\Big(\E^{{Z}_0}\Big(\exp\Big(16 \theta^{\eta} R^{2\eta-4}|a-b|^{-2\eta} \langle M(\phi_a-\phi_b)\rangle_{{T_\theta^R}+1}\Big)\Big)\Big)^{1/2}.
\end{align}
It suffices to bound
\begin{align}\label{4e7.61}
\E^{{Z}_0}\Big(\exp\Big(16 \theta^{\eta} R^{2\eta-4}|a-b|^{-2\eta} \langle M(\phi_a-\phi_b)\rangle_{{T_\theta^R}+1}\Big)\Big).
\end{align}
By \eqref{4e7.42}, the above quadratic variation is bounded by
\begin{align}\label{4e8.23}
\langle M(\phi_a-\phi_b)\rangle_{{T_\theta^R}+1}\leq & 
2\sum_{n=0}^{{T_\theta^R}}\sum_{x\in \Z^d_R} Z_n(x) \cdot \frac{1}{V(R)} \sum_{i=1}^{V(R)} \Big(\phi_a({x+e_i})-\phi_b(x+e_i)\Big)^2.
\end{align}
Apply \eqref{4e7.12} to see that for any $a\neq b \in \Z^d$ and $y\in \Z^d_R$, we have
\begin{align}\label{4e7.24}
&|\phi_a(y)-\phi_b(y)|^2\leq CR^2 |a-b|^{2\eta} \Big(\Big(\sum_{k=1}^\infty  \frac{1}{k^{(3+\eta)/2}} e^{-\frac{|y-a|^2}{64k}}\Big)^2+\Big(\sum_{k=1}^\infty  \frac{1}{k^{(3+\eta)/2}} e^{-\frac{|y-b|^2}{64k}}\Big)^2\Big)\nn\\
&\leq CR^2 |a-b|^{2\eta} C_{\ref{4l4.1.1}}(\frac{1+\eta}{2})\Big(\sum_{k=1}^\infty  \frac{1}{k^{2+\eta}} e^{-\frac{|y-a|^2}{64k}}+\sum_{k=1}^\infty  \frac{1}{k^{2+\eta}} e^{-\frac{|y-b|^2}{64k}}\Big),
\end{align}
where the last inequality is by Lemma \ref{4l4.1.1} with $\alpha=\frac{1+\eta}{2}$. Define for any $a\in \Z^d_R$ that
\begin{align}\label{4e8.22}
f_a(y):=\sum_{k=1}^\infty  \frac{1}{k^{2+\eta}} e^{-\frac{|y-a|^2}{64k}}\text{ and write } \overline{f_a}(y)=\frac{1}{V(R)}\sum_{i=1}^{V(R)} f_a(y+e_i).
\end{align}
Therefore we apply \eqref{4e7.24} to see that \eqref{4e8.23} becomes
\begin{align*}
\langle M(\phi_a-\phi_b)\rangle_{{T_\theta^R}+1}
\leq &2\sum_{n=0}^{{T_\theta^R}} \sum_{x\in \Z^d_R} Z_n(x) \cdot C(\eta) R^2 |a-b|^{2\eta} \frac{1}{V(R)} \sum_{i=1}^{V(R)} \Big(f_a({x+e_i})+f_b(x+e_i)\Big)\\
 =& 2\sum_{n=0}^{{T_\theta^R}} \sum_{x\in \Z^d_R} Z_n(x)\cdot C(\eta) R^2 |a-b|^{2\eta} ( \overline{f_a}(x) +\overline{f_b}(x))\\
\leq &CR^2 |a-b|^{2\eta} \sum_{n=0}^{{T_\theta^R}}  (Z_n(\overline{f_a})+Z_n(\overline{f_b})).
\end{align*}
Returning to \eqref{4e7.61}, we get
\begin{align}\label{4e7.25}
&\E^{{Z}_0}\Big(\exp\Big(16 \theta^{\eta} R^{2\eta-4}|a-b|^{-2\eta} \langle M(\phi_a-\phi_b)\rangle_{{T_\theta^R}+1}\Big)\Big)\\
\leq& \E^{{Z}_0}\Big(\exp\Big(16C \theta^{\eta} R^{2\eta-2} \sum_{n=0}^{{T_\theta^R}}  (Z_n(\overline{f_a})+Z_n(\overline{f_b})) \Big)\Big)\nn\\
\leq& \Big(\E^{{Z}_0}\Big(\exp\Big(32C \theta^{\eta} R^{2\eta-2} \sum_{n=0}^{{T_\theta^R}}  Z_n(\overline{f_a}) \Big)\Big)\Big)^{1/2} \Big(\E^{{Z}_0}\Big(\exp\Big(32C \theta^{\eta} R^{2\eta-2} \sum_{n=0}^{{T_\theta^R}}  Z_n(\overline{f_b}) \Big)\Big)\Big)^{1/2},\nn
\end{align}
where the last inequality is by the Cauchy-Schwartz inequality. 
Hence it suffices to bound
\begin{align}\label{4e8.75}
\E^{{Z}_0}\Big(\exp\Big(32C \theta^{\eta} R^{2\eta-2} \sum_{n=0}^{{T_\theta^R}}  Z_n(\overline{f_a}) \Big)\Big)\ \text{ for any } a\in \Z^3.
\end{align}
We state in the following proposition a stronger result (with a higher exponent on $\theta$).

\begin{proposition}\label{4p5.5}
Let $\eta=1/8$. For any $K>0$, $m>0$ and $T\geq 100$, there exist constants $C_{\ref{4p5.5}}>0$ and $\theta_{\ref{4p5.5}}\geq 100$ depending only on $T,m,K$ such that for all $\theta \geq \theta_{\ref{4p5.5}}$,  there is some $C_{\ref{4p5.5}}(T,\theta, m, K)\geq 4\theta$ such that for any $R\geq  C_{\ref{4p5.5}}$ and any $Z_0$ satisfying \eqref{4eb1.7}, we have
\begin{align}\label{4e8.04}
\E^{{Z}_0}\Big(\exp\Big(K \theta^{2\eta} R^{2\eta-2} \sum_{n=0}^{{T_\theta^R}}  Z_n(\overline{f_a}) \Big)\Big)\leq C_{\ref{4p5.5}}(T, m, K), \ \forall a\in \Z^3_R.
\end{align}
\end{proposition}

\begin{proof}[Proof of Proposition \ref{4p3.2}(ii) assuming Proposition \ref{4p5.5}] Let $\eta=1/8$, $m>0$ and $T\geq 100$. Let $K=32C$ and $\theta, R$ be as in Proposition \ref{4p5.5} so that \eqref{4e8.04} holds. Hence it follows that
\begin{align}\label{4ec1.4}
\E^{{Z}_0}\Big(\exp\Big(32C \theta^{\eta} R^{2\eta-2} \sum_{n=0}^{{T_\theta^R}}  Z_n(\overline{f_a}) \Big)\Big) \leq C_{\ref{4p5.5}}(T, m, 32C),\quad \forall a \in \Z^3.
\end{align}
 Combining \eqref{4e8.46}, \eqref{4e7.25} and \eqref{4ec1.4}, we may conclude
\begin{align}\label{4e8.74}
&\E^{{Z}_0}\Big(\exp\Big(\frac{\theta^{\eta/2} R^{\eta-2}}{|a-b|^{\eta}} \Big||M_{{T_\theta^R}+1}(\phi_a)|-|M_{{T_\theta^R}+1}(\phi_b)|\Big|\Big)\Big)\leq2  C_{\ref{4p5.5}}(T, m, 32C) ^{1/2},
\end{align}
thus completing the proof of Proposition \ref{4p3.2}(ii).
\end{proof}
The proof of Proposition \ref{4p5.5} is rather complicated and so we postpone its proof till the end of this section. 
The reason for considering a different exponent on $\theta$ is because the same term will appear in the proof of Proposition \ref{4p3.2}(i), which we now give.

\subsubsection{Proof of Proposition \ref{4p3.2} (i)}
We move next to the exponential moment of $M_{{T_\theta^R}+1}(\phi_a)$.  By \eqref{4e7.10},  for any $R\geq 4\theta$ with $\theta\geq 100$ large, we have 
\[
\theta^{\eta} R^{-2}\| \phi_a\|_\infty\leq \theta^{\eta} R^{-2} \cdot CR\leq C\theta^{\eta-1} \leq 1.
\]
 So we may apply Proposition \ref{4p5.1} with $\phi=\phi_a$ and $\lambda=\theta^{\eta} R^{-2} $ to get
\begin{align}\label{4e8.09}
&\E^{{Z}_0}\Big(\exp\Big(\theta^{\eta} R^{-2} |M_{{T_\theta^R}+1}(\phi_a)|\Big)\Big)\leq 2\Big(\E^{{Z}_0}\Big(\exp\Big(16 \theta^{2\eta} R^{-4} \langle M(\phi_a)\rangle_{{T_\theta^R}+1}\Big)\Big)\Big)^{1/2}.
\end{align}
It suffices to bound
\begin{align}\label{4e8.43}
\E^{{Z}_0}\Big(\exp\Big(16 \theta^{2\eta} R^{-4} \langle M(\phi_a)\rangle_{{T_\theta^R}+1}\Big)\Big).
\end{align}
Recall from \eqref{4e7.42} that
\begin{align}\label{4e8.31}
\langle M(\phi_a)\rangle_{{T_\theta^R}+1}\leq&2\sum_{n=0}^{T_\theta^R}\sum_{x\in \Z^d_R} Z_n(x) \cdot  \frac{1}{V(R)}\sum_{i=1}^{V(R)} ({\phi_a}({x+e_i}))^2 \nn\\
\leq&2C^2\sum_{n=0}^{T_\theta^R} \sum_{x\in \Z^d_R} Z_n(x) \cdot  \frac{1}{V(R)} \sum_{i=1}^{V(R)} ({g_{a,3}}({x+e_i}))^2.
\end{align}
where the last inequality is by \eqref{4e7.10a}. Recall from \eqref{4e8.22}  that
\[f_a(y)=\sum_{k=1}^\infty  \frac{1}{k^{2+\eta}} e^{-\frac{|y-a|^2}{64k}},\]
where $\eta=1/8$. We first establish the following bounds for $g_{a,3}^2$.

\begin{lemma}\label{4l1.93}
There is some absolute constant $C_{\ref{4l1.93}}>0$ such that for any $R\geq 400$ and $a\in \Z^d_R$,
\begin{align}\label{4eb1.21}
(g_{a,3}(y))^2\leq C_{\ref{4l1.93}}R^{2+2\eta} f_a(y)+C_{\ref{4l1.93}}, \quad \forall  y\in \Z_R^d.
\end{align}
\end{lemma}
\begin{proof}
Recall that
\begin{align}\label{4eb1.24}
g_{a,3}(y)=R\sum_{k=1}^\infty \frac{1}{k^{3/2}} e^{-|y-a|^2/(32k)}.
\end{align}
We first consider $|y-a|\geq R\geq 1$. Use Lemma \ref{4l4.1} to see that
\begin{align*}
(g_{a,3}(y))^2 1_{\{|y-a|\geq R\}} \leq R^2 C_{\ref{4l4.1}}(\frac{1}{2})^2 \frac{32}{|y-a|^{2}} 1_{\{|y-a|\geq R\}} \leq C.
\end{align*}
Next, if $1< |y-a|< R$, we use Lemma \ref{4l4.1} again to get
\begin{align}\label{ec1.5}
(g_{a,3}(y))^2 1_{\{1<|y-a|<R\}} \leq&R^2 C_{\ref{4l4.1}}(\frac{1}{2})^2 \frac{32}{|y-a|^{2}} 1_{\{1<|y-a|<R\}}\nn\\
 \leq& C R^{2+2\eta} \frac{64^{1+\eta} }{|y-a|^{2+2\eta}} 1_{\{1<|y-a|<R\}},
 \end{align}
where in the last inequality we have used $|y-a|<R$. On the other hand, we apply  Lemma \ref{4l4.1} to $f_a$ to get
\begin{align}\label{ec1.6}
f_a(y)1_{\{1<|y-a|<R\}}\geq c_{\ref{4l4.1}}(1+\eta) \frac{64^{1+\eta}}{|y-a|^{2+2\eta}}1_{\{1<|y-a|<R\}}.
\end{align}
Combine \eqref{ec1.5} and \eqref{ec1.6} to arrive at
\begin{align*}
(g_{a,3}(y))^2 1_{\{1<|y-a|<R\}} \leq C R^{2+2\eta}c_{\ref{4l4.1}}(1+\eta)^{-1}f_a(y)1_{\{1<|y-a|<R\}}\leq CR^{2+2\eta}f_a(y).
 \end{align*}
Finally for $|y-a|\leq 1$, we have
\begin{align*}
(g_{a,3}(y))^2 1_{\{|y-a|\leq 1\}} \leq R^2 \Big(\sum_{k=1}^\infty \frac{1}{k^{3/2}}\Big)^2 \leq c_1 R^2,
\end{align*}
for some constant $c_1>0$.  Next we have
\begin{align*}
f_{a}(y) 1_{\{|y-a|\leq 1\}} \geq e^{-\frac{1}{64}}\sum_{k=1}^\infty  \frac{1}{k^{2+\eta}} \geq c_2
\end{align*}
for some constant $c_2>0$. Therefore it follows that
 \begin{align*}
(g_{a,3}(y))^2 1_{\{|y-a|\leq 1\}} \leq c_1 R^2\leq \frac{c_1}{c_2} R^{2+2\eta} f_{a}(y) 1_{\{|y-a|\leq 1\}}\leq CR^{2+2\eta}f_a(y).
\end{align*}
By adjusting constants, we complete the proof.
\end{proof}
Apply the above lemma to see that \eqref{4e8.31} becomes
\begin{align}\label{4e8.44}
\langle M(\phi_a)\rangle_{{T_\theta^R}+1}&\leq 2C^2\sum_{n=0}^{T_\theta^R} \sum_{x\in \Z^d_R} Z_n(x) \cdot  \frac{1}{V(R)}\sum_{i=1}^{V(R)} C_{\ref{4l1.93}}R^{2+2\eta} f_a(x+e_i)\\
&\quad   +2C^2\sum_{n=0}^{T_\theta^R} \sum_{x\in \Z^d_R} Z_n(x) \cdot C_{\ref{4l1.93}}\leq CR^{2+2\eta}\sum_{n=0}^{T_\theta^R}  Z_n(\overline{f_a})+C\sum_{n=0}^{T_\theta^R} Z_n(1).\nn
\end{align}
Returning to \eqref{4e8.43}, we use \eqref{4e8.44} to arrive at
\begin{align}\label{4e8.82}
&\E^{{Z}_0}\Big(\exp\Big(16 \theta^{2\eta} R^{-4} \langle M(\phi_a)\rangle_{{T_\theta^R}+1}\Big)\Big)\\
\leq& \E^{{Z}_0}\Big(\exp\Big(16 \theta^{2\eta} R^{-4} \Big(CR^{2+2\eta}\sum_{n=0}^{T_\theta^R}  Z_n(\overline{f_a})+C\sum_{n=0}^{T_\theta^R} Z_n(1)\Big)  \Big)\Big)\nn\\
\leq& \Big(\E^{{Z}_0}\Big(\exp\Big(32C\theta^{2\eta} R^{2\eta-2}\sum_{n=0}^{T_\theta^R}  Z_n(\overline{f_a})\Big) \Big)\Big)^{1/2} \Big(\E^{{Z}_0}\Big(\exp\Big(32C\theta^{2\eta}  R^{-4} \sum_{n=0}^{T_\theta^R} Z_n(1)\Big)  \Big)\Big)^{1/2},\nn
\end{align}
where the last inequality is by the Cauchy-Schwartz inequality. Now we are ready to finish the proof of Proposition \ref{4p3.2}(i).
\begin{proof}[Proof of Proposition \ref{4p3.2}(i)]
Let $\eta=1/8$, $m>0$ and $T\geq 100$. Let $K=32C$ and $\theta, R, Z_0$ be as in Proposition \ref{4p5.5} so that \eqref{4e8.04} holds. Hence we have for any $a \in \Z^d$, 
\begin{align}\label{4e8.81}
\E^{{Z}_0}\Big(\exp\Big(32C\theta^{2\eta} R^{2\eta-2}\sum_{n=0}^{T_\theta^R}  Z_n(\overline{f_a})\Big) \Big)\leq C_{\ref{4p5.5}}(T, m, 32C),
\end{align}
thus giving bounds for the first term on the right-hand side of \eqref{4e8.82}.
For the second term, we note that 
\begin{align}\label{4e5.65}
G(1,T_\theta^R)=3+T_\theta^R \leq 2 \frac{TR^2}{\theta},
\end{align}
where the last is by \eqref{4e10.06}. Let $\lambda=32C\theta^{2\eta}  R^{-4}$ and $n=T_\theta^R\leq \frac{TR^2}{\theta}$. Then \eqref{4e5.65} implies
\begin{align}\label{4e5.66}
 2\lambda T_\theta^R e^{\frac{T_\theta^R\theta}{ R^{2}}} G(1,T_\theta^R)\leq 64C\theta^{2\eta}  R^{-4} \frac{TR^2}{\theta}e^{T} \cdot 2 \frac{TR^2}{\theta}\leq C(T) \frac{1}{\theta^{2-2\eta}}.
\end{align}
 If we pick $\theta>0$ large enough so that $C(T)/{\theta^{2-2\eta}}\leq 1/2$, then we may apply Proposition \ref{4p1.4} to get (recall $|Z_0|\leq 2R^2\log\theta/{\theta}$)
\begin{align}\label{4e8.80}
&\E^{{Z}_0}\Big(\exp\Big(32C\theta^{2\eta} R^{-4} \sum_{n=0}^{T_\theta^R} Z_n(1)\Big)  \Big)=\E^{{Z}_0}\Big(\exp\Big(\lambda\sum_{n=0}^{T_\theta^R} Z_n(1)\Big)  \Big)\nn\\
\leq& \exp\Big(\lambda |{Z}_0| e^{\frac{T_\theta^R\theta}{ R^{2}}} G(1,T_\theta^R) (1-2\lambda  T_\theta^R e^{\frac{T_\theta^R\theta}{ R^{2}}} G(1,T_\theta^R))^{-1}\Big)\nn\\
\leq &\exp\Big(\lambda \frac{2R^2 \log \theta}{\theta} e^{T} \cdot 2 \frac{TR^2}{\theta} (1-C(T)/{\theta^{2-2\eta}})^{-1}\Big)\nn\\
\leq &\exp\Big(c(T)\frac{\log \theta}{\theta^{2-2\eta}} \cdot 2\Big)\leq e^{c(T)},
\end{align}
where in the second inequality we have used \eqref{4e5.65} and \eqref{4e5.66}. The second last inequality is by $C(T)/{\theta^{2-2\eta}}\leq 1/2$ and the last inequality uses $\log \theta \leq  \theta^{2-2\eta}$ for $\theta\geq 100$. Now combine \eqref{4e8.81} and \eqref{4e8.80} to see that   \eqref{4e8.82} becomes
\begin{align}\label{4e8.83}
&\E^{{Z}_0}\Big(\exp\Big(16 \theta^{2\eta} R^{-4} \langle M(\phi_a)\rangle_{{T_\theta^R}+1}\Big)\Big)\leq C_{\ref{4p5.5}}(T, m, 32C)^{1/2} e^{c(T)/2}=C(T,m).
\end{align}
Hence we conclude from \eqref{4e8.09}, \eqref{4e8.83} that
\begin{align}\label{4e8.84}
\E^{{Z}_0}\Big(\exp\Big(\theta^{\eta} R^{-2} |M_{{T_\theta^R}+1}(\phi_a)|\Big)\Big)\leq 2{C(T,m)}^{1/2},
\end{align}
thus finishing the proof of Proposition \ref{4p3.2}(i).
\end{proof}

\subsubsection{Proof of Proposition \ref{4p5.5}}
Finally we will prove Proposition \ref{4p5.5}, thus completing the proof of Proposition \ref{4p3.2}. Nevertheless, applying Proposition \ref{4p1.4} won't direct us to the conclusion immediately. Recall $f_a(y)=\sum_{k=1}^\infty  \frac{1}{k^{2+\eta}} e^{-\frac{|y-a|^2}{64k}}$ where $\eta=1/8$. By Lemma \ref{4l4.1}, for any $|y-a|>1$ we have $f_a(y)$ is bounded above and below by $|y-a|^{-2-2\eta}$ up to some constants, which is too singular a function to integrate in $d=3$. To solve this issue, by recalling the generator $\cL$ from \eqref{4e1.7}, we will find some $\psi_a$ such that $\cL \psi_a(x)=-f_a(y)$ and then use the martingale problem \eqref{4e1.1} to get the desired bounds. By Green's function representation (see, e.g., (4.24) and (4.25) of \cite{LL10}), for any $a\in \Z_R^3$, we define
\begin{align}\label{4e8.00}
\psi_a(x)&:=\sum_{y\in \Z_R^3} \sum_{n=0}^\infty p_n(x-y)  f_a(y), \quad \forall x\in \Z_R^3.
\end{align}
The following lemma justifies the absolute convergence of the above summation. This idea also originates from the fact that 
\[
\Delta^{-1} \frac{1}{|x|^{2+2\eta}}= -c(\eta) \frac{1}{|x|^{2\eta}},
\]
where $\Delta^{-1}$ is the inverse Laplacian operator on $\R^3$. Similar idea has been used in the proof of Lemma 2.2 in \cite{Hong18}. 
\begin{lemma}\label{4l5.6}
There is some absolute constant $c_{\ref{4l5.6}}>0$ such that for any $R\geq K_{\ref{4p1.1}}$ and any $x,a \in \Z_R^3$, 
\[
\psi_a(x)=\sum_{y\in \Z_R^3} \sum_{n=0}^\infty p_n(x-y)  f_a(y) \leq c_{\ref{4l5.6}} \sum_{k=1}^\infty  \frac{1}{k^{1+\eta}} e^{-\frac{|x-a|^2}{64k}}.
\]
\end{lemma}
\begin{proof}
First we use Proposition \ref{4p1.1}(i) to get
\begin{align}\label{4e8.88}
\psi_a(x)=&\sum_{y\in \Z_R^3} \sum_{n=0}^\infty p_n(x-y)  \sum_{k=1}^\infty  \frac{1}{k^{2+\eta}} e^{-\frac{|y-a|^2}{64k}}\nn\\
\leq&  \sum_{k=1}^\infty  \frac{1}{k^{2+\eta}} e^{-\frac{|x-a|^2}{64k}}+\sum_{y\in \Z^d_R} \sum_{n=1}^\infty \frac{c_{\ref{4p1.1}}}{n^{3/2}R^3} e^{-\frac{|x-y|^2}{32n}} \sum_{k=1}^\infty \frac{1}{k^{2+\eta}} e^{-\frac{|y-a|^2}{64k}}\nn \\
\leq &\sum_{k=1}^\infty  \frac{1}{k^{1+\eta}} e^{-\frac{|x-a|^2}{64k}}+ \sum_{k=1}^\infty \frac{1}{k^{2+\eta}}  \sum_{n=1}^\infty \frac{c_{\ref{4p1.1}}}{n^{3/2}R^3} \sum_{y\in \Z^d_R} e^{-\frac{|x-y|^2}{64n}} e^{-\frac{|y-a|^2}{64k}},
\end{align}
where in the last inequality we have used $k^{2+\eta}\geq k^{1+\eta}$, $32<64$ and  Fubini's theorem. It suffices to bound the second term above.
Use Lemma \ref{4l4.2}(i) with $s=1/(64n)$ and $t=1/(64k)$ to get 
\begin{align*}
\sum_{y\in \Z^d_R} e^{-\frac{|x-y|^2}{64n}} e^{-\frac{|y-a|^2}{64k}}
 \leq&2^3 e^{-\frac{|x-a|^2}{64(k+n)}}  \sum_{y\in \Z^d_R} e^{-\frac{|y|^2}{64n}} e^{-\frac{|y|^2}{64k}} =8 e^{-\frac{|x-a|^2}{64(k+n)}} \sum_{y\in \Z^d_R} e^{-\frac{|y|^2}{64\frac{nk}{n+k}}}  \\
\leq& 8 e^{-\frac{|x-a|^2}{64(k+n)}} c_{\ref{4l4.2}} R^3 (\frac{32nk}{n+k})^{3/2},
\end{align*}
where the last inequality is by Lemma \ref{4l4.2}(ii) applied with $u=\frac{32nk}{n+k}>1$. Hence it follows that 
\begin{align*}
 I:=&\sum_{k=1}^\infty \frac{1}{k^{2+\eta}}  \sum_{n=1}^\infty \frac{c_{\ref{4p1.1}}}{n^{3/2}R^3} \sum_{y\in \Z^d_R} e^{-\frac{|x-y|^2}{64n}} e^{-\frac{|y-a|^2}{64k}}\\
 \leq&  \sum_{k=1}^\infty \frac{1}{k^{2+\eta}} \sum_{n=1}^\infty \frac{c_{\ref{4p1.1}}}{n^{3/2}R^3} 8 e^{-\frac{|x-a|^2}{64(k+n)}} c_{\ref{4l4.2}} R^3 (\frac{32nk}{n+k})^{3/2}\\
\leq &C  \sum_{k=1}^\infty \frac{1}{k^{\eta+1/2}}  \sum_{n=1}^\infty e^{-\frac{|x-a|^2}{64(k+n)}}  (\frac{1}{n+k})^{3/2}=C  \sum_{k=1}^\infty \frac{1}{k^{\eta+1/2}}  \sum_{n=k+1}^\infty e^{-\frac{|x-a|^2}{64n}}  \frac{1}{n^{3/2}}.
\end{align*}
Apply Fubini's theorem to the  right-hand side term above to get
\begin{align*}
I\leq &C \sum_{n=2}^\infty e^{-\frac{|x-a|^2}{64n}}   \frac{1}{n^{3/2}} \sum_{k=1}^{n-1} \frac{1}{k^{\eta+1/2}} \leq C  \sum_{n=2}^\infty e^{-\frac{|x-a|^2}{64n}}  \frac{1}{n^{3/2}} \cdot C(\eta)n^{\frac{1}{2}-\eta} \\
\leq &C  \sum_{n=1}^\infty e^{-\frac{|x-a|^2}{64n}}  \frac{1}{n^{1+\eta}},
\end{align*}
and so the proof is complete by \eqref{4e8.88}.
\end{proof}
The above lemma gives the absolute convergence of $\psi_a$ and so we have (recall \eqref{4e1.7})
\begin{align}\label{4e8.25}
\cL \psi_a(x)=&\frac{1}{V(R)}\sum_{i=1}^{V(R)} (\psi_a(x+e_i)-\psi_a(x))\nn\\
=&\sum_{y\in \Z_R^d}  f_a(y) \sum_{n=0}^\infty \frac{1}{V(R)}\sum_{i=1}^{V(R)} (p_n(x+e_i-y)-p_n(x-y))\nn \\
=&\sum_{y\in \Z_R^d}  f_a(y) \sum_{n=0}^\infty  (p_{n+1}(x-y)-p_n(x-y))\nn \\
=&\sum_{y\in \Z_R^d}  f_a(y) (-p_0(x-y))=-f_a(x),
\end{align}
where the third equality uses \eqref{4e2.4}. 
By the linearity of $\cL$, if we define
\begin{align}\label{4e8.01}
\overline{\psi_a}(x)&=\frac{1}{V(R)}\sum_{i=1}^{V(R)} \psi_a(x+e_i),
\end{align}
then it follows that
\begin{align}\label{4e8.02}
\cL \overline{\psi_a}(x)&=-\frac{1}{V(R)}\sum_{i=1}^{V(R)} f_a(x+e_i)=-\overline{f_a}(x).
\end{align}
Replace $\phi$ in \eqref{4e1.1} with $\overline{\psi_a}$ and use \eqref{4e8.02} to get for any $N\geq 1$,
\begin{align*}
Z_{N}( \overline{\psi_a})=&Z_0( \overline{\psi_a})- (1+\frac{\theta}{R^{2}})\sum_{n=0}^{N-1} \sum_{|\alpha|=n} \overline{f_a}(Y^\alpha)+M_N( \overline{\psi_a})+\frac{\theta}{R^{2}}\sum_{n=0}^{N-1} Z_n( \overline{\psi_a})\\
=&Z_0( \overline{\psi_a})- (1+\frac{\theta}{R^{2}})\sum_{n=0}^{N-1} Z_n( \overline{f_a}) +M_N( \overline{\psi_a})+\frac{\theta}{R^{2}}\sum_{n=0}^{N-1} Z_n( \overline{\psi_a}).
\end{align*}
Let $N=T_\theta^R+1$ and rearrange terms to arrive at
\begin{align}\label{4e8.26}
(1+\frac{\theta}{R^{2}}) \sum_{n=0}^{T_\theta^R} Z_n( \overline{f_a}) =&Z_0( \overline{\psi_a})-Z_{T_\theta^R+1}( \overline{\psi_a})+M_{T_\theta^R+1}( \overline{\psi_a})+\frac{\theta}{R^{2}}\sum_{n=0}^{T_\theta^R} Z_n( \overline{\psi_a})\nn\\
\leq &Z_0( \overline{\psi_a})+M_{T_\theta^R+1}( \overline{\psi_a})+\frac{\theta}{R^{2}}\sum_{n=0}^{T_\theta^R} Z_n( \overline{\psi_a}).
\end{align}
Now we are ready to give the proof of Proposition \ref{4p5.5}.
\begin{proof}[Proof of Proposition \ref{4p5.5}]
 Let $K>0$, $m>0$ and $T\geq 100$. For any $\theta\geq 100$ and $R\geq 4\theta$,  we use \eqref{4e8.26} to get for any $a\in \Z^3_R$ and any $Z_0$ as in \eqref{4eb1.7},
\begin{align}\label{4e8.05}
&\E^{{Z}_0}\Big(\exp\Big(K \theta^{2\eta} R^{2\eta-2} \sum_{n=0}^{{T_\theta^R}}  Z_n(\overline{f_a}) \Big)\Big)\nn\\
\leq &\E^{{Z}_0}\Big(\exp\Big(K \theta^{2\eta} R^{2\eta-2} \Big(Z_0( \overline{\psi_a})+M_{T_\theta^R+1}( \overline{\psi_a})+\frac{\theta}{R^{2}}\sum_{n=0}^{T_\theta^R} Z_n( \overline{\psi_a}) \Big)\Big)\nn\\
\leq &\exp\Big(K \theta^{2\eta} R^{2\eta-2} Z_0( \overline{\psi_a})\Big)  \times \Big(\E^{{Z}_0}\Big(\exp\Big(2K \theta^{2\eta} R^{2\eta-2} M_{T_\theta^R+1}( \overline{\psi_a}) \Big)\Big)\Big)^{1/2}\nn\\
&\quad  \quad \times \Big(\E^{{Z}_0}\Big(\exp\Big(2K \theta^{2\eta} R^{2\eta-2}\frac{\theta}{R^{2}} \sum_{n=0}^{T_\theta^R} Z_n( \overline{\psi_a}) \Big)\Big)\Big)^{1/2},
\end{align}
where we have used the Cauchy-Schwartz inequality in the last inequality. 
It suffices to bound the three terms on the right-hand side of \eqref{4e8.05}, which we now give.\\

(i) First we consider $\sum_{n=0}^{T_\theta^R} Z_n( \overline{\psi_a})$. By Lemma \ref{4l5.6} and Lemma \ref{4l1.3}, for any $a,x\in \Z_R^d$, we have for any $n\geq 1$,
\begin{align*}
\sum_{y\in \Z^d_R} p_n(y-x) \psi_a(y)\leq c_{\ref{4l5.6}} \sum_{y\in \Z^d_R} p_n(y-x) \sum_{k=1}^\infty \frac{1}{k^{1+\eta}} e^{-\frac{|y-a|^2}{64k}} \leq c_{\ref{4l5.6}}  c_{\ref{4l1.3}}     n^{-\eta},
\end{align*}
and so it follows that
\begin{align}\label{4e5.47}
\sum_{y\in \Z^d_R} p_n(y-x) \overline{\psi_a}(y)\leq c_{\ref{4l5.6}} c_{\ref{4l1.3}}   n^{-\eta}.
\end{align}
By Lemma \ref{4l5.6}, we also have 
\begin{align}\label{4e5.60}
\|\overline{\psi_a}\|_\infty \leq C
\end{align}
 for some constant $C>0$. Apply \eqref{4e5.47} and \eqref{4e5.60} to get
\begin{align}\label{4e8.06}
G(\overline{\psi_a},{T_\theta^R})=&3\|\overline{\psi_a}\|_\infty+\sum_{k=1}^{{T_\theta^R}} \sup_{y \in \Z_R^d} \sum_{z\in \Z^d_R} p_k(y-z) \overline{\psi_a}(z)\nn\\
\leq &C+\sum_{k=1}^{{T_\theta^R}} c_{\ref{4l5.6}}c_{\ref{4l1.3}}     k^{-\eta}\leq C +C (T_\theta^R)^{1-\eta}\leq c(T) \frac{R^{2-2\eta}}{\theta^{1-\eta}}.
\end{align}
Let $\lambda=2K \theta^{1+2\eta} R^{2\eta-4}$ and $n=T_\theta^R\leq \frac{TR^2}{\theta}$. Then by \eqref{4e8.06} we have
\begin{align}\label{4e5.48}
 2\lambda T_\theta^R e^{\frac{T_\theta^R\theta}{ R^{2}}} G(\overline{\psi_a},T_\theta^R)\leq 4K \theta^{1+2\eta} R^{2\eta-4} \frac{TR^2}{\theta}e^{T} \cdot c(T) \frac{R^{2-2\eta}}{\theta^{1-\eta}}\leq C(T) K \frac{1}{\theta^{1-3\eta}}.
\end{align}
 If we pick $\theta>0$ large enough so that $C(T)K/{\theta^{1-3\eta}}\leq 1/2$, then we may apply Proposition \ref{4p1.4} to get (recall $|Z_0|\leq 2R^2\log\theta/{\theta}$)
\begin{align}\label{4e8.72}
&\E^{Z_0}\Big(\exp\Big(2K \theta^{1+2\eta} R^{2\eta-4} \sum_{k=0}^{{T_\theta^R}} Z_k(\overline{\psi_a})\Big)\Big)=\E^{Z_0}\Big(\exp\Big(\lambda\sum_{k=0}^{{T_\theta^R}} Z_k(\overline{\psi_a})\Big)\Big)\nn\\
\leq& \exp\Big(\lambda |Z_0| e^{\frac{T_\theta^R\theta}{ R^{2}}} G(\overline{\psi_a},{T_\theta^R})  (1-2\lambda {T_\theta^R} e^{\frac{T_\theta^R\theta}{ R^{2}}} G(\overline{\psi_a},T_\theta^R))^{-1}\Big)\nn\\
\leq &\exp\Big(\lambda \frac{2R^2\log \theta}{{\theta}} e^{T} \cdot c(T)\frac{R^{2-2\eta}}{\theta^{1-\eta}}(1-C(T)K/{\theta^{1-3\eta}})^{-1}\Big)\nn\\
\leq &\exp\Big(Kc(T) \frac{\log \theta}{\theta^{1-3\eta}}  \cdot 2\Big)\leq \exp(2Kc(T)),
\end{align}
where in the second inequality we have used \eqref{4e8.06}, \eqref{4e5.48}. The second last inequality is by $C(T)K/{\theta^{1-3\eta}}\leq 1/2$ and the last inequality uses $\log \theta \leq \theta^{1-3\eta}$ for $\theta\geq 100$. \\

(ii) Next we consider $M_{T_\theta^R+1}( \overline{\psi_a})$. Use $R\geq 4\theta$ and \eqref{4e5.60} to get 
\begin{align}
2K \theta^{2\eta} R^{2\eta-2}  \|\overline{\psi_a}\|_\infty \leq 2K\theta^{4\eta-2} 4^{2\eta-2} C\leq 1 
\end{align}
if we set $\theta\geq 100$ to be large. Then we may apply Proposition \ref{4p5.1} with $\phi= \overline{\psi_a}$ and $\lambda=2K \theta^{2\eta} R^{2\eta-2} $ to get
\begin{align}\label{4e8.50}
&\E^{{Z}_0}\Big(\exp\Big(2K \theta^{2\eta} R^{2\eta-2} M_{T_\theta^R+1}( \overline{\psi_a}) \Big)\Big)\nn\\
\leq& 2\Big(\E^{{Z}_0}\Big(\exp\Big(64K^2 \theta^{4\eta} R^{4\eta-4} \langle M( \overline{\psi_a})\rangle_{T_\theta^R+1} \Big)\Big)\Big)^{1/2}.
\end{align}
Use \eqref{4e7.42} to see that
\begin{align}\label{4e5.61}
 \langle M( \overline{\psi_a})\rangle_{T_\theta^R+1}\leq &2\sum_{n=0}^{T_\theta^R} \sum_{x\in \Z^d_R} Z_n(x) \cdot
 \frac{1}{V(R)} \sum_{i=1}^{V(R)} (\overline{\psi_a}({x+e_i}))^2  .
 \end{align}
By Jensen's inequality, we have for any $y\in \Z_R^d$,
\begin{align*}
(\overline{\psi_a}(y))^2&=\Big(\frac{1}{V(R)}\sum_{i=1}^{V(R)} \psi_a(y+e_i)\Big)^2\leq \frac{1}{V(R)}\sum_{i=1}^{V(R)} (\psi_a(y+e_i))^2,
\end{align*}
and so \eqref{4e5.61} becomes
\begin{align*}
 \langle M( \overline{\psi_a})\rangle_{T_\theta^R+1} \leq&2\sum_{n=0}^{T_\theta^R} \sum_{x\in \Z^d_R} Z_n(x) \cdot \frac{1}{V(R)}
\sum_{i=1}^{V(R)}  \frac{1}{V(R)}
\sum_{j=1}^{V(R)}(\psi_a({x+e_i+e_j}))^2 .
\end{align*}
For any $a \in \Z^d_R$, we define
\begin{align}\label{4e8.28}
h_a(x)= \frac{1}{V(R)}
\sum_{i=1}^{V(R)}  \frac{1}{V(R)}
\sum_{j=1}^{V(R)}(\psi_a({x+e_i+e_j}))^2, \quad \forall  x \in \Z^d_R,
\end{align}
so that
\begin{align*}
 \langle M( \overline{\psi_a})\rangle_{T_\theta^R+1} \leq&2\sum_{n=0}^{T_\theta^R} \sum_{x\in \Z^d_R} Z_n(x) \cdot h_a(x) = 2\sum_{n=0}^{T_\theta^R} Z_n( h_a).
\end{align*}
Returning to \eqref{4e8.50}, we have
\begin{align}\label{4e8.51}
&\E^{{Z}_0}\Big(\exp\Big(2K \theta^{2\eta} R^{2\eta-2} M_{T_\theta^R+1}( \overline{\psi_a}) \Big)\Big)\nn\\
\leq& 2\Big(\E^{{Z}_0}\Big(\exp\Big(128K^2 \theta^{4\eta} R^{4\eta-4} \sum_{n=0}^{T_\theta^R} Z_n( h_a) \Big)\Big)\Big)^{1/2}.
\end{align}
It suffices to bound
\begin{align}\label{4e8.52}
\E^{{Z}_0}\Big(\exp\Big(128 K^2 \theta^{4\eta} R^{4\eta-4} \sum_{n=0}^{T_\theta^R} Z_n( h_a) \Big)\Big).
\end{align}
Use Lemma \ref{4l5.6} and then Lemma \ref{4l4.1.1} to see that
\begin{align}\label{4e8.07}
(\psi_a(y))^2 \leq& c_{\ref{4l5.6}}^2 \Big(\sum_{k=1}^\infty  \frac{1}{k^{1+\eta}} e^{-\frac{|x-a|^2}{64k}}\Big)^2\nn\\
\leq& c_{\ref{4l5.6}}^2 C_{\ref{4l4.1.1}}(\eta)\sum_{k=1}^\infty  \frac{1}{k^{1+2\eta}} e^{-\frac{|y-a|^2}{64k}}, \quad \forall y\in \Z_R^d.
\end{align}
Now apply  \eqref{4e8.07} and Lemma \ref{4l1.3} to get for any $n\geq  1$ and $x\in \Z_R^d$,
\begin{align*}
\sum_{y\in \Z^d_R} p_n(y-x) (\psi_a(y))^2 \leq C(\eta) \sum_{y\in \Z^d_R} p_n(y-x)\sum_{k=1}^\infty  \frac{1}{k^{1+2\eta}} e^{-\frac{|y-a|^2}{64k}}\leq   C(\eta) c_{\ref{4l1.3}}  \cdot  n^{-2\eta}.
\end{align*}
We conclude from \eqref{4e8.28} and the above that
\begin{align}\label{4e5.63}
\sum_{y\in \Z^d_R} p_n(y-x) h_a(y) \leq   C(\eta) c_{\ref{4l1.3}}  \cdot n^{-2\eta}\leq C n^{-2\eta}.
\end{align}
Recall \eqref{4e8.28} again and use Lemma \ref{4l5.6} to get $\|h_a\|_\infty \leq \|\psi_a\|_\infty^2\leq c$ for some $c>0$. Now use \eqref{4e5.63} to arrive at
\begin{align}\label{4e8.08}
G(h_a,{T_\theta^R})=&3\|h_a\|_\infty+\sum_{n=1}^{{T_\theta^R}} \sup_{y \in \Z_R^d} \sum_{z\in \Z^d_R} p_n(y-z) h_a(z)\nn\\
\leq &c+\sum_{n=1}^{{T_\theta^R}}C    n^{-2\eta}\leq c +C (T_\theta^R)^{1-2\eta}\leq c(T) \frac{R^{2-4\eta}}{\theta^{1-2\eta}}.
\end{align}

\no  Let $\lambda=128 K^2 \theta^{4\eta} R^{4\eta-4}$ and $n=T_\theta^R\leq \frac{TR^2}{\theta}$. By \eqref{4e8.08} we have
\begin{align}\label{4e5.64}
 2\lambda T_\theta^R e^{\frac{T_\theta^R\theta}{ R^{2}}} G(h_a,T_\theta^R)\leq 256 K^2 \theta^{4\eta} R^{4\eta-4} \frac{TR^2}{\theta}e^{T} \cdot c(T) \frac{R^{2-4\eta}}{\theta^{1-2\eta}}\leq C(T) K^2 \frac{1}{\theta^{2-6\eta}}.
\end{align}
 If we pick $\theta>0$ large enough so that $C(T)K^2/{\theta^{2-6\eta}}\leq 1/2$, then we may apply Proposition \ref{4p1.4} to get (recall $|Z_0|\leq 2R^2\log\theta/{\theta}$)
\begin{align}\label{4e8.60}
&\E^{Z_0}\Big(\exp\Big(128 K^2 \theta^{2\eta} R^{4\eta-4} \sum_{k=0}^{{T_\theta^R}} Z_k(h_a)\Big)\Big)\nn\\
\leq &\exp\Big(\lambda |Z_0| e^{\frac{T_\theta^R\theta}{ R^{2}}} G(h_a,{T_\theta^R})  (1-2\lambda {T_\theta^R}e^{\frac{T_\theta^R\theta}{ R^{2}}} G(h_a,T_\theta^R))^{-1}\Big)\nn\\
\leq &\exp\Big(\lambda \frac{2R^2\log \theta}{{\theta}} e^{T}\cdot  c(T)\frac{R^{2-4\eta}}{\theta^{1-2\eta}}(1-C(T)K^2/{\theta^{2-6\eta}})^{-1}\Big)\nn\\
\leq&\exp\Big( c(T) K^2\frac{\log \theta}{\theta^{2-6\eta}} \cdot 2\Big)\leq e^{c(T) K^2},
\end{align}
where in the second inequality we have used \eqref{4e8.08} and \eqref{4e5.64}. The second last inequality is by $C(T)K^2/{\theta^{2-6\eta}}\leq 1/2$ and the last inequality uses $\log \theta \leq \theta^{2-6\eta}$ for $\theta\geq 100$. 

Now combine \eqref{4e8.51} and \eqref{4e8.60} to see that if $\theta>100$ is chosen large enough, then we have
\begin{align}\label{4e8.73}
&\E^{{Z}_0}\Big(\exp\Big(2K \theta^{2\eta} R^{2\eta-2} M_{T_\theta^R+1}( \overline{\psi_a}) \Big)\Big)\leq 2e^{c(T) K^2/2}.
\end{align}

(iii) It remains to bound $Z_0( \overline{\psi_a})$. We first give the following bound on $\psi_a$.  
\begin{lemma}\label{4l1.04}
There is some absolute constant $C_{\ref{4l1.04}}>0$ such that for any $R\geq K_{\ref{4p1.1}}$ and $a\in \Z^d_R$,
\begin{align}\label{4eb1.23}
R^{2\eta} \psi_a(x)\leq C_{\ref{4l1.04}}g_{a,3}(x)+C_{\ref{4l1.04}}, \quad \forall  x\in \Z_R^d.
\end{align}
\end{lemma}
\begin{proof}
For any $a, x\in \Z^d_R$, we use Lemma \ref{4l5.6} and Lemma \ref{4l4.1} to get 
\begin{align}\label{4e8.29}
R^{2\eta} \psi_a(x)\leq &R^{2\eta} c_{\ref{4l5.6}}  \sum_{k=1}^\infty \frac{1}{k^{1+\eta}} e^{-\frac{|x-a|^2}{64k}}\nn \\
\leq& R^{2\eta}C 1_{\{|x-a|\leq 1\}}+R^{2\eta} c_{\ref{4l5.6}}   C_{\ref{4l4.1}}(\eta) \frac{64^\eta}{|x-a|^{2\eta}} 1_{\{|x-a|>1\}}\nn\\
\leq &R^{2\eta}C  1_{\{|x-a|\leq 1\}}+C   \frac{R^{2\eta}}{|x-a|^{2\eta}} 1_{\{1<|x-a|<R\}}+ C 1_{\{|x-a|\geq R\}}\nn\\
\leq &CR   1_{\{|x-a|\leq 1\}}+C  \frac{R}{|x-a|} 1_{\{1<|x-a|<R\}}+ C 1_{\{|x-a|\geq R\}},
\end{align}
where in the last inequality we have used $R^{2\eta}\leq R$ for the first term (recall $R\geq 1$) and $R/|x-a|>1$ for the second term.
Recall from \eqref{4eb1.24} that
\begin{align}\label{4e8.30}
g_{a,3}(x)=R\sum_{k=1}^\infty \frac{1}{k^{3/2}} e^{-|x-a|^2/(32k)}.
\end{align}
If $|x-a|>R$, by \eqref{4e8.29} we get $R^{2\eta} \psi_a(x) 1_{\{|x-a|\geq R\}}\leq C$, thus giving \eqref{4eb1.23}. Turning to $|x-a|\leq 1$, we can find some constant $c_1>0$ such that
\begin{align}
g_{a,3}(x)1_{\{|x-a|\leq 1\}} \geq R e^{-1/32}\sum_{k=1}^\infty \frac{1}{k^{3/2}} 1_{\{|x-a|\leq 1\}} \geq c_1 R 1_{\{|x-a|\leq 1\}},
\end{align}
thus giving
\begin{align*}
R^{2\eta} \psi_a(x) 1_{\{|x-a|\leq 1\}}\leq CR 1_{\{|x-a|\leq 1\}}\leq \frac{C}{c_1} g_{a,3}(x)1_{\{|x-a|\leq 1\}}\leq Cg_{a,3}(x).
\end{align*}
Finally if $1<|x-a|<R$, we may apply Lemma \ref{4l4.1} to get
\begin{align*}
g_{a,3}(x)1_{\{1<|x-a|<R\}} \geq R c_{\ref{4l4.1}}(\frac{1}{2})  \frac{32^{1/2}}{|x-a|} 1_{\{1<|x-a|<R\}}\geq c_2 \frac{R}{|x-a|} 1_{\{1<|x-a|<R\}}
\end{align*}
for some constant $c_2>0$.
By \eqref{4e8.29}, we get 
\begin{align*}
R^{2\eta} \psi_a(x) 1_{\{1<|x-a|<R\}}\leq C\frac{R}{|x-a|} 1_{\{1<|x-a|<R\}}\leq \frac{C}{c_2} g_{a,3}(x)1_{\{1<|x-a|<R\}}\leq Cg_{a,3}(x).
\end{align*}
By adjusting constants, we complete the proof.
\end{proof}
Now we may apply the above lemma to get for any $a\in \Z^d_R$,
\begin{align*}
R^{2\eta-2} Z_0( \overline{\psi_a})=&R^{-2} \frac{1}{V(R)}\sum_{i=1}^{V(R)} \sum_{x\in \Z^d_R} Z_0(x) R^{2\eta}\psi_{a}(x+e_i)\\
\leq &C_{\ref{4l1.04}} R^{-2} Z_0(1)+ C_{\ref{4l1.04}} R^{-2} \frac{1}{V(R)}\sum_{i=1}^{V(R)} \sum_{x\in \Z^d_R} Z_0(x)  g_{a,3}(x+e_i)\\
\leq &C R^{-2} Z_0(1)+ C R^{-2} \frac{1}{V(R)}\sum_{i=1}^{V(R)} Z_0(g_{a-e_i,3}).
\end{align*}
Recall that $Z_0$ is as in \eqref{4eb1.7} and we use conditions (ii) and (iii) to see that the above becomes
\begin{align}\label{4ed8.71}
R^{2\eta-2} Z_0( \overline{\psi_a})\leq &C   \frac{\log \theta}{\theta}+ C  \frac{m}{\theta^{1/4}} \leq C  \frac{m+1}{\theta^{1/4}},
\end{align}
where the last inequality uses $\log \theta \leq \theta^{3/4}$ for $\theta>0$. This is the only place we use the regularity condition (iii) of $Z_0$ when calculating the exponential moments. Returning to \eqref{4e8.05}, we use the above to get
\begin{align}\label{4e8.71}
\exp\Big(K \theta^{2\eta}  R^{2\eta-2} Z_0( \overline{\psi_a})\Big)\leq \exp\Big(K \theta^{2\eta}  C  \frac{m+1}{\theta^{1/4}}\Big)= e^{CK(m+1)},
\end{align}
where the last equality is by $\eta=1/8$.

Finally we combine  \eqref{4e8.05}, \eqref{4e8.72}, \eqref{4e8.73} and \eqref{4e8.71} to see that
\begin{align}
&\E^{{Z}_0}\Big(\exp\Big(K \theta^{2\eta} R^{2\eta-2} \sum_{n=0}^{{T_\theta^R}}  Z_n(\overline{f_a}) \Big)\Big)\nn\\
\leq &e^{CK(m+1)}\cdot (2e^{c(T) K^2/2})^{1/2} (e^{2Kc(T)})^{1/2}=C(T,m,K),
\end{align}
thus completing the proof.
\end{proof}

\section{Regularity of branching random walk}\label{4s7}
 In this section we give the proof of Proposition \ref{4p3}. Recall from \eqref{4e12.01} that we assume
\begin{align}\label{4eb1.31}
Z_0(1)\leq 2R^{d-1} f_d(\theta)/\theta.
\end{align}
Recall $T_\theta^R=[{TR^{d-1}}/{\theta}]$ and
\begin{align}\label{4eb1.34}
200\leq \frac{1}{2}  \frac{TR^{d-1}}{\theta} \leq T_\theta^R\leq \frac{TR^{d-1}}{\theta}.
\end{align}
 Before proceeding to the proof of Proposition \ref{4p3}, we first give the proof of Corollary \ref{4c0.1} by assuming Proposition \ref{4p3}.

\begin{proof}[Proof of Corollary \ref{4c0.1} assuming Proposition \ref{4p3}] 
Fix $\eps_0\in (0,1)$, $T\geq 100+\eps_0^{-1}$. Let $\theta_{\ref{4c0.1}}=\theta_{\ref{4p3}}\geq 100$ and choose $\theta\geq \theta_{\ref{4c0.1}}$. Let $C_{\ref{4c0.1}}=C_{\ref{4p3}}\geq 4\theta$ and choose $R\geq C_{\ref{4p3}}$. Let $Z_0$ be as in \eqref{4eb1.31}. First we use \eqref{4ea4.5} to see that \[\E^{Z_0}(Z_{T_\theta^R}(1))=(1+\frac{\theta}{R^{d-1}})^{T_\theta^R}Z_0(1)\leq e^TZ_0(1),\] where the inequality uses \eqref{4eb1.34}. By Markov's inequality, it follows that
\begin{align}\label{4e11.42}
\P^{Z_0}\Big(Z_{T_\theta^R}(1)\geq e^{2T}Z_0(1)\Big)\leq e^{-T}\leq \eps_0,
\end{align}
where the last inequality is by $e^T\geq T\geq \eps_0^{-1}$. Hence with probability larger than $1-\eps_0$, we have $Z_{T_\theta^R}(1)\leq e^{2T}Z_0(1)$.

Next, recalling $R_\theta=\sqrt{R^{d-1}/\theta}$, we fix $u\in Q_{8 \sqrt{\log f_d(\theta)}R_\theta} (0)^c$. In view of Proposition \ref{4p3}, it suffices to get a uniform in $u$ bound for $\tilde{Z}_{T_\theta^R}(g_{u,d})$ where $\tilde{Z}_{T_\theta^R}(\cdot)=Z_{T_\theta^R}(\cdot \cap Q_{4R_\theta}(0))$, 
 Notice that $8\sqrt{\log f_d(\theta)}\geq 8$ for $\theta\geq 100$. Hence for any $x\in Q_{4R_\theta}(0)$, we have $|u-x|\geq 4R_\theta$.  In $d=2$, we use \eqref{4e10.32} to see that for any $x\in Q_{4R_\theta}(0)$,
\begin{align*}
g_{u,2}(x)\leq C\Big(1+\log^+ \Big(\frac{R}{\theta}\frac{1}{16R_\theta^2} \Big)\Big)= C.
\end{align*}
 and so it follows that
\begin{align*}
\tilde{Z}_{T_\theta^R}(g_{u,2})=&\sum_{x \in \Z_R^d \cap Q_{4R_\theta}(0)} {Z}_{T_\theta^R}(x) g_{u,2}(x)\leq CZ_{T_\theta^R}(1).
\end{align*}
 On the event $\{{Z}_{T_\theta^R}(1)\leq e^{2T} Z_0(1)\}$, the above becomes
\begin{align}\label{4ea4.6}
\tilde{Z}_{T_\theta^R}(g_{u,2}) \leq C\cdot e^{2T} Z_0(1) \leq Ce^{2T} \frac{2R \sqrt{\theta}}{\theta}\leq C(T)\frac{R}{\theta^{1/4}},
\end{align}
where the second inequality uses \eqref{4eb1.31}.
Similarly in $d=3$, by \eqref{4e10.32} we have for any $x\in Q_{4R_\theta}(0)$,
\begin{align}\label{4e11.44}
g_{u,3}(x)\leq C\frac{R}{4R_\theta}\leq  C\sqrt{\theta}.
\end{align}
It follows that
\begin{align*}
\tilde{Z}_{T_\theta^R}(g_{u,3})=&\sum_{x \in \Z_R^d \cap Q_{4R_\theta}(0)} {Z}_{T_\theta^R}(x) g_{u,3}(x) \leq C\sqrt{\theta} Z_{T_\theta^R}(1).
\end{align*}
On the event $\{{Z}_{T_\theta^R}(1)\leq e^{2T} Z_0(1)\}$, we have
\begin{align}\label{4ea4.7}
\tilde{Z}_{T_\theta^R}(g_{u,3})\leq C\sqrt{\theta} \cdot e^{2T} Z_0(1) \leq C\sqrt{\theta} \cdot e^{2T}  \frac{2R^2 \log \theta}{\theta}\leq C(T)\frac{R^2}{\theta^{1/4}},
\end{align}
where the last inequality is by  $\log \theta\leq \theta^{1/4}$ for $\theta\geq 100$.
Now we conclude from \eqref{4e11.42}, \eqref{4ea4.6}, \eqref{4ea4.7} that
\begin{align}
\P^{Z_0}\Big(\tilde{Z}_{T_\theta^R}(g_{u,d})\leq C(T) \frac{R^{d-1}}{\theta^{1/4}},\quad \forall u\in Q_{8 \sqrt{\log f_d(\theta)}R_\theta} (0)^c \Big)\geq 1-\eps_0.
\end{align}
By Proposition \ref{4p3}, the proof is complete by letting $m_{\ref{4c0.1}}=m_{\ref{4p3}}+C(T)$.
\end{proof}

Now we return to Proposition \ref{4p3}. To do this, we will calculate the corresponding exponential moments.
Throughout the rest of this section, we fix $\eta=1/8$. 
\begin{proposition}\label{4p3.6}
Let $d=2$. For any $T\geq 100$, there exist constants $C_{\ref{4p3.6}}(T)>0$ and $\theta_{\ref{4p3.6}}(T)\geq 100$ such that for all $\theta \geq \theta_{\ref{4p3.6}}(T)$,  there is some $K_{\ref{4p3.6}}(T,\theta)\geq 4\theta$ such that for any $R\geq  K_{\ref{4p3.6}}$ and any $Z_0$ satisfying \eqref{4eb1.31}, we have
\begin{align*}
&\text{(i) } \E^{Z_0}\Big(\exp\Big({\theta}^{1/2} R^{-1}  Z_{T_\theta^R}(g_{u,2})\Big)\Big)\leq C_{\ref{4p3.6}}(T), \quad \forall u \in \R^2;\\
&\text{(ii) }\E^{{Z}_0}\Big(\exp\Big({\theta}^{1/2} R^{-1}  \frac{(R/\theta)^{\eta/2}}{|u-v|^{\eta}} |{Z}_{T_\theta^R}(g_{u,2})-{Z}_{T_\theta^R}(g_{v,2})|\Big)\Big)\leq C_{\ref{4p3.6}}(T), \quad \forall u\neq v \in \R^3.
\end{align*}
\end{proposition}

\begin{proposition}\label{4p3.7}
Let $d=3$. For any $T\geq 100$, there exist constants $C_{\ref{4p3.7}}(T)>0$ and $\theta_{\ref{4p3.7}}(T)\geq 100$ such that for all $\theta \geq \theta_{\ref{4p3.7}}(T)$,  there is some $K_{\ref{4p3.7}}(T,\theta)\geq 4\theta$ such that for any $R\geq  K_{\ref{4p3.7}}$ and any $Z_0$ satisfying \eqref{4eb1.31}, we have
\begin{align*}
&\text{(i) } \E^{Z_0}\Big(\exp\Big(\frac{\theta^{1/2}R^{-2} }{\log \theta}  Z_{T_\theta^R}(g_{u,3})\Big)\Big)\leq C_{\ref{4p3.7}}(T), \quad \forall u \in \R^3;\\
&\text{(ii) }\E^{{Z}_0}\Big(\exp\Big(\frac{\theta^{1/2}R^{-2} }{\log \theta}  \frac{(R^2/\theta)^{\eta/2}}{|u-v|^{\eta}} |{Z}_{T_\theta^R}(g_{u,3})-{Z}_{T_\theta^R}(g_{v,3})|\Big)\Big)\leq C_{\ref{4p3.7}}(T), \quad \forall u\neq v \in \R^3.
\end{align*}
\end{proposition}
\begin{proof}[Proof of Proposition \ref{4p3} assuming Propositions \ref{4p3.6}, \ref{4p3.7}]
For any $T, \theta, R$ fixed, the random measure ${Z}_{T_\theta^R}$ is a.s. finite and $\{g_{u,d}: u\in \R^d\}$ are continuous functions that are uniformly bounded (see, e.g., \eqref{4eb1.32} and \eqref{4e6.20}). Hence the family $\{{Z}_{T_\theta^R}(g_{u,d}): u\in \R^d\}$ is  an almost surely continuous random field. By applying Lemma \ref{4l2.2}, we may finish the proof of  Proposition \ref{4p3} in a  way similar to that of Corollary \ref{4c2.1}. So the details are omitted. 
\end{proof}

It remains to prove Propositions \ref{4p3.6}, \ref{4p3.7}, which we now give.

\subsection{Exponential moments of $Z_{{T_\theta^R}}(g_{u,2})$ in $d=2$}

Let $d=2$. For any $u\in \R^2$ and $x\in \Z_R^2$, Lemma \ref{4l3.3} with $n=T_\theta^R$ and $\beta=1$ implies
\begin{align}\label{4eb1.33}
 \sum_{y\in \Z^d_R} p_{T_\theta^R}(x-y)  g_{u,2}(y)\leq& c_{\ref{4l3.3}} \Big(1+ \frac{1}{T_\theta^R} \log \frac{2R}{\theta}+\Big(\frac{R}{T_\theta^R\theta}\Big)^{1/2}\Big)\nn\\
 \leq &C\Big(1+ \frac{2\theta}{TR} \frac{2R}{\theta}+\Big(\frac{R}{\theta}\frac{2\theta}{TR}\Big)^{1/2}\Big)\leq C(T),
\end{align}
where in the second inequality we have used \eqref{4eb1.34} and $\log s\leq s$, $\forall s>0$ .  
Apply Proposition \ref{4p1.2} and   \eqref{4eb1.34}  to get
 \begin{align*}
 \E^{x}({Z}_{{T_\theta^R}}(g_{u,2})) & =(1+\frac{\theta}{R})^{{T_\theta^R}} \sum_{y\in \Z^d_R} p_{{T_\theta^R}}(x-y)  g_{u,2}(y) \leq e^{\frac{T_\theta^R \theta}{R}} C(T)\leq c(T).
\end{align*}
So it follows that (recall $|Z_0|\leq 2R/\sqrt{\theta}$ in $d=2$)
 \begin{align}\label{4e5.34}
 \E^{Z_0}({Z}_{{T_\theta^R}}(g_{u,2}))\leq c(T) |Z_0|\leq C(T)\frac{R}{\theta^{1/2}}.
\end{align}
\begin{proof}[Proof of Proposition \ref{4p3.6}(i)]
 Let $\lambda={\theta}^{1/2} R^{-1}$ and $n=T_\theta^R\leq \frac{TR}{\theta}$. Next, recall from \eqref{4ec7.4} to see that 
\begin{align}\label{4eb1.36}
 \lambda  e^{\frac{T_\theta^R\theta}{ R}} G(g_{u,2},T_\theta^R)\leq {\theta}^{1/2} R^{-1} e^{T} \cdot C(T) \frac{R}{\theta}\leq c(T) \frac{1}{\theta^{1/2}}.
\end{align}
 If we pick $\theta>0$ large enough so that $c(T)/{\theta^{1/2}}\leq 1/2$, then we may apply Corollary \ref{4c1.2} to get \begin{align*}
&\E^{Z_0}\Big(\exp\Big(\lambda Z_{{T_\theta^R}}(g_{u,2})\Big)\Big)\leq \exp\Big(\lambda \E^{{Z}_0}({Z}_{{T_\theta^R}}(g_{u,2})) (1-\lambda e^{\frac{T_\theta^R\theta}{ R}}  G(g_{u,2},{T_\theta^R}))^{-1}\Big)\\
\leq &\exp\Big(\lambda C(T)\frac{R}{\theta^{1/2}}  (1-c(T)/{\theta^{1/2}})^{-1}\Big)\leq \exp\Big(\lambda C(T)\frac{R}{\theta^{1/2}}  \cdot 2\Big)=e^{2C(T)},
\end{align*}
where we have used \eqref{4e5.34}, \eqref{4eb1.36} in the second inequality and the last inequality is by $c(T)/{\theta^{1/2}}\leq 1/2$. Thus the proof of Proposition \ref{4p3.6}(i) is finished.
\end{proof}

Turning to the difference moments in Proposition \ref{4p3.6}(ii), we fix any $u\neq v \in \R^d$. By (3.44) of \cite{Sug89}, for any $0<\alpha\leq 1$, there exists some constant $C(\alpha)>0$ such that
\begin{align}\label{4e6.18a}
|e^{-\frac{|x|^2}{2t}}-e^{-\frac{|y|^2}{2t}}|\leq C(\alpha) t^{-\alpha/2} |x-y|^\alpha (e^{-\frac{|x|^2}{4t}}+e^{-\frac{|y|^2}{4t}}),\ \forall t>0, x,y\in \R^d.
\end{align}
Use the above with $\alpha=\eta$ to see that for any $y\in \R^d$, we have
\begin{align}\label{4e10.42}
|g_{u,2}(y)-g_{v,2}(y)|\leq&  \sum_{k=1}^\infty e^{-k\theta/R}  \frac{1}{k} |e^{-\frac{|y-u|^2}{32k}}-e^{-\frac{|y-v|^2}{32k}}|\nn\\
\leq & \sum_{k=1}^\infty e^{-k\theta/R}   \frac{1}{k} C(\eta) (16k)^{-\eta/2} |u-v|^\eta (e^{-\frac{|y-u|^2}{64k}}+e^{-\frac{|y-v|^2}{64k}})\nn\\
\leq& C |u-v|^{\eta}   \sum_{k=1}^\infty  \frac{1}{k^{(2+\eta)/2}} (e^{-\frac{|y-u|^2}{64k}}+e^{-\frac{|y-v|^2}{64k}}).
\end{align}
It follows that for any $n\geq 1$ and $x\in \Z_R^d$,
\begin{align}\label{4e10.43}
&\sum_{y\in \Z^d_R} p_n(y-x)  |g_{u,2}(y)-g_{v,2}(y)|\nn\\
\leq &C |u-v|^{\eta} \sum_{y\in \Z^d_R} p_n(y-x) \sum_{k=1}^\infty \frac{1}{k^{(2+\eta)/2}} (e^{-\frac{|y-u|^2}{64k}}+e^{-\frac{|y-v|^2}{64k}})\nn\\
 \leq & C |u-v|^{\eta} \cdot 2 c_{\ref{4l1.3}} n^{-\eta/2}\leq C |u-v|^{\eta} n^{-\eta/2},
\end{align}
where the second inequality is by Lemma \ref{4l1.3}. Apply Proposition \ref{4p1.2} and \eqref{4e10.43} with $n=T_\theta^R$ to get
\begin{align*}
 \E^{x}({Z}_{{T_\theta^R}}(|g_{u,2}-g_{v,2}|))=&(1+\frac{\theta}{R})^{{T_\theta^R}} \sum_{y\in \Z^d_R} p_{{T_\theta^R}}(y-x) |g_{u,2}(y)-g_{v,2}(y)|\\
\leq &e^{\frac{T_\theta^R \theta}{R}}  C |u-v|^{\eta}   (T_\theta^R)^{-\eta/2}\leq C e^T |u-v|^{\eta}  \Big(\frac{1}{2} \frac{TR}{\theta}\Big)^{-\eta/2}\\
=&C(T) |u-v|^\eta R^{-\eta/2} \theta^{\eta/2},
\end{align*}
where in the last inequality we have used \eqref{4eb1.34}. So we have (recall $|Z_0|\leq 2R/\theta^{1/2}$)
 \begin{align}\label{4e5.35}
 \E^{Z_0}({Z}_{{T_\theta^R}}(|g_{u,2}-g_{v,2}|))&\leq  |Z_0| C(T) |u-v|^\eta R^{-\eta/2} \theta^{\eta/2}\nn\\
 &\leq C(T) |u-v|^\eta R^{1-\eta/2} \theta^{(\eta-1)/2}.
\end{align}
Next, we apply \eqref{4e10.42} and \eqref{4e10.43} again to see that
\begin{align}\label{4eb1.35}
G(|g_{u,2}-g_{v,2}|,{T_\theta^R})=&3\|g_{u,2}-g_{v,2}\|_\infty+\sum_{k=1}^{{T_\theta^R}} \sup_{y \in \Z_R^d} \sum_{z\in \Z^d_R} p_k(y-z) |g_{u,2}(y)-g_{v,2}(y)|\nn\\  
\leq& C |u-v|^{\eta}+C |u-v|^{\eta} \sum_{k=1}^{{T_\theta^R}}  k^{-\frac{\eta}{2}}\nn\\
\leq& C |u-v|^{\eta} (T_\theta^R)^{1-\eta/2}\leq c(T) |u-v|^\eta \frac{R^{1-\eta/2}}{\theta^{1-\eta/2}}.
\end{align}
where in the last inequality we have used \eqref{4eb1.34}.

\begin{proof}[Proof of Proposition \ref{4p3.6}(ii)]
 Let $\lambda=\theta^{(1-\eta)/2} R^{\eta/2-1} |u-v|^{-\eta}$ and $n=T_\theta^R\leq \frac{TR}{\theta}$. By \eqref{4eb1.35} we have
\begin{align}\label{4e5.72}
 \lambda  e^{\frac{T_\theta^R\theta}{R}} G(|g_{u,2}-g_{v,2}|,T_\theta^R)\leq \frac{\theta^{(1-\eta)/2} R^{\eta/2-1}}{ |u-v|^{\eta}} e^{T} \cdot c(T) |u-v|^\eta \frac{R^{1-\eta/2}}{\theta^{1-\eta/2}}\leq c(T) \frac{1}{\theta^{1/2}}.
\end{align}
 If we pick $\theta>0$ large enough so that $c(T)/{\theta^{1/2}}\leq 1/2$, then we may apply Corollary \ref{4c1.2} to get
\begin{align*}
&\E^{Z_0}\Big(\exp\Big(\lambda |Z_{{T_\theta^R}}(g_{u,2})-Z_{{T_\theta^R}}(g_{v,2})|\Big)\Big)\leq \E^{Z_0}\Big(\exp\Big(\lambda Z_{{T_\theta^R}}(|g_{u,2}-g_{v,2}|)\Big)\Big)\\
\leq& \exp\Big(\lambda \E^{{Z}_0}({Z}_{{T_\theta^R}}(|g_{u,2}-g_{v,2}|)) (1-\lambda e^{T} G(|g_{u,2}-g_{v,2}|,{T_\theta^R}))^{-1}\Big)\\
\leq &\exp\Big(\lambda C(T) |u-v|^\eta R^{1-\eta/2} \theta^{(\eta-1)/2}  (1-c(T)/{\theta^{1/2}} )^{-1}\Big)\\
\leq &\exp\Big(\lambda C(T) |u-v|^\eta R^{1-\eta/2} \theta^{(\eta-1)/2}  \cdot 2\Big)=e^{2C(T)},
\end{align*}
where we have used \eqref{4e5.35}, \eqref{4e5.72}  in the third inequality and the last inequality is by $c(T)/{\theta^{1/2}}\leq 1/2$. So the proof is complete.
\end{proof}

\subsection{Exponential moments of $Z_{{T_\theta^R}}(g_{u,3})$ in $d=3$}

Let $d=3$. Fix $u\in \R^d$. For any $x\in \Z_R^d$,  we may apply  Lemma \ref{4l1.3} to get for any $n\geq 1$,
\begin{align}\label{4e5.49}
\sum_{y\in \Z^d_R} p_n(y-x)  g_{u,3}(y)= R \sum_{y\in \Z^d_R} p_n(y-x) \sum_{k=1}^\infty \frac{1}{k^{3/2}} e^{-\frac{|y-u|^2}{64k}} \leq  R \cdot c_{\ref{4l1.3}} n^{-1/2}.
\end{align}
By Proposition \ref{4p1.2}, we have
  \begin{align*}
 \E^{x}({Z}_{{T_\theta^R}}(g_{u,3}))=&(1+\frac{\theta}{R^{2}})^{{T_\theta^R}} \sum_{y\in \Z^d_R} p_{{T_\theta^R}}(x-y) g_{u,3}(y)   \\
  \leq& e^{\frac{T_\theta^R\theta}{R^{2}}} R \cdot c_{\ref{4l1.3}} (T_\theta^R)^{-1/2} \leq  C(T) \theta^{1/2},
 \end{align*}
where in the last inequality we have used \eqref{4eb1.34}. Hence it follows that (recall in $d=3$ that $|Z_0|\leq 2R^2 \log \theta/\theta$)
 \begin{align}\label{4e5.36}
 \E^{Z_0}({Z}_{{T_\theta^R}}(g_{u,3}))\leq  |Z_0| C(T) \theta^{1/2}\leq C(T)R^2 \frac{\log \theta}{\theta^{1/2}}.
 \end{align}
Next we use \eqref{4eb1.32} and \eqref{4e5.49} to see that
\begin{align}\label{4e5.50}
G(g_{u,3}, {T_\theta^R})=& 3\|g_{u,3}\|_\infty+\sum_{k=1}^{T_\theta^R} \sup_{y \in \Z_R^d} \sum_{z\in \Z^d_R} p_k(y-z) g_{u,3}(z)\nn\\
\leq &CR+\sum_{k=1}^{T_\theta^R} R\cdot c_{\ref{4l1.3}}  k^{-1/2}\leq  CR\sqrt{T_\theta^R}\leq c(T)\frac{R^2}{\theta^{1/2}},
\end{align}
where in the last inequality we have used \eqref{4eb1.34}.
\begin{proof}[Proof of Proposition \ref{4p3.7}(i)]
Let $\lambda={\theta}^{1/2} R^{-2}/\log \theta$ and $n=T_\theta^R\leq \frac{TR^2}{\theta}$. By \eqref{4e5.50} we have
\begin{align}\label{4e5.71}
 \lambda  e^{\frac{T_\theta^R\theta}{R^{2}}} G(g_{u,3},T_\theta^R)\leq \frac{{\theta}^{1/2} R^{-2}}{\log \theta} e^{T} \cdot c(T)\frac{R^2}{\theta^{1/2}}\leq c(T) \frac{1}{\log \theta}.
\end{align}
 If we pick $\theta>0$ large enough so that $c(T)/{\log \theta}\leq 1/2$, then we may apply Corollary \ref{4c1.2} to get
\begin{align*}
&\E^{Z_0}\Big(\exp\Big(\lambda Z_{{T_\theta^R}}(g_{u,3})\Big)\Big)\leq \exp\Big(\lambda \E^{{Z}_0}({Z}_{{T_\theta^R}}(g_{u,3})) (1-\lambda e^{\frac{T_\theta^R\theta}{R^{2}}} G(g_{u,3},{T_\theta^R}))^{-1}\Big)\\
\leq &\exp\Big(\lambda C(T)R^2 \frac{\log \theta}{\theta^{1/2}}  (1-c(T)/{\log \theta} )^{-1}\Big)\\
\leq&\exp\Big(\lambda C(T) \frac{R^2 \log \theta}{{\theta}^{1/2}} \cdot 2  \Big)= e^{2C(T)},
\end{align*}
where we have used \eqref{4e5.36}, \eqref{4e5.71} in the second inequality and the last inequality is by $c(T)/ \log \theta \leq 1/2$.
\end{proof}

Turning to the difference moments, we fix $u\neq v \in \R^d$. For any $y\in \R^d$, by \eqref{4e6.18a} with $\alpha=\eta$ we have
\begin{align}\label{4e5.51}
|g_{u,3}(y)-g_{v,3}(y)|\leq&  R \sum_{k=1}^\infty   \frac{1}{k^{3/2}} |e^{-\frac{|y-u|^2}{32k}}-e^{-\frac{|y-v|^2}{32k}}|\nn\\
\leq &R \sum_{k=1}^\infty   \frac{1}{k^{3/2}} C(\eta) (16k)^{-\eta/2} |u-v|^\eta (e^{-\frac{|y-u|^2}{64k}}+e^{-\frac{|y-v|^2}{64k}})\nn\\
=& CR |u-v|^{\eta}   \sum_{k=1}^\infty  \frac{1}{k^{(3+\eta)/2}} (e^{-\frac{|y-u|^2}{64k}}+e^{-\frac{|y-v|^2}{64k}}).
\end{align}
For any $x\in \Z_R^d$,  we may use the above and Lemma \ref{4l1.3} to get for any $n\geq 1$,
\begin{align}\label{4e5.52}
&\sum_{y\in \Z^d_R} p_n(y-x)  |g_{u,3}(y)-g_{v,3}(y)|\nn\\
\leq &CR |u-v|^{\eta} \sum_{y\in \Z^d_R} p_n(y-x) \sum_{k=1}^\infty \frac{1}{k^{(3+\eta)/2}} (e^{-\frac{|y-u|^2}{64k}}+e^{-\frac{|y-v|^2}{64k}})\nn\\
 \leq & CR |u-v|^{\eta} \cdot 2 c_{\ref{4l1.3}} n^{-(1+\eta)/2} \leq CR |u-v|^{\eta} n^{-(1+\eta)/2}.
\end{align}
By using Proposition \ref{4p1.2} and the above, we get for any $x\in \Z_R^d$,
\begin{align*}
 \E^{x}({Z}_{{T_\theta^R}}(|g_{u,3}-g_{v,3}|))=&(1+\frac{\theta}{R})^{{T_\theta^R}} \sum_{y\in \Z^d_R} p_{{T_\theta^R}}(y-x) |g_{u,3}(y)-g_{v,3}(y)|\nn\\
\leq &e^{\frac{T_\theta^R\theta}{R^{2}}} CR |u-v|^{\eta} \cdot 2 c_{\ref{4l1.3}}   (T_\theta^R)^{-(1+\eta)/2}\leq C(T) R^{-\eta} \theta^{(1+\eta)/2} |u-v|^\eta,
\end{align*}
where in the last inequality we have used \eqref{4eb1.34}. Hence it follows that (recall $|Z_0|\leq 2R^2 \log \theta/\theta$)
 \begin{align}\label{4e5.37}
 \E^{Z_0}({Z}_{{T_\theta^R}}(|g_{u,3}-g_{v,3}|))&\leq  |Z_0| C(T) R^{-\eta} \theta^{(1+\eta)/2} |u-v|^\eta\nn\\
 &\leq C(T) R^{2-\eta} \frac{\log \theta}{ \theta^{(1-\eta)/2}} |u-v|^\eta.
 \end{align}
 Next we use \eqref{4e5.51} and \eqref{4e5.52} to get
 \begin{align}\label{4e7.14}
G(|g_{u,3}-g_{v,3}|,{T_\theta^R})=&3\|g_{u,3}-g_{v,3}\|_\infty+\sum_{k=1}^{{T_\theta^R}} \sup_{y \in \Z_R^d} \sum_{z\in \Z^d_R} p_k(y-z) |g_{u,3}(y)-g_{v,3}(y)|\nn\\
\leq& C R |u-v|^{\eta}+CR |u-v|^{\eta} \sum_{k=1}^{{T_\theta^R}}   k^{-(1+\eta)/2} \nn\\
\leq&  C R |u-v|^{\eta} (T_\theta^R)^{(1-\eta)/2}\leq c(T)  |u-v|^{\eta} \frac{R^{2-\eta}}{\theta^{(1-\eta)/2}}.
\end{align}
\begin{proof}[Proof of Proposition \ref{4p3.7}(ii)]
Let $\lambda=|u-v|^{-\eta}{\theta}^{(1-\eta)/2} R^{\eta-2}/\log \theta$ and $n=T_\theta^R\leq \frac{TR^2}{\theta}$. By \eqref{4e7.14} we have
\begin{align}\label{4e10.52}
 \lambda  e^{\frac{T_\theta^R\theta}{R^{2}}} G(|g_{u,3}-g_{v,3}|,T_\theta^R)\leq \frac{{\theta}^{(1-\eta)/2} R^{\eta-2}}{|u-v|^{\eta}\log \theta} e^{T} \cdot  c(T)  |u-v|^{\eta} \frac{R^{2-\eta}}{\theta^{(1-\eta)/2}}\leq \frac{c(T) }{\log \theta}.
\end{align}
 If we pick $\theta>0$ large enough so that $c(T)/{\log \theta}\leq 1/2$, then we may apply Corollary \ref{4c1.2} to get (recall $|Z_0|\leq 2R^2 \log \theta/{\theta}$)
\begin{align*}
&\E^{Z_0}\Big(\exp\Big(\lambda |Z_{{T_\theta^R}}(g_{u,3})-Z_{{T_\theta^R}}(g_{v,3})|\Big)\Big)\leq \E^{Z_0}\Big(\exp\Big(\lambda Z_{{T_\theta^R}}(|g_{u,3}-g_{v,3}|)\Big)\Big)\\
\leq& \exp\Big(\lambda \E^{{Z}_0}({Z}_{{T_\theta^R}}(|g_{u,3}-g_{v,3}|)) (1-\lambda e^{\frac{T_\theta^R\theta}{R^{2}}} G(|g_{u,3}-g_{v,3}|,{T_\theta^R}))^{-1}\Big)\\
\leq &\exp\Big(\lambda  C(T) R^{2-\eta} \frac{\log \theta}{ \theta^{(1-\eta)/2}} |u-v|^\eta
 (1-c(T)/{\log \theta} )^{-1}\Big)\\
\leq &\exp\Big(\lambda C(T) R^{2-\eta} \frac{\log \theta}{ \theta^{(1-\eta)/2}} |u-v|^\eta
  \cdot 2 \Big)= e^{2C(T)},
\end{align*}
where we have used \eqref{4e5.37}, \eqref{4e10.52} in the third inequality and the last inequality is by $c(T)/ \log \theta \leq 1/2$. So the proof is complete.
\end{proof}

\section{Mass propagation of SIR epidemic}\label{4s8}

Finally we will prove Proposition \ref{4p4} in this section. Fix any $\eps_0\in (0,1)$, $\kappa>0$ and $T\geq 100$ satisfying \eqref{4ea10.04}. We will choose $\theta\geq 100$ and $R\geq 4\theta$ large below. Recall that we assume $\eta_0\subseteq \Z_R^d$ is as in \eqref{4ea10.23} such that 
\begin{align}\label{4eb2.1}
\begin{dcases}
\text{(i) }\eta_0\subseteq Q_{R_\theta}(0)\\
\text{(ii) } R^{d-1} f_d(\theta)/\theta \leq |\eta_0|\leq 1+R^{d-1} f_d(\theta)/\theta\\
\text{(iii) }|\eta_0 \cap Q(y)|\leq K \beta_d(R), \forall y\in \Z^d,
\end{dcases}
\end{align}
where $\beta_d$ is defined as in \eqref{4ea10.45} and $K\geq 100$ is some large constant to be chosen. Let $\eta$ be an SIR epidemic process starting from $(\eta_0, \rho_0)$ where $\rho_0$ is  any finite subset of $\Z_R^d$ disjoint from $\eta_0$. Recall that $\cA(0)=\{(0,1), (1,0)\}$ in $d=2$ and $\cA(0)=\{(0,1,0), (1,0,0)\}$ in $d=3$. Fix any $y\in \cA(0)$. It suffices to show that
\[
\P \Big(\{|\hat{\eta}_{T_\theta^R}^{K_{\ref{4p4}} }\cap Q_{R_\theta}(y R_\theta)|< |\eta_0| \} \cap N(\kappa)\Big)\leq \frac{\eps_0}{2},
\] where \[N(\kappa)=\{|\rho_{T_\theta^R}\cap \cN(x) |\leq \kappa R, \forall x\in \Z_R^d\}.\]
Recall that we also write $\eta_0(x)=1(x\in \eta_0)$ for $x\in \Z_R^d$ so that $\eta_0\in M_F(\Z_R^d)$. Define $Z_0=\eta_0$ so that $Z_0$ dominates $\eta_0$ and $|Z_0|=|\eta_0|$. Then we may apply Lemma \ref{4ea10.45} to couple a branching random walk $(Z_n)$ starting from $Z_0$ with $\eta$ so that $Z_n$ dominates $\eta_n$ for any $n\geq 0$, i.e. $Z_n(x)\geq \eta_n(x)$ for any $x\in \Z_R^d$.  

The outline for the proof of Proposition \ref{4p4} is as follows: We first prove that with high probability, $Z_{T_\theta^R}(Q_{R_\theta}(y R_\theta))\geq 6|\eta_0|$. Next on the event $N(\kappa)$, we show that the SIR epidemic $\eta$ satisfies with high probability, $\eta_{T_\theta^R}(Q_{R_\theta}(y R_\theta))\geq 2|\eta_0|$. Finally we use the dominating branching random walk again to show that  with high probability, the difference between $\eta_{T_\theta^R}(Q_{R_\theta}(y R_\theta))$ and thinned version $\hat\eta_{T_\theta^R}^K(Q_{R_\theta}(y R_\theta))$ is no larger than $|\eta_0|$, thus completing the proof.

\subsection{Mass propagation of branching envelope}
We show in this subsection that $Z_{T_\theta^R}(Q_{R_\theta}(y R_\theta))\geq 6|\eta_0|$ holds with probability larger than $1-\eps_0/8$, which is done by calculating its first and second moments. First we consider the branching random walk $Z=(Z_n)$ starting from a single particle at $x\in Q_{R_\theta}(0) \cap \Z_R^d$ whose law is denoted by $\P^x$. By \eqref{4ea4.5} we know $(Z_n(1))$ is a Galton-Watson process with offspring distribution $Bin(V(R),p(R))$. Use the mean and variance formula for Galton-Watson process to see that (see, e.g., Chapter 4.7 of \cite{Ross})
\begin{align}\label{4e5.78}
 \E^{x}({{Z}}_{T_\theta^R}(1))=(V(R)p(R))^{T_\theta^R}=(1+\frac{\theta}{R^{d-1}})^{T_\theta^R} \leq e^T,
\end{align}
and
\begin{align*}
\text{Var}^{x}({{Z}}_{T_\theta^R}(1))=& (1+\frac{\theta}{R^{d-1}})(1-p(R))\cdot (1+\frac{\theta}{R^{d-1}})^{T_\theta^R-1} \frac{(1+\frac{\theta}{R^{d-1}})^{T_\theta^R}-1}{\frac{\theta}{R^{d-1}}}  \nn\\
\leq&  (1+\frac{\theta}{R^{d-1}})^{T_\theta^R}  (1+\frac{\theta}{R^{d-1}})^{T_\theta^R} \frac{R^{d-1}}{\theta}\leq e^{2T} \frac{R^{d-1}}{\theta}.
\end{align*}
Hence it follows that
\begin{align}\label{4e5.74}
\text{Var}^{x}({{Z}}_{T_\theta^R}(Q_{R_\theta}(yR_\theta)))\leq&\E^{x}\Big({{Z}}_{T_\theta^R}(Q_{R_\theta}(yR_\theta))^2\Big)\leq \E^{x}\Big({{Z}}_{T_\theta^R}(1)^2\Big)\nn\\
\leq& e^{2T} \frac{R^{d-1}}{\theta}+(e^T)^2\leq 2e^{2T}  \frac{R^{{d-1}}}{\theta}.
\end{align}

\no Next, recall from \eqref{4eb2.2} that $(S_n)$ is the random walk taking uniform steps in $\cN(0)$. By Proposition \ref{4p1.2}, we have
\begin{align}\label{4e2.23}
\E^{x}(Z_{T_\theta^R}(Q_{R_\theta}(yR_\theta)))&=(1+\frac{\theta}{R^{d-1}})^{T_\theta^R}  \P(S_{T_\theta^R}+x\in Q_{R_\theta}(yR_\theta))\nn\\
&\geq e^{\frac{\theta T_\theta^R }{2R^{d-1}}}\P(S_{T_\theta^R}+x\in Q_{R_\theta}(yR_\theta))\nn\\
& \geq e^{T/4} \P(S_{T_\theta^R}+x\in Q_{R_\theta}(yR_\theta)) ,
\end{align}
where the first inequality is by $1+x\geq e^{x/2}$ for $0\leq x\leq 1$ and the last inequality uses \eqref{4eb1.34}. Recall we pick $x\in Q_{R_\theta}(0)$. So we may write $x=z_R \cdot R_\theta$ with $z_R \to z\in [-1,1]^d$ as $R\to \infty$ and it follows from Central Limit Theorem that
\begin{align}
\P\Big(S_{T_\theta^R}+x \in Q_{R_\theta}(yR_\theta) \Big)&=\P\Big(\frac{S_{T_\theta^R}+x}{R_\theta} \in Q(y) \Big) \nn\\
&\to \P(\zeta_T^z\in Q(y)) \text{ as } R\to \infty,
\end{align}
where $\zeta_T^z$ is a $d$-dimensional Gaussian random variable with mean $z$ and variance $T/3$.
Returning to \eqref{4e2.23}, the above implies if $R>R_0(\theta, T)$ for some constant $R_0(\theta,T)>0$, we have
\begin{align}\label{4e5.75}
\E^{x}(Z_{T_\theta^R}(Q_{R_\theta}(yR_\theta)))\geq e^{T/4} \frac{1}{2} \inf_{z\in Q(0)} \inf_{y\in Q(0)}\P(\xi_T^z\in Q(y))\geq 8,
\end{align}
where the last inequality is by \eqref{4ea10.04}.

Now we return to the BRW $Z=(Z_n)$ starting from $Z_0$ whose law is denoted by $\P^{Z_0}$. Since \eqref{4e5.74} and \eqref{4e5.75} hold for any $x\in Q_{R_\theta}(0)\cap \Z_R^d$, we may conclude
\begin{align}\label{4ed2.24}
\E^{Z_0}({{Z}}_{T_\theta^R}(Q_{R_\theta}(yR_\theta)))\geq  8 |Z_0|,
\end{align}
and
\begin{align*}
\text{Var}^{Z_0}({{Z}}_{T_\theta^R}(Q_{R_\theta}(yR_\theta)))\leq 2e^{2T}  \frac{R^{{d-1}}}{\theta} |Z_0|.
\end{align*}
Use \eqref{4ed2.24} and Chebyshev's inequality to get 
\begin{align}\label{4e2.24}
\P^{Z_0}\Big({{Z}}_{T_\theta^R}(Q_{R_\theta}(yR_\theta))\leq 6 |Z_0|  \Big)\leq& \P^{Z_0}\Big(\Big|{{Z}}_{T_\theta^R}(Q_{R_\theta}(yR_\theta))-\E^{Z_0}({{Z}}_{T_\theta^R}(Q_{R_\theta}(yR_\theta)))\Big|\geq 2 |Z_0|  \Big)\nn\\
\leq& \frac{2e^{2T}  \frac{R^{{d-1}}}{\theta} |Z_0|}{(2 |Z_0| )^2}\leq \frac{1}{2}e^{2T}  \frac{1}{f_d(\theta)},
\end{align}
where the last inequality uses $|Z_0|=|\eta_0|\geq R^{d-1} f_d(\theta)/\theta$ by \eqref{4eb2.1}. Pick $\theta\geq 100$ large so that
\begin{align}\label{4eb2.4}
\P^{Z_0}\Big({{Z}}_{T_\theta^R}(Q_{R_\theta}(yR_\theta))\leq 6 |Z_0|  \Big)\leq \frac{\eps_0}{8}.
\end{align}
\subsection{Mass propagation of SIR epidemic}

To show that $\eta_{T_\theta^R}(Q_{R_\theta}(y R_\theta))\geq 2|\eta_0|$ holds with high probability on the event $N(\kappa)$, we will couple the original epidemic $(\eta_n)$ with a Modified SIR epidemic $(\bar{\eta}_n)$: Let $\bar{\eta}_0=\eta_0$. At time $n\geq 1$, any particle in location $x$ will produce offspring to each of its neighbouring sites in $\cN(x)$ while avoiding birth to the recovered sites in $\rho_n$. In other words, the particle located at $x$  will produce $Bin(V(R)-|\rho_n \cap \cN(x)|, p(R))$ to its neighbouring sites. In this way we allow two different particles to give birth to the same location (multiple occupancy). One can construct $(\bar{\eta}_n)$ together with the original SIR $(\eta_n)$ and the branching envelope $(Z_n)$ so that (i) the Modified SIR always dominates the original SIR; (ii) the branching envelope $(Z_n)$ always dominates the Modified SIR. This coupling can be done in a way similar to that of Lemma \ref{4l1.5}. Denote by $\P$ the joint law of $(Z,\bar \eta, \eta)$.

The difference between $(\eta_n)$ and $(\bar{\eta}_n)$ comes from the event called ``collision'': when two infected sites simultaneously attempt to infect the same susceptible site, all but one of the attempts fail. Let $\Gamma_{n}(x)$ be the number of collisions at site $x$ and time $n$ in the SIR epidemic $(\eta_n).$  Write $f(R)=o(h(R))$ if $f(R)/h(R) \to 0$ as $R\to \infty$. The following lemma is from Lemma 2.26 of \cite{LPZ14} (see also Lemma 9 of \cite{LZ10}), whose proof will be contained in Appendix \ref{a2}.
\begin{lemma}\label{4l10.01}
For any $T\geq 100$ and $\theta\geq 100$, we have
\[
\E \Big(\sum_{n=1}^{T_\theta^R}  \sum_{x\in \Z^d_R}\Gamma_{n}(x)\Big)=o(R^{d-1}).
\]
\end{lemma}
By an argument similar to the proof of (2.41) of \cite{LPZ14}, one may notice that the difference between $\bar{\eta}_{T_\theta^R}(1)$ and $\eta_{T_\theta^R}(1)$ is at most the sum of all the offsprings of the ``lost'' particles due to collisions. Hence it follows that
\begin{align}\label{4e5.82}
\E (\bar{\eta}_{T_\theta^R}(1)-\eta_{T_\theta^R}(1))\leq& \E  \Big(\sum_{n=1}^{T_\theta^R} (1+\frac{\theta}{R^{d-1}})^{T_\theta^R-n}  \sum_{x\in \Z^d_R}\Gamma_{n}(x)\Big)\nn\\
\leq &(1+\frac{\theta}{R^{d-1}})^{T_\theta^R} \E  \Big(\sum_{n=1}^{T_\theta^R}  \sum_{x\in \Z^d_R}\Gamma_{n}(x)\Big)\leq e^T o(R^{d-1}).
\end{align}
where the last inequality is by \eqref{4eb1.34}. This gives that
  \begin{align*}
&\P\Big(\bar{\eta}_{T_\theta^R}(Q_{R_\theta}(yR_\theta))-{\eta}_{T_\theta^R}(Q_{R_\theta}(yR_\theta))\geq  |\eta_0| \Big)\nn\\
\leq &\P\Big(\bar{\eta}_{T_\theta^R}(1)-{\eta}_{T_\theta^R}(1)\geq  |\eta_0| \Big)\leq  \frac{1}{|\eta_0|} \E\Big(\bar{\eta}_{T_\theta^R}(1)-{\eta}_{T_\theta^R}(1) \Big)\nn\\
\leq & \frac{\theta}{R^{d-1} f_d(\theta)} \E\Big(\bar{\eta}_{T_\theta^R}(1)-{\eta}_{T_\theta^R}(1) \Big) \to 0 \text{ as } R\to \infty.
\end{align*}
Hence if $R$ is large, we have
  \begin{align}\label{4e10.46}
\P\Big(\bar{\eta}_{T_\theta^R}(Q_{R_\theta}(yR_\theta))-{\eta}_{T_\theta^R}(Q_{R_\theta}(yR_\theta))\geq  |\eta_0| \Big)\leq \frac{\eps_0}{8}.
\end{align}

\no Set $\gamma_{0}=|\bar{\eta}_{0}|=|{\eta}_{0}|$ and let $(\gamma_{n}, n\geq 0)$ be a Galton-Watson process with offspring distribution $Bin(V(R)-\kappa R, p(R))$. On the event 
\[
N(\kappa)=\{ |\rho_{T_\theta^R} \cap \cN(x)| \leq \kappa R, \quad \forall x\in \Z^d_R\},
 \]
one may check that the process $(\bar{\eta}_{n}(1), n\geq 0)$ will dominate $(\gamma_{n}, n\geq 0)$ up to time $T_\theta^R$, that is, we may define $(\gamma_n)$ on the same probability space so that
  \begin{align}\label{4ed5.79}
 \gamma_{n}\leq \bar{\eta}_{n}(1), \quad \forall  n\leq T_\theta^R, \text{ on the event } N(\kappa).
\end{align}
For the Galton-Watson process $(\gamma_n)$, we have
  \begin{align}\label{4e5.79}
\E(\gamma_{T_\theta^R}) =& \gamma_{0} \Big((V(R)-\kappa R) p(R)\Big)^{T_\theta^R}\nn\\
= &|{\eta}_{0}|(1+\frac{\theta}{R^{d-1}})^{T_\theta^R}  (1-\frac{\kappa R}{V(R)})^{T_\theta^R}.
\end{align}
On the other hand, by \eqref{4e5.78} we have the branching random walk $Z$ satisfies $\E^{Z_0}(Z_{T_\theta^R}(1))=(1+\frac{\theta}{R^{d-1}})^{T_\theta^R} |Z_0|$.  Choose $R$ large so that ${\kappa R}< V(R)/2$. Then we may use \eqref{4e5.79} to see that (recall $|Z_0|=|\eta_0|$)
  \begin{align}\label{4e5.80}
\E({{Z}}_{T_\theta^R}(1)- \gamma_{T_\theta^R})= &  (1+\frac{\theta}{R^{d-1}})^{T_\theta^R} |\eta_0| \Big[1-(1-\frac{\kappa R}{V(R)})^{T_\theta^R}\Big]\nn\\
\leq & e^{T}|\eta_0| \cdot  (-T_\theta^R \log (1-\frac{\kappa R}{V(R)}))  \leq e^{T}|\eta_0| T_\theta^R \cdot 2\frac{\kappa R}{V(R)}\nn\\
\leq & e^{T}|\eta_0| \frac{TR^{d-1}}{\theta} \cdot 2\frac{\kappa R}{V(R)}\leq C(T) |\eta_0| \frac{\kappa}{\theta},
\end{align}
where the first inequality is by $1-(1-x)^n=1-e^{n\log(1-x)}\leq -n\log(1-x)$ for any $x\in (0,\frac{1}{2})$ and $n\geq 1$. The second inequality uses $-\log (1-x) \leq 2x$ for $x\in (0,\frac{1}{2})$. Apply Markov's inequality and \eqref{4e5.80} to get 
  \begin{align}\label{4e2.25}
\P\Big({{Z}}_{T_\theta^R}(1)-\gamma_{T_\theta^R}\geq 3   |\eta_0| \Big)\leq\frac{ C(T) |\eta_0| \frac{\kappa}{\theta}}{3   |\eta_0|}=\frac{1}{3}C(T) \frac{\kappa}{\theta}< \frac{\eps_0}{8},
\end{align}
if we pick $\theta>0$ large. Since $(Z_n)$ dominates $(\bar{\eta}_{n})$, we have for any $A\subseteq \Z^d$,
  \begin{align}\label{4e5.81}
 {{Z}}_{T_\theta^R}(A)-\bar{\eta}_{T_\theta^R}(A)\leq {{Z}}_{T_\theta^R}(1)-\bar{\eta}_{T_\theta^R}(1)\leq {{Z}}_{T_\theta^R}(1)-\gamma_{T_\theta^R} \text{ on } N(\kappa),
\end{align}
where the last inequality is by \eqref{4ed5.79}.
Now we conclude that
  \begin{align}\label{4e10.47}
&\P\Big(\Big\{\bar{\eta}_{T_\theta^R}(Q_{R_\theta}(yR_\theta))\leq  3|\eta_0| \Big\} \cap N(\kappa)\Big)\nn\\
\leq & \P\Big(\Big\{{{Z}}_{T_\theta^R}(Q_{R_\theta}(yR_\theta))-\bar{\eta}_{T_\theta^R}(Q_{R_\theta}(yR_\theta))\geq  3|\eta_0|\Big\} \cap N(\kappa) \Big)\nn\\
&\quad \quad +\P\Big({{Z}}_{T_\theta^R}(Q_{R_\theta}(yR_\theta))\leq 6 |\eta_0| \Big)\nn\\
\leq & \P\Big({{Z}}_{T_\theta^R}(1)-\gamma_{T_\theta^R}\geq 3 |\eta_0| \Big)+\frac{\eps_0}{8}\leq \frac{\eps_0}{4},
\end{align}
where the second last inequality uses \eqref{4e5.81}, \eqref{4eb2.4} and the last inequality is by  \eqref{4e2.25}. Recall \eqref{4e10.46} to get for $R$ large,
  \begin{align}\label{4e2.30}
&\P\Big(\Big\{{\eta}_{T_\theta^R}(Q_{R_\theta}(yR_\theta))\leq  2|\eta_0|\Big\} \cap N(\kappa)\Big)\nn\\
&\leq\P\Big(\bar{\eta}_{T_\theta^R}(Q_{R_\theta}(yR_\theta))-{\eta}_{T_\theta^R}(Q_{R_\theta}(yR_\theta))\geq  |\eta_0| \Big) \nn\\
&\quad + \P\Big(\Big\{\bar{\eta}_{T_\theta^R}(Q_{R_\theta}(yR_\theta))\leq  3|\eta_0| \Big\} \cap N(\kappa)\Big)\leq \frac{3\eps_0}{8},
\end{align}
where the last inequality uses \eqref{4e10.47}.

\subsection{Mass propagation of the thinned SIR epidemic}
 Finally we will turn to the thinned process $\hat{\eta}_{T_\theta^R}^K$ and show that 
 \[
 \P \Big(\Big\{|\hat{\eta}_{T_\theta^R}^{K} \cap Q_{R_\theta}(y R_\theta)|< |\eta_0| \Big\} \cap N(\kappa)\Big)\leq \frac{\eps_0}{2},
 \]  
 if we pick $K>0$ large.
Recall that the thinned version  $\hat{\eta}_{T_\theta^R}^{K}$ is obtained by deleting all the vertices in ${\eta}_{T_\theta^R} \cap Q(y)$ for each $y\in \Z^d$ if $|{\eta}_{T_\theta^R} \cap Q(y)|>K\beta_d(R)$. We will use the dominating BRW $Z$ to show that with high probability, the amount of the deleted particles will be small.
 
 Recall that $Z_0(x)=1(x\in \eta_0)$ where $\eta_0$ is a subset of $\Z_R^d$. Then we have for any set $D$ and $n\geq 1$,
\begin{align}\label{4e10.27}
\E^{Z_0} (Z_n(D)^2 )&= \sum_{x\in \eta_0} \E^x(Z_n(D)^2)+ \sum_{x\in \eta_0} \sum_{y\in \eta_0, y\neq x} \E^x(Z_n(D)) \E^y(Z_n(D))\nn\\
&\leq  \sum_{x\in \eta_0} \E^x(Z_n(D)^2)+ \Big(\sum_{x\in \eta_0} \E^x(Z_n(D))\Big)^2.
\end{align}
Take $D=Q(a)$ for $a\in \Z^d$ and let $n=T_\theta^R$ to get
 \begin{align}\label{4e10.28}
&\E^{Z_0} (Z_{T_\theta^R}(Q(a))^2 )\leq  \sum_{x\in \eta_0} \E^x(Z_{T_\theta^R}(Q(a))^2)+ \Big(\sum_{x\in \eta_0} \E^x(Z_{T_\theta^R}(Q(a)))\Big)^2.
\end{align}
 
\no  Apply Proposition \ref{4p1.2}(i) and Proposition \ref{4p1.1}(i) to see that for any $x\in \Z_R^d$,
\begin{align}\label{4e10.29}
\E^x(Z_{T_\theta^R}(Q(a)))=&(1+\frac{\theta}{R^{d-1}})^{T_\theta^R} \sum_{y\in \Z_R^d} 1_{Q(a)}(y) p_{T_\theta^R}(x-y)\nn\\
\leq & e^T \sum_{y\in \Z_R^d} 1_{Q(a)}(y) \frac{c_{\ref{4p1.1}}}{(T_\theta^R)^{d/2} R^d} \leq C(T) \frac{1}{(T_\theta^R)^{d/2}}.
 \end{align}
Recall $G(\phi,n)$ from \eqref{4e5.90}. Use Proposition \ref{4p1.1}(i) to see that
\begin{align}\label{4e10.30}
 G(1_{Q(a)},T_\theta^R)=&3\|1_{Q(a)}\|_\infty+\sum_{k=1}^{T_\theta^R} \sup_{y \in \Z_R^d} \sum_{z\in \Z^d_R} p_k(y-z) 1_{Q(a)}(z)\nn\\
 \leq &3+\sum_{k=1}^{T_\theta^R} \sup_{y \in \Z_R^d} \sum_{z\in \Z^d_R} \frac{c_{\ref{4p1.1}}}{k^{d/2} R^d} 1_{Q(a)}(z)\nn\\
  \leq &3+\sum_{k=1}^{T_\theta^R} \frac{c_{\ref{4p1.1}}}{k^{d/2}}  c(d) \leq C \sum_{k=1}^{T_\theta^R} \frac{1}{k^{d/2}}.
\end{align}
Write $h_d(n):=\sum_{k=1}^{n} \frac{1}{k^{d/2}}$ for $n\geq 1$. Then it follows that 
\begin{align}\label{4eb2.8}
G(1_{Q(a)},T_\theta^R)\leq C h_d(T_\theta^R)\leq 
\begin{cases}
C+C \log T_\theta^R\leq C(T)\log R, &d=2;\\
C, &d=3.
\end{cases}
\end{align}
Next, by Proposition \ref{4p1.2}(ii), we have
\begin{align}\label{4e10.48}
\E^x\Big(Z_{T_\theta^R}(Q(a))^2\Big)\leq&
e^{\frac{\theta T_\theta^R}{R^{d-1}}}  G(1_{Q(a)},T_\theta^R)  \E^{x}\Big({Z}_{T_\theta^R}(Q(a))\Big)\leq  e^T  C  h_d(T_\theta^R) \cdot C(T) \frac{1}{(T_\theta^R)^{d/2}}\nn\\
\leq &C(T) h_d(T_\theta^R) (T_\theta^R)^{-d/2},
\end{align}
where the second inequality follows from \eqref{4e10.29} and \eqref{4eb2.8}.
Returning to \eqref{4e10.28}, we use \eqref{4e10.29} and \eqref{4e10.48} to see that
 \begin{align}\label{4e10.03}
\E^{Z_0} \Big(Z_{T_\theta^R}(Q(a))^2 \Big)\leq& |\eta_0| C(T) h_d(T_\theta^R) (T_\theta^R)^{-d/2}+ \Big(|\eta_0|C(T) \frac{1}{(T_\theta^R)^{d/2}}\Big)^2\nn\\
\leq&\frac{2R^{d-1} f_d(\theta)}{\theta} C(T)h_d(T_\theta^R) \Big(\frac{2\theta}{TR^{d-1}}\Big)^{d/2}\nn\\
&\quad \quad +\Big(\frac{2R^{d-1} f_d(\theta)}{\theta}\Big)^2 C(T) \Big(\frac{2\theta}{TR^{d-1}}\Big)^{d}:=I,
\end{align}
where we have used \eqref{4eb1.34} in the second inequality. In $d=2$, we use \eqref{4eb2.8} to get 
 \begin{align}\label{4eb2.11}
I\leq& \frac{2R \sqrt{\theta}}{\theta} C(T) \log R \cdot \frac{2\theta}{TR}+\Big(\frac{2R\sqrt{\theta}}{\theta}\Big)^2 C(T) \Big(\frac{2\theta}{TR}\Big)^{2}\nn\\
\leq& C(T)\sqrt{\theta} \log R+C(T)\theta\leq  C(T)\sqrt{\theta} \log R,
\end{align}
where the last inequality is by $\theta\leq \sqrt{\theta} \log R$ when $R$ is large. In $d=3$, by \eqref{4eb2.8} we have
 \begin{align}\label{4eb2.10}
I\leq& \frac{2R^{2} \log \theta}{\theta} C(T) C \Big(\frac{2\theta}{TR^{2}}\Big)^{3/2}+\Big(\frac{2R^{2} \log\theta}{\theta}\Big)^2 C(T) \Big(\frac{2\theta}{TR^{2}}\Big)^{3}\nn\\
\leq& C(T) \frac{1}{R} \sqrt{\theta} \log \theta+C(T)\frac{1}{R^2}\theta (\log \theta)^2 \leq C(T) \frac{1}{R} \sqrt{\theta} \log \theta,
\end{align}
where in the last inequality we have used $\frac{1}{R^2}\theta (\log \theta)^2 \leq \frac{1}{R} \sqrt{\theta} \log \theta$ when $R$ is large. Hence we conclude from \eqref{4e10.03}, \eqref{4eb2.11}, \eqref{4eb2.10} that 
 \begin{align}\label{4eb2.9}
\E^{Z_0} \Big(Z_{T_\theta^R}(Q(a))^2 \Big)\leq
\begin{cases}
C(T)\sqrt{\theta} \log R &d=2\\
C(T) \frac{1}{R} \sqrt{\theta} \log \theta, &d=3.
\end{cases}
\end{align}

 Write $V(a)={{Z}}_{T_\theta^R}(Q(a))$. It follows that
\begin{align}\label{4e2.29}
&\E(V(a)\cdot 1_{\{V(a)>K \beta_d(R)\}})\leq \frac{\E(V(a)^2)}{K\beta_d(R)} 
\leq
\begin{cases}
C(T) \frac{\sqrt{\theta}}{K}, &d=2\\
C(T) \sqrt{\theta}  \frac{\log \theta}{KR}, &d=3.
\end{cases}
\end{align}
Let \[D=\sum_{a\in A} V(a) 1_{\{V(a)>K\beta_d(R)\}},\] where $A=Q_{R_\theta}(y R_\theta) \cap \Z^d$ so that \[Q_{R_\theta}(yR_\theta) \subseteq \bigcup_{a\in A} Q(a).\] Observe that $|A|\leq C (R_\theta)^d $.  Recall $R_\theta=\sqrt{R^{d-1}/\theta}$ and use \eqref{4e2.29} to see that
\[
\E (D)\leq 
\begin{dcases}
C(T) \frac{\sqrt{\theta}}{K} C (R_\theta)^2 \leq  \frac{C(T)}{K} \frac{R\sqrt{\theta}}{\theta}\leq \frac{C(T)}{K} |\eta_0|, &d=2;\\
C(T) \sqrt{\theta}  \frac{\log \theta}{KR} C (R_\theta)^3\leq \frac{C(T)}{K} \frac{R^2 \log \theta}{\theta}\leq \frac{C(T)}{K} |\eta_0|, &d=3.
\end{dcases}
\]
 Since $Z_{T_\theta^R}$ dominates $\eta_{T_\theta^R}$, we get
\begin{align}
 0\leq \eta_{T_\theta^R}(Q_{R_\theta}(yR_\theta))-\hat{\eta}_{T_\theta^R}^K(Q_{R_\theta}(yR_\theta))\leq& \sum_{a\in A} \eta_{T_\theta^R}(Q(a)) 1_{\{\eta_{T_\theta^R}(Q(a)) >K\beta_d(R)\}}\nn\\
 \leq& \sum_{a\in A} Z_{T_\theta^R}(Q(a)) 1_{\{Z_{T_\theta^R}(Q(a)) >K\beta_d(R)\}}= D.
\end{align}
  By Markov's inequality, we have
\begin{align}
&\P\Big({\eta}_{T_\theta^R}(Q_{R_\theta}(yR_\theta))-\hat{\eta}_{T_\theta^R}^K(Q_{R_\theta}(yR_\theta)) \geq |\eta_0|\Big)\leq \frac{\E (D)}{|\eta_0|}\leq C(T) \frac{1}{K}\leq \frac{\eps_0}{8},
\end{align}
if we pick $K>0$ large. Recall \eqref{4e2.30} and use the above to see that
\begin{align}
& \P \Big(\Big\{|\hat{\eta}_{T_\theta^R}^{K} \cap Q_{R_\theta}(y R_\theta)|< |\eta_0| \Big\} \cap N(\kappa)\Big)\nn\\
&\leq \P\Big({\eta}_{T_\theta^R}(Q_{R_\theta}(yR_\theta))-\hat{\eta}_{T_\theta^R}^K(Q_{R_\theta}(yR_\theta)) \geq |\eta_0|\Big)\nn\\
&\quad \quad + \P \Big(\Big\{{\eta}_{T_\theta^R}(Q_{R_\theta}(yR_\theta))\leq 2 |\eta_0| \Big\} \cap N(\kappa)\Big)\leq \frac{\eps_0}{2},
\end{align}
and so the proof of Proposition \ref{4p4} is complete.

\bibliographystyle{plain}

\begin{thebibliography}{10}


\bibitem{BDS89}
M. Bramson, R. Durrett, G. Swindle. 
\newblock Statistical Mechanics of Crabgrass. 
\newblock {\em Ann. Probab.}, {\bf 17}, no. 2, 444-481, (1989).

%
%
\bibitem{DIP89}
D. Dawson, I. Iscoe and E. Perkins.
\newblock Super-Brownian motion: Path properties and hitting probabilities.
\newblock {\em Prob. Th. Rel. Fields} 83: 135--205, (1989).


%

\bibitem{EK86}
S. N. Ethier and T. G. Kurtz.
\newblock Markov Processes: Characterization and Convergence.
\newblock  John Wiley $\&$ Sons, Inc., New York, (1986).  


\bibitem{Du10}
R. Durrett.
\newblock Probability: Theory and Examples.
\newblock {\em Cambridge Series in Statistical and Probabilistic Mathematics.}
\newblock Cambridge University Press, New York, (2010).

\bibitem{Free75}
D. Freedman.
\newblock On Tail Probabilities for Martingales.
\newblock {\em Ann. Probab.}, {\bf 3} (1): 100-118, (1975).


\bibitem{FP16}
S. Frei and E. Perkins. 
\newblock A lower bound for $p_c$ in range-$R$ bond percolation in two and three dimensions. 
\newblock {\em Electron. J. Probab.}, 21: no. 56, 1--22, (2016).


\bibitem {HS05}
R. van der Hofstad and A. Sakai.
\newblock Critical points for spread-out self-avoiding walk, percolation
and the contact process above the the upper critical dimensions.
\newblock {\em Prob. Th. Rel. Fields} 132: 438-470, (2005). 



%
%
%
%
%
%
%
\bibitem{Hong18}
J. Hong.
\newblock Renormalization of local times of super-Brownian motion.
\newblock {\em Electron. J. Probab.}, 23: no. 109, 1--45, (2018).
%
%
%
%
%
%
%
%
%
%

\bibitem{LPZ14}
S. Lalley, E. Perkins and X. Zheng. 
\newblock A phase transition for measure-valued SIR epidemic processes. 
\newblock {\em Ann. Probab.}, {\bf 42} (1): 237--310, (2014).



\bibitem{LZ10}
S. Lalley and X. Zheng. 
\newblock Spatial epidemics and local times for critical branching random walks
in dimensions $2$ and $3$.
\newblock {\em Prob. Th. Rel. Fields}, {\bf 148} (3–4): 527–566, (2010). 


\bibitem{LL10}
G. Lawler and V. Limic.
\newblock  Random walk: a modern introduction.
\newblock  {\em Cambridge University Press}, Cambridge, (2010). 


\bibitem{Mol77}
D. Mollison.
\newblock Spatial contact models for ecological and epidemic spread.
\newblock {\em J. Roy. Statist. Soc. B}, {\bf 39}: 283–326, (1977). 

\bibitem{Pen93}
M. D. Penrose.
\newblock On the spread-out limit for bond and continuum percolation.
\newblock {\em Ann. Probab.}, {\bf 3}: 253–276, 1993. 




%
%
%
%
%
%
%
%
\bibitem{Per88}
E.A.~Perkins.
\newblock A Space-Time Property of a Class of Measure-Valued Branching Diffusions.
\newblock {\em Transactions of the American Mathematical Society}, {\bf 305}, no. 2: 743--795, (1988).



\bibitem{Per02}
E.A.~Perkins.
\newblock Dawson-Watanabe Superprocesses and Measure-valued Diffusions.
\newblock {\em Lectures on Probability Theory and Statistics, no.\
1781, Ecole d'Et\'e de Probabilit\'es de Saint Flour 1999}
\newblock Springer, Berlin (2002).

\bibitem{Pin95}
R.G.~Pinsky.
\newblock On the large time growth rate of the support of supercritical super-Brownian motion. 
\newblock {\em Ann. Probab.}, {\bf 23}, 1748--1754, (1995). 




\bibitem{Ross}
S.~Ross.
\newblock Introduction to probability models. Twelfth edition.
\newblock {\em Academic Press},
\newblock  London, (2019).



%
%
%
\bibitem{Sug89}
S. Sugitani.
\newblock Some properties for the measure-valued branching diffusion processes.
\newblock {\em J. Math. Soc. Japan}, {\bf 41}: 437--462, (1989).

%


\end{thebibliography}
\def\cprime{$'$}

\clearpage

\appendix

\section{Approximation by characteristic function} \label{a3}

This section is devoted to the proof of Proposition \ref{4p1.1} where $d\geq 1$.
Let $Y_1, Y_2, \cdots$ be i.i.d. random variables uniform on $\cN(0)$. Define $\rho(t)=\E e^{it\cdot Y_1},  t\in \R^d$ to be the characteristic function of $Y_1$ and denote by $\Gamma$ the covariance matrix of $Y_1$, which is given by
\begin{align}\label{4e1.8}
\Gamma=[\E(Y_1^i Y_1^j)]_{1\leq i,j\leq d}=\frac{R(R+1)}{3R^2} \frac{(2R+1)^d}{V(R)}I:=\frac{\lambda_0(R,d)}{3}I.
\end{align}
In the above, $\lambda_0(R,d)$ is some constant which will converge to $1$ as $R\to \infty$. Throughout the rest of this section, we will write $\lambda_0=\lambda_0(R,d)$ for simplicity and only consider $R>0$ large so that $1/2\leq \lambda_0\leq 3/2$. \\

Write $S_n=Y_1+\cdots+Y_n$ for each $n\geq 1$. The characteristic function of $S_n$ will be given by $\rho_{S_n}(t)=\rho^n(t)$. For any $x\in \Z_R^d$, we let $xR=(x_1R, \cdots, x_d R)\in \Z^d$. Then by applying Proposition 2.2.2 of \cite{LL10}, we have for any $x\in \Z_R^d$,
\begin{align}\label{4e6.62}
p_n(x)=&\P(S_n=x)=\P(S_n R=xR)=\frac{1}{(2\pi)^d} \int_{[-\pi,\pi]^d} \rho^n(tR) e^{-it \cdot xR} dt\nn\\
=&\frac{1}{(2\pi)^d n^{d/2} R^d} \int_{[-\sqrt{n} R\pi,\sqrt{n}R\pi]^d} \rho^n\Big(\frac{t}{\sqrt{n}}\Big) e^{-i\frac{t\cdot x}{\sqrt{n}}} dt.
\end{align}
Following (2.2) of \cite{LL10}, we will approximate $p_n(x)$ by (recall $\Gamma$ from \eqref{4e1.8})
\begin{align}\label{4e2.14}
\bar{p}_n(x):=\frac{1}{(2\pi)^d n^{d/2} R^d} \int  e^{-\frac{\lambda_0}{6} |t|^2} e^{-i\frac{t\cdot x}{\sqrt{n}}} dt=\frac{(3/\lambda_0)^{d/2}}{(2\pi)^{d/2} n^{d/2} R^d} e^{-\frac{3|x|^2}{2n\lambda_0}},
\end{align}
where $\lambda_0=\lambda_0(R,d)$ is as in \eqref{4e1.8}. Before giving the error estimates between $p_n(x)$ and $\bar{p}_n(x)$, we first state some preliminary results on $\rho(t)$.

\begin{lemma}\label{4l1.2}
(i) There is some constant $c_{\ref{4l1.2}}=c_{\ref{4l1.2}}(d)>0$ so that for all $R\geq 1$,
\[
\sup_{|t|\leq \sqrt{d} \pi R} |\rho(t)|\leq \frac{c_{\ref{4l1.2}}}{|t|}. 
\]
(ii) For any $0<\delta<1$, there are constants $C_{\ref{4l1.2}}>0$, $K_{\ref{4l1.2}}>0$ depending only on $d, \delta$ such that for any $R\geq C_{\ref{4l1.2}}$,
\[
\sup_{\delta\leq |t|\leq \delta^{-1}} |\rho(t)|\leq e^{-K_{\ref{4l1.2}}}.
\]
\end{lemma}
\begin{proof}
(i) For any $t=(t_1,\cdots, t_d) \in \R^d$, we have
\begin{align}\label{4e4.1}
\rho(t)=&\frac{1}{V(R)}\Big(\sum_{k_1=-R}^R \cdots \sum_{k_d=-R}^R e^{it_1 \frac{k_1}{R}}\cdots e^{it_d \frac{k_d}{R}} -e^{it\cdot 0}\Big)\nn\\
=&\frac{1}{V(R)}\Big(\prod_{k=1}^d \frac{e^{-it_k}-e^{it_k \frac{R+1}{R}}}{1-e^{it_k\frac{1}{R}}}-1\Big)=\frac{1}{V(R)}\Big(\prod_{k=1}^d \Big(e^{it_k}+\frac{e^{it_k}-e^{-it_k}}{e^{it_k\frac{1}{R}}-1}\Big)-1\Big)\nn\\
=&\frac{1}{V(R)}\Big(\prod_{k=1}^d \Big(e^{it_k}+\frac{e^{it_k}-e^{-it_k}}{it_k\frac{1}{R}}\frac{it_k\frac{1}{R}}{e^{it_k\frac{1}{R}}-1}\Big)-1\Big)\nn\\
=&\frac{(2R)^d}{V(R)}\Big(\prod_{k=1}^d \Big(\frac{e^{it_k}}{2R}+\frac{\sin t_k}{t_k}\frac{it_k\frac{1}{R}}{e^{it_k\frac{1}{R}}-1}\Big)-\frac{1}{(2R)^d}\Big).
\end{align}
If $ |t|\leq \sqrt{d} \pi R$, we have $|t_k/R|\leq \sqrt{d} \pi$ for any $1\leq k\leq d$, and so we may use $\frac{|s|}{|e^{is}-1|} \leq c$, $\forall |s|\leq \sqrt{d} \pi$ for some constant $c>0$ to get
\begin{align}\label{4e5.20}
|\rho(t)|\leq & \frac{1}{(2R)^d}+\prod_{k=1}^d \Big(\frac{1}{2R}+\frac{|\sin t_k|}{|t_k|}\frac{|t_k\frac{1}{R}|}{|e^{it_k\frac{1}{R}}-1|}\Big)\nn\\
\leq & \frac{1}{(2R)^d}+\prod_{k=1}^d \Big(\frac{1}{2R}+c \frac{|\sin t_k|}{|t_k|}\Big).
\end{align}
For any $t=(t_1,\cdots, t_d)\in \R^d$, there is some $1\leq j\leq d$ so that $|t_j|=\max\{|t_k|, 1\leq k\leq d\}$ and hence $|t|\leq \sqrt{d}|t_j|$. Use $|\sin t_k|\leq |t_k|$ and $|\sin t_k|\leq 1$ to arrive at
\begin{align*}
|\rho(t)|\leq & \frac{1}{(2R)^d}+(1+c)^{d-1}\Big(\frac{1}{2R}+c \frac{|\sin t_j|}{|t_j|}\Big) \\
\leq & \frac{1}{(2R)^d}+(1+c)^{d-1}\Big(\frac{1}{2R}+c\frac{1}{|t|/\sqrt{d}}\Big) \leq C(d)\frac{1}{|t|},
\end{align*}
where the last inequality is by $|t|\leq \sqrt{d} \pi R$. The proof of (i) is then complete.\\

(ii) For any $\delta\leq |t|\leq \delta^{-1}$, there is some $1\leq j\leq d$ so that $|t_j|=\max\{|t_k|, 1\leq i\leq d\}$ and hence $\delta \leq |t|\leq \sqrt{d}|t_j|$. It follows that
\begin{align}\label{4eb3.1a}
\frac{|\sin (t_j)|}{|t_j|} \leq \sup_{|x|>\delta/\sqrt{d}} \frac{|\sin x|}{|x|}\leq e^{-K},
\end{align}
for some $K=K(d, \delta)>0$. Since $\lim_{x\to 0} \frac{|x|}{|e^{ix}-1|}=1$ and $|t_k/R|\leq |t|/R\leq  \delta^{-1}/R$ for all $1\leq k\leq d$, we get for $R$ large,
\begin{align}\label{4eb3.2a}
 \sup_{\delta\leq |t|\leq \delta^{-1}} \frac{|t_k\frac{1}{R}|}{|e^{it_k\frac{1}{R}}-1|} \leq  \sup_{|x|\leq \delta^{-1}/R} \frac{|x|}{|e^{ix}-1|} \leq 1+\frac{K}{2d}, \quad \forall 1\leq k\leq d.
\end{align}
Recall the first inequality in \eqref{4e5.20}. We may apply \eqref{4eb3.1a}, \eqref{4eb3.2a} to get for $R$ large,
\begin{align*}
\sup_{\delta\leq |t|\leq \delta^{-1}}|\rho(t)|\leq & \frac{1}{(2R)^d}+\prod_{k=1}^d \Big(\frac{1}{2R}+ (1+\frac{K}{2d}) \frac{|\sin t_k|}{|t_k|}\Big)\\
\leq &\frac{1}{(2R)^d}+\Big (\frac{1}{2R}+ (1+\frac{K}{2d})\Big)^{d-1} \Big(\frac{1}{2R}+ (1+\frac{K}{2d}) e^{-K}\Big).
\end{align*}
Let $R\to \infty$ to see that 
\begin{align*}
\limsup_{R\to \infty} \sup_{\delta\leq |t|\leq \delta^{-1}}|\rho(t)|\leq  (1+\frac{K}{2d})^{d} e^{-K}\leq e^{-\frac{1}{2}K}.
\end{align*}
So for $R$ large enough, we have
\[ \sup_{\delta\leq |t|\leq \delta^{-1}} |\rho(t)| \leq e^{-\frac{1}{2}K},\]
and the proof is complete.
\end{proof}

\begin{lemma}\label{4l3.4}
There are constants $c_{\ref{4l3.4}}(d), C_{\ref{4l3.4}}(d)>0$ such that for any $R\geq C_{\ref{4l3.4}}(d)$,
\begin{align}\label{4e6.19}
 \sup_{x\in \Z_R^d} |p_n(x)-\bar{p}_n(x)|\leq  \frac{c_{\ref{4l3.4}}}{n^{d/2+1} R^d}, \quad \forall n\geq 1.
\end{align}
\end{lemma}
\begin{proof}
Recall $p_1(x)=\frac{1}{V(R)} 1(x\in \cN(0))$. By using $p_{n+1}(x)=\sum_{y} p_n(x-y) p
_1(y)$, one may easily conclude by induction that
\[
\sup_{x\in \Z_R^d} p_n(x) \leq C(d)\frac{1}{R^d}  \text{ for all } n\geq 1.
\]
On the other hand, recall $\bar{p}_n(x)$ from \eqref{4e2.14} to see that 
\begin{align}\label{4eb2.41}
\sup_{x\in \Z_R^d} \bar{p}_n(x) \leq C(d)\frac{1}{n^{d/2}R^d} \text{ for all } n\geq 1.
\end{align}
Hence it follows that for $1\leq n\leq 2d$,
\[
 \sup_{x\in \Z_R^d} |p_n(x)-\bar{p}_n(x)|\leq   \sup_{x\in \Z_R^d} p_n(x)+ \sup_{x\in \Z_R^d} \bar{p}_n(x)\leq C(d)\frac{1}{R^d} \leq \frac{C(d) (2d)^{d/2+1}}{n^{d/2+1} R^d}.
\]
It suffices to prove \eqref{4e6.19} for any $n\geq 2d.$

Use the symmetry of $Y_1$ to get for any $t\in \R^d$,
\begin{align}\label{4e5.1}
f(t):=\E e^{it\cdot Y_1}-\sum_{k=0}^3 \E\frac{(it\cdot Y_1)^k}{k!}=\rho(t)-(1-\frac{\lambda_0}{6}|t|^2),
\end{align}
where $\lambda_0=\lambda_0(R,d)$ is as in \eqref{4e1.8}. 
Apply Jensen's inequality and Lemma 3.3.7 of \cite{Du10} to get
\begin{align}\label{4e5.2}
|f(t)|\leq \E\Big( \Big|e^{it\cdot Y_1}-\sum_{k=0}^3 \frac{(it\cdot Y_1)^k}{k!}\Big|\Big)\leq \E \frac{|t\cdot Y_1|^{4}}{4!}\leq \frac{|t|^{4}}{4!}\E|Y_1|^4\leq \frac{1}{24}|t|^4.
\end{align}
Rearrange terms in \eqref{4e5.1} to see $\rho(t)=1-\frac{\lambda_0}{6}|t|^2+f(t)$. Define
 \begin{align}\label{4e5.3}
g(t):=\log \rho(t)-(-\frac{\lambda_0}{6}|t|^2+f(t)).
\end{align}
Since $|\log(1+x)-x|\leq x^2$ when $|x|$ is small, by \eqref{4e5.2} we get
 \begin{align}\label{4e5.4}
|g(t)|\leq (-\frac{\lambda_0}{6}|t|^2+f(t))^2\leq \frac{1}{12}|t|^4, \text{ for } |t|>0 \text{ small, }
\end{align}
where in the last inequality we have used $\lambda_0\leq 3/2$.
Now use \eqref{4e5.3} to see that for any $n\geq 2d$,
 \begin{align}\label{4e5.5}
\rho^n\Big(\frac{t}{\sqrt{n}}\Big)=e^{n\log \rho(\frac{t}{\sqrt{n}})}=\exp\Big(-\frac{\lambda_0}{6}|t|^2+nf(\frac{t}{\sqrt{n}})+ng(\frac{t}{\sqrt{n}})\Big)=e^{-\frac{\lambda_0}{6}|t|^2}F(t,n),
\end{align}
where
 \begin{align}\label{4e5.6}
F(t,n)=\exp\Big(nf(\frac{t}{\sqrt{n}})+ng(\frac{t}{\sqrt{n}})\Big).
\end{align}
 Pick $\delta \in(0,1/2)$ small so that $2c_{\ref{4l1.2}}\leq \delta^{-1}$ and \eqref{4e5.4} holds for any $|t|\leq \delta$. By \eqref{4e6.62} and \eqref{4e2.14}, we have
\begin{align}\label{4e5.7}
&{(2\pi)^d n^{d/2} R^d} |p_n(x)-\bar{p}_n(x)|\nn\\
= &\Big|\int_{[-\sqrt{n} R\pi,\sqrt{n}R\pi]^d} \rho^n(\frac{t}{\sqrt{n}}) e^{-i\frac{t\cdot x}{\sqrt{n}}} dt- \int e^{-\frac{\lambda_0}{6}|t|^2} e^{-i\frac{t\cdot x}{\sqrt{n}}} dt\Big|\nn\\
\leq & \Big|\int_{|t|\leq \delta\sqrt{n}} (\rho^n(\frac{t}{\sqrt{n}})-e^{-\frac{\lambda_0}{6}|t|^2}) e^{-i\frac{t\cdot x}{\sqrt{n}}} dt\Big|+\int_{[-\sqrt{n} R\pi,\sqrt{n}R\pi]^d} 1_{\{|t|>\delta\sqrt{n}\}} \Big|\rho^n(\frac{t}{\sqrt{n}})\Big|  dt\nn\\
&\quad  +\int 1_{\{|t|>\delta\sqrt{n}\}} e^{-\frac{\lambda_0}{6}|t|^2}  dt:=I_1+I_2+I_3.
\end{align}
For $I_3$, one can easily check that for some constants $c_1, c_2>0$ depending on $d,\delta$, we have
 \begin{align}\label{4e5.8}
I_3\leq c_1 e^{-c_2n}.
 \end{align}

\no Turning to $I_2$, for any $n\geq 2d$, we may apply Lemma \ref{4l1.2}(i) to get for any $R\geq C_{\ref{4l1.2}},$
  \begin{align}\label{4e5.11}
I_2 \leq &\int_{\delta\sqrt{n}\leq |t|\leq \sqrt{d}  \sqrt{n} R\pi }  \Big|\rho^n(\frac{t}{\sqrt{n}})\Big| dt=n^{d/2}  \int_{\delta\leq |t|\leq \sqrt{d} \pi R} |\rho({t})|^n dt\nn\\
\leq& n^{d/2} \int_{\delta \leq |t|\leq 2c_{\ref{4l1.2}}} |\rho({t})|^n dt+n^{d/2}\int_{2c_{\ref{4l1.2}}\leq |t|\leq \sqrt{d} \pi R} (\frac{c_{\ref{4l1.2}}}{|t|})^n dt\nn \\
\leq& n^{d/2} \int_{\delta \leq |t|\leq \delta^{-1}} e^{-nK_{\ref{4l1.2}} } dt+n^{d/2}\int_{|t|\geq 2c_{\ref{4l1.2}}} (\frac{c_{\ref{4l1.2}}}{|t|})^n dt\nn \\
 \leq& n^{d/2} C(\delta) e^{-nK_{\ref{4l1.2}} }+n^{d/2}C(d) c_{\ref{4l1.2}}^d\frac{1}{n-d}2^{d-n} \leq c_3 e^{-c_4n},
\end{align}
for some constants $c_3, c_4>0$ depending on $d,\delta$. In the third inequality we have used $2c_{\ref{4l1.2}}\leq \delta^{-1}$ and Lemma  \ref{4l1.2}(ii).

It remains to bound $I_1$. By \eqref{4e5.5} we have
 \begin{align}\label{4e6.75}
I_1=&\Big|\int_{|t|\leq \delta\sqrt{n}} e^{-\frac{\lambda_0}{6}|t|^2}  (F(t,n)-1) e^{-i\frac{t\cdot x}{\sqrt{n}}} dt\Big|\nn\\
=&\Big|\int_{n^{1/8}\leq |t|\leq \delta\sqrt{n}} e^{-\frac{\lambda_0}{6}|t|^2}  (F(t,n)-1) e^{-i\frac{t\cdot x}{\sqrt{n}}} dt\Big|\nn\\
&+\Big|\int_{|t|\leq n^{1/8}} e^{-\frac{\lambda_0}{6}|t|^2}  (F(t,n)-1) e^{-i\frac{t\cdot x}{\sqrt{n}}} dt\Big|:=J_1+J_2.
  \end{align}
We first deal with $J_1$. Since $n^{1/8}\leq |t|\leq \delta\sqrt{n}$ and we have chosen $0<\delta<1/2$ small, we may apply \eqref{4e5.2} and \eqref{4e5.4} to get
  \begin{align}\label{4e6.77}
|nf(\frac{t}{\sqrt{n}})+ng(\frac{t}{\sqrt{n}})|\leq \frac{1}{8}n |\frac{t}{\sqrt{n}}|^4\leq \frac{1}{32}n |\frac{t}{\sqrt{n}}|^2=\frac{1}{32} |t|^2,
  \end{align}
  where in the last inequliaty we have used $|t|/\sqrt{n}\leq \delta\leq 1/2$.
 Recall $F(t,n)$ from \eqref{4e5.6} and apply \eqref{4e6.77} to see that
  \begin{align}\label{4e6.76}
J_1\leq& \int_{n^{1/8}\leq |t|\leq \delta\sqrt{n}} e^{-\frac{\lambda_0}{6}|t|^2}  (1+e^{\frac{1}{32} |t|^2}) dt\leq 2\int_{|t|\geq n^{1/8}} e^{-\frac{1}{24}|t|^2} dt\leq c_5 e^{-c_6 n^{1/4}},
  \end{align}
  for some constant $c_5, c_6>0$ depending on $d$. In the second inequality we have used $\lambda_0\geq 1/2$.
Turning to $J_2$,  we will use the first inequality in \eqref{4e6.77} to see that
  \[
 |nf(\frac{t}{\sqrt{n}})+ng(\frac{t}{\sqrt{n}})|\leq \frac{1}{8}n |\frac{t}{\sqrt{n}}|^4=\frac{1}{8n} |{t}|^4.
 \]
Apply $|e^x-1|\leq 2|x|$ for $|x|<1/2$ and the above to get for $|t|\leq n^{1/8}$,
 \[
 |F(t,n)-1|\leq 2 |nf(\frac{t}{\sqrt{n}})+ng(\frac{t}{\sqrt{n}})|\leq \frac{1}{4n}|t|^4.
 \]
Hence $J_2$ becomes
  \begin{align}\label{4e5.10}
J_2\leq & \int_{|t|\leq n^{1/8}}e^{-\frac{\lambda_0}{6}|t|^2}  \frac{1}{4n}|t|^4 dt  \leq  \frac{1}{4n} \int e^{-\frac{1}{12}|t|^2} |t|^4 dt \leq C(d)\frac{1}{n}.
  \end{align}
Apply \eqref{4e6.76}, \eqref{4e5.10} in \eqref{4e6.75} to get
 \begin{align}\label{4eb3.31}
I_1\leq c_5 e^{-c_6 n^{1/4}}+C(d)\frac{1}{n}\leq C(d)\frac{1}{n}.
  \end{align}
Finally combine \eqref{4e5.7}, \eqref{4e5.8}, \eqref{4e5.11}, \eqref{4eb3.31} to conclude for any $n\geq 2d$,
 \begin{align}\label{4e5.12}
 {(2\pi)^d n^{d/2} R^d} |p_n(x)-\bar{p}_n(x)|\leq c_1 e^{-c_2n}+c_3 e^{-c_4n}+\frac{C(d)}{n} \leq \frac{C(d)}{n},
\end{align}
as required.
\end{proof}
An easy consequence of the above lemma is
 \begin{align}\label{4eb2.45}
\sup_{x\in \Z_R^d} p_n(x)\leq \sup_{x\in \Z_R^d} |p_n(x)-\bar{p}_n(x)|+\sup_{x\in \Z_R^d} \bar{p}_n(x)\leq C(d) \frac{1}{n^{d/2}R^d}, \quad \forall n\geq 1,
\end{align}
where the last inequality uses \eqref{4e6.19} and \eqref{4eb2.41}. Now we are ready to give the proof of Proposition \ref{4p1.1}(i).

\begin{proof}[Proof of Proposition \ref{4p1.1}(i)]

For any $t\in \R$, we let $\phi(t)=\E e^{tY_1^1}$ be the moment generating function of the first coordinate of $Y_1$. Let 
\begin{align}\label{4eb3.10}
f(t):=\phi(t)-\E\Big(\sum_{k=0}^3 \frac{(tY_1^1)^k}{k!}\Big)=\phi(t)-(1+\frac{\lambda_0}{6}t^2). 
\end{align}
For $|t|\leq 1$, we use $|e^x-\sum_{k=0}^3 \frac{x^k}{k!}|\leq \frac{x^4}{12}$ for all $|x|\leq 1$ to get
\begin{align}\label{4e6.78}
|f(t)|\leq \E\Big( \Big|e^{tY_1^1}-\sum_{k=0}^3 \frac{(tY_1^1)^k}{k!}\Big|\Big)\leq \E \Big(\frac{(tY_1^1)^4}{12}\Big)\leq \frac{1}{12}t^4,
\end{align}
where we have used $|Y_1^1|\leq 1$ in the second and the last inequalities. Fix any $n\geq 1$. Recall $S_n=Y_1+\cdots+Y_n$ and define $S_n^1=Y_1^1+\cdots+Y_n^1$. Then $\E e^{t S_n^1}=\phi(t)^n$. For any $0\leq t\leq \sqrt{n}$, we apply \eqref{4eb3.10}, \eqref{4e6.78} to get
\begin{align}\label{4eb3.11}
\E e^{\frac{t}{\sqrt{n}} S_n^1}=\phi(\frac{t}{\sqrt{n}})^n&=\Big(1+\frac{\lambda_0}{6}\frac{t^2}{n}+f(\frac{t}{\sqrt{n}})\Big)^n\nn\\
&\leq \exp\Big(\frac{\lambda_0}{6}{t^2}+\frac{1}{12} \frac{t^4}{n}\Big)\leq \exp\Big(\frac{\lambda_0}{6}{t^2}+\frac{1}{12} {t^2}\Big),
\end{align}
where the last inequality uses $t^2\leq n$. Since $\E (Y_1^1)=0$, we have $\{S_n^1, n\geq 1\}$ is a martingale w.r.t. the filtration generated by $\{Y_n^1\}$. Hence we may use the symmetry of $S_n$ and apply Martingale Maximal Inequality (see, e.g., Theorem 12.2.5 of \cite{LL10}) to get
\begin{align*}
\P(\max_{1\leq k\leq n} \|S_k\|_\infty \geq t\sqrt{n})\leq & d\cdot \P(\max_{1\leq k\leq n} |S_k^1| \geq t\sqrt{n})\leq d\cdot \frac{\E e^{\frac{t}{\sqrt{n}} S_n^1} }{e^{t^2}}\\ 
\leq &de^{-t^2} \exp\Big(\frac{\lambda_0}{6}t^2+\frac{1}{12} {t^2}\Big)\leq de^{-t^2/2},
\end{align*}
where the second last inequality is by \eqref{4eb3.11} and the last inequality uses $\lambda_0\leq 3/2$.
For $t>\sqrt{n}$, the above inequality is immediate since $\|S_k\|_\infty \leq n$ for all $1\leq k\leq n$. Therefore we get for any $t\geq 0$, 
\begin{align*}
&\P(\max_{1\leq k\leq n} |S_n| \geq t\sqrt{n}\sqrt{d}) \leq \P(\max_{1\leq k\leq n} \|S_n\|_\infty \geq t\sqrt{n}) \leq de^{-t^2/2}.
\end{align*}
For any $x\in \Z_R^d$, set $t=|x|/\sqrt{nd}$ in the above to get
\begin{align}\label{4eb2.50}
\P(\max_{1\leq k\leq n} |S_k| \geq |x|)\leq de^{-|x|^2/(2nd)}.
\end{align}

Now we return tot $p_n(x)=\P(S_n=x)$. Notice that when $n=1$, we have for any $x\in \Z_R^d,$
 \[
p_1(x)=\frac{1}{V(R)} 1(x\in \cN(0)) \leq \frac{C(d)}{R^d} e^{-\frac{1}{8}}1(|x|\leq \sqrt{d}) \leq \frac{C(d)}{R^d} e^{-\frac{|x|^2}{8d}},
\]
and so we may assume $n\geq 2$ below. Let $m=n/2$ if $n$ even and $m=(n+1)/2$ if $n$ odd. Then we have $n-m\geq 1$ and
\begin{align}\label{4eb2.46}
\{S_n=x\}=\{S_n=x, |S_m|\geq |x|/2 \} \cup \{S_n=x, |S_n-S_m|\geq |x|/2\}.
\end{align}
It suffices to bound the probabilities of the events on the right-hand side. Apply \eqref{4eb2.50} to get
\begin{align}\label{4eb2.47}
\P(S_n=x, |S_m|\geq |x|/2)=&\P(|S_m|\geq |x|/2)\P\Big(S_n=x\Big| |S_m|\geq |x|/2\Big)\nn\\
\leq &\P(\max_{1\leq k\leq n} |S_k| \geq |x|/2)\cdot \sup_{y\in\Z_R^d} p_{n-m}(x-y)\nn\\
\leq&de^{-\frac{|x|^2}{8nd}} \frac{C(d)}{(n-m)^{d/2} R^d}\leq \frac{C(d)}{n^{d/2} R^d} e^{-\frac{|x|^2}{8nd}},
\end{align}
where we have used \eqref{4eb2.45} in the second last inequality.
The probability of the other event on the right-hand side of \eqref{4eb2.46} can be estimated in a similar way if one notices
  \begin{align}
\P(S_n=x, |S_n-S_m|\geq |x|/2)=\P(S_n=x, |S_{n-m}|\geq |x|/2).
\end{align}
Now it follows from \eqref{4eb2.46}, \eqref{4eb2.47} that
 \[
p_n(x)\leq \frac{C(d)}{n^{d/2} R^d} e^{-\frac{|x|^2}{8nd}},
\]
as required.
\end{proof}

The proof of Proposition \ref{4p1.1}(ii) follows in a similar way to that of Lemma 3 in \cite{LZ10}.

\begin{proof}[Proof of Proposition \ref{4p1.1}(ii)]
It suffices to show that for any $x,y\in \Z^d_R$ with $|x-y|\geq 1$,
\begin{align*}
  |p_n(x)-{p}_n(y)|\leq  C(d)\frac{1}{n^{d/2} R^d} \Big(\frac{|x-y|}{\sqrt{n}} \wedge 1\Big) (e^{-\frac{|x|^2}{16nd}}+e^{-\frac{|y|^2}{16nd}}).
\end{align*}
By Proposition \ref{4p1.1}(i) we have for any $n\geq 1$,
\begin{align}\label{4e8.33}
|p_n(x)-{p}_n(y)|\leq p_n(x)+p_n(y)&\leq \frac{c_{\ref{4p1.1}}}{n^{d/2} R^d} (e^{-\frac{|x|^2}{8nd}}+e^{-\frac{|y|^2}{8nd}}).
\end{align}
Therefore it suffices to show that for any $x,y\in \Z^d_R$ with $|x-y|\geq 1$,
\begin{align}\label{4e8.32}
  |p_n(x)-{p}_n(y)|\leq  C(d)\frac{1}{n^{d/2} R^d} \frac{|x-y|}{\sqrt{n}} (e^{-\frac{|x|^2}{16nd}}+e^{-\frac{|y|^2}{16nd}}).
\end{align}
Since \eqref{4e8.32} holds trivially for $n\leq 2d$ by \eqref{4e8.33}, we may assume that $n\geq 2d$.\\

\no ${\bf Case\ 1.}$  We first consider $|x|,|y| \geq \sqrt{16nd\log n}$. Then we have
\begin{align*}
e^{-\frac{|x|^2}{16nd}} \leq e^{-\log n}\leq \frac{1}{\sqrt{n}}, \text{ and } e^{-\frac{|y|^2}{16nd}} \leq  \frac{1}{\sqrt{n}}.
\end{align*}
If $|x-y|\geq 1$, we may use \eqref{4e8.33} and the above to get
\begin{align}\label{4e8.34}
|p_n(x)-{p}_n(y)|\leq& \frac{c_{\ref{4p1.1}}}{n^{d/2} R^d} (e^{-\frac{|x|^2}{16nd}} +e^{-\frac{|y|^2}{16nd}}) \frac{1}{\sqrt{n}}\nn\\
 \leq& \frac{c_{\ref{4p1.1}}}{n^{d/2} R^d} (e^{-\frac{|x|^2}{16nd}} +e^{-\frac{|y|^2}{16nd}}) \frac{|x-y|}{\sqrt{n}}.
\end{align}

\no ${\bf Case\ 2.}$ Next we consider $|x|,|y| \leq \sqrt{8nd\log n}$ and $|x-y|\geq 1$. Then it follows
\begin{align*}
e^{-\frac{|x|^2}{16nd}} \geq e^{-\frac{1}{2}\log n}= \frac{1}{\sqrt{n}}, \text{ and } e^{-\frac{|y|^2}{16nd}} \geq  \frac{1}{\sqrt{n}}.
\end{align*}
Now use Lemma \ref{4l3.4} and the above to get for all $n\geq 1$ and $R\geq C_{\ref{4l3.4}}$,
\begin{align}\label{4e8.62}
|p_n(x)-\bar{p}_n(x)|\leq& \frac{c_{\ref{4l3.4}}}{n^{d/2+1} R^d}\leq \frac{C(d)}{n^{d/2} R^d} e^{- \frac{|x|^2}{16nd}} \frac{1}{\sqrt{n}} \leq \frac{C(d)}{n^{d/2} R^d} \frac{|x-y|}{\sqrt{n}} e^{- \frac{|x|^2}{16nd}}.
\end{align}
Similarly the above holds for $|p_n(y)-\bar{p}_n(y)|$. Turning to $\bar{p}_n(x)-\bar{p}_n(y)$, we recall from \eqref{4e6.18a} with $\alpha=1$ that there exists some constant $C>0$ such that
\begin{align}\label{4e6.18}
|e^{-\frac{|x|^2}{2t}}-e^{-\frac{|y|^2}{2t}}|\leq Ct^{-1/2} |x-y| (e^{-\frac{|x|^2}{4t}}+e^{-\frac{|y|^2}{4t}}),\ \forall t>0, x,y\in \R^d.
\end{align}
Apply the above  to get
\begin{align}\label{4e8.63}
  |\bar{p}_n(x)-\bar{p}_n(y)|\leq & C(d)\frac{1}{n^{d/2} R^d} \frac{|x-y|}{\sqrt{n}} (e^{-\frac{3|x|^2}{4n\lambda_0}}+e^{-\frac{3|y|^2}{4n\lambda_0}})\nn\\
  \leq & C(d)\frac{1}{n^{d/2} R^d} \frac{|x-y|}{\sqrt{n}} (e^{-\frac{|x|^2}{16nd}}+e^{-\frac{|y|^2}{16nd}}).
\end{align}
Combine \eqref{4e8.62} and \eqref{4e8.63} to see that 
\begin{align*}
  |{p}_n(x)-{p}_n(y)|\leq &   |{p}_n(x)-\bar{p}_n(x)|+  |\bar{p}_n(x)-\bar{p}_n(y)|+  |\bar{p}_n(y)-{p}_n(y)|\\
  \leq & C(d)\frac{1}{n^{d/2} R^d} \frac{|x-y|}{\sqrt{n}} (e^{-\frac{|x|^2}{16nd}}+e^{-\frac{|y|^2}{16nd}}).
\end{align*}

\no ${\bf Case\ 3.}$ Finally if $|x| \leq \sqrt{8nd\log n}$ and $|y| \geq \sqrt{16nd\log n}$ or vice-versa, we have \[|x-y|\geq (4-2\sqrt{2})\sqrt{d} \sqrt{n\log n} \geq \frac{1}{2}\sqrt{n},\] and so by \eqref{4e8.33}, 
\begin{align*}
|p_n(x)-{p}_n(y)| &\leq  \frac{c_{\ref{4p1.1}}}{n^{d/2} R^d} (e^{-\frac{|x|^2}{8nd}}+e^{-\frac{|y|^2}{8nd}}) \\
& \leq \frac{c_{\ref{4p1.1}}}{n^{d/2} R^d} \frac{2 |x-y|}{\sqrt{n}}(e^{-\frac{|x|^2}{16nd}}+e^{-\frac{|y|^2}{16nd}}).
\end{align*}
Now the proof is complete with the above three cases.
\end{proof}

\section{Moments and exponential moments of BRW}\label{a1}

\subsection{Moments and exponential moments of $Z_n$} \label{4ap1.1}
This section gives the proofs of Proposition \ref{4p1.2} and Corollary \ref{4c1.2} which are restated as Proposition \ref{4cp1.2} and Corollary \ref{ac1.2} below. To begin with, we will introduce another labelling system for our BRW. 

Let $\tilde{I}=\cup_{n=0}^\infty \N\times \{1,\cdots, V(R)\}^n$.  If $\beta=(\beta_0, \beta_1, \cdots, \beta_n)\in\tilde{I}$, we set $|\beta|=n$ to be the generation of $\beta$ and write
$\beta|k=(\beta_0, \cdots, \beta_k)$ for each $0\leq k\leq n$. Let $\pi \beta=(\beta_0, \beta_1, \cdots, \beta_{n-1})$ be the parent of $\beta$ and set $\beta \vee i=(\beta_0, \beta_1, \cdots, \beta_n, i)$ to be the $i$-th offspring of $\beta$ for $1\leq i\leq V(R)$. Let $\{{B}^\beta: \beta \in \tilde{I}, |\beta|>0\}$ be i.i.d. Bernoulli random variables with parameter $p(R)$ indicating whether the birth from $\pi \beta$ to $\beta$ is valid. Assume $\{{W}^{\beta \vee i}, 1\leq i\leq V(R)\}_{\beta \in \tilde{I}}$ is a collection of i.i.d. random vectors, each uniformly distributed on  $\cN(0)^{(V(R))}=\{(e_1,\cdots, e_{V(R)}): \{e_i\} \text{ all distinct}\}$.  Let $\{B^{\beta}\}$ and $\{W^\beta\}$ be mutually independent.

Fix any $\tilde{Z}_0\in M_F(\Z_R^d)$. Again we may rewrite $\tilde{Z}_0$ as $\tilde{Z}_0=\sum_{i=1}^{|\tilde{Z}_0|} \delta_{x_i}$ for some $x_i\in \Z_R^d$. If  $i>|\tilde{Z}_0|$, we set $x_i$ to be the cemetery state $\Delta$. Write $\beta \approx n$ if $|\beta|=n$, $\beta_0\leq |\tilde{Z}_0|$ and ${B}^{\beta|i}=1$ for all $1\leq i\leq n$ so that such a $\beta$ labels a particle alive in generation $n$, whose historical path would be given by 
\begin{align}\label{4e6.51}
\tilde{Y}_k^\beta=x_{\beta_0}+\sum_{i=1}^{|\beta|} 1(i\leq k) W^{\beta|i},  \quad \forall k\geq 0.
\end{align}
We denote the current location of the particle $\beta$ by
\begin{align}\label{4e7.31}
\tilde{Y}^\beta=
\begin{cases}
x_{\beta_0}+\sum_{i=1}^{|\beta|}  W^{\beta|i}, &\text{ if } \beta\approx |\beta|,\\
\Delta, &\text{ otherwise. }
\end{cases}
\end{align}
If $|\beta|=0$, we have $\tilde{Y}^\beta=x_{\beta_0}$ for all $1\leq \beta_0\leq |{\tilde{Z}}_0|$ and $\tilde{Y}^\beta=\Delta$ otherwise. 
For any Borel function $\phi$, we define
\begin{align}\label{4ed7.31}
\tilde{Z}_n(\phi)=\sum_{|\beta|=n} \phi(\tilde{Y}^\beta),
\end{align}
where it is understood that $\phi(\Delta)=0$. In this way, $\tilde{Z}$ gives the empirical distribution of a branching random walk where in generation $n$, each particle gives birth to one offspring to its $V(R)$ neighboring positions independently with probability $p(R)$.
Recall the labelling system for BRW $Z=(Z_n)$ from \eqref{4eb2.21}. One can easily check that if $Z_0=\tilde{Z}_0$, then for any $\phi$ and $n\geq 0$ we have
\[
Z_n(\phi) \text{ is equal to } \tilde{Z}_n(\phi) \text{ in distribution. }
\]
We slightly abuse the notation and use $\P^{\tilde{Z}_0}$ to denote the law of $\tilde{Z}=(\tilde{Z}_n)$ as in \eqref{4ed7.31}. In particular, we write $\P^x$ for the case when $\tilde{Z}_0=\delta_x$.
The two labelling systems have their own uses: $Z=(Z_n)$ is tailor-made to couple BRW with SIR epidemic as in Lemma \ref{4l1.5}; $\tilde{Z}=(\tilde{Z}_n)$ is more suitable for calculating its moments, which we will give below.

For any $\beta \in \tilde{I}$, if $S$ is a subset of $\tilde{I}$ so that all the indices in $S$ have length $|\beta|$, we define
\begin{align}\label{4eb2.13}
\sigma(S,\beta)=
\begin{cases}
|\beta|-\inf\{j: \beta|j\neq \gamma|j \text{ for all } \gamma \in S\} &\text{ if } \beta \notin S;\\
-1 &\text{ if } \beta \in S.\\
\end{cases}
\end{align}
In this way, $\sigma(S,\beta)$ denotes the number of generations back that $\beta$ first split off from the family tree generated by $S$. Set 
\begin{align}\label{4eb2.18}
\cF(S)=\sigma\{B^{{\gamma}|k}: {\gamma}\in S, 1\leq k\leq |{\gamma}|\} \vee  \sigma\{W^{{\gamma}|k}: {\gamma}\in S, 1\leq k\leq |{\gamma}|\}
\end{align}
to be the $\sigma$-field containing the information of the family tree generated by $S$.

Recall $S_n$ from \eqref{4eb2.2}.  For convenience we let $S_k=0$ if $k\leq 0$.  For any $n\geq 1$ and any Borel function $\phi$, we define 
\begin{align}\label{4eb2.15}
 G(\phi,n)= 3\|\phi\|_\infty+\sum_{k=1}^{n} \sup_{y\in \Z_R^d}  \E(\phi(y+S_k))=3\|\phi\|_\infty +\sum_{k=1}^{n} \sup_{y\in \Z_R^d}\sum_{z\in \Z^d_R} \phi(y+z) p_k(z).
\end{align}
The following lemma is proved in a similar way to that of Lemma 2.4 in \cite{Per88}.

\begin{lemma}\label{4la.1}
For any $n,m\geq 1$ and $\phi\geq 0$, we let $S\subseteq \tilde{I}$ be a set of $m$ indices of length $n$. Then for any $x\in \Z_R^d$ we have
 \[
 \E^x\Big(\sum_{\substack{|\beta| = n \\ \sigma(S,\beta)\leq n-1}} \phi(\tilde{Y}^{\beta}) \Big|\cF(S)\Big)
 \leq m e^{\frac{n\theta}{R^{d-1}}}G(\phi,n).
 \]
 \end{lemma}
 \begin{proof}
 Fix $x\in \Z^d_R$ and $n\geq 1$. We label the ancestor particle at $x$ by $1$ and only consider $\beta$ with $\beta_0=1$ below. 
Let $|\beta|= n$ and assume $\sigma(S,\beta)=i$ for some $i\in \{-1, 0,\cdots, n-1\}$. Then by \eqref{4eb2.13} we get
\[
\{\beta|k:|\beta|-i\leq k\leq |\beta| \} \cap \{{\gamma}|k: {\gamma} \in S, k\leq |\beta|\}=\emptyset.
\]
Hence $\sigma\{B^{\beta|k}: |\beta|-i\leq k\leq |\beta|\}\vee \sigma\{W^{\beta|k}: |\beta|-i+1\leq k\leq |\beta|\}$ is independent of $\cF(S)$. Let $i^+=i\vee 0$. Since $\beta|(n-i^+-1)\in S$, we have $\tilde{Y}^{\beta|(n-i^+-1)}$ is $\cF(S)$-measurable and so
\begin{align*}
 &\E^x(\phi(\tilde{Y}^{\beta}) |\cF(S))=1(\tilde{Y}^{\beta|(n-i^+-1)} \neq \Delta) \times \E^x\Big(1(B^{\beta|k}=1, k=|\beta|-i, \cdots, |\beta|)\\
 &\times \phi\Big(\tilde{Y}^{\beta|(n-i^+-1)}+W^{\beta|(n-i^+)}+\sum_{k=n-i+1}^{n} W^{\beta|k}\Big)\Big| \cF(S)\Big)\\
 \leq & \sup_{e\in \cN(0)} \E^x \Big(\phi\Big(\tilde{Y}^{\beta|(n-i^+-1)}+e+\sum_{k=n-i+1}^{n} W^{\beta|k}\Big)\Big) \cdot \E^x\Big(\prod_{k=|\beta|-i}^{|\beta|}1(B^{\beta|k}=1)\Big)\\
 \leq & \sup_{y\in \Z_R^d}\E (\phi(y+S_{i}))\cdot p(R)^{i+1},
\end{align*}
where the first inequality follows by conditioning on $W^{\beta|(n-i^+)}=e$ for $e\in \cN(0)$ and then using that $B^{\beta|k}, W^{\beta|k}$ are independent of $\cF(S)$ and finally taking sup over $e\in \cN(0)$. The last inequality follows if one notices that $\{W^{\beta|k}\}$ are i.i.d. random variables uniform on $\cN(0)$ and $\{B^{\beta|k}\}$ are i.i.d. Bernoulli. 
 Notice $\{\beta: \sigma(S,\beta)=i\} \subseteq \cup_{{\gamma} \in S} \{\beta: \sigma({\gamma},\beta)=i\}$. For each $\gamma \in S$, we have the number of particles $\beta$ satisfying $\beta_0=1$, $|\beta|=n$ and $\sigma({\gamma},\beta)=i$ is at most $V(R)^{i+1}$ and so it follows that
\begin{align*}
 \E^x\Big(\sum_{\substack{|\beta| = n \\ \sigma(S,\beta)=i}} \phi(\tilde{Y}^{\beta}) \Big|\cF(S)\Big) \leq&  \sup_{y\in \Z_R^d}\E (\phi(y+S_{i}))   p(R)^{i+1} \cdot mV(R)^{i+1}\\
 \leq&m(1+\frac{\theta}{R^{d-1}})^n \sup_{y\in \Z_R^d}\E (\phi(y+S_{i})).
\end{align*}
Sum $i$ over $-1\leq i\leq n-1$ to get
\begin{align*}
 &\E^x\Big(\sum_{\substack{|\beta|= n \\ \sigma(S,\beta)\leq n-1}} \phi(\tilde{Y}_n^{\beta}) \Big|\cF(S)\Big)\\
 &\leq m(1+\frac{\theta}{R^{d-1}})^n \Big(3\|\phi\|_\infty+\sum_{i=1}^n \sup_{y\in \Z_R^d}\E (\phi(y+S_{i}))\Big)\leq  m e^{\frac{n\theta}{R^{d-1}}} G(\phi,n),
\end{align*}
as required.
\end{proof}

\begin{proposition}\label{4cp1.2}
For any $x\in \Z^d_R$, $\phi\geq 0$ and $n\geq 1$, we have 
\begin{align*}
\E^{x}(\tilde{Z}_{n}(\phi))=(1+\frac{\theta}{R^{d-1}})^n \E(\phi(S_n+x)).
\end{align*}
For any $p\geq 2$, 
\begin{align*}
&\E^{x}(\tilde{Z}_{n}(\phi)^p)\leq (p-1)! e^{\frac{n\theta(p-1)}{R^{d-1}}} G(\phi,n)^{p-1}\E^{x}(\tilde{Z}_{n}(\phi)).
\end{align*}
\end{proposition}

\begin{proof}
Fix $x\in \Z^d_R$ and $n\geq 1$. We label the particle at $x$ by $1$ and only consider $\beta$ with $\beta_0=1$ below. For any $\phi: \Z^d_R \to \R$, we have
\begin{align}\label{4e6.9}
\E^{x}(\tilde{Z}_{n}(\phi))=&\E^x\Big(\sum_{|\beta|= n} \phi(\tilde{Y}^\beta)\Big)=\sum_{|\beta|= n} \E^x(\phi(\tilde{Y}^\beta)| \beta\approx n)\P^x( \beta\approx n)\nn\\
=& \sum_{|\beta|= n} \E(\phi(x+S_n))   p(R)^n=(1+\frac{\theta}{R^{d-1}})^n \E(\phi(S_n+x)).
\end{align}

\no Turning to $p\geq 2$, we have
\begin{align*}
I:=&\E^{x}(\tilde{Z}_{n}(\phi)^p)=\E^x\Big(\sum_{|\beta^1|=  n}\cdots \sum_{|\beta^p| = n} \prod_{i=1}^p \phi(\tilde{Y}^{\beta^i})\Big)\\
=&\E^x\Big(\sum_{|\beta^1| = n}\cdots \sum_{|\beta^{p-1}| = n} \prod_{i=1}^{p-1} \phi(\tilde{Y}^{\beta^i}) \E^x\Big(\sum_{|\beta^{p}| = n} \phi(\tilde{Y}^{\beta^i}) \Big|\cF(S)\Big)\Big),
\end{align*}
where $S=\{\beta^1,\cdots \beta^{p-1}\}$ is a set of $p-1$ indices of length $n$. Since all $\beta^j$ have a common ancestor $x$, we have $\sigma(S,\beta^p)\leq n-1$. Hence
\begin{align*}
I=&\E^x\Big(\sum_{|\beta^1| = n}\cdots \sum_{|\beta^{p-1}| = n} \prod_{i=1}^{p-1} \phi(\tilde{Y}^{\beta^i}) \times \E^x\Big(\sum_{\substack{|\beta^{p}| = n \\ \sigma(S,\beta^p)\leq n-1}} \phi(\tilde{Y}^{\beta^i}) \Big|\cF(S)\Big)\Big)\\
\leq&\E^x\Big(\sum_{|\beta^1| = n}\cdots \sum_{|\beta^{p-1}| = n} \prod_{i=1}^{p-1} \phi(\tilde{Y}^{\beta^i}) \times (p-1)e^{\frac{n\theta}{R^{d-1}}}G(\phi,n)\Big)\\
= &\E^{x}(\tilde{Z}_{n}(\phi)^{p-1}) (p-1)e^{\frac{n\theta}{R^{d-1}}}G(\phi,n),
\end{align*}
where the inequality is by Lemma \ref{4la.1}.
Use induction to conclude
\begin{align*}
&\E^{x}(\tilde{Z}_{n}(\phi)^p)\leq (p-1)! e^{\frac{n\theta(p-1)}{R^{d-1}}} G(\phi,n)^{p-1}\E^{x}(\tilde{Z}_{n}(\phi)),
\end{align*}
as required.
\end{proof}

\begin{corollary}\label{ac1.2}
For any $\tilde{Z}_0\in M_F(\Z^d_R)$, $\phi\geq 0$, $\lambda>0$ and $n\geq 1$, if $\lambda e^{\frac{n\theta}{R^{d-1}}} G(\phi,n)<1$ is satisfied, we have
\begin{align*}
&\E^{\tilde{Z}_0}(e^{\lambda \tilde{Z}_{n}(\phi)})\leq \exp\Big(\lambda \E^{\tilde{Z}_0}(\tilde{Z}_{n}(\phi)) (1-\lambda e^{\frac{n\theta}{R^{d-1}}} G(\phi,n))^{-1}\Big).
\end{align*}
\end{corollary}
\begin{proof}
Write $\tilde{Z}_0=\sum_{i=1}^{|\tilde{Z}_0|} \delta_{x_i}$ for some $x_i\in \Z_R^d$. We first consider $\P^x$ for $x=x_i$ with $1\leq i\leq |\tilde{Z}_0|$.
For any $\phi\geq 0$, $\lambda>0$ and $n\geq 1$ such that $\lambda e^{\frac{n\theta}{R^{d-1}}} G(\phi,n)<1$,  we may apply Proposition \ref{4cp1.2} to get
\begin{align*}
\E^{x}(e^{\lambda \tilde{Z}_{n}(\phi)})=&1+\sum_{p=1}^\infty \frac{1}{p!} \lambda^p \E^{x}(\tilde{Z}_{n}(\phi)^p)\\
\leq &1+\sum_{p=1}^\infty \frac{1}{p} \lambda^p e^{\frac{n\theta(p-1)}{R^{d-1}}} G(\phi,n)^{p-1}\E^{x}(\tilde{Z}_{n}(\phi))\\
\leq &1+\lambda \E^{x}(\tilde{Z}_{n}(\phi)) (1-\lambda e^{\frac{n\theta}{R^{d-1}}} G(\phi,n))^{-1}\\
\leq &\exp\Big(\lambda \E^{x}(\tilde{Z}_{n}(\phi)) (1-\lambda e^{\frac{n\theta}{R^{d-1}}} G(\phi,n))^{-1}\Big).
\end{align*}
Returning to $\P^{\tilde{Z}_0}$, we use the above to arrive at
\begin{align*}
\E^{\tilde{Z}_0}(e^{\lambda \tilde{Z}_{n}(\phi)})=\prod_{i=1}^{|\tilde{Z}_0|}\E^{x_i}(e^{\lambda \tilde{Z}_{n}(\phi)})\leq &\prod_{i=1}^{|\tilde{Z}_0|}\exp\Big(\lambda \E^{x_i}(\tilde{Z}_{n}(\phi)) (1-\lambda e^{\frac{n\theta}{R^{d-1}}} G(\phi,n))^{-1}\Big)\\
= & \exp\Big(\lambda \E^{\tilde{Z}_0}(\tilde{Z}_{n}(\phi)) (1-\lambda e^{\frac{n\theta}{R^{d-1}}} G(\phi,n))^{-1}\Big),
\end{align*}
as required.
\end{proof}

\subsection{Exponential moment for occupation measure}\label{4ap1.2}
By using similar arguments with the above, we will prove Proposition \ref{4p1.4} in this section.
For any $n\geq 1$ and  $\phi\geq 0$, we define 
\begin{align}\label{4eb2.16}
 F(\phi,n)= 3\|\phi\|_\infty+ \sup_{y\in \Z_R^d}\sum_{k=1}^{n}  \E(\phi(y+S_k))=3\|\phi\|_\infty + \sup_{y\in \Z_R^d} \sum_{k=1}^{n} \sum_{z\in \Z^d_R} \phi(y+z) p_k(z).
\end{align}
Recall $G(\phi, n)$ from \eqref{4eb2.15}. It is immediate that $F(\phi, n)\leq G(\phi, n)$ and so Proposition \ref{4p1.4} will be an easy consequence of the following proposition. In fact, there is almost no difference between $F(\phi, n)$ and $G(\phi, n)$ for our application in this paper but we feel it may require this stronger result in some cases.
\begin{proposition}\label{4ap3.1}
For any $\tilde{Z}_0\in M_F(\Z^d_R)$, $\phi\geq 0$, $\lambda>0$, $n\geq 1$, if $2\lambda n e^{\frac{n\theta}{R^{d-1}}} F(\phi,n)<1$ is satisfied, we have
\begin{align}\label{4e100a}
&\E^{\tilde{Z}_0}\Big(\exp\Big({\lambda\sum_{k=0}^n \tilde{Z}_{k}(\phi)}\Big)\Big)\leq \exp\Big(\lambda |\tilde{Z}_0| e^{\frac{n\theta}{R^{d-1}}}  F(\phi,n)  (1-2\lambda  n e^{\frac{n\theta}{R^{d-1}}} F(\phi,n))^{-1}\Big).
\end{align}
\end{proposition}
\begin{proof}
Write $\tilde{Z}_0=\sum_{i=1}^{|\tilde{Z}_0|} \delta_{x_i}$ for some $x_i\in \Z_R^d$. Again we first consider $\P^x$ for $x=x_i$ with $1\leq i\leq |\tilde{Z}_0|$.
For any $\phi\geq 0$, $\lambda>0$ and $n\geq 1$ such that $2\lambda n e^{\frac{n\theta}{R^{d-1}}} F(\phi,n)<1$,  by Proposition \ref{4cp1.2} we have
\begin{align}\label{4eb1.41}
\E^x\Big(\sum_{k=0}^n \tilde{Z}_k(\phi)\Big)=&\sum_{k=0}^n (1+\frac{\theta}{R^{d-1}})^k \E(\phi(x+S_k))\nn\\
\leq& e^{n\theta/R^{d-1} }\sum_{k=0}^n\E(\phi(x+S_k))\leq e^{n\theta/R^{d-1} } F(\phi,n).
\end{align}
Next we will calculate the following $p$-th moment for any $p\geq 2$:
\begin{align}
&\E^x\Big(\Big(\sum_{k=0}^n \tilde{Z}_k(\phi)\Big)^p\Big)=\E^x\Big(\Big(\sum_{ |\beta|\leq n} \phi(\tilde{Y}^{\beta})\Big)^p\Big)=\E^x\Big(\sum_{|\beta^1|\leq n} \cdots \sum_{|\beta^p|\leq n} \prod_{i=1}^p \phi(\tilde{Y}^{\beta^i})\Big).
\end{align}
Let $S=\{\beta^1,\cdots \beta^{p-1}\}$ and recall the $\sigma$-field $\cF(S)$ from \eqref{4eb2.18} so that $\tilde{Y}^{\beta^i} \in \cF(S)$ for all $1\leq i\leq p-1$. Then it follows the above that
\begin{align}\label{4e8.20}
&\E^x\Big(\Big(\sum_{k=0}^n \tilde{Z}_k(\phi)\Big)^p\Big)
=&\E^x\Big(\sum_{|\beta^1|\leq n} \cdots \sum_{|\beta^{p-1}|\leq n} \prod_{i=1}^{p-1} \phi(\tilde{Y}^{\beta^i}) \E^x\Big(\sum_{|\beta^p|\leq n}   \phi(\tilde{Y}^{\beta^p}) \Big|\cF(S)\Big)\Big).
\end{align}
For any $\beta^p$ with $|\beta^p|\leq n$, we let $\alpha \in \tilde{I}$ denote the position where $\beta^p$ first split off from the family tree generated by $S=\{\beta^1,\cdots \beta^{p-1}\}$ so that $\alpha=\beta^i|j$ for some $1\leq i\leq p-1$ and $0\leq j\leq n$, that is, $\beta^p=\alpha$ or $\beta^p=\alpha \vee \tilde{\beta}^p$ for some $0\leq |\tilde{\beta}^p|\leq n-1$.  Here we use $\gamma \vee \delta=(\gamma_0, \cdots, \gamma_m, \delta_0, \cdots, \delta_l)$ to denote the concatenation in $\tilde{I}$.  One can see that $\tilde{Y}^{\alpha}\in \cF(S)$ and there are at most $(p-1) \cdot (n+1)$ such $\alpha$.  Now we have
\begin{align*}
\E^x\Big(\sum_{|\beta^p|\leq n}   \phi(\tilde{Y}^{\beta^p}) \Big|\cF(S)\Big)\leq  \E^x\Big( \sum_{\alpha} \phi(\tilde{Y}^{\alpha}) \Big|\cF(S)\Big)+ \E^x\Big(\sum_{\alpha} \sum_{0\leq |\tilde{\beta}^p|\leq n-1}   \phi(\tilde{Y}^{\alpha\vee \tilde{\beta}^p}) \Big|\cF(S)\Big).
\end{align*}
The first term can be simply bounded by $(p-1) (n+1) \|\phi\|_\infty$. For the second term, we have
\begin{align*}
I:=& \E^x\Big(\sum_{\alpha} \sum_{0\leq |\tilde{\beta}^p|\leq n-1}   \phi(\tilde{Y}^{\alpha\vee \tilde{\beta}^p}) \Big|\cF(S)\Big)\\
\leq&  \sum_{\alpha} \sum_{k=0}^{n-1} \sum_{|\tilde{\beta}^p|=k} \E^x\Big( \phi\Big(\tilde{Y}^{\alpha}+W^{\alpha\vee (\tilde{\beta}^p|0)}+\sum_{j=1}^{k} W^{\alpha\vee (\tilde{\beta}^p|j)}\Big) \Big|\cF(S)\Big) \E^x\Big(\prod_{j=0}^{k} 1(B^{\alpha\vee (\tilde{\beta}^p|j)}=1)\Big)\\
\leq &\sum_{\alpha}  \sum_{e\in \cN(0)}  \E^x\Big(1_{\{W^{\alpha\vee (\tilde{\beta}^p|0)}=e\}} \Big|\cF(S)\Big)  \sum_{k=0}^{n-1} \sum_{|\tilde{\beta}^p|=k} \E^{x}\Big( \phi\Big(\tilde{Y}^{\alpha}+e+\sum_{j=1}^{k} W^{\alpha\vee (\tilde{\beta}^p|j)}\Big)\Big) p(R)^{k+1},
\end{align*}
where the last inequality follows by indicating on the event $W^{\alpha\vee (\tilde{\beta}^p|0)}=e$ for any $e\in \cN(0)$ and by noticing that $\{W^{\alpha\vee (\tilde{\beta}^p|j)}, j\geq 1\}$ are independent of $\cF(S)$. Now take sup over $y=\tilde{Y}^{\alpha}+e$ in $\Z_R^d$ to get
\begin{align*}
I\leq& \sum_{\alpha} \sup_{y\in \Z_R^d} \sum_{k=0}^{n-1} \sum_{|\tilde{\beta}^p|=k}  \E^{x}\Big( \phi\Big(y+\sum_{j=1}^{k} W^{\alpha\vee (\tilde{\beta}^p|j)}\Big)\Big)p(R)^{k+1}\\
=&\sum_{\alpha} \sup_{y\in \Z_R^d} \sum_{k=0}^{n-1}  \E\Big(\phi(S_{k}+y)\Big) V(R)^{k+1}p(R)^{k+1}\\
\leq& (p-1)(n+1)e^{n\theta/R^{d-1}} \Big(\|\phi\|_\infty+ \sup_{y\in \Z_R^d} \sum_{k=1}^n \E(  \phi(y+S_{k} ))\Big).
\end{align*}
Now we conclude that
\begin{align*}
\E^x\Big(\sum_{|\beta^p|\leq n}   \phi(\tilde{Y}^{\beta^p}) \Big|\cF(S)\Big)\leq (p-1)(n+1)e^{n\theta/R^{d-1}} F(\phi,n).
\end{align*}
Returning to \eqref{4e8.20}, we use to above to arrive at
\begin{align*}
&\E^x\Big(\Big(\sum_{k=0}^n \tilde{Z}_k(\phi)\Big)^p\Big)\leq (p-1)(2n)e^{n\theta/R^{d-1}} F(\phi,n) \E^x\Big(\Big(\sum_{k=0}^n \tilde{Z}_k(\phi)\Big)^{p-1}\Big).
\end{align*}
By induction we get
\begin{align*}
\E^x\Big(\Big(\sum_{k=0}^n \tilde{Z}_k(\phi)\Big)^p\Big)\leq& (p-1)!(2n)^{p-1} e^{n(p-1)\theta/R^{d-1}} F(\phi,n)^{p-1} \E^x\Big(\sum_{k=0}^n \tilde{Z}_k(\phi)\Big)\\
\leq &(p-1)!(2n)^{p-1} e^{pn\theta/R^{d-1}} F(\phi,n)^{p},
\end{align*}
where the last inequality uses \eqref{4eb1.41}.
Hence it follows that
\begin{align*}
\E^{x}\Big(\exp\Big({\lambda\sum_{k=0}^n \tilde{Z}_{k}(\phi)}\Big)\Big)= &1+\sum_{p=1}^\infty \frac{1}{p!} \lambda^p \E^{x}\Big(\Big(\sum_{k=0}^n \tilde{Z}_{k}(\phi)\Big)^p\Big)\\
\leq &1+\sum_{p=1}^\infty \frac{1}{p} \lambda^p (2n)^{p-1} e^{\frac{pn\theta }{R^{d-1}}} F(\phi,n)^{p}\\
\leq &1+\lambda e^{\frac{n\theta }{R^{d-1}}}F(\phi,n) (1-2\lambda n e^{\frac{n\theta}{R^{d-1}}} F(\phi,n))^{-1}\\
\leq &\exp\Big(\lambda e^{\frac{n\theta }{R^{d-1}}} F(\phi,n)  (1-2\lambda  n e^{\frac{n\theta}{R^{d-1}}} F(\phi,n))^{-1}\Big).
\end{align*}
Returning to $\P^{\tilde{Z}_0}$, we use the above to arrive at
\begin{align*}
\E^{\tilde{Z}_0}\Big(\exp\Big({\lambda\sum_{k=0}^n \tilde{Z}_{k}(\phi)}\Big)\Big)&=\prod_{i=1}^{|\tilde{Z}_0|}\E^{x_i}\Big(\exp\Big({\lambda\sum_{k=0}^n \tilde{Z}_{k}(\phi)}\Big)\Big)\\
&\leq \exp\Big(\lambda |\tilde{Z}_0| e^{\frac{n\theta}{R^{d-1}}}  F(\phi,n)  (1-2\lambda  n e^{\frac{n\theta}{R^{d-1}}} F(\phi,n))^{-1}\Big),
\end{align*}
as required.
\end{proof}

\subsection{Exponential moments of the martingale term}\label{4ap1.3}
We give in this section the proof of Proposition \ref{4p5.1}.

\begin{proof}[Proof of Proposition \ref{4p5.1}]
Recall $M_N(\phi)$ from \eqref{4e1.22} and $\langle M(\phi) \rangle_{N}$ from \eqref{4eb1.51}.  Notice that $\langle M(\phi) \rangle_{N}=\langle M(-\phi) \rangle_{N}$ and
\begin{align*}
\E^{Z_0}( \exp(\lambda |M_{N}(\phi)|))\leq \E^{Z_0}( \exp(\lambda M_{N}(\phi)))+\E^{Z_0}( \exp(\lambda M_{N}(-\phi))).
\end{align*}
It suffices to show that
\begin{align}\label{4eb1.53}
\E^{Z_0}( e^{\lambda M_{N}(\phi)})\leq \Big(\E^{Z_0}\Big( e^{ 16\lambda^2 \langle M(\phi) \rangle_{N}}\Big)\Big)^{1/2}.
\end{align}
For each $n\geq 1$, we define 
\begin{align}
Y_n:=\lambda M_{n}(\phi)-\lambda M_{n-1}(\phi)=\sum_{|\alpha|=n-1} \sum_{i=1}^{V(R)}  \lambda\phi({Y^\alpha+e_i}) (B^{\alpha \vee e_i}-p(R)).
\end{align}
 Then $\sum_{n=1}^N Y_n=\lambda M_N(\phi)$ for each $N\geq 1$. By recalling $\cG_n=\sigma(\{B^\alpha: |\alpha|\leq n\})$, we have $Y_n\in \cG_n$. Further define for each $n\geq 1$ that
\begin{align}
V_n:=\E(Y_n^2|\cG_{n-1})=\sum_{|\alpha|=n-1} \sum_{i=1}^{V(R)}  \lambda^2 \phi({Y^\alpha+e_i})^2 p(R)(1-p(R)),
\end{align}
where in the last equality we have used the independence of $B^{\alpha \vee e_i}$.
It is immediate that $V_n\in \cG_{n-1}$ and $\sum_{n=1}^N V_n=\lambda^2\langle M(\phi) \rangle_{N}$ for each $N\geq 1$. Hence we may rewrite  \eqref{4eb1.53} as
\begin{align}\label{4ec1.66}
\E^{Z_0}\Big( e^{\sum_{n=1}^N Y_n}\Big)\leq \Big(\E^{Z_0}\Big( e^{ 16 \sum_{n=1}^N V_n}\Big)\Big)^{1/2}.
\end{align}

To prove the above inequality, we apply the Cauchy-Schwartz inequality to get
\begin{align*}
\E^{Z_0}\Big( e^{\sum_{n=1}^N Y_n}\Big)=&\E^{Z_0}\Big( e^{\sum_{n=1}^N Y_n-8\sum_{n=1}^N V_n} \cdot e^{8\sum_{n=1}^N V_n}\Big)\\
\leq &\Big(\E^{Z_0}\Big( e^{2\sum_{n=1}^N Y_n-16\sum_{n=1}^N V_n}\Big)\Big)^{1/2}\Big(\E^{Z_0}\Big(e^{16\sum_{n=1}^N V_n}\Big)\Big)^{1/2}.
\end{align*}
It suffices to prove
 \begin{align}\label{4ed2.2}
&\E^{Z_0}\Big( e^{2\sum_{n=1}^N Y_n-16\sum_{n=1}^N V_n}\Big)\leq 1.
\end{align}
Observe that
 \begin{align}\label{4ed1.53}
\E^{Z_0}(e^{2Y_n}|\cG_{n-1})=&\E^{Z_0}\Big( \exp\Big(\sum_{|\alpha|=n-1} \sum_{i=1}^{V(R)}  2\lambda {\phi}({Y^\alpha+e_i}) (B^{\alpha \vee e_i}-p(R))\Big)\Big|\cG_{n-1}\Big)\nn\\
=&\prod_{|\alpha|=n-1} \prod_{i=1}^{V(R)}\E^{Z_0}\Big( \exp\Big(2\lambda \phi({Y^\alpha+e_i}) (B^{\alpha \vee e_i}-p(R))\Big)\Big|\cG_{n-1}\Big).
\end{align}
Lemma 1.3(a) of Freedman \cite{Free75} gives that if a random variable $X$ satisfies $|X|\leq 1$, $\E(X)=0$ and $\E(X^2)=V$, then
 \begin{align}\label{4ed2.1}
\E(e^{2X})\leq e^{(e^2-3)V}\leq e^{16V}.
\end{align}
The constant $16$ above is in fact unimportant and we simply pick a large one.  Since $\lambda \|\phi\|_\infty\leq 1$ and $B^{\alpha \vee e_i}$ is a  Bernoulli random variable with mean $p(R)$, we have $X=\lambda \phi({Y^\alpha+e_i}) (B^{\alpha \vee e_i}-p(R))$ satisfies the assumption of Freedman's lemma. By \eqref{4ed2.1}, we get
\[
\E^{Z_0}\Big( \exp\Big(2\lambda \phi({Y^\alpha+e_i}) (B^{\alpha \vee e_i}-p(R))\Big)\Big|\cG_{n-1}\Big)\leq \exp\Big(16\lambda^2 \phi({Y^\alpha+e_i})^2 p(R)(1-p(R))\Big),
\]
Use the above to see that \eqref{4ed1.53} becomes
\begin{align*}
\E^{Z_0}(e^{2Y_n}|\cG_{n-1})\leq &\prod_{|\alpha|=n-1} \prod_{i=1}^{V(R)}\exp\Big(16\lambda^2 \phi({Y^\alpha+e_i})^2 p(R)(1-p(R))\Big)\\
=&\exp\Big(16\sum_{|\alpha|=n-1} \sum_{i=1}^{V(R)}\lambda^2 \phi({Y^\alpha+e_i})^2 p(R)(1-p(R))\Big)=e^{16V_n},
\end{align*}
thus giving
\begin{align*}
\E^{Z_0}(e^{2Y_n-16V_n}|\cG_{n-1})\leq 1, \forall n\geq 1.
\end{align*}
For each $N\geq 1$, we have
\begin{align*}
\E^{Z_0}\Big( e^{2\sum_{n=1}^N Y_n-16\sum_{n=1}^N V_n}\Big|\cG_{N-1}\Big)&=e^{2\sum_{n=1}^{N-1} Y_n-16\sum_{n=1}^{N-1} V_n}\E^{Z_0}\Big(e^{2Y_N-16V_N} \Big|\cG_{N-1}\Big)\\
&\leq e^{2\sum_{n=1}^{N-1} Y_n-16\sum_{n=1}^{N-1} V_n}.
\end{align*}
Use induction with above to get \eqref{4ed2.2} and so the proof is complete as noted above.
\end{proof}

\section{Proofs of Lemmas \ref{4l2.1}, \ref{4l4.2} and \ref{4l1.3}}\label{a2}

\subsection{Proof of Lemma \ref{4l2.1}}
We first consider $d=1$. For any $n\in \Z$, by interpolation we have
\begin{align}\label{4e8.13}
g(x)=(n+1-x)f(n)+(x-n)f(n+1), \text{ if } n\leq x\leq n+1.
\end{align}
Let $\mu_1=\mu/4$. For any $x\in \R$, if we let $n\in \Z$ so that $n\leq x<n+1$, then by \eqref{4e8.13} and the Cauchy-Schwartz inequality, we have
\begin{align}\label{4e8.14}
\E(e^{\mu_1 g(x)})=&\E(e^{\mu_1 (n+1-x)f(n)+\mu_1 (x-n)f(n+1)})\leq \Big(\E(e^{ 2\mu_1 (n+1-x)f(n)})\Big)^{1/2} \Big(\E(e^{2\mu_1(x-n)f(n+1)})\Big)^{1/2}\nn\\
\leq &\Big(\E(e^{ \mu f(n)})\Big)^{1/2} \Big(\E(e^{\mu f(n+1)})\Big)^{1/2}\leq C_1,
\end{align}
where the last inequality uses \eqref{4eb3.21}.
Next, for any $x<y$ in $\R$, we let $n\in \Z$ so that $n\leq x<n+1$. To prove the remaining inequality in \eqref{4e8.12}, we will proceed by three cases.\\

\no ${\bf Case\ 1.}$ If $n\leq x<y<n+1$, then by \eqref{4e8.13} we have
\begin{align*}
|g(x)-g(y)|
=|x-y||f(n+1)-f(n)|,
\end{align*}
Let $\lambda_1=\lambda/4$ to see that
\begin{align*}
&\E(e^{\lambda_1 \frac{|g(x)-g(y)|}{|x-y|^\eta}})=\E(e^{\lambda_1 |x-y|^{1-\eta} |f(n+1)-f(n)|})\leq\E(e^{\lambda |f(n+1)-f(n)|})\leq C_1,
\end{align*}
where we have used $|x-y|\leq 1$ in the first inequality and the last inequality is by \eqref{4eb3.21}.

\no ${\bf Case\ 2.}$ If $n+1\leq y<n+2$, then again by \eqref{4e8.13} we have 
\begin{align*}
g(x)-g(y)&=(n+1-x)f(n)+(x-n)f(n+1)\\
&\quad -\big((n+2-y)f(n+1)+(y-n-1)f(n+2)\big)\\
&=(n+1-x)[f(n)-f(n+1)]+(y-n-1)[f(n+1)-f(n+2)].
\end{align*}
Note $|x-y|=y-(n+1)+(n+1)-x \geq \max\{y-n-1, n+1-x\}$. So the above becomes
\begin{align}\label{4eb3.22}
&|g(x)-g(y)| \leq |x-y||f(n)-f(n+1)|+|x-y||f(n+1)-f(n+2)|.
\end{align}
 Apply \eqref{4eb3.22} and the Cauchy-Schwartz inequality to get
\begin{align*}
&\E(e^{\lambda_1 \frac{|g(x)-g(y)|}{|x-y|^\eta}})\leq \E(e^{\lambda_1 |x-y|^{1-\eta} |f(n)-f(n+1)|} e^{\lambda_1 |x-y|^{1-\eta} |f(n+1)-f(n+2)|})\\
\leq & \Big(\E(e^{ 2\lambda_1|x-y|^{1-\eta} |f(n)-f(n+1)|})\Big)^{1/2} \Big(\E(e^{2\lambda_1|x-y|^{1-\eta}|f(n+1)-f(n+2)|})\Big)^{1/2}\\
\leq & \Big(\E(e^{ \lambda |f(n)-f(n+1)|})\Big)^{1/2} \Big(\E(e^{\lambda |f(n+1)-f(n+2)|})\Big)^{1/2}\leq C_1.
\end{align*}
where we have used $|x-y|\leq 2$ in the second last inequality and the last inequality is by \eqref{4eb3.21}.

\no ${\bf Case\ 3.}$ If $n+m\leq y<n+m+1$ for some $m\geq 2$, then by \eqref{4e8.13} we have 
\begin{align*}
&g(x)-g(y)=(n+1-x)f(n)+(x-n)f(n+1)\\
&\quad -\Big((n+m+1-y)f(n+m)+(y-n-m)f(n+m+1)\Big)\\
&=(n+1-x)[f(n)-f(n+1)]+[f(n+1)-f(n+m)]\\
&\quad+(y-n-m)[f(n+m)-f(n+m+1)].
\end{align*}
It follows that
\begin{align*}
&|g(x)-g(y)| \leq |f(n)-f(n+1)|+|f(n+1)-f(n+m)|+|f(n+m)-f(n+m+1)|.
\end{align*}
Note in this case we have $|x-y|\geq m-1\geq 1$ and hence
\begin{align*}
&\E(e^{\lambda_1 \frac{|g(x)-g(y)|}{|x-y|^\eta}})\leq \E(e^{\lambda_1  (|f(n)-f(n+1)|+|f(n+m)-f(n+m+1)|)} e^{\lambda_1 \frac{|f(n+1)-f(n+m)|}{(m-1)^\eta}})\\
\leq & \Big(\E(e^{ 2\lambda_1(|f(n)-f(n+1)|+|f(n+m)-f(n+m+1)|)})\Big)^{1/2} \Big(\E(e^{2\lambda_1\frac{|f(n+1)-f(n+m)|}{(m-1)^\eta} })\Big)^{1/2}\\
\leq & \Big(\E(e^{ 4\lambda_1|f(n)-f(n+1)|})\Big)^{1/4}\Big(\E(e^{ 4\lambda_1|f(n+m)-f(n+m+1)|})\Big)^{1/4} \Big(\E(e^{2\lambda_1\frac{|f(n+1)-f(n+m)|}{(m-1)^\eta} })\Big)^{1/2}\leq C_1.
\end{align*}
Combine the above three cases and \eqref{4e8.14} to conclude \eqref{4e8.12} holds by letting $c_{\ref{4l2.1}}=1/4$ in $d=1$.\\

We continue to the case $d=2$. Fixing any $y_0\in \R$, we first show that
\begin{align}\label{4e8.16}
\begin{cases}
&\E\Big(\exp\Big(\frac{\lambda}{2} \frac{|g(n,y_0)-g(m,y_0)|}{|n-m|^\eta}\Big)\Big)\leq C_1, \quad \forall n\neq m \in \Z,\\
&\E\Big(\exp(\frac{\mu}{2} g(n,y_0))\Big)\leq C_1, \quad \forall n\in \Z.
\end{cases}
\end{align}
To see this, we let $k\in \Z$ so that $k\leq y_0<k+1$. By linear interpolation we have for any $n\in \Z$,
\begin{align}\label{4e8.17}
g(n,y_0)=(k+1-y_0)f(n,k)+(y_0-k) f(n,k+1).
\end{align}
Similar to derivation of \eqref{4e8.14}, we may use the above and \eqref{4eb3.21} to get
\begin{align*}
\E\Big(\exp\Big(\frac{\mu}{2} g(n,y_0)\Big)\Big)\leq C_1, \forall n\in \Z.
\end{align*}
Next, for any $n\neq m\in \Z$, we use \eqref{4e8.17} to see that
\begin{align*}
|g(n,y_0)-g(m,y_0)|=&\Big|(k+1-y_0)[f(n,k)-f(m,k)]+(y_0-k) [f(n,k+1)-f(m,k+1)]\Big|\\
\leq& |f(n,k)-f(m,k)|+|f(n,k+1)-f(m,k+1)|.
\end{align*}
It follows that
\begin{align*}
&\E\Big(\exp\Big(\frac{\lambda}{2} \frac{|g(n,y_0)-g(m,y_0)|}{|n-m|^\eta}\Big)\Big)\leq \E\Big(e^{\frac{\lambda}{2} \frac{|f(n,k)-f(m,k)|+|f(n,k+1)-f(m,k+1)|}{|n-m|^\eta}}\Big)\\
\leq &\Big(\E(e^{\lambda \frac{|f(n,k)-f(m,k)|}{|n-m|^\eta}})\Big)^{1/2} \Big(\E(e^{\lambda \frac{|f(n,k+1)-f(m,k+1)|}{|n-m|^\eta}})\Big)^{1/2}\leq C_1,
\end{align*}
where the last inequality is by \eqref{4eb3.21},
thus giving \eqref{4e8.16}.
By the case in $d=1$, we may apply \eqref{4e8.16} to see that there exists some constant $c_1=1/4>0$ such that if we let $\lambda_1=c_1\frac{\lambda}{2}$ and $\mu_1=c_1\frac{\mu}{2}$, then
\begin{align}\label{4e8.18}
\begin{cases}
&\E\Big(\exp\Big(\lambda_1 \frac{|g(x,y_0)-g(y,y_0)|}{|x-y|^\eta}\Big)\Big)\leq C_1, \forall x\neq y \in \R\\
&\E(\exp(\mu_1 g(x,y_0)))\leq C_1, \forall x\in \R.
\end{cases}
\end{align}
By symmetry we may repeat the above and show that for any $z_0\in \R$,
\begin{align}\label{4e8.19}
\begin{cases}
&\E\Big(\exp\Big(\lambda_1 \frac{|g(z_0,x)-g(z_0,y)|}{|x-y|^\eta}\Big)\Big)\leq C_1, \forall x\neq y \in \R\\
&\E(\exp(\mu_1 g(z_0,x)))\leq C_1, \forall x\in \R.
\end{cases}
\end{align}
The second inequality in \eqref{4e8.12} is now included in \eqref{4e8.18} and \eqref{4e8.19}. Let $\lambda_2=\frac{1}{2}\lambda_1=\frac{\lambda}{16}$. It suffices to show that
\begin{align}\label{4e8.21}
\E\Big(\exp\Big(\lambda_2 \frac{|g(x)-g(y)|}{|x-y|^\eta}\Big)\Big)\leq C_1, \forall  x\neq y \in \R^2.
 \end{align}
For any $x=(x_1, x_2)$ and $y=(y_1,y_2)$ in $\R^2$, we have
\begin{align}\label{4e10.60}
|g(x_1,x_2)-g(y_1,y_2)|\leq |g(x_1,x_2)-g(x_1,y_2)|+|g(x_1,y_2)-g(y_1,y_2)|.
\end{align}
Use the Cauchy-Schwartz inequality and $|x-y|\geq \max\{|x_1-y_1|, |x_2-y_2|\}$ to get
\begin{align*}
&\E\Big(e^{\lambda_2 \frac{|g(x)-g(y)|}{|x-y|^\eta}}\Big)\leq \Big(\E(e^{2\lambda_2 \frac{|g(x_1,x_2)-g(x_1,y_2)|}{|x_2-y_2|^\eta}})\Big)^{1/2} \Big(\E(e^{2\lambda_2 \frac{|g(x_1,y_2)-g(y_1,y_2)|}{|x_1-y_1|^\eta}})\Big)^{1/2}\leq C_1. 
\end{align*}
where the last inequality is by $\lambda_2=\lambda_1/2$ and  \eqref{4e8.18}, \eqref{4e8.19}, thus finishing the case $d=2$ by letting $c_{\ref{4l2.1}}=1/16$.

The case for $d\geq 3$ can be proved by induction in a similar way to that of the case $d=2$: we fix one coordinate and use linear interpolation and the $d-1$ case to prove equations like \eqref{4e8.18} and \eqref{4e8.19} hold.  Then use triangle inequality as in \eqref{4e10.60} to prove that \eqref{4e8.21} holds in $\R^d$, thus finishing the proof.

\subsection{Proofs of Lemma \ref{4l4.2} and Lemma \ref{4l1.3}} 
\begin{proof}[Proof of Lemma \ref{4l4.2}]
(i) For any $s,t>0$ and $x_1, x_2\in \Z_R^d$, we use translation invariance to get
\begin{align*}
&\sum_{y\in \Z^d_R} e^{-t|y-x_1|^2} e^{-s|y-x_2|^2}=\sum_{y\in \Z^d_R} e^{-t|y-(x_1-x_2)|^2}e^{-s|y|^2} \\
=&e^{-\frac{st}{s+t}|x_1-x_2|^2} \sum_{y\in \Z^d_R} e^{-(s+t)|y-\frac{t}{s+t}(x_1-x_2)|^2}:= e^{-\frac{st}{s+t}|x_1-x_2|^2} I.
\end{align*}
Let $k=yR\in \Z^d$ to see that
\begin{align*}
I=&\sum_{k\in \Z^d}  e^{-(s+t)|\frac{k}{R}-\frac{t}{s+t}(x_1-x_2)|^2}=\sum_{k\in \Z^d}  e^{-\frac{s+t}{R^2}|k-\frac{tR}{s+t}(x_1-x_2)|^2}.
\end{align*}
Write $a=\frac{s+t}{R^2}>0$ and $u=\frac{tR}{s+t}(x_1-x_2) \in \R^d$. Then the above becomes 
\begin{align}\label{4e9.1}
I=&\sum_{k_1\in \Z}\cdots \sum_{k_d\in \Z}  e^{-a\sum_{i=1}^d|k_i-u_i|^2}=\prod_{i=1}^d \sum_{k_i\in \Z} e^{-a|k_i-u_i|^2}.
\end{align}
For any $u_i\in \R$, if we let $\{u_i\}=u_i-[u_i] \in [0,1)$, then 
\begin{align*}
 \sum_{k_i\in \Z} e^{-a|k_i-u_i|^2}= \sum_{k_i\in \Z} e^{-a|k_i-\{u_i\}|^2} \leq \sum_{k_i\in \Z} (e^{-a|k_i|^2}+e^{-a|k_i-1|^2})=2\sum_{k_i\in \Z} e^{-a|k_i|^2}.
\end{align*}
Apply the above in \eqref{4e9.1} to get
\begin{align*}
I\leq&\prod_{i=1}^d 2\sum_{k_i\in \Z} e^{-a|k_i|^2}=2^d \sum_{k\in \Z^d} e^{-a|k|^2}=2^d \sum_{k\in \Z^d} e^{-\frac{s+t}{R^2}|k|^2}=2^d \sum_{y\in \Z^d_R} e^{-(s+t)|y|^2},
\end{align*}
thus completing the proof of (i).\\

(ii) For any $u\geq 1$, we let $s=1/(2u)<1$ and write $y=k/R$ for $k\in \Z^d$ to get
\begin{align*}
J:=&\sum_{y\in \Z^d_R} e^{-|y|^2/(2u)}=\sum_{y\in \Z^d_R} e^{-s|y|^2}\\
=&\sum_{k\in \Z^d} e^{-s|k|^2/R^2}=\Big(\sum_{k\in \Z} e^{-s|k|^2/R^2}\Big)^d=\Big(1+2\sum_{k=1}^\infty e^{-sk^2/R^2}\Big)^d.
\end{align*}
For any $k\geq 1$, we have 
\[
e^{-sk^2/R^2}\leq \int_{k-1}^k e^{-st^2/R^2} dt,
\]
and so
\[
\sum_{k=1}^\infty e^{-sk^2/R^2}\leq \int_0^\infty e^{-st^2/R^2} dt=\frac{1}{2} \sqrt{2\pi \frac{R^2}{2s}}.
\]
Therefore it follows that
\begin{align*}
J&\leq 2^d+ 2^d \Big(2\sum_{k=1}^\infty e^{-sk^2/R^2}\Big)^d \leq 2^d+ 4^d \Big(\frac{1}{2} \sqrt{2\pi \frac{R^2}{2s}}\Big)^d \leq C(d) R^d \frac{1}{s^{d/2}},
\end{align*}
where the last is by $s\leq 1$ and $R\geq 1$. The proof is complete by noting $s=1/(2u)$.
\end{proof}

\begin{proof}[Proof of Lemma \ref{4l1.3}]
Let $d=2$ or $d=3$ and $1<\alpha<(d+1)/2$. For any $n\geq 1$, $R\geq K_{\ref{4p1.1}}$, and $a,x\in \Z^d_R$, we use Proposition \ref{4p1.1}(ii) and Fubini's theorem to get
\begin{align}\label{4e9.21}
&\sum_{y\in \Z^d_R} p_n(y-x) \sum_{k=1}^\infty \frac{1}{k^{\alpha}} e^{-\frac{|y-a|^2}{64k}}\leq \sum_{y\in \Z^d_R} \frac{c_{\ref{4p1.1}}}{n^{d/2}R^d} e^{-\frac{|y-x|^2}{32n}} \sum_{k=1}^\infty \frac{1}{k^{\alpha}} e^{-\frac{|y-a|^2}{64k}}\nn \\
=&\frac{c_{\ref{4p1.1}}}{n^{d/2}R^d} \sum_{k=1}^\infty \frac{1}{k^{\alpha}}  \sum_{y\in \Z^d_R} e^{-\frac{|y-x|^2}{32n}} e^{-\frac{|y-a|^2}{64k}} \leq \frac{c_{\ref{4p1.1}}}{n^{d/2}R^d} \sum_{k=1}^\infty \frac{1}{k^{\alpha}} \cdot 2^d\sum_{y\in \Z^d_R} e^{-\frac{|y|^2}{32n}} e^{-\frac{|y|^2}{64k}} \nn \\
\leq&C(d)\frac{1}{n^{d/2}R^d} \sum_{y\in \Z^d_R} e^{-\frac{|y|^2}{32n}} \sum_{k=1}^\infty \frac{1}{k^{\alpha}} e^{-\frac{|y|^2}{64k}}:=C(d)\frac{1}{n^{d/2}R^d} \cdot I,
\end{align}
where the second inequality uses Lemma \ref{4l4.2}.
It suffices to bound $I$. For $|y|\leq 1$, we have
\begin{align}\label{4eb3.1}
\sum_{y\in \Z^d_R, |y|\leq 1} e^{-{\frac{|y|^2}{32n}}} \sum_{k=1}^\infty \frac{1}{k^{\alpha}} e^{-\frac{|y|^2}{64k}}\leq \sum_{y\in \Z^d_R, |y|\leq 1} C(\alpha) \leq C(\alpha)(2R+1)^d.
\end{align}
For $|y|\geq 1$, we use Lemma \ref{4l4.1} to see that
\begin{align}\label{4eb3.2}
&\sum_{y\in \Z^d_R, |y|\geq 1} e^{-\frac{|y|^2}{32n}}  \sum_{k=1}^\infty \frac{1}{k^{\alpha}} e^{-\frac{|y|^2}{64k}}\leq 64^{\alpha-1} C_{\ref{4l4.1}}(\alpha-1)  \sum_{y\in \Z^d_R, |y|\geq 1} e^{-\frac{|y|^2}{32n}}  \frac{1}{|y|^{2\alpha-2}}\nn\\
= &C(\alpha) \sum_{k=1}^\infty \sum_{\substack{y\in \Z^d_R \\ k\leq |y|<k+1}} e^{-\frac{|y|^2}{32n}}  \frac{1}{|y|^{2\alpha-2}}\leq C(\alpha)\sum_{k=1}^\infty\sum_{\substack{y\in \Z^d_R\\ k\leq |y|<k+1}} e^{-\frac{k^2}{32n}}  \frac{1}{k^{2\alpha-2}}\nn\\
\leq &C(\alpha)  \sum_{k=1}^\infty C(d) k^{d-1} R^d \cdot e^{-\frac{k^2}{32n}}  \frac{1}{k^{2\alpha-2}}\leq C(\alpha, d)  R^d \sum_{k=1}^\infty k^{d+1-2\alpha}e^{-\frac{k^2}{32n}}.
\end{align}
Since $\alpha<(d+1)/2$, one may get
\begin{align*}
 \sum_{k=1}^\infty k^{d+1-2\alpha}e^{-\frac{k^2}{32n}} \leq &   \sum_{k=1}^\infty \int_k^{k+1}s^{d+1-2\alpha}e^{-\frac{(s-1)^2}{32n}} ds=  \int_0^{\infty} (s+1)^{d+1-2\alpha} e^{-\frac{s^2}{32n}} ds\\
 =&  \int_0^{\infty} (\sqrt{t} \sqrt{n}+1)^{d+1-2\alpha} e^{-\frac{t}{32}}  \frac{\sqrt{n}}{2\sqrt{t}}dt \\
 \leq& (\sqrt{n})^{d+1-2\alpha}\sqrt{n} \int_0^{\infty}  \frac{(\sqrt{t}+1)^{d+1-2\alpha}}{2\sqrt{t}}e^{-\frac{t}{32}} dt  \leq Cn^{1+d/2-\alpha}.
\end{align*}
Returning to \eqref{4eb3.2}, we get
\begin{align}\label{4eb3.3}
&\sum_{y\in \Z^d_R, |y|\geq 1} e^{-\frac{|y|^2}{32n}}  \sum_{k=1}^\infty \frac{1}{k^{\alpha}} e^{-\frac{|y|^2}{64k}}\leq C(\alpha, d)  R^d Cn^{1+d/2-\alpha}.
\end{align}
Combine \eqref{4eb3.1} and \eqref{4eb3.3} to arrive at
\[
I\leq C(\alpha)(2R+1)^d+C(\alpha, d) R^d \cdot C n^{1+d/2-\alpha}\leq C(\alpha, d) R^d \cdot n^{1+d/2-\alpha}.
\]
The proof is complete by \eqref{4e9.21}.
\end{proof}

\section{Collision estimates for SIR epidemic} \label{4a0}

In this section, we give the proof of Lemma \ref{4l10.01}. Recall that in an SIR epidemic, when two (or more) infected individuals simultaneously attempt to infect the same susceptible individual, all but one of the attempts fail. We call such an occurrence a collision. For any $x\in \Z_R^d$, we let $\Gamma_n(x)$ denote the number of collisions at site $x$ and time $n$. For the susceptible individual at $x$, a collision occurs at $x$ if and only if there is some pair $u,v$ of infected individuals at neighboring sites that simultaneously attempt to infect
$x$. For example, if $k\geq 2$ infected individuals simultaneously attempt to infect $x$, then the number of collisions at $x$ is $\binom{k}{2}$. Therefore given that $|\eta_n \cap \cN(x)|=N_0$, the conditional expectation of $\Gamma_{n+1}(x)$ is given by
\begin{align}
\sum_{k=2}^{N_0} \binom{N_0}{k} p(R)^k (1-p(R))^{N_0-k} \binom{k}{2}\leq \frac{N_0(N_0-1)}{2}p(R)^2\leq |\eta_n \cap \cN(x)|^2 p(R)^2.
\end{align}
It follows that
\begin{align}\label{4eb2.33}
\E \Big(\sum_{n=1}^{T_\theta^R}  \sum_{x\in \Z^d_R}\Gamma_{n}(x)\Big)\leq& \E \Big(\sum_{n=0}^{T_\theta^R}  \sum_{x\in \Z^d_R}|\eta_n \cap \cN(x)|^2 p(R)^2\Big).
\end{align}
Use the dominating BRW $Z=(Z_n)$ to see that the right-hand side of \eqref{4eb2.33} is bounded by
\begin{align}\label{4e5.83}
 \E \Big(\sum_{n=0}^{T_\theta^R}  \sum_{x\in \Z^d_R}Z_n(\cN(x))^2\Big) p(R)^2\leq& \E \Big(\sum_{n=0}^{T_\theta^R}  \sum_{a\in \Z^d}Z_n(Q_3(a))^2 (2R+1)^d \Big)p(R)^2\nn\\
\leq& C\frac{1}{R^d} \E \Big(\sum_{n=0}^{T_\theta^R}  \sum_{a\in \Z^d}Z_n(Q_3(a))^2 \Big),
\end{align}
where the first inequality uses the fact that $\cN(x) \subseteq Q_3(a)$ holds for any $\|x-a\|_\infty \leq 1$ with $x\in \Z^d_R$ and $a\in \Z^d$. It suffices to show that
\begin{align}\label{4e5.91}
\frac{1}{R^{d}} \E \Big(\sum_{n=0}^{T_\theta^R}  \sum_{a\in \Z^d}Z_n(Q_3(a))^2 \Big)=o(R^{d-1}).
\end{align}
Recall that $Z_0(x)=1(x\in \eta_0)$ where $\eta_0$ is a subset of $\Z_R^d$ as in \eqref{4eb2.1}. Hence it is immediate that
\begin{align}\label{4e5.91a}
\frac{1}{R^{2d-1}} \sum_{a\in \Z^d}Z_0(Q_3(a))^2 &\leq \frac{1}{R^{2d-1}} \sum_{a\in \Z^d}(6^d K\beta_d(R))^2 1_{\{\|a\|_\infty \leq R_\theta+4\}}\nn\\
&\leq \frac{1}{R^{2d-1}} (6^d K\beta_d(R))^2 (2R_\theta+9)^d =o(1),
\end{align}
where the last follows by $\beta_d(R)\leq \log R$ and $R_\theta=\sqrt{R^{d-1}/\theta}$.

Next we consider $ \sum_{a\in \Z^d} \E (Z_n(Q_3(a))^2 )$ for any $1\leq n\leq T_\theta^R$. Recall that $\P^x$ denotes the law of the BRW starting from a single ancestor at $x\in \Z_R^d$. By \eqref{4e10.27} with $D=Q_3(a)$, we have
\begin{align}\label{4eb3.27}
\sum_{a\in \Z^d} \E (Z_n(Q_3(a))^2 )\leq &\sum_{a\in \Z^d} \sum_{x\in \eta_0} \E^x(Z_n(Q_3(a))^2)\nn\\
&+\sum_{a\in \Z^d} \Big(\sum_{x\in \eta_0} \E^x(Z_n(Q_3(a)))\Big)^2:=I_1+I_2.
\end{align}
We first deal with $I_1$. Recall $G(\phi,n)$ from \eqref{4e5.90}. Recall from \eqref{4e10.30} to see that
\begin{align}\label{4e5.85}
 G(1_{Q_3(a)},n)\leq C(d) h_d(n),
\end{align}
where $h_d(n)=\sum_{k=1}^n \frac{1}{n^{d/2}}$.
Although \eqref{4e10.30} deals with $Q(a)$, the conclusion still holds by adjusting the constants $C(d)$. 
Now apply Proposition \ref{4p1.2}(ii)  to see that
\begin{align}\label{4e5.84}
\E^x((Z_n(Q_3(a)))^2 )\leq& e^{\frac{n\theta}{R^{d-1}}} G(1_{Q_3(a)},n) \E^{x}({Z}_{n}(Q_3(a)))\nn\\
\leq& e^T C(d) h_d(n) \E^{x}({Z}_{n}(Q_3(a))),
\end{align}
where the last inequality uses $n\leq T_\theta^R$ and \eqref{4e5.85}.
Returning to $I_1$, we apply \eqref{4e5.84} to get
\begin{align*}
I_1\leq&\sum_{a\in \Z^d} \sum_{x\in \eta_0}  e^T C(d) h_d(n) \E^{x}({Z}_{n}(Q_3(a)))\\
\leq& C(d,T) h_d(n)\sum_{x\in \eta_0}  \sum_{a\in \Z^d} \E^{x}({Z}_{n}(Q_3(a)))\\
\leq& C(d,T) h_d(n)\sum_{x\in \eta_0}C(d)\E^{x}({Z}_{n}(1)).
\end{align*}
By \eqref{4ea4.5}, we have 
\begin{align}\label{4eb3.25}
\E^{x}({Z}_{n}(1))=(1+\frac{\theta}{R^{d-1}})^n\leq e^{\frac{n\theta}{R^{d-1}}} \leq e^T.
\end{align}
 It follows that
\begin{align}\label{4eb3.26}
I_1\leq  C(d,T) h_d(n)\sum_{x\in \eta_0}C(d) e^T\leq C(d,T) |\eta_0| h_d(n).
\end{align}

Turning to $I_2$, we observe that
\begin{align}\label{4e10.53}
I_2\leq &\Big(\sup_{a\in \Z^d}\sum_{x\in \eta_0} \E^x(Z_n(Q_3(a)))\Big)  \sum_{a\in \Z^d}  \sum_{x\in \eta_0} \E^x(Z_n(Q_3(a)))\nn\\
\leq &\Big(\sup_{a\in \Z^d} \sum_{x\in \eta_0} \E^x(Z_n(Q_3(a)))\Big)  \sum_{x\in \eta_0}  C(d)   \E^{x}(Z_n(1))\nn\\
\leq &\Big(\sup_{a\in \Z^d}\sum_{x\in \eta_0} \E^x(Z_n(Q_3(a)))\Big)  C(d) e^T |\eta_0|,
\end{align}
where the last inequality uses \eqref{4eb3.25}.
It remains to bound $\sup_{a\in \Z^d} \sum_{x\in \eta_0} \E^x(Z_n(Q_3(a)))$. For any $a\in \Z^d$, we apply Proposition \ref{4p1.2}(i) to see that
\begin{align}\label{4e10.54}
\sum_{x\in \eta_0} \E^x(Z_n(Q_3(a)))=&(1+\frac{\theta}{R^{d-1}})^n \sum_{x\in \eta_0} \P(S_n+x\in Q_3(a))\nn\\
\leq &e^{\frac{n\theta}{R^{d-1}}} \sum_{m\in \Z^d} |\eta_0 \cap Q(m)| \sup_{x\in Q(m)} \P(S_n+x\in Q_3(a))\nn\\
\leq &e^{T} K\beta_d(R) \sum_{m\in \Z^d}  \sup_{x\in Q(m)} \P(S_n+x\in Q_3(a)),
\end{align}
where in the last inequality we have used the condition (iii) from \eqref{4eb2.1}.
For any $x\in Q(m)$, we have $\P(S_n+x\in Q_3(a))\leq \P(S_n\in Q_5(a-m)),$ and so
\begin{align*}
\sum_{m\in \Z^d}  \sup_{x\in Q(m)} \P(S_n\in Q_3(a))\leq \sum_{m\in \Z^d}  \P(S_n\in Q_5(a-m)) \leq C(d).
\end{align*}
Use the above in \eqref{4e10.54} to arrive at
\begin{align}\label{4e10.55}
\sum_{x\in \eta_0} \E^x(Z_n(Q_3(a))) \leq &e^{T} K\beta_d(R)  C(d).
\end{align} 
Returning to \eqref{4e10.53}, we have
\begin{align}\label{4e10.56}
I_2\leq C(d)   e^T |\eta_0| \cdot e^{T} K\beta_d(R)  C(d)\leq C(d,T) |\eta_0| K\beta_d(R).
\end{align}
Finally combine \eqref{4eb3.26} and \eqref{4e10.56} to see that \eqref{4eb3.27} becomes
\begin{align*}
&\sum_{a\in \Z^d} \E (Z_n(Q_3(a))^2 )\leq C(d,T) |\eta_0| h_d(n)+C(d,T) |\eta_0| K\beta_d(R).
\end{align*}
Sum $n$ over $1\leq n\leq T_\theta^R$ to get
\begin{align*}
 &\E \Big(\sum_{n=1}^{T_\theta^R}  \sum_{a\in \Z^d}Z_n(Q_3(a))^2 \Big)\leq  C(d,T) |\eta_0| \sum_{n=1}^{T_\theta^R} h_d(n)+T_\theta^R C(d,T) |\eta_0| K\beta_d(R)\\
&\leq C(d,T) |\eta_0| \cdot T_\theta^R C(T)\log R+T_\theta^R C(d,T) |\eta_0| K \log R\\
&\leq C(d,T) \frac{2R^{d-1}f_d(\theta)}{\theta} \frac{TR^{d-1}}{\theta}  \log R+ \frac{TR^{d-1}}{\theta}  C(d,T) \frac{2R^{d-1}f_d(\theta)}{\theta} K \log R=o(R^{2d-1}),
\end{align*}
where the second inequality uses \eqref{4eb2.8} and \eqref{4ea10.45}. The proof of \eqref{4e5.91} is complete by \eqref{4e5.91a} and the above.

\end{document}